%% file: article.tex
\makeatletter
\@ifundefined{ifhs}{\let\ifhs\iffalse}{}
\@ifundefined{ifdebugmode}{\let\ifdebugmode\iffalse}{}
\makeatother
\documentclass[twoLevelNum,regularSubmission,english]{higherStructures-my}
\usepackage{macros}

\title{Unifying notions of pasting diagrams}
\author{Simon Forest}
\email{simon.forest@normalesup.org}

\begin{document}
\maketitle

\input{abstract}

\ifhs
\else
  \tableofcontents
\fi

\ifdebugmode
  \tableoftodos
\fi

\section*{Introduction}
\input{introduction}

\section{Formalisms of pasting diagrams}
\label{sec:pd}

\input{def}

\section[The free \texorpdfstring{$\omega$}{omega}-category of cells]{The free $\omega$-category of cells}
\label{sec:cat-cells}
\input{cell}



\section{Relating formalisms}
\label{sec:alt-rep}
\input{relating}

\clearpage
\appendix

\section{Details on the gluing theorem}
\label{sec:details-gluing-thm}
\input{details-gluing-thm}

\section{Details on the freeness property}
\label{sec:details-freeness}
\input{details-freeness}

\clearpage
\printbibliography

\end{document}

%% file: abstract.tex
\begin{abstract}
  \noindent In this work, we relate the three main formalisms for the notion of
  pasting diagram in strict $\omega$\categories: Street's \emph{parity
    complexes}, Johnson's \emph{pasting schemes} and Steiner's \emph{augmented
    directed complexes}. In the process, we show that the axioms of parity
  complexes and pasting schemes are not strong enough for them to correctly
  represent pasting diagrams, and we do so by providing a counter-example. Then,
  we introduce a new formalism, called \emph{torsion-free complexes}, which aims
  at encompassing the three other ones. We prove its correctness by providing a
  detailed proof that an instance induces a free \ocat. Next, we prove that the
  three other formalisms can be embedded in some sense in the new one. Finally,
  we show that there are no other embedding between these four formalisms.
\end{abstract}


%% file: introduction.tex
\hspace*{\fill} \textit{From an original idea of S.\,Mimram.}

\paragraph*{Pasting diagrams}
Central to the theory of strict $\omega$-categories is the notion of pasting
diagram, which gives a simple representation for formal composites of cells of
strict \ocats. Indeed, the standard representation, as equivalence classes of
expressions under the axioms of \ocats, can be difficult to handle in practice,
since the equivalence relation induced by the axioms is hard to describe.
Instead, a graphical representation of the cells involved in the composite is
often sufficient to designate a cell.
For instance, consider the two formal composites
\begin{gather*}
  a \comp_0 (\alpha \comp_{1} \beta) \comp_{0} ( (\gamma \comp_{0} h) \comp_{1} (\delta \comp_{0} h))
  \shortintertext{and}
  (a \comp_{0} \alpha \comp_{0} e \comp_{0} h) \comp_{1} (a \comp_{0} c \comp_{0} \gamma \comp_{0}
  h) \comp_{1} (a \comp_{0} \beta \comp_{0} \delta \comp_{0} h ).
\end{gather*}
Under the axioms of \ocats, it can be checked, though it is not immediate, that
they represent the same cell. However, both are formal composites of the
elements of the following diagram
\begin{equation}
  \label{eq:2-pd}
  \begin{tikzcd}[ampersand replacement=\&]
    u \arrow[rr,"a"] \& \& v \arrow[rr,"c"{description},""{auto=false,name=c}]
    \arrow[rr,out=70,in=110,"b",""{auto=false,name=b}] 
    \arrow[rr,out=-70,in=-110,"d"',""{auto=false,name=p}] 
    \arrow[phantom,"\Downarrow \alpha",from=b,to=c]
    \arrow[phantom,"\Downarrow \beta",from=c,to=p] 
    \& \& w \arrow[rr,"f"{description},""{auto=false,name=f}] \arrow[rr,out=70,in=110,"e",""{auto=false,name=e}]
    \arrow[rr,out=-70,in=-110,"g"',""{auto=false,name=g}] \& \& x \arrow[rr,"h",""{auto=false,name=h}] \& \& y
    \arrow[phantom,"\Downarrow \gamma",from=e,to=f]
    \arrow[phantom,"\Downarrow \delta",from=f,to=g]
  \end{tikzcd}
\end{equation}
More generally, all formal composites involving all the generators of this diagram are
equal and the data of the diagram enables referring to the cell obtained by
composing $u$, $v, \ldots, y$, $a$, $b, \ldots, h$, $\alpha$, $\beta$, $\gamma$,
$\delta$ together unambiguously without giving an explicit composite for them.
We call \emph{pasting diagrams} the diagrams satisfying this property. It can be
observed that this pasting diagram is made of smaller pasting diagrams like
\[
  \begin{tikzcd}[ampersand replacement=\&]
    v \arrow[rr,"c"{description},""{auto=false,name=c}]
    \arrow[rr,out=70,in=110,"b",""{auto=false,name=b}] 
    \arrow[rr,out=-70,in=-110,"d"',""{auto=false,name=p}] 
    \arrow[phantom,"\Downarrow \alpha",from=b,to=c]
    \arrow[phantom,"\Downarrow \beta",from=c,to=p] 
    \& \& w
  \end{tikzcd}
  \quad\qtand\quad
  \begin{tikzcd}[ampersand replacement=\&]
    w \arrow[rr,"f"',""{auto=false,name=f}] \arrow[rr,out=70,in=110,"e",""{auto=false,name=e}]
     \& \& x \arrow[rr,"h",""{auto=false,name=h}] \& \& y
    \arrow[phantom,"\Downarrow \gamma",from=e,to=f]
  \end{tikzcd}
  .
\]
Moreover, the two can be composed along $w$ by taking the union of the pasting
diagrams. Thus, given a set of generators and a specification of sources and
targets for them satisfying sufficient properties, one can obtain an \ocat of
pasting diagrams on such a set, which is actually free on the generators. This
fact justifies the use of pasting diagrams as an adequate replacement to formal
composites to designate particular cells.

\paragraph*{Pasting diagrams in $1$-categories}
The simplest instances of pasting diagrams are the ones of dimension~$1$: in this case,
they are of the form
\begin{equation}
  \label{eq:1-pd}
  \begin{tikzcd}
    x_0 \arrow[r,"a_1"] & x_1\ar[r,"a_2"]&x_2\ar[r,"a_3"]&\cdots\ar[r,"a_n"]&x_n    
  \end{tikzcd}
\end{equation}
and admit $a_n\circ\cdots\circ a_1$ as composite. On the contrary,
diagrams such as
\begin{equation}
  \label{eq:1-non-pd}
  \begin{tikzcd}
    y&\ar[l,"a"']x\ar[r,"b"]&z
  \end{tikzcd}
  \qquad\qquad\text{or}\qquad\qquad
  \begin{tikzcd}
    x\ar[loop right,"a"]
  \end{tikzcd}
\end{equation}
are not expected to be pasting diagrams: in the first one, the two arrows are
not even composable, and the second one is ambiguous in the sense that it might
denote $a$, or $a\circ a$, etc. Note that the diagram~\eqref{eq:1-pd} can
be freely obtained as the composite of generating diagrams of the form
\[
  \begin{tikzcd}
    x_i\ar[r,"a_i"]&x_{i+1}
  \end{tikzcd}
\]
(composition amounts here to identify the target object of a diagram with the
source of the second), whereas this is not the case for the diagrams
of~\eqref{eq:1-non-pd}. Pasting diagrams of the form~\eqref{eq:1-pd} can
also be characterized as finite graphs which are connected, acyclic and
non-branching, in the sense that no two arrows have the same source or the same
target.

\paragraph*{Pasting diagrams in higher dimensions}%

In order to extend the definition of pasting diagrams to higher dimensions, we
first need to extend the one of graph: an \emph{\ohg} is the data of sets of
hyperedges, or \emph{generators}, for each dimension, where each generator of
dimension~$i{+}1$ has specified source and target sets of edges of
dimension~$i$. Then, the definition of higher-dimensional pasting diagrams can
be sketched as follows: an $(n{+}1)$-pasting diagram is an \ohg with edges up to
dimension~$n{+}1$ whose source and target are valid $n$-pasting diagrams, and
whose $(n{+}1)$-generator can be composed unambiguously in an
$(n{+}1)$\category. The conditions for which the composition is unambiguous cannot be formulated as easily as in dimension~$1$. Indeed, on the one hand, the
complexity of the definition of higher strict categories makes it difficult to
check whether a set of generators can be composed in at least one way, and if
two composites are formally equivalent. On the other hand, the sources for
non-composability or ambiguity are much more varied.
For example, the order in which we are supposed to compose the elements
of~\eqref{eq:2-pd} is ambiguous. Considering only the $2$-generators, the orders
of composition $\alpha,\beta,\gamma,\delta$ and $\alpha,\gamma,\delta,\beta$ are
both possible. However, it can be proved that all possible orders of composition
are equivalent by the axioms of strict $\omega$-categories, so this ambiguity is
not important. On the contrary, given the $2$-cells~$\alpha$ and~$\beta$
described by the diagrams
\begin{equation*}
  \label{eq:2-pd-ambi}
  \begin{tikzcd}[row sep={2.5em,between origins},column sep={3.5em,between origins},cramped]
    & x \arrow[rd,"b"]& \\
    w \arrow[ru,"a"] \arrow[rd,"a'"']& \Downarrow \alpha & y \\
    & x \arrow[ru,"b"']& 
  \end{tikzcd}
  \quad
  \qtand
  \quad
  \begin{tikzcd}[row sep={2.5em,between origins},column sep={3.5em,between origins},cramped]
    & y \arrow[rd,"c"]& \\
    x \arrow[ru,"b"] \arrow[rd,"b"']& \Downarrow \beta & z \\
    & y \arrow[ru,"c'"']& 
  \end{tikzcd},
\end{equation*}
$\alpha$ and $\beta$ can be composed together in two possible orders: $\alpha$
then $\beta$ or $\beta$ then $\alpha$, which can be represented as
\[
  \begin{tikzcd}[row sep={2.5em,between origins},column sep={3.5em,between origins}]
    & x \arrow[rd,"b"] \ar[dd,phantom,"\Downarrow \alpha"]& \\
    w \arrow[ru,"a"] \arrow[rd,"a'"']& & y \ar[rd,"c"]
    \ar[dd,phantom,"\Downarrow \beta"] & \\
    & x \arrow[ru,"b"'{description}] \ar[rd,"b"'] &  & z \\
    & & y \ar[ru,"c'"']
  \end{tikzcd}
  \quad
  \qtand
  \quad
  \begin{tikzcd}[row sep={2.5em,between origins},column sep={3.5em,between origins}]
    &   & y \arrow[rd,"c"] \ar[dd,phantom,"\Downarrow \beta"]& \\
    &   x \arrow[ru,"b"] \arrow[rd,"b"'{description}] \ar[dd,phantom,"\Downarrow \alpha"]& & z \\
    w \ar[ru,"a"] \ar[rd,"a'"'] &   & y \arrow[ru,"c'"']& \\
    & x \ar[ru,"b"']
  \end{tikzcd}
  .
\]
But here, these two composites are different. Even more subtle problems arise
starting from dimension $3$, justifying the use of sophisticated formalisms for
recognizing pasting diagrams.

\paragraph*{Pasting diagram formalisms}
Until now, different proposals of pasting diagram formalisms, each one giving a
set of conditions to recognize pasting diagrams, have been made. The main ones
are Johnson's \emph{pasting schemes}~\cite{johnson1989combinatorics}, Street's
\emph{parity complexes}~\cite{street1991parity,street1994parity} and Steiner's
\emph{augmented directed complexes}~\cite{steiner2004omega}.
Even though the ideas underlying the definitions of these formalisms are quite
similar, they differ on many points and comparing them precisely is uneasy, and
actually, to the best of our knowledge, no formal account of the differences was
ever made.

One of the main difference between these formalisms is the notion of sub-pasting
diagram, or \emph{cell}, which is used. The formalism of pasting schemes has the
most evident notion of cell: it is simply a set which gathers all the generators which appear in the pasting
diagram it represents, regardless of their dimension or source/target status. For example, the
diagram~\eqref{eq:2-pd} will be represented by the set
\[
  X = \set{u,v,w,x,y,a,b,c,d,e,f,g,h,\alpha,\beta,\gamma,\delta}.
\]
Parity complexes use cells defined as
tuples of sets of generators that are kept organized by dimension and by
source/target status. For instance, the pasting diagram~\eqref{eq:2-pd} is
represented by five sets
\[
  \arraycolsep=0pt
  \begin{array}{rclrcl}
    \multicolumn{6}{c}{X_2 = \set{\alpha,\beta,\gamma,\delta},} \\[3pt]
    X_{1,-}\; & \spaceeq &\;  \set{a,b,e,h}, &\hspace{3.5em} X_{1,+}\; &\spaceeq&\; \set{a,d,g,h}, \\[3pt]
    X_{0,-}\; &\spaceeq&\; \set{u}, & X_{0,+}\;&\spaceeq&\; \set{y}
  \end{array}
\]
where $X_{i,-}$ represents the $i$-source, $X_{i,+}$ the $i$-target, and $X_2$
the $2$\dimensional part of the pasting diagram. This notion of cell seems less
natural at first than the one of pasting schemes, since the translations from a
pasting diagram to a cell, and \viceversa, are not evident. The notion of cell
used by augmented directed complexes can be obtained by considering the abelian
groups induced by an \ohg. As a variant of the ones of parity complexes, cells
are now given by sums of generators for each dimension and source/target status.
For example, the pasting diagram~\eqref{eq:2-pd} will be represented by the five
elements
\begin{carray}{X_2 = \alpha + \beta + \gamma + \delta,}%
    X_{1,-}\;\;&=&\; a + b + e + h, &\hspace{5em} X_{1,+} \;\;&=&\; a + d + g + h, \\[3pt]
    X_{0,-} \;\;&=&\; u, & X_{0,+}\;\;&=&\; y.
\end{carray}%
\noindent Thus, one can use tools from group
theory and commutative algebra when manipulating augmented directed complexes,
which make them an interesting alternative to the two other set-based formalisms.

Another important point of divergence between the different formalisms is the
conditions, or \emph{axioms}, they require on diagrams in order for them to be
pasting diagrams. This naturally raises the question of the difference of
expressivity, \ie ability to recognize more or fewer
diagrams, between these formalisms. Since the axioms are quite sophisticated and rely on different definitions used by
each formalism and, in particular, the different notions of cells, these
comparisons cannot be done so easily.

\paragraph*{Outline and results}
In \Cref{sec:pd}, we recall the definitions of the main structures involved in
this article. We first introduce the definitions of globular sets
(\Cref{text:globular-sets}) and strict categories (\Cref{text:strcat-firstdef}),
and then recall the definitions of the three existing pasting diagram formalisms
that we consider: parity complexes (\Cref{ssec:street}), pasting schemes
(\Cref{ssec:johnson}) and augmented directed complexes (\Cref{ssec:steiner}). We
relate each definition to the unifying notion of \ohg (\Cref{ssec:hypergraph}):
a formalism is then a class of \ohg{}s (defined by axioms) together with a
notion of cell and operations on these cells. In
\Cref{ssec:street}\oldparref{par:ce-pc}, we discuss a counter-example to the
freeness property claimed in the respective articles of parity complexes and
pasting schemes, \ie that the diagrams they accept are pasting diagrams. It
involves the diagram made of
\begin{equation*}
  \begin{tikzcd}[ampersand replacement=\&]
    x \arrow[rr,"b"{name=b,description}]
    \arrow[rr,out=70,in=110,"a"{name=a}] 
    \arrow[rr,out=-70,in=-110,"c"'{name=p}] 
    \arrow[phantom,"\alpha\!\Downarrow\ \Downarrow\!\alpha'",pos=0.6,from=a,to=b]
    \arrow[phantom,"\beta\!\Downarrow\ \Downarrow\!\beta'",pos=0.4,from=b,to=p] 
    \& \& y \arrow[rr,"e"{name=e,description}]
    \arrow[rr,out=70,in=110,"d"{name=p}]
    \arrow[rr,out=-70,in=-110,"f"'{name=f}] \& \& z
    \arrow[phantom,"\gamma\!\Downarrow\ \Downarrow\!\gamma'",pos=0.6,from=p,to=e]
    \arrow[phantom,"\delta\!\Downarrow\ \Downarrow\!\delta'",pos=0.4,from=e,to=f]
  \end{tikzcd}
\end{equation*}
together with two 3-generators
\begin{align*}
  \begin{tikzcd}[ampersand replacement=\&]
    x \arrow[rr,"b"{name=b,description}]
    \arrow[rr,out=70,in=110,"a"{name=a}] 
    \arrow[phantom,"\Downarrow \alpha",pos=0.6,from=a,to=b]
    \& \& y \arrow[rr,"e"{name=e,description}]
    \arrow[rr,out=-70,in=-110,"f"'{name=f}] \& \& z
    \arrow[phantom,"\Downarrow \delta",pos=0.4,from=e,to=f]
  \end{tikzcd}
  &
  \quad\overset{A}{\Rrightarrow}\quad
  \begin{tikzcd}[ampersand replacement=\&]
    x \arrow[rr,"b"{name=b,description}]
    \arrow[rr,out=70,in=110,"a"{name=a}] 
    \arrow[phantom,"\Downarrow \alpha'",pos=0.6,from=a,to=b]
    \& \& y \arrow[rr,"e"{name=e,description}]
    \arrow[rr,out=-70,in=-110,"f"'{name=f}] \& \& z
    \arrow[phantom,"\Downarrow \delta'",pos=0.4,from=e,to=f]
  \end{tikzcd},
    \shortintertext{and}
  \begin{tikzcd}[ampersand replacement=\&]
    x \arrow[rr,"b"{name=b,description}]
    \arrow[rr,out=-70,in=-110,"c"'{name=p}] 
    \arrow[phantom,"\Downarrow \beta",pos=0.4,from=b,to=p] 
    \& \& y \arrow[rr,"e"{name=e,description}]
    \arrow[rr,out=70,in=110,"d"{name=q}]
    \& \& z
    \arrow[phantom,"\Downarrow \gamma",pos=0.6,from=q,to=e]
  \end{tikzcd}
  &
  \quad\overset{B}{\Rrightarrow}\quad
  \begin{tikzcd}[ampersand replacement=\&]
    x \arrow[rr,"b"{name=b}]
    \arrow[rr,out=-70,in=-110,"c"'{name=p,description}] 
    \arrow[phantom,"\Downarrow \beta'",pos=0.4,from=b,to=p] 
    \& \& y \arrow[rr,"e"{name=e,description}]
    \arrow[rr,out=70,in=110,"d"{name=q}]
    \& \& z
    \arrow[phantom,"\Downarrow \gamma'",pos=0.6,from=q,to=e]
  \end{tikzcd}
    .
\end{align*}
This shortcoming motivated the introduction of a new
formalism, called \emph{torsion-free complexes}, whose axioms aim at correcting
and generalizing the ones of parity complexes and pasting schemes
(\Cref{ssec:gen-pc}).

In \Cref{sec:cat-cells}, we show the correctness of torsion-free complexes as a
pasting diagram formalism, \ie that the set of cells associated with a
torsion-free complex has a canonical structure of a free \ocat. For this
purpose, we state in \Cref{ssec:gluing-thm} the correctness of a \eq{gluing}
operation (\Cref{thm:cell-gluing}), as an adapted version of an existing result
for parity complexes~\cite[Lemma 3.2]{street1991parity}. This operation allows
constructing new cells by gluing higher-dimensional generators on existing
cells. In \Cref{ssec:cells-are-cat}, we prove that the cells of a torsion-free
complex admit a structure of an \ocat (\Cref{thm:cell-is-ocat}). Then, in
\Cref{ssec:freeness-def}, in order to show that the \ocat is free, we first
introduce the notion of freeness that we use by recalling the definition of
polygraphs~\cite{street1976limits,burroni1993higher}, which describe sets of
generators of different dimensions from which one can generate a free \ocat.
Finally, in \Cref{ssec:freeness-def}, we state the freeness properties of the
\ocat of cells of a torsion-free complex
(\Cref{thm:ext-freeness,coro:cells-are-freely-gen}).

In \Cref{sec:alt-rep}, we relate the different pasting diagram formalisms that were
introduced. We first make the link between torsion-free complexes and the three
other ones. For this purpose, in \Cref{text:cl-max-cells}, we define other
notions of cells for torsion-free complexes, namely \emph{maximal-well-formed}
and \emph{closed-well-formed sets}. Closed-well-formed sets should be understood
as the equivalent of the notion of cell for pasting schemes in torsion-free
complexes. Maximal-well-formed sets are then a convenient intermediate for
proofs between the original notion of cell for parity complexes and
closed-well-formed sets. We show that the both new notions induce \ocats of
cells isomorphic to the original one
(\Cref{thm:ctoprinc-m-iso,thm:ctocl-m-iso}). We then prove the embedding results
into torsion-free complexes for the three other formalisms: in
\Cref{ssec:enc-street}, we show that parity complexes are torsion-free complexes
(\Cref{thm:pc-is-gpc}); in \Cref{ssec:enc-johnson}, we show that loop-free
pasting schemes are torsion-free complexes (\Cref{thm:ps-is-gpc}) and that both
formalisms induce isomorphic $\omega$-categories (\Cref{thm:cells-wfset-isom});
in \Cref{ssec:enc-steiner}, we show that loop-free unital augmented directed
complexes are torsion-free complexes (\Cref{thm:adc-is-gpc}) and that both
formalisms induce isomorphic $\omega$-categories (\Cref{thm:ctost-functor}).
Finally, in \Cref{ssec:incl-ce}, we give counter-examples to the other
embeddings between the formalisms.

\paragraph*{Applications and related works}
Pasting diagram formalisms are an effective description of cells of free
\ocats. In particular, they give a precise definition to the notion of
commutative diagram and can represent generic compositions. Moreover, they make
it possible to study higher categories by probing them through pasting diagrams.
For example, augmented directed complexes were used to give an effective
description of the Gray tensor product in~\cite{steiner2004omega}. In a related
manner, Kapranov and Voevodsky studied topological properties of pasting schemes
in~\cite{kapranov1991combinatorial} and used them in an attempt to give a
description of $\omega$-groupoids in~\cite{kapranov1991infty}, but their results
were shown paradoxical~\cite{simpson1998homotopy}.

Several other works studied pasting diagrams. In~\cite{buckley2015formal},
Buckley gives a mechanized Coq proof of the results of~\cite{street1991parity}
but stops at the excision theorem~\cite[Theorem 4.1]{street1991parity}. In
particular, the proof of the freeness claim~\cite[Theorem 4.2]{street1991parity}
was not formally verified, and could not be, since this claim does not hold in
general, as is shown in the present paper.
In~\cite{campbell2016higher}, Campbell isolates a common structure behind parity
complexes and pasting schemes, called \emph{parity structure},
and gives stronger axioms
than the ones of parity complexes and pasting schemes, taking an opposite path
from this work which seeks a more general formalism. In~\cite{nguyen2017parity},
Nguyen studies \emph{pre-polytopes with labeled structures} and shows that they
give a parity structure that satisfies a variant of Campbell's axioms that are
enough to obtain another correct notion of pasting diagrams.
In~\cite{henry2017non}, Henry defines a theoretical notion of pasting diagrams,
called polyplexes, to show that certain classes of polygraphs are presheaf
categories, and uses them to prove a variant of the Simpson's conjecture
in~\cite{henry2018regular}. However, his pasting diagrams can involve some
looping behaviors, and are then out of the scope of the formalisms studied in
the present work. Using similar ideas,
Hadzihasanovic~\cite{hadzihasanovic2018combinatorial} defines a class of pasting
diagrams, called \emph{regular polygraphs}, that is ``big enough'' to study
semi-strict categories and which is well-behaved for several constructions
(notably, their realizations as topological spaces are CW complexes).

\paragraph*{Acknowledgements}
I would like to deeply thank Samuel Mimram and Yves Guiraud for their
supervision, help and useful feedback during this work. I would also like to
thank Simon Henry, Ross Street and Léonard Guetta for the interesting exchanges
on the subject. Finally, I would like to thank École Normale Supérieure de Paris
for funding my PhD thesis.

\paragraph*{Notations} We write $\N$ for the set of natural numbers, $\N^*$
for~$\N\setminus\set 0$ and~$\omega$ for the first infinite ordinal. Given $n
\in \N$, we write~$\N_n$ for the set~$\set{0,\ldots,n}$ and~$\N^*_n$ for $\N_n
\setminus \set 0$. We use the convention that~$\N_\omega$ denotes~$\N$.

%% file: def.tex
In this section, we recall some basic definitions about strict categories and
introduce the definitions of the formalisms of pasting diagrams that we will
consider in this article. We present them through the common perspective of
\emph{\ohg{}s}, that are structures which encode the information in diagrams of
generators like~\eqref{eq:2-pd}. Then, the definition of each formalism roughly
follows the same pattern. First, a definition for cells that represent pasting
diagrams is introduced, together with an identity and composition operations
that aim at equipping those cells with a structure of \ocat. Then, a class of
\ohg{}s that are correctly handled by the considered formalism is defined by the
mean of axioms or conditions.

We first recall the definition of globular sets (\Cref{text:globular-sets}) and
of strict categories (\Cref{text:strcat-firstdef}) as
globular sets with additional operations. We then introduce \ohg{}s
(\Cref{ssec:hypergraph}) and recall the definitions of the three main existing
formalisms for pasting diagrams: \emph{parity complexes} (\Cref{ssec:street}),
\emph{pasting schemes} (\Cref{ssec:johnson}) and \emph{augmented directed
  complexes} (\Cref{ssec:steiner}). Then, we introduce the new formalism of
\emph{torsion-free complexes} that share the definitions of parity complexes but
have different axioms on \ohg{}s (\Cref{ssec:gen-pc}).

\subsection{Globular sets}
\label{text:globular-sets}

Given $n \in \Ninf$, an \emph{$n$\globular set}~$(X,\csrc,\ctgt)$ is the data of
sets~$X_k$ for~$k \in \N_n$, the elements of~$X_k$ being called
\emph{$k$-cells}, together with, for~$i \in \N_{n-1}$, functions
\[
  \csrc_i,\ctgt_i\co X_{i+1}\to X_i
\]
as in
\[
  \begin{tikzcd}
    X_0
    &
    \ar[l,shift right,"\csrc_0"']
    \ar[l,shift left,"\ctgt_0"]
    X_1
    &
    \ar[l,shift right,"\csrc_1"']
    \ar[l,shift left,"\ctgt_1"]
    X_2
    &
    \ar[l,shift right,"\csrc_2"']
    \ar[l,shift left,"\ctgt_2"]
    \cdots
    &
    \ar[l,shift right,"\csrc_{n-2}"']
    \ar[l,shift left,"\ctgt_{n-2}"]
    X_{n-1}
    &
    \ar[l,shift right,"\csrc_{n-1}"']
    \ar[l,shift left,"\ctgt_{n-1}"]
    X_{n}
  \end{tikzcd}
\]
satisfying the following globular identities for every~$i \in \N_{n-2}$:
\[
  \csrc_i\circ\csrc_{i+1}
  =
  \csrc_i\circ\ctgt_{i+1}
  \qquad\text{and}\qquad
  \ctgt_i\circ\csrc_{i+1}
  =
  \ctgt_i\circ\ctgt_{i+1}
  .
\]
Given~$k \in \N^*_n$, the elements of~$X_k$ are called the \emph{$k$\globes}
of~$X$. Given~$i,j\in\N$ with~$i\leq j$, by abusing notation, we write
$\csrc_i\colon X_j\to X_i$ for the function
\[
  \csrc_i=\csrc_i\circ\csrc_{i+1}\circ\cdots\circ\csrc_{j-1}
\]
and similarly for $\ctgt$. Given $i,k\in \N_n$ with $i \le k$, for~$u \in X_k$,
$\csrc_i(u)$ and~$\ctgt_i(u)$ are respectively the \emph{$i$\source} and \emph{$i$\target}
of~$u$. We write $X_k\times_i X_k$ for the pullback
\[
  \begin{tikzcd}[column sep={4em,between origins},row sep={2.5em,between origins}]
    &\ar[dl,dotted]X_k\times_i X_k\ar[dr,dotted]&\\
    X_k\ar[dr,"\ctgt_i"']&&\ar[dl,"\csrc_i"]X_k.\\
    &X_i&
  \end{tikzcd}
\]
Given $u,v \in X_k$, we say that $u$ and $v$ are \emph{$i$-composable}
when~$\ctgt_i (u) = \csrc_i (v)$. More generally, given $p \ge 0$ and
$u_1,\ldots,u_p \in X_k$, we say that $u_1,\ldots,u_p$ are \emph{$i$-composable}
when, for $j\in \N^*_{p-1}$, $u_j$ and $u_{j+1}$ are $i$-composable.

Given two $n$\globular sets~$X$ and~$Y$, a \emph{morphism} $F\co X \to Y$ between~$X$
and~$Y$ is the data of functions~$F_k \co X_k \to Y_k$ for~$k \in \N_n$ such
that $F_i \circ \csrctgt\eps_i = \csrctgt\eps_i \circ F_{i+1}$ for~$i \in
\N_{n-1}$ and~$\eps \in \set{-,+}$.

\subsection{Strict categories}
\label{text:strcat-firstdef}

Given~$n \in \N \cup \set{\omega}$, a \index{strict category}\emph{strict
  $n$\category}~$(C,\csrc,\ctgt,\unit{},\comp)$ (often simply denoted~$C$) is an
$n$\globular set~$(C,\csrc,\ctgt)$ together with, for~$k \in \N$ with~$k < n$,
identity operations
\[
  \unitp{k+1} {}\colon C_{k}\to C_{k+1}
\]
often written~$\unit {}$ when there is no ambiguity on~$k$, and, for~$i,k \in
\N_n$ with~$i < k$, composition operations
\[
  \comp_{i,k}\colon C_k\times_i C_k\to C_k 
\]%
often denoted~$\comp_i$ when there is no ambiguity on~$k$, which satisfy the
axioms~\ref{cat:first} to~\ref{cat:last} below. Given~$k,l \in \N_n$ such
that~${k \le l}$ and~${u \in C_k}$, we extend the notations for identity
operations and write~$\unitp{l}{}(u)$ for
\[
  \unitp{l}{}(u) = \unitp{l}{} \circ \cdots \circ \unitp{k+1}{}(u)
\]
and, for the sake of conciseness, we often write~$\unitp l u$ for~$\unitp l {}
(u)$, or even~$\unit u$ when~$l = k+1$. The axioms are the following:
\begin{enumerate}[label=(S-\roman*), ref=(S-\roman*)]
\item\label{cat:first} \label{cat:id-srctgt} for~$k \in \N_{n-1}$ and~$u \in
  C_k$,
  \[
  \csrc_k (\unitp{k+1}u) = \ctgt_k(\unitp{k+1}u)
  = u,
  \]
\item \label{cat:src-tgt}for~$i,k \in \N_n$ with~$i < k$,~$(u,v) \in C_k
  \times_i C_k$ and~$\eps \in \set{-,+}$,
  \[
    \csrctgt\eps_{k-1}(u\comp_i v) =
    \begin{cases}
      \csrctgt\eps_{k-1}(u) \comp_{i} \csrctgt\eps_{k-1} (v)&\text{if~$i < k-1$,} \\
      \csrc_{k-1}(u)&\text{if~$i=k-1$ and~$\epsilon=-$,}\\
      \ctgt_{k-1}(v)&\text{if~$i=k-1$ and~$\epsilon=+$,}
    \end{cases}
  \]
\item\label{cat:unital} for~$i,k \in \N_n$ such that~$i < k$, and~$u \in C_k$,
  \[
    \unitp{k}{} (\csrc_i (u)) \comp_{i} u = u = u \comp_{i} \unitp{k}{}
    (\ctgt_i (u)),
  \]
\item\label{cat:assoc} for~$i,k \in \N_n$ such that~$i < k$, and $i$\composable~$u,v,w \in C_k$,
  \[
    (u \comp_{i} v) \comp_{i} w = u \comp_{i} (v \comp_{i} w) ,
  \]
\item\label{cat:id-xch} for~$i,k \in \N_{n-1}$ such that~$i < k$, and~$(u,v) \in C_k
  \times_i C_k$,
  \[
    \unitp{k+1}{}(u \comp_{i} v) = \unitp{k+1}u \comp_{i} \unitp{k+1}v ,
  \]
\item\label{cat:xch}\label{cat:last} (\eq{exchange law}) for~$i,j,k \in \N_n$ such that~$i < j < k$,
  and~$u,u',v,v' \in C_k$ such that~$u,v$ are $i$\composable, and~${u,u'}$ are
  $j$\composable, and~$v,v'$ are $j$\composable,
  \[
    (u \comp_i v) \comp_j (u' \comp_i v') = (u \comp_j u') \comp_i (v \comp_j v').
  \]
\end{enumerate}
\noindent Given two strict $n$\categories~$C$ and~$D$, a
\index{morphism!of strict categories}\emph{morphism}~$F$ between~$C$ and~$D$ is
the data of an $n$\globular morphism~$F \co C \to D$ which moreover satisfies
that
\begin{itemize}
\item $F(\unitp {k+1} u) = \unitp {k+1} {F(u)}$ for every~$k \in \N_{n-1}$ and~$u \in
  C_k$,
\item $F(u \comp_i v) = F(u) \comp_i F(v)$ for every~$i,k \in \N_n$ with~$i < k$
  and $i$\composable~$u,v \in C_k$.
\end{itemize}
We often call such morphisms \index{functor@$n$-functor}\emph{$n$\functors}. We
write~$\nCat n$ for the category of strict $n$\categories. Given~$k,l\in \Ninf$ with~$k <
l$,
there is an evident truncation functor
\[
  \restrict k - \co \nCat l \to \nCat k
\]
which forgets the cells of dimension~$> k$ from an $l$\category.

\begin{rem}
  \label{rem:ocat-limit}
  Using the truncation functors, the category~$\oCat$ could equivalently be
  defined as the strict limit in~$\CAT$ on the diagram
  \[
    \begin{tikzcd}
      \nCat 0
      &
      \ar[l,"\restrict 0 -"']
      \nCat 1
      &
      \ar[l,"\restrict 1 -"']
      \nCat 2
      &
      \ar[l,"\restrict 2 -"']
      \cdots
      &
      \ar[l,"\restrict {k-1} -"']
      \nCat{k}
      &
      \ar[l,"\restrict {k} -"']
      \nCat{k+1}
      &
      \ar[l,"\restrict {k+1} -"']
      \cdots
    \end{tikzcd}
  \]
\end{rem}

\subsection{Hypergraphs}
\label{ssec:hypergraph}
\input{hypergraph}

\subsection{Parity complexes}
\label{ssec:street}
\input{street}

\subsection{Pasting schemes}
\label{ssec:johnson}
\input{johnson}

\subsection{Augmented directed complexes}
\label{ssec:steiner}
\input{steiner}

\subsection{Torsion-free complexes}
\label{ssec:gen-pc}
\input{forest}


%% file: hypergraph.tex
Here, we introduce the notion of \emph{\ohg}. It is essentially the same as the
one of \emph{parity structure} introduced by \citeauthor{campbell2016higher}
in~\cite{campbell2016higher} when defining a new formalism whose instances are
both parity complexes and pasting schemes. It is also similar to the notion of
\emph{oriented graded poset} that, in a related context,
\citeauthor{hadzihasanovic2018combinatorial} used to define presentations of
polygraphs~\cite{hadzihasanovic2018combinatorial}.

\paragraph{Definition}

A \index{graded set}\emph{graded set} is a set~$P$ equipped with a partition $ P
= \sqcup_{n\in\N}P_n $. An \index{hypergraph@$\omega$-hypergraph}\emph{\ohg} is
a graded set~$P$, the elements of $P_n$ being called \index{generator!of
  an $\omega$-hypergraph}\emph{$n$\generators}, together with, for~$n \in \N$
and~$u \in P_{n+1}$, two finite \glossary(.j){$(-)^-$,
  $(-)^+$}{the source and target operations of an \ohg}subsets~$u^-,u^+
\subseteq P_n$ called the \index{source!for an $\omega$-hypergraph}\emph{source}
and \index{target!for an $\omega$-hypergraph}\emph{target} of~$u$. Given a
subset~$U\subseteq P$ and~$\epsilon\in\set{-,+}$, we write~$U^\epsilon$ for
$\cup_{u\in U}u^\epsilon$.

Simple \ohg{}s can be represented graphically using \emph{diagrams}, where
$0$\generators are represented by their names, and higher generators
by arrows~$\to$,~$\To$,~$\Tto$, \etc that represent respectively
$1$\generators, $2$\generators, $3$\generators etc.
%
\begin{example}
  The diagram
  \begin{equation}
    \label{eq:ex-ppc}
    \begin{tikzcd}[ampersand replacement=\&,row sep=1ex,column sep=1.5ex,cramped]
      \& y \arrow[rd,"c"] \& \\
      x \arrow[ru,"a"] \arrow[rd,"b"'] \& \Downarrow \alpha \& z \\
      \& y' \arrow[ru,"d"'] \&   
    \end{tikzcd}
  \end{equation}
  represents the \ohg~$P$ with $P_0=\set{x,y,y',z}$, $P_1=\set{a,b,c,d}$,
  $P_2=\set{\alpha}$, and~$P_n=\emptyset$ for~$n\geq 3$, sources and targets
  being
    $a^-=\set{x}$,
      $a^+=\set{y}$,
      $\alpha^-=\set{a,c}$,
      $\alpha^+=\set{b,d}$,
  and so on.
\end{example}
%


\noindent
%
  %


\paragraph{Fork-freeness}
\parlabel{text:fork-freeness}

Given an \ohg~$P$ and~$n\in\N$, a subset~${U\subseteq P_n}$ is
\index{fork-free!for an \ohg}\emph{fork-free} (also called \emph{well-formed} in~\cite{street1991parity})
when:
\begin{itemize}
\item either~$n = 0$ and~$\setsize U = 1$,
\item or~$n > 0$ and for all~$u,v \in U$ and~$\eps \in \set{-,+}$, we
  have~$u^\epsilon\cap v^\epsilon=\emptyset$.
\end{itemize}
For example, the subset~$\set{a,b}$
of~\eqref{eq:ex-ppc} is not fork-free since~${a^-\cap b^-=\set{x}}$,
but~$\set{a,c}$ is.
\begin{remark}
  Note that the definition of fork-freeness depends on the intended
  dimension~$n$. This subtlety is important in the case of the empty
  set:~$\emptyset$ is not well-formed as a subset of~$P_0$ but it is as a subset
  of~$P_n$ when~$n > 0$.
\end{remark}

\paragraph[The relation \texorpdfstring{$\tl$}{triangle}]{The relation $\bm{\tl}$}
\parlabel{par:tl-rel}

Given an \ohg~$P$,~$n \in \N^*$ and~$U \subseteq P_n$, for~$u,v \in U$, we
write~$u \tlone_U v$ when~$u^+ \cap v^- \neq \emptyset$ and we define the
\glossary(.k){$\tl_U$, $\tl$}{a pre-order on the generators of an
  \ohg}relation~$\tl_U$ on~$U$ as the transitive closure of~$\tlone_U$. Given
subsets~$V,W \subseteq U$, we write~$V \tl_U W$ when there exist~$u \in V$
and~$v \in W$ such that~$u \tl_U v$.
We define the relation~$\tl$ on~$P$ by putting~$u \tl v$ when there exists~$n
\in \N^*$ such that~$u,v \in P_n$ and~$u \tl_{P_n} v$. The \ohg~$P$ is then said
\index{acyclic}\emph{acyclic} when~$\tl$ is irreflexive.
\begin{example}
  The \ohg represented by
  \begin{equation}
    \label{eq:not-acyclic}
    \begin{tikzcd}[ampersand replacement=\&,cramped]
      x\ar[r,bend left,"a"]\&\ar[l,bend left,"b"]y
    \end{tikzcd}
  \end{equation}
  is not acyclic since~$a\triangleleft b\triangleleft a$. On the contrary, the
  \ohg represented by~\eqref{eq:2-pd} is acyclic.
\end{example}

\smallpar Given a subset~${V \subseteq U}$, we say that~$V$ is a
\index{segment}\emph{segment for~$\tl_U$} when for all~$u_1,u_2,u_3 \in U$ such
that
\[
  {u_1,u_3 \in V}
  \quad\qtand\quad
  {u_1 \tl_U
    u_2 \tl_U u_3},
\]
it holds that~$u_2 \in V$.

\paragraph{Other source and target operations}
\parlabel{par:other-src-tgt-ops}

Given an \ohg~$P$, for~$n \ge 2$,~$u \in P_n$ and~${\eps,\eta \in \set{-,+}}$,
we write~$u^{\eps \eta}$ for~$(u^\eps)^\eta$. We extend the notation to
subsets~$U \subseteq P_n$ and write~$U^{\eps\eta}$ for~$(U^\eps)^\eta$.
Moreover, we \glossary(.ka){$(-)^\mp$, $(-)^\pm$}{the border operations for sets
  of generators of an \ohg}write~$u^\mp$ and~$u^\pm$ for
\[
  u^\mp = u^- \setminus u^+
  \quad\qtand\quad
  u^\pm = u^+ \setminus u^-.
\]
We also extend the notation to subsets~$U \subseteq P_n$ and write~$U^\mp$ and~$U^\pm$ for
\[
  U^\mp = U^- \setminus U^+
  \quad\qtand\quad
  U^\pm = U^+ \setminus U^-\pbox.
\]

\begin{example}
  Consider the \ohg represented by the diagram
  \begin{equation}
    \label{eq:ohg-cells}
    \begin{tikzcd}[row sep={2.5em,between origins},column sep={4em,between origins}]
      & &         & \mathmakebox[10pt]{w} \arrow[rd,"d"] \arrow[phantom,dd,"\Downarrow\beta"] &         &  & \\
      & & \mathmakebox[10pt]{v} \arrow[ru,"c"] \arrow[rd,"c''" description] \arrow[phantom,dd,"\Downarrow\alpha"] &         & \mathmakebox[10pt]{x} \arrow[rd,"e"] \arrow[phantom,dd,"\Downarrow\delta"]&  & \\
      \mathmakebox[10pt]{t} \arrow[r,"a"] & \mathmakebox[10pt]{u} \arrow[ru,"b"] \arrow[rd,"b'"']& & w' \arrow[ru,"d''" description]
      \arrow[rd,"d'''" description] \arrow[phantom,dd,"\Downarrow\gamma"] &         & \mathmakebox[10pt]{y} \arrow[r,"f"]& \mathmakebox[10pt]{z} \\
      & & v' \arrow[ru,"c'''" description] \arrow[rd,"c'"'] &         & x' \arrow[ru,"e'"'] &  & \\
      & &         & w'' \arrow[ru,"d'"']&         &  & 
    \end{tikzcd}
    \pbox.
  \end{equation}
  For this \ohg, we have
  \[
    \alpha^{--} = \set{u,v}, \qquad \alpha^{+-} = \set{u,v'}, \qquad
    \alpha^{-\mp} = \set{u}, \qquad \alpha^{+\pm} = \set{w'}
  \]
  and, writing~$U$ for the set~$\set{a,b,c,d,e,f}$,
  \begin{align*}
    U^- &= \set{t,u,v,w,x,y}\zbox, & U^+ &= \set{u,v,w,x,y,z}\zbox, \\
    U^\mp &= \set{t}\zbox, & U^\pm &= \set{z}
                                     \intertext{and, writing~$V$ for the set~$\set{\alpha,\beta,\gamma,\delta}$,}
                                     V^- &= \set{b,c,c'',c''',d,d'',d''',e}\zbox,
                           &
                             V^+ &= \set{b',c',c'',c''',d',d'',d''',e'}\zbox,
    \\
    V^\mp &= \set{b,c,d,e}\zbox,
                                   & V^\pm &= \set{b',c',d',e'}\zbox.
  \end{align*}%
  From the above examples, one can intuitively describe the operations~$(-)^-$
  and~$(-)^+$ as computing the ``inner'' sources and targets of a set of
  generators, whereas the operations~$(-)^\mp$ and~$(-)^\pm$ compute the source
  and target ``borders'' of a set of generators.
\end{example}


%% file: street.tex
In this section, we recall the formalism of parity complexes developed by
Street in~\cite{street1991parity}. Most of the content will be reused when
defining torsion-free complexes. The idea behind the formalism is to represent
an $(n{+}1)$\cell as a pair of source and target $n$\cells together with a
subset of~$P_{n+1}$ which ``moves'' the source $n$\cell to the target $n$\cell.

\paragraph{Pre-cells}
Let~$P$ be an \ohg. For~$n \in \N$, an \index{pre-cell!of an
  $\omega$-hypergraph}\emph{$n$\precell of~$P$} is a tuple
\[
  X = (X_{0,-},X_{0,+},\ldots,X_{n-1,-},X_{n-1,+},X_n)
\]
of finite subsets of~$P$, such that~$X_{i,\epsilon} \subseteq P_i$ for~$i \in
\N_{n-1}$ and~$\epsilon \in \set{-,+}$, and~$X_n \subseteq P_n$. By convention,
we often denote~$X_n$ by~$X_{n,-}$ or~$X_{n,+}$. We
\glossary(PCellP){$\PCell(P)$}{the graded set of pre-cells
  of~$P$}write~$\PCell(P)$ for the graded set of pre-cells of~$P$.
Given~$n \in \N$,~$\eps \in \set{-,+}$ and an $(n{+}1)$\precell~$X$
of~$P$, we define the $n$\precell~$\csrctgt\eps_n (X)$ as
\[
  \csrctgt\eps_n (X)=(X_{0,-},X_{0,+},\ldots,X_{n-1,-},X_{n-1,+},X_{n,\eps}).
\]
The functions~$\csrc,\ctgt$ then equip~$\PCell(P)$ with a structure of an
$\omega$\globular set.
\paragraph{Movement and orthogonality}
\parlabel{par:move} Let~$P$ be an \ohg. Given~$n\in\N$ and finite sets
$M\subseteq P_{n+1}$, $U\subseteq P_n$ and $V\subseteq P_n$, we say that~$M$
\index{movement}\emph{moves}~$U$ to~$V$ when
\[
  U = (V \cup M^{-}) \setminus M^{+} \qqtand V = (U \cup M^{+}) \setminus M^{-}
  \zbox.
\]
Intuitively, the first equation means that~$U$ is the subset obtained from~$V$
by replacing the target of~$M$ by its source, and the second equation has a dual
meaning.
\begin{example}
  In the \ohg~\eqref{eq:ohg-cells}, the set~$\set{\alpha,\beta,\gamma,\delta}$
  moves the set~$\set{a,b,c,d,e,f}$ to the set~$\set{a,b',c',d',e',f}$.
\end{example}

\paragraph{Cells}
\parlabel{par:cell-def}

Let~$P$ be an \ohg. Given~$n \in \N$, an \index{cell!of an \ohg}\emph{$n$\cell of~$P$} is an
$n$\precell of~$P$, such that
\begin{enumerate}[label=(\roman*),ref=(\roman*)]
\item \label{pc:cell:move} $X_{i+1,\epsilon}$ moves~$X_{i,-}$ to~$X_{i,+}$
  for~$i \in \N_{n-1}$ and~$\epsilon \in \set{-,+}$,
\item \label{pc:cell:fork-free}$X_{i,\epsilon}$ is fork-free for~$i \in \N_n$ and~$\epsilon\in\set{-,+}$.
\end{enumerate}
We denote \glossary(CellP){$\Cell(P)$}{the graded set of cells
  of the \ohg~$P$}by~$\Cell(P)$ the graded set of cells of~$P$, which inherits the
structure of globular set from~$\PCell(P)$. An $n$\cell~$X$ can be represented
as on \Cref{fig:cell-movements}
\begin{figure}
  \centering
  \[
    \begin{tikzpicture}[xscale=2,yscale=1.5]
      \node (X0s) at (0,0) {$X_{0,-}$};
      \node (X0t) at (0,-1) {$X_{0,+}$};
      \node (X1s) at ($(X0s) + (1,0)$) {$X_{1,-}$};
      \node (X1t) at ($(X0t) + (1,0)$) {$X_{1,+}$};
      \node (Xds) at ($(X1s) + (1,0)$) {$\cdots$};
      \node (Xdt) at ($(X1t) + (1,0)$) {$\cdots$};
      \node (Xn-2s) at ($(Xds) + (1,0)$) {$X_{n-2,-}$};
      \node (Xn-2t) at ($(Xdt) + (1,0)$) {$X_{n-2,+}$};
      \node (Xn-1s) at ($(Xn-2s) + (1,0)$) {$X_{n-1,-}$};
      \node (Xn-1t) at ($(Xn-2t) + (1,0)$) {$X_{n-1,+}$};
      \node (Xn) at ($(Xn-1s)!0.5!(Xn-1t) + (1,0)$) {$X_{n}$};
      \draw[->] (X0s) .. controls ($(X1s) + (-0.1,0)$) .. (X0t);
      \draw[->] (X0s) .. controls ($(X1t) + (-0.1,0)$) .. (X0t);
      \draw[->] (X1s) .. controls ($(Xds) + (-0.1,0)$) .. (X1t);
      \draw[->] (X1s) .. controls ($(Xdt) + (-0.1,0)$) .. (X1t);
      \draw[->] (Xds) .. controls ($(Xn-2s) + (-0.2,0)$) .. (Xdt);
      \draw[->] (Xds) .. controls ($(Xn-2t) + (-0.2,0)$) .. (Xdt);
      \draw[->] (Xn-2s) .. controls ($(Xn-1s) + (-0.2,0)$) .. (Xn-2t);
      \draw[->] (Xn-2s) .. controls ($(Xn-1t) + (-0.2,0)$) .. (Xn-2t);
      \draw[->] (Xn-1s) .. controls ($(Xn) + (-0.1,0)$) .. (Xn-1t);
    \end{tikzpicture}
  \]
  \caption{Movements in a cell}
  \label{fig:cell-movements}
\end{figure}
where each arrow
\[
  \begin{tikzcd}[cramped] U \arrow[r,bend left,"M"] & V
  \end{tikzcd}
\]
means that~$M$ moves~$U$ to~$V$.
\begin{example}
  The \ohg represented by~\eqref{eq:ohg-cells} has, among others,
  \begin{itemize}
  \item a $0$\cell~$(\set{t})$,
  \item a $1$\cell~$(\set{t},\set{w'},\set{a,b,c''},\set{a,b,c'''},\set{\alpha})$,
  \item a $2$\cell~$(\set{t},\set{z},\set{a,b,c,d,e,f},\set{a,b',c',d',e',f},\set{\alpha,\beta,\gamma,\delta})$, \etcend
  \end{itemize}
\end{example}

\paragraph{Identity and composition of operations}
\parlabel{text:pc-identity-composition}
Let~$P$ be an \ohg. Given~$n \in \N$ and an $n$\cell~$X$, the
\index{identity operation!for \ohg cells}\emph{identity of~$X$} is the $(n{+}1)$\cell
\[
  \unitp{n+1}{}(X) = (X_{0,-},X_{0,+},\ldots,X_{n-1,-},X_{n-1,+},X_n,X_n,\emptyset).
\]
Given~$i,n\in \N$ with~$i < n$, and $i$\composable $n$\cells~$X,Y \in
\Cell(P)_n$, the \index{composition operation!for \ohg cells}\emph{$i$-composition~$X \comp_i Y$ of~$X$ and~$Y$} is defined as
the $n$\precell~$Z$
such that, for~$j \in \N_n$ and~$\eps \in \set{-,+}$,
\[
  Z_{j,\epsilon}
  =
  \begin{cases}
    X_{j,\epsilon} & \text{if~$j < i$,} \\
    X_{i,-} & \text{if~$j = i$ and~$\epsilon = -$,} \\
    Y_{i,+} & \text{if~$j = i$ and~$\eps = +$,} \\
    X_{j,\eps} \cup Y_{j,\eps} & \text{if~$j > i$.}
  \end{cases}
\]
It will be shown in \Secr{cat-cells} that, under suitable assumptions, the
composite of two $n$\cells is actually an $n$\cell. 
\paragraph{Atoms and relevance}
%
Let~$P$ be an \ohg. Given~$n \in \N$ and~$u \in P_n$, we define sets~$\gen{u}_{i,\eps}
\subseteq P_i$ for~$i \in \N_n$ and~$\eps \in \set{-,+}$ with a downward
induction by
\begin{gather*}
  \gen{u}_{n,-} = \gen{u}_{n,+} = \set{u}
  \shortintertext{and}
  \gen{u}_{j,-} = \gen{u}_{j+1,-}^\mp \qquad \gen{u}_{j,+} = \gen{u}_{j+1,+}^\pm
\end{gather*}
for~$j \in \N_{n-1}$. We often write~$\gen{u}_n$ for both~$\gen{u}_{n,-}$
and~$\gen u_{n,+}$. The \glossary(.kb){$\gen u$}{the atom associated to the
  generator~$u$}\index{atom!of an \ohg}\emph{atom associated to~$u$} is then the
$n$\precell of~$P$
\[
  \gen{u}=(\gen{u}_{0,-},\gen{u}_{0,+},\ldots,\gen{u}_{n-1,-},\gen{u}_{n-1,+},\gen{u}_n) .
\]
A generator~$u$ is said \index{relevant}\emph{relevant} when the atom~$\gen{u}$ is a cell.
When~$P$ is a parity complex, the relevant generators of~$P$ will have the role
of generating cells in the $\omega$\category~$\Cell(P)$.\tobackpropagate{mauvais
  générateurs ici}%
\begin{example}
  The atom associated to~$\alpha$ in~\eqref{eq:ex-ppc} is~$\gen{\alpha}$ with
  \begin{align*}
    \gen{\alpha}_{0,-} &= \set{u}, & \gen{\alpha}_{1,-} &= \set{a,c}, &\gen{\alpha}_2 &= \set{\alpha}\zbox,\\
    \gen{\alpha}_{0,+} &= \set{z}, & \gen{\alpha}_{1,+} &= \set{b,d}\zbox,
  \end{align*}
  and, since it is a cell,~$\alpha$ is relevant.
\end{example}

\paragraph{Tightness}

Some defects were found in the first definition of parity complexes given
in~\cite{street1991parity}, so that Street fixed his definition
in~\cite{street1994parity}. His correction involves the notion of tightness
defined as follows. Given~$n \in \nats$, a subset~$U \subseteq P_n$ is said to
be \index{tight}\emph{tight} when, for all~$u,v \in P_n$ such that~$u \tl v$ and~$v\in U$, we
have~$u^-\cap U^\pm = \emptyset$.
\begin{example}
  In~\eqref{eq:ohg-cells},~$U = \set{\beta,\gamma}$ is not tight since~$\alpha
  \tl \gamma$ and~$c'' \in \alpha^- \cap U^\pm$. However, the set~$U' =
  \set{\alpha,\beta,\gamma,\delta}$ is tight.
\end{example}

\paragraph{Parity complexes}
\parlabel{par:pc} A \index{parity complex}\emph{parity
  complex}~\cite{street1991parity,street1994parity} is an \ohg~$P$ satisfying the
axioms \streetax{0} to \streetax{5} below:
\begin{enumerateaxioms}
  \item[{\streetaxiom[d]{0}}]  for~$n \in \N^*$ and~$u \in P_n$,~$u^- \neq \emptyset$ and~$u^+ \neq \emptyset$;
  \item[{\streetaxiom[d]{1}}] for~$n \in \N$ with~$n \ge 2$ and~$u \in P_n$,~$u^{--} \cup u^{++} = u^{-+} \cup u^{+-}$;
  \item[{\streetaxiom[d]{2}}]  for~$n \in \N^*$ and~$u \in P_n$,~$u^- \text{ and } u^+ \text{ are fork-free}$;
  \item[{\streetaxiom[d]{3}}]  $P$ is acyclic; 
  \item[{\streetaxiom[d]{4}}] for~$n \in \N^*$,~$u,v \in P_n$,~$w \in P_{n+1}$,
    if~$u \tl v$,~$u \in w^\epsilon$ and~$v \in w^\eta$ for some~$\eps,\eta \in \set{-,+}$, then~$\epsilon = \eta$;
  \item[{\streetaxiom[d]{5}}] for~$i,n \in \N$ with~$i < n$ and~$u \in P_n$,~$\gen{u}_{i,-}$ is tight.
\end{enumerateaxioms}
\streetaxiom{0} ensures that each generator has defined source and target.
\streetaxiom{1} enforces basic globular properties on generators. For example,
it forbids the \ohg
\begin{equation}
  \label{eq:johnson-non-glob}
  \begin{tikzcd}[ampersand replacement=\&,cramped]
    w \arrow["a", ""{name=x, below},rrr, bend left] \& x
    \arrow["b"',""{name=y, above}, r, bend right] \arrow["\Downarrow \alpha", from=x, to=y,phantom] \& y \& z
  \end{tikzcd}
\end{equation}
since~$\alpha^{--} \cup \alpha^{++} = \set{w,y}$ and~$\alpha^{+-} \cup \alpha^{-+} = \set{x,z}$.
\streetaxiom{2} forbids generators with parallel elements in their sources or targets.
For example, it forbids the \ohg
\begin{equation}
  \label{eq:ex-non-wfs}
  \begin{tikzcd}[cramped]
    \begin{array}{c}
    x \\ y
    \end{array}
    \arrow[r,"a"] & z
  \end{tikzcd}
\end{equation}
since~$a^-=\set{x,y}$ is not fork-free. \streetaxiom{3} forbids \ohg{}s with
some loops like~\eqref{eq:not-acyclic}. \streetaxiom{4} can be informally
described as forbidding ``bridges''. For instance, the \ohg
\begin{equation}
  \label{eq:not-3b}
\begin{tikzcd}[baseline=(current bounding box.center),sep=small,cramped]
  & y \arrow[to=Z,"b"] \ar[dd,"c",to path={-| ([xshift=0.5cm]Z.east) [pos=1]\tikztonodes |- 
    (\tikztotarget)}] &   \\
  x \arrow[ru,"a"] \arrow[rd,"a'"'] & \Downarrow \alpha
  & |[name=Z]|z \\
  & y' \arrow[to=Z,"b'"'] &
\end{tikzcd}
\end{equation}
does not satisfy \streetaxiom{4}. Indeed,~$a \tl c \tl b'$ and~$a \in \alpha^-$
and~$b' \in \alpha^+$. \streetaxiom{5} prevents more subtle problems, like the
one exposed by the \ohg~\eqref{eq:ex-non-segment}
(even though the latter does not satisfy~\streetaxiom{3} in the
first place). It entails that the sources and targets of each generator are
segments
(as defined in \Cref{ssec:hypergraph}), which is a condition that we will
motivate
in~\Cref{ssec:gen-pc} when discussing \forestaxiom{3} of torsion-free
complexes.

\paragraph{A counter-example to the freeness property}
\parlabel{par:ce-pc}
Given a parity complex~$P$, the main result claimed in~\cite{street1991parity}
is that the globular set~$\Cell(P)$ together with the source, target, identity
and composition operations, has the structure of an $\omega$\category, which is
freely generated by the atoms~$\gen{u}$ for~$u \in P$ (\cite[Theorem
4.2]{street1991parity}). More precisely, using the terminology
of~\Cref{ssec:freeness-def}, this result states that there is an \opol~$\Q$ and an
$\omega$\functor~$F \co \freecat\Q \to \Cell(P)$ such that
\begin{itemize}
\item $\Q_k = P_k$ for~$k \in \N$,
\item $F(u) = \gen{u}$ for~$u \in \Q$,
\item $F$ is an isomorphism.
\end{itemize}
Since the cells of a free \ocat on a polygraph are equivalent classes of formal
composites of generators (\cf \cite[Lemma~4.1]{metayer2008cofibrant}), this
property intuitively says that the cells of parity complexes adequately
represent pasting diagrams, \ie diagrams associated with a unique class of
equivalent formal composites of generators. However, this property does not
hold as we illustrate with a counter-example. Consider the \ohg~$P$ defined by
the diagram given by
\begin{equation}
  \label{eqn:ce-tf}
  \begin{tikzcd}[ampersand replacement=\&]
    x \arrow[rr,"b"{description},""{auto=false,name=b}]
    \arrow[rr,out=70,in=110,"a",""{auto=false,name=a}] 
    \arrow[rr,out=-70,in=-110,"c"',""{auto=false,name=p}] 
    \arrow[phantom,"\alpha\!\Downarrow\ \Downarrow\!\alpha'",from=a,to=b]
    \arrow[phantom,"\beta\!\Downarrow\ \Downarrow\!\beta'",from=b,to=p] 
    \& \& y \arrow[rr,"e"{description},""{auto=false,name=e}]
    \arrow[rr,out=70,in=110,"d",""{auto=false,name=p}]
    \arrow[rr,out=-70,in=-110,"f"',""{auto=false,name=f}] \& \& z
    \arrow[phantom,"\gamma\!\Downarrow\ \Downarrow\!\gamma'",from=p,to=e]
    \arrow[phantom,"\delta\!\Downarrow\ \Downarrow\!\delta'",from=e,to=f]
  \end{tikzcd}
\end{equation}
together with two 3-generators
\begin{align*}
  \begin{tikzcd}[ampersand replacement=\&]
    x \arrow[rr,"b"{description},""{auto=false,name=b}]
    \arrow[rr,out=70,in=110,"a",""{auto=false,name=a}] 
    \arrow[phantom,"\Downarrow \alpha",from=a,to=b]
    \& \& y \arrow[rr,"e"{description},""{auto=false,name=e}]
    \arrow[rr,out=-70,in=-110,"f"',""{auto=false,name=f}] \& \& z
    \arrow[phantom,"\Downarrow \delta",from=e,to=f]
  \end{tikzcd}
  &
  \quad\overset{A}{\Rrightarrow}\quad
  \begin{tikzcd}[ampersand replacement=\&]
    x \arrow[rr,"b"{description},""{auto=false,name=b}]
    \arrow[rr,out=70,in=110,"a",""{auto=false,name=a}] 
    \arrow[phantom,"\Downarrow \alpha'",from=a,to=b]
    \& \& y \arrow[rr,"e"{description},""{auto=false,name=e}]
    \arrow[rr,out=-70,in=-110,"f"',""{auto=false,name=f}] \& \& z
    \arrow[phantom,"\Downarrow \delta'",from=e,to=f]
  \end{tikzcd},
  \\
  \begin{tikzcd}[ampersand replacement=\&]
    x \arrow[rr,"b"{description},""{auto=false,name=b}]
    \arrow[rr,out=-70,in=-110,"c"',""{auto=false,name=p}] 
    \arrow[phantom,"\Downarrow \beta",from=b,to=p] 
    \& \& y \arrow[rr,"e"{description},""{auto=false,name=e}]
    \arrow[rr,out=70,in=110,"d",""{auto=false,name=q}]
    \& \& z
    \arrow[phantom,"\Downarrow \gamma",from=q,to=e]
  \end{tikzcd}
  &
  \quad\overset{B}{\Rrightarrow}\quad
  \begin{tikzcd}[ampersand replacement=\&]
    x \arrow[rr,"b"{description},""{auto=false,name=b}]
    \arrow[rr,out=-70,in=-110,"c"{description},""{auto=false,name=p}] 
    \arrow[phantom,"\Downarrow \beta'",from=b,to=p] 
    \& \& y \arrow[rr,"e"{description},""{auto=false,name=e}]
    \arrow[rr,out=70,in=110,"d",""{auto=false,name=q}]
    \& \& z
    \arrow[phantom,"\Downarrow \gamma'",from=q,to=e]
  \end{tikzcd}
    \pbox.
\end{align*}
By carefully checking Axioms~\streetax{0} to~\streetax{5}, it can be
shown that~$P$ is a parity complex. The diagram~\eqref{eqn:ce-tf} moreover
defines a polygraph~$\Q$, whose induced \ocat~$\freecat{\Q}$ is
supposed to be isomorphic to~$\Cell(P)$, as a consequence of~\cite[Theorem
4.2]{street1991parity}, but it is not the case here. Indeed, we can find two
expressions that compose together the $3$\generators~$A$ and~$B$ in~$\freecat\Q$,
inducing two $3$\cells~$\cecompn_1$ and~$\cecompn_2$ with
\begin{align*}
  \cecompn_1 &= ((a \comp_0 \gamma) \comp_1 A \comp_1 (\beta \comp_0 f)) \comp_2 ((\alpha'
  \comp_0 d) \comp_1 B \comp_1 (c \comp_0 \delta'))
               \\
  \cecompn_2 &= ((\alpha \comp_0 d) \comp_1 B \comp_1 (c \comp_0 \delta))\comp_2 ((a
  \comp_0 \gamma') \comp_1 A \comp_1 (\beta' \comp_0 f))
\end{align*}
where, for the sake of readability, we omitted the operations~$\unit{}$ required
to lift the generators to dimension~$3$. The canonical morphism~$F\co \freecat\Q
\to \Cell(P)$ maps~$\cecompn_1$ and~$\cecompn_2$ to the same $3$\cell~$X$
defined by:
\begin{carray}{X_3 = \set{A,B},}
  X_{2,-} \;\;&=&\; \set{\alpha,\beta,\gamma,\delta}, \hspace*{5em}& X_{2,+} \;\;&=&\; \set{\alpha',\beta',\gamma',\delta'}, \\[3pt]
  X_{1,-} \;\;&=&\; \set{a,d}, & X_{1,+} \;\;&=&\; \set{c,f}, \\[3pt]
  X_{0,-} \;\;&=&\; \set{x}, & X_{0,-} \;\;&=&\; \set{z}.
\end{carray}%
However,~$H_1$ and~$H_2$ are different cells in~$\freecat\Q$. This has first
been proved using \hbox{Agda} as a proof assistant~\cite{forest2019describing},
but can be proved more quickly using a solver for the word problem on strict
categories like \cateq~\cite{forestcateq}. Hence, the distinct
cells~$\cecompn_1$ and~$\cecompn_2$ of~$\freecat\Q$ are sent to the same cell
of~$\Cell(P)$ by~$F$, since the information that makes~$\cecompn_1$
and~$\cecompn_2$ different is the order in which~$A$ and~$B$ are composed, which
can not be expressed by a cell of a parity complex. This
refutes~\cite[Theorem~4.2]{street1991parity} which asserts that~$F$ is an
isomorphism. Thus, parity complexes do not necessarily induce free
$\omega$\categories in general.



%% file: johnson.tex
Johnson's loop-free pasting schemes~\cite{johnson1989combinatorics} is another
proposed formalism for pasting diagrams. Like parity complexes, they are based
on \ohg{}s, but the cells will now be represented as single subsets of
generators instead of tuples like for parity complexes. This formalism relies on
set relations, namely~$\jB$ and~$\jE$, to encode which generators to remove in
order to obtain respectively the target and the source of a cell.

\paragraph{Conventions for relations} A \emph{relation between two sets~$X$
  and~$Y$} is a subset~$\genrel \subseteq X \times Y$. For~$(x,y) \in X \times
Y$, we write~$x \genrel y$ when~$(x,y) \in \genrel$. The \emph{identity relation
  on a set~$X$} is the relation~$\genrel \subseteq X \times X$ such that~$x
\genrel y$ if and only if $x = y$\tobackpropagate{utiliser~$y$ ici au lieu de~$x'$}. Given
a relation~$\genrel$ between~$X$ and~$Y$, and~$x \in X$, we write~$\genrel(x)$
for the set ${\genrel(x) = \set{ y \in Y \mid x \genrel y }}$. Given a
subset~$X' \subseteq X$, we denote by~$\genrel(X')$ the set ${\set{ y \in Y \mid
    \exists x \in X', x \genrel y}}$. The relation~$\genrel$ is said
\index{finitary!relation}\emph{finitary} when, for all~$x \in X$,~$\genrel(x)$
is a finite set. If~$\genrel$ is a relation on a graded set~$P = \sqcup_{n \in
  \N} P_n$, given~$k,l \in \N$, we write~$\genrel^l_k$ for the relation
between~$P_l$ and~$P_k$ defined as~${\genrel \cap (P_l \times P_k)}$. Similarly,
we write~$\genrel^l$ for the relation between~$P_l$ and~$P$ defined as~$\genrel
\cap (P_l \times P)$. Given relations~$\genrel$ between~$X$ and~$Y$
and~$\genrel'$ between~$Y$ and~$Z$, we write~$\comprel{\genrel}{\genrel'}$ for
the relation between~$X$ and~$Z$ which is the composite relation defined as
\[
  \comprel{\genrel}{\genrel'} = \set{ (x,z) \in X \times Z \mid \exists y \in
    Y, x \genrel y~\text{and}~ y \genrel'z}.
\]


\paragraph{Pre-pasting schemes} A \index{pre-pasting scheme}\emph{pre-pasting
  scheme}~$(P,\jB,\jE)$ is given by a graded set~$P$ and two
\glossary(B){$\jB$}{the source relation of a pasting
  scheme}\glossary(E){$\jE$}{the target relation of a pasting
  scheme}relations~$\jB,\jE$ (for ``beginning'' and ``end'') on~$P$ such that
\begin{enumerate}[label=(\roman*),ref=(\roman*)]
 \item $\jB$ and~$\jE$ are finitary,
 \item for~$k,l \in \N$ with~$l < k$,~$\jB^l_k = \jE^l_k = \emptyset$,
 \item $\jB_k^k$ (\resp~$\jE_k^k$) is the identity relation on~$P_k$,
 \item for~$k,l \in \N$ with~$k < l$,~$\genrel \in \set{\jB,\jE}$,~$u \in
   P_{l+1}$ and~$v \in P_k$,~$u \genrel^{l+1}_k v$ if and only if
    \[
      u \comprel{\genrel^{l+1}_{l}}{\jB^{l}_k} v \qtand u
      \comprel{\genrel^{l+1}_{l}}{\jE^{l}_k} v.
    \]
\end{enumerate}
\begin{example}
  The diagram \eqref{eq:ex-ppc} can be encoded as a pre\nbd-pasting
  scheme
  \begin{align*}
    \jB^2_1(\alpha) &= \set{a,c}, & \jB^2_0(\alpha) &= \set{y}, & \jB^1_0(a) &= \set{x}, \\
    \jE^2_1(\alpha) &= \set{b,d}, & \jE^2_0(\alpha) &= \set{y'}, & \jE^1_0(a) &= \set{y} \ldots
  \end{align*}%
  \tobackpropagate{ce n'est pas~$f$ mais~$\alpha$ ici, aussi~$\backslash\backslash$ à enlever}%
\end{example}
\noindent Note that the relations~$\jB$ and~$\jE$ of a pre-pasting scheme~$P$
are completely determined by the data of~$\jB^{k+1}_{k}(u)$
and~$\jE^{k+1}_{k}(u)$ for~$k \in \N$ and~$u \in P_k$. As a consequence, the
data of a pre-pasting scheme structure on~$P$ is equivalent to the data of an
\ohg structure on~$P$: the correspondence is given by $u^- = \jB^{k+1}_{k}(u)$
and $u^+ = \jE^{k+1}_{k}(u)$ for~$k \in \N^*$ and~$u \in P_{k+1}$. In
particular, the relation~$\tl$ on a pasting scheme is defined as the one on the
associated \ohg.


\paragraph{Direct loops}
Given an \ohg~$P$,~$P$ \index{direct loop}\emph{has a direct loop} when
\begin{enumerate}[label=(\roman*),ref=(\roman*)]
\item either there exist~$n \in \N^*$ and~$u,v \in P_n$ such that~$u \tl v$
  and~$\jE(v) \cap \jB(u) \neq \emptyset$,
\item or there exists~$w\in P$ such that~$\jE(w) \cap \jB(w) \neq \set{w}$.
\end{enumerate}
\begin{example}
  The \ohg
  \begin{equation}
    \label{eq:ex-johncyclic}
    \begin{tikzcd}[ampersand replacement=\&,cramped]
      \& |[alias=y]| y  \arrow[rd,"a_2"] \& \\
      x \arrow[rr,"b"{description},""{name=b},""{name=bup,above},""{name=bdown,below}] \arrow[ru, "a_1"] \arrow[rd, "c_1"'] \& \& z \\
      \& |[alias=ybis]| y \arrow[ru, "c_2"'] \& \arrow[phantom, from=y,
      to=bup, "\Downarrow \alpha"] \arrow[phantom, from=b, to=ybis,
      "\Downarrow \beta"]
    \end{tikzcd}
  \end{equation}
  has a direct loop by the first criterion, because~$\alpha \tl \beta$ and~$y \in
  \jB(\alpha) \cap \jE(\beta)$. Examples of direct loops by the second condition
  are given by the \ohg{}s
  \begin{equation}
    \label{eq:2-pd-direct-loops}
    P^1 = 
    \begin{tikzcd}[row sep={3em,between origins},column sep={4em,between origins}]
      v
      \ar[r,bend left=70,"a",""{name=top,auto=false}]
      \ar[r,bend right=70,"a"',""{name=bot,auto=false}]
      &
      w
      \ar[from=top,to=bot,phantom,"\Downarrow\alpha"]
    \end{tikzcd}
    \qtand
    P^2 =
    \begin{tikzcd}[row sep={2em,between origins},column sep={3em,between origins},baseline=(x.base)]
      & y \arrow[rd,"c"]& \\
      |[alias=x]| x
      \arrow[ru,"b"]
      \arrow[rd,"b'"']
      &
      \Downarrow \beta
      &
      z \\
      & y \arrow[ru,"c'"']& 
    \end{tikzcd}
    \pbox.
  \end{equation}
\end{example}

\paragraph{Finite graded subsets} 
\parlabel{text:fgs}

Let~$P$ be a pre-pasting scheme. We define the \glossary(R){$\clrel$}{the
  closure relation for a pasting scheme or, more generally, an
  \ohg}relation~$\clrel \subseteq P \times P$ as the smallest reflexive
transitive relation on~$P$ such that, for all~$k \in \N$ and~$x \in P_{k+1}$, we
have
\[
  \jB(x) \cup \jE(x) \subseteq \clset(x)\zbox.
\]
\vskip-\baselineskip
\begin{example}
  In the case of the \ohg~\eqref{eq:ex-johncyclic}, we
  have
  \[
    \clrel(\alpha) = \set{x,y,z,a_1,a_2,b,\alpha}
    \qtand
    \clrel(\beta) = \set{x,y,z,b,c_1,c_2,\beta}.
  \]
\end{example}
\noindent A \glossary(fgs){fgs}{``finite graded subset''}\index{finite graded
  subset (fgs)}\emph{finite graded subset of dimension~$n$ of~$P$} (abbreviated
\emph{$n$\fgs}) is an $(n{+}1)$\nb{tuple}
\[
  X = (X_0,\ldots,X_n)
\]
such that~$X_k
\subseteq P_k$ and~$X_k$ is finite for~$k \in \N_n$. We often identify the
$n$\fgs~$X$ with the set~${\cup_{k \in \N_n} X_k}$, but one should keep in mind
that the $n$\fgs~$X$ and the $(n{+}1)$\fgs~$(X_0,\ldots,X_n,\emptyset)$ are two
different objects. We say that~$X$ is~\index{closed fgs}\emph{closed}
when~$\clset(X) = X$.
Given~$n \in \N$ and an $(n{+}1)$\fgs~$X$ of~$P$, we define the
\index{source!of a pre-pasting scheme fgs}\emph{source}
and the \index{target!of a pre-pasting fgs}\emph{target} of~$X$ as the $n$\fgs's~$\csrc_n (X)$ and~$\ctgt_n (X)$
of~$P$ such that
\[
  \csrc_n (X) = X \setminus \jE^{n}(X) \qtand \ctgt_n (X) = X \setminus
  \jB^{n}(Y).
\]
\vskip-\baselineskip
\begin{example}
Considering the \ohg~\eqref{eq:ex-johncyclic}, we have
\[
  \csrc_n(\clrel(\alpha)) = \clrel(\alpha) \setminus \set {b,\alpha} = \set{x,y,z,a_1,a_2}
  \qtand
  \ctgt_n(\clrel(\alpha)) = \clrel(\alpha) \setminus \set{y,a_1,a_2} = \set{x,z,b}.
\]
\end{example}
\begin{remark}
  The fgs's of the form~$\clrel(u)$ for~$u \in P$ are the analogue of the
  \index{atom!of a pre-pasting scheme}atoms defined for parity complexes.
\end{remark}

\tobackpropagate{utiliser des parenthèses pour~$\csrc$ et~$\ctgt$}

\paragraph{Well-formed sets}
\parlabel{text:wf-sets}

Given a pre-pasting scheme~$P$, we define by induction on~$n$ the notion of
\glossary(wfs){wfs}{``well-formed fgs''}\index{well-formed fgs
  (wfs)}\emph{well-formed $n$\fgs} (abbreviated \emph{$n$\wfs}): given~$n\in
\N$, an $n$\fgs~$X$ of~$P$ is \emph{well-formed} when
\begin{enumerate}[label=(\roman*),ref=(\roman*)]
\item $X$ is closed,
\item $X_n$ is fork-free,
\item when~$n > 0$,~$\csrc_n (X)$ and~$\ctgt_n (X)$ are well-formed $(n{-}1)$\fgs.
\end{enumerate}
We denote \glossary(WFP){$\wfset(P)$}{the graded set of well-formed $n$\fgs's
  (or $n$\wfs's) on~$P$}by~$\wfset(P)$ the graded set of $n$\wfs's of~$P$ for~$n
\in \N$. By~\cite[Theorem 3]{johnson1989combinatorics}, for~$n \in \N$, the
operations~$\csrc_n$ and~$\ctgt_n$ on $(n{+}1)$\fgs's restrict to functions
\[
  \csrc_n \co \wfset(P)_{n+1} \to \wfset(P)_n
  \qtand
  \ctgt_n \co \wfset(P)_{n+1} \to \wfset(P)_n
\]
and they equip~$\wfset(P)$ with a structure of $\omega$\globular set. In the
following, the wfs's will be the ``cells'' of the pasting diagram formalism of
pasting schemes.
\begin{example}
  The pre-pasting scheme
  \begin{equation}
    \label{eq:ex-wfs}
    \begin{tikzcd}[ampersand replacement=\&,cramped]
      \& |[alias=y]| y_1  \arrow[rd,"a_2"] \& \\
      x \arrow[rr,"b"{description},""{name=b},""{name=bup,above},""{name=bdown,below}] \arrow[ru, "a_1"] \arrow[rd, "c_1"'] \& \& z \\
      \& |[alias=ybis]| y_2 \arrow[ru, "c_2"'] \& \arrow[phantom, from=y,
      to=bup, "\Downarrow \alpha"] \arrow[phantom, from=b, to=ybis,
      "\Downarrow \beta"]
    \end{tikzcd}
  \end{equation}
  has, among others,
  the $0$\wfs's~$\set{x}$ and~$\set{z}$,
  the $1$\wfs's~$\set{x,y_1,z,a_1,a_2}$ and~$\set{x,y_2,z,c_1,c_2}$,
  and the $2$\wfs~$\set{x,y_1,y_2,z,a_1,a_2,b,c_1,c_2,\alpha,\beta}$.
\end{example}

\paragraph{Identity and composition operations}

Let~$P$ be a pre-pasting scheme. Given~$n \in \N$ and an $n$\wfs~$X =
(X_0,\ldots,X_n)$ of~$P$, the \index{identity operation!for pre-pasting scheme
  fgs's}\emph{identity of~$X$} is the $(n{+}1)$\wfs~$\unitp{n+1}{}(X)$ defined
by
\[
   \unitp{n+1}{}(X) = (X_0,\ldots,X_n,\emptyset) .
\]
Given~$i,n \in \N$ with~$i < n$ and~$X,Y$ two $n$\wfs such that~$\ctgt_i (X) = \csrc_i
(Y)$, the \index{composition operation!for pre-pasting scheme fgs's}\emph{$i$\composition} of~$X$ and~$Y$ is the $n$\fgs~$X \comp_i Y$ such
that
\[
  X \comp_{i} Y = X \cup Y.
\]
Under the conditions of a pre-pasting scheme, it is not necessarily the case
that the composite of two $n$\wfs's is an $n$\wfs, but it will under the axioms
of a pasting scheme introduced below.

\paragraph{Loop-free pasting schemes}
\parlabel{text:loop-free-ps}

A \index{pasting scheme}\emph{pasting scheme}~\cite{johnson1989combinatorics} is
a pre\nbd-pasting scheme~$P$ satisfying the following two axioms:
\begin{enumerateaxioms}
    \item[{\johnsonaxiom[d]{0}}] for~$k \in \N$ and~$u \in P_{k+1}$,~$\jB^{k+1}_{k}(u) \neq \emptyset$ and~$\jE^{k+1}_{k}(u) \neq \emptyset$;
    \item[{\johnsonaxiom[d]{1}}] for~$k,l \in \N$ with~$k \le l$,~$\genrel \in \set{\jB,\jE}$,~$u \in P_{l+1}$
    and~$v \in P_k$,
    \begin{enumerate}[label=--]
    \item if~$u \comprel{\jE^{l+1}_l}{\genrel^l_k} v$ then~$u
    \jE^{l+1}_k v$ or~$u \comprel{\jB^{l+1}_k}{\genrel^l_k} v$,
    \item if~$u \comprel{\jB^{l+1}_l}{\genrel^l_k} v$ then~$u \jB^{l+1}_k
    v$ or~$u \comprel{\jE^{l+1}_l}{\genrel^l_k} v$.
    \end{enumerate}
\end{enumerateaxioms}
\tobackpropagate{erreur d'indice dans le dernier or}%
The pasting scheme~$P$ is a \index{loop-free pasting scheme}\emph{loop-free pasting scheme} when it
moreover satisfies the following:
\begin{enumerateaxioms}
    \item[{\johnsonaxiom[d]{2}}] $P$ has no direct loops;
    \item[{\johnsonaxiom[d]{3}}] for~$u \in P$,~$\clset(u) \in \wfset(P)$;
    \item[{\johnsonaxiom[d]{4}}] for~$k,n \in \N$ with~$k < n$,~$X \in \wfset(P)_k$ and~$u \in P_n$,
    \begin{enumerate}[label=--]
    \item if~$\csrc_k (\clset(u)) \subseteq X$, then~$\gen{u}_{k,-}$ is a segment for~$\tl_{X_k}$,
    \item if~$\ctgt_k (\clset(u)) \subseteq X$, then~$\gen{u}_{k,+}$ is a
      segment for~$\tl_{X_k}$;
    \end{enumerate}
    \item[{\johnsonaxiom[d]{5}}] for~$n \in \N$,~$X \in \wfset(P)_n$ and~$u \in P_{n+1}$
    with~$\csrc_n(\clrel(u)) \subseteq X$, the following hold:
    \begin{enumerate}[label=(\alph*),ref=(\alph*)]
    \item $X \cap \jE(u) = \emptyset$,
    \item for~$y \in X$, if~$\jB(u) \cap \clset(y) \neq \emptyset$, then~$y \in \jB(u)$.
    \end{enumerate}
\end{enumerateaxioms}
\tobackpropagate{IMPORTANT: condition oubliée dans \johnsonaxiom{5} !!!}%
\johnsonaxiom{1} enforces basic globular properties on generators and forbids,
the \ohg~\eqref{eq:johnson-non-glob} for example. \johnsonaxiom{2} forbids
\ohg{}s with loops like~\eqref{eq:not-acyclic},~\eqref{eq:ex-johncyclic}
and~\eqref{eq:2-pd-direct-loops}.\tobackpropagate{axiome 2 non traité !}
\johnsonaxiom{3} enforces fork-freeness on the iterated sources and targets of a
generator (for example, it forbids the \ohg~\eqref{eq:ex-non-wfs}).
\johnsonaxiom{4} relates to \streetaxiom{5} of parity complexes and prevent
situations in the spirit of the \ohg~\eqref{eq:ex-non-segment}
(even though the latter does not satisfy~\johnsonaxiom{2} in the first place).
We motivate this axiom in \Cref{ssec:gen-pc} when we discuss a similar axiom for
torsion-free complexes. \johnsonaxiom{5}~can be deduced from the other axioms
(\cf~\cite[Theorem~3.7]{johnson1987pasting}) but it simplifies the proofs
of~\cite{johnson1989combinatorics}. An example of a sensible pre-pasting scheme
that satisfy Axioms~\johnsonax{0} to~\johnsonax{3}, but neither \johnsonaxiom{4}
nor \johnsonaxiom{5}, exists in dimension four (see~\cite[Example
3.11]{power1991n}).


\paragraph{A counter-example to the freeness property}
\parlabel{par:ce-ps}
The main result claimed in~\cite{johnson1989combinatorics} is similar to the one
of~\cite{street1991parity}: given a loop-free pasting scheme~$P$, the globular
set~$\wfset(P)$ together with the source, target, identity and composition
operations has the structure of an \ocat, which is freely generated by the
wfs's~$\clset(u)$ for~$u \in P$ (\cite[Theorem 13]{johnson1989combinatorics}).
Using the terminology of~\Cref{ssec:freeness-def}, this amounts to say that there
exist an \opol~$\Q$ and an $\omega$\functor~$F \co \freecat\Q \to \wfset(P)$
such that
\begin{itemize}
\item $\Q_k = P_k$ for~$k \in \N$,
\item $F(u) = \clset (u)$ for~$u \in \Q$,
\item $F$ is an isomorphism.
\end{itemize}
But the same flaw as for parity complexes is present here too, which makes the
freeness result wrong. In fact, the counter-example to the freeness property of
parity complexes, introduced in \Cref{ssec:street}, is also a counter-example to the
freeness property of pasting schemes: the \ohg~$P$ is a loop-free pasting scheme
and the canonical morphism~$F\co\freecat\Q \to \wfset(P)$ sends~$H_1$ and~$H_2$
to the same $3$\wfs~$X = \set{x,y,z,\alpha,\beta,\gamma,\delta,\alpha',\beta',\gamma',\delta',A,B}$
refuting the freeness property~\cite[Theorem 13]{johnson1989combinatorics}.


%% file: steiner.tex
Augmented directed complexes, designed by Steiner in~\cite{steiner2004omega},
are not directly based on \ohg{}s, but on chain complexes. Under the conditions
required by~\citeauthor{steiner2004omega}, it happens that the data of a chain
complex is equivalent to the data of an \ohg. The definition of
cells for this formalism strongly resembles the one of parity complexes. The
only difference is that the cells are tuples of group elements instead of
subsets of an \ohg.

\paragraph{Augmented directed complex}

A \glossary(preadc){pre-adc}{``pre-augmented directed
  complex''}\index{pre-augmented directed complex}\emph{pre-augmented directed
  complexes~$\genadc$} (abbreviated \emph{pre-adc}) is the data of
\begin{itemize}
\item for~$n \in \N$, an abelian group~$\defgrpname_n$ together with a
  distinguished submonoid~$\deffreemon_n \subseteq \defgrpname_n$,
\item for~$n \in \N$, group morphisms called \index{boundary operator}\emph{boundary operators}
  \[
    \diffop_n \colon K_{n+1} \to K_n,
  \]
\item an \index{augmentation}\emph{augmentation}, that is, a group morphism
  \[
    \augop \colon K_0 \to\Z.
  \]
\end{itemize}
An \glossary(adc){adc}{``augmented directed complex''}\index{augmented directed
  complex (adc)}\emph{augmented directed complex}, abbreviated \emph{adc}, is a
pre-adc~$\genadc$ such that
\[
  \augop \circ \diffop_0 = 0
  \qtand
  \diffop_n \circ
  \diffop_{n+1} = 0 \text{ for~$n \in \N$.}
\]\tobackpropagate{un $=0$ oublié !!!}%

\paragraph{Bases for pre-adc's}
\parlabel{par:adc-basis}

Given a pre-adc~$\genadc$, a \index{basis of an adc}\emph{basis} of~$\genadc$ is the data of a graded
set~$P \subseteq \bigsqcup_{n \in \N} K_n$ such that
each~$\deffreemon_n$ is the free commutative monoid on~$P_n$ and
each~$K_n$ is the free abelian group on~$\deffreemon_n$. Given a
basis~$P$ of~$\genadc$, every element $u \in K_n$\tobackpropagate{se
  placer dans~$K_n$ et pas~$K_n^\ast$, et utiliser~$\Z$ et pas~$\N$ après} can
be uniquely written as
$ u = \sum_{g \in P_n} u_g g $,
with~$u_g \in \Z$ such that~$u_g \neq 0$ for a finite number of~$g\in
P_n$. This representation induces a partial order~$\le$ where, for~$n
\in \N$ and~$u,v \in K_n$,~$u \le v$ when~$u_g \le v_g$ for all~$g \in
P_n$. Given~$n \in \N$ and~$u,v \in K_n$ we can define a
\glossary(.m){$u \land v$}{the greatest lower bound of~$u$
  and~$v$}\index{greatest lower bound}\emph{greatest lower bound}~${u \land v}$
of~$u$ and~$v$ by
\[
  u \land v = \sum_{g \in P_n} \min(u_g,v_g)g
  \zbox.
\]
Given~$n \in \N$ and~$u \in K_{n+1}$, we write~$u^\mp,u^\pm \in \freemon{K}_{n}$
for the unique elements \glossary(.kb){$(-)^\mp$, $(-)^\pm$}{the border
  operations on a pre-adc} which satisfies that
  $\diffop_n (u) = u^\pm - u^\mp$ and $u^\mp \land u^\pm = 0$.
Moreover, we write~$u^-,u^+$ for
\[
  u^- = \sum_{g \in P_{n+1}} u_g g^\mp \qtand u^+ = \sum_{g \in P_{n+1}} u_g g^\pm\zbox.
\]

\begin{remark}
  The elements~$u^\mp$ and~$u^\pm$ are respectively denoted by~$\csrc (u)$ and~$\ctgt (u)$ in~\cite{steiner2004omega}. We adopt the former notation for
  consistency with those of \Ssecr{street}.
\end{remark}

\paragraph[From \texorpdfstring{$\omega$}{omega}-hypergraphs to pre-adc's with basis]{From \bmohgwe{}s to pre-adc's with basis}
\parlabel{par:ohg-to-pre-adc}
\ndr{ce paragraphe permet une translation directe des exemples donnés pour les
  autres structures aux \ohg{}s}
Given an \ohg~$P$, we define the \emph{pre-adc~$\genadc$ associated to~$P$} as
follows. For~$n \in \N$,~$\freemon{K}_n$ is defined as the free commutative
monoid on~$P_n$ and~$K_n$ as the free abelian group on~$\freemon{K}_n$. The
augmentation~$\augop \colon K_0 \to \mathbb{Z}$ is defined as the unique
morphism such that~$\augop(x) = 1$ for~$x \in P_0$. Given~$n \in \N$ and a
finite subset~$U \subseteq P_n$, we write~$\stom_n (U)$ for~$\sum_{u \in U} u
\in K_n$. Then,~$\diffop_n \colon K_{n+1} \to K_n$ is defined as the unique
morphism such that
$
  \diffop_n(u) = \stom_n(u^+) - \stom_n(u^-)
  $
for~$u \in P_{n+1}$. Then,~$K$ canonically admits~$P$ as a basis. We say that
\emph{$P$ is an adc} when~$K$ is an adc.
\begin{example}
  We explicitly describe the pre-adc associated to the \ohg~\eqref{eq:ex-wfs} as
  follows. Writing~$\freemon{S}$ for the free commutative monoid on a set~$S$,
  we put
  \[
    \freemon{K}_0 = \freemon{\set{x,y_1,y_2,z}}, \quad \freemon{K}_1 =
    \freemon{\set{a_1,a_2,b,c_1,c_2}},\quad \freemon{K}_2 =
    \freemon{\set{\alpha,\beta}}
  \]
  and~$\freemon{K}_n = \set{0}$ for~$n \ge 3$.~$K_0$,~$K_1$,~$K_2$ and~$K_n$
  for~$n \ge 3$ are then the induced free abelian groups on these monoids. The
  operations~$\augop$ and~$\diffop$ are defined by universal property to be the
  unique morphisms such that $\augop(x) = \augop(y_1) = \augop(y_2) = \augop(z)
  = 1$ and
  \begin{gather*}
    \diffop_0(a_1) = y_1 - x\zbox,\qquad
    \diffop_0(a_2) = z - y_1\zbox,\qquad
    \diffop_0(b) = z - x\zbox, \\
    \begin{aligned}
      \diffop_0(c_1) &= y_2 - x\zbox, &  \diffop_0(c_2) &= z-y_2\zbox, \\
      \diffop_1(\alpha) &= b - (a_1 + a_2)\zbox,\qquad
      & \diffop_1(\beta)
      &= \mathmakebox[0pt][l]{(c_1 + c_2) - b\zbox.}
    \end{aligned}
  \end{gather*}
  We can now give some examples for the operations~$(-)^\mp$ and~$(-)^\pm$
  operations defined above:
  \begin{align*}
    (a_1 + a_2)^\mp &= x, & (a_1 + a_2)^\pm &= z, &
    (\alpha + \beta)^\mp &= a_1 + a_2, & (\alpha + \beta)^\pm &= c_1 + c_2\zbox.
  \end{align*}
  We moreover illustrate the operations~$(-)^-$ and~$(-)^+$:
  \begin{align*}
    (a_1 + a_2)^- & =  x + y_1, & (a_1 + a_2)^+ &= y_1 + z, &
    (\alpha + \beta)^- &= a_1 + a_2 + b, & (\alpha + \beta)^+ &= b + c_1 + c_2\zbox.
  \end{align*}
  Thus, the operations~$(-)^\mp$ and~$(-)^\pm$ compute the source and target
  ``borders'' of an element of~$K_n$, whereas the operations~$(-)^-$ and~$(-)^+$
  compute the sum of the ``inner'' sources and targets of an element of~$K_n$.
  They are the analogues of the operations defined for \ohg
  in~\Cref{ssec:hypergraph}.
\end{example}

\paragraph{Cells}

Let~$K$ be a pre-adc. Given~$n \in \N$, an \index{pre-cell!of a (pre-)adc}\emph{$n$\precell of~$K$} is given by
an $(2n{+}1)$\nb{tuple}
\[
  X=(X_{0,-},X_{0,+},\ldots,X_{n-1,-},X_{n-1,+},X_n)
\]
with~$X_n \in \freemon{K}_n$ and $X_{i,-},X_{i,+} \in
\freemon{K}_i$\tobackpropagate{problème d'indice ici} for~$i \in \N_{n-1}$. For
the sake of conciseness, we often refer to~$X_n$ by~$X_{n,-}$ or~$X_{n,+}$. We
\glossary(PCellK){$\adcPCell(K)$}{the set of pre-cells of the pre-adc~$K$}write~$\adcPCell(K)$ for the graded set of pre-cells of~$K$. When~$n > 0$,
given~$\eps \in \set{-,+}$, we define the $n$\precell~$\csrctgt\eps_n (X)$ as
\[
  \csrctgt\eps_n (X)=(X_{0,-},X_{0,+},\ldots,X_{n-1,-},X_{n-1,+},X_{n,\eps}) .
\]
The functions~$\csrc,\ctgt$ then equip~$\adcPCell(K)$ with a structure of
$\omega$\globular set.

\smallpar Given $n \in \N$, an \index{cell!of a (pre-)adc}\emph{$n$\cell of~$K$} is an $n$\precell~$X$
of~$K$ such that
\begin{enumerate}[label=(\roman*),ref=(\roman*)]
\item \label{adc:cell:move} for~$i \in \N_{n-1}$,~$\diffop_{i} (X_{i+1,-}) = \diffop_{i} (X_{i+1,+}) =
  X_{i,+} - X_{i,-}$,
\item \label{adc:cell:0cell} $\augop (X_{0,-}) = \augop (X_{0,+}) = 1$.
\end{enumerate}
We denote \glossary(CellK){$\adcCell(K)$}{the set of cells of the
  pre-adc~$K$}by~$\adcCell(K)$ the graded set of cells of~$K$, which inherits
the $\omega$\globular structure from~$\adcPCell(K)$.
\begin{remark}
  The condition~\ref{adc:cell:move} is analogous to the moving
  condition~\ref{pc:cell:move} of parity complex cells, and the
  condition~\ref{adc:cell:0cell} is related to the fork-freeness
  condition~\ref{pc:cell:fork-free} of parity complex cells instantiated in
  dimension~$0$.
\end{remark}

\paragraph{Identity and composition operations}

Let~$K$ be a pre-adc. Given~$n \in \N$ and an $n$\precell~$X$ of~$K$, we define
the \index{identity operation!for (pre-)adc cells}\emph{identity of~$X$} as the
$(n{+}1)$\precell~$\unitp{n+1}{}(X)$ of~$K$ such that
\[
  \unitp{n+1}{}(X) = (X_{0,-},X_{0,+},\ldots,X_{n-1,-},X_{n-1,+},X_{n},X_{n}, 0)\zbox.
\]
Given~$i,n \in \N$ with~$i < n$ and $i$\composable $n$\cells~$X,Y$, we define
the \index{composition operation!for (pre-)adc cells}\emph{$i$\composition~$X
  \comp_{i} Y$} of~$X$ and~$Y$ as the $n$\precell~$Z$ such that, for~$j \in
\N_n$ and~$\eps \in \set{-,+}$,
\[ Z_{j,\eps} = 
  \begin{cases}
    X_{j,\eps} + Y_{j,\eps} & \text{when~$j > i$,} \\
    X_{i,-} & \text{when~$j = i$ and~$\eps = -$,} \\
    Y_{i,+} & \text{when~$j = i$ and~$\eps = +$,} \\
    X_{j,\eps} \text{ (or equivalently~$Y_{j,\eps}$) } & \text{when~$j < i$.}
  \end{cases}
\]
We then easily verify that~$Z \in \adcCell(K)$.

\paragraph{Atoms}
Let~$K$ be a pre-adc equipped with a basis~$P$. Given~$n \in \N$ and~$u \in
P_n$, we define~${\stgen{u}_{i,\eps} \subseteq P_i}$ for~$i \in \N_n$ and~${\eps
  \in \set{-,+}}$ using a downward induction by $\stgen{u}_{n,-} =
\stgen{u}_{n,+} = u$ and, for~$j \in \N_{n-1}$, $\stgen{u}_{j,-} =
\stgen{u}_{j+1,-}^\mp$ and $\stgen{u}_{j,+} = \stgen{u}_{j+1,+}^\pm$. For
simplicity, we sometimes write~$\stgen{u}_{n,-}$ or~$\stgen{u}_{n,+}$
for~$\stgen{u}_n$. The \glossary(.o){$\stgen u$}{the atom associated to the
  basis generator~$u$ of a pre-adc}\index{atom!of a pre-adc}\emph{atom
  associated to~$u$} is then the $n$\precell of~$K$
\[
  \stgen{u}=(\stgen{u}_{0,-},\stgen{u}_{0,+},\ldots,\stgen{u}_{n-1,-},\stgen{u}_{n-1,+},\stgen{u}_n) \zbox.
\]
\vskip-\baselineskip
\begin{example}
  In the pre-adc associated to the \ohg~\eqref{eq:ex-wfs}, the
  atom~$\stgen{\alpha}$ associated to~$\alpha$ is defined by
  \begin{carray}{\stgen{\alpha}_2 = \alpha,}
    \stgen{\alpha}_{1,-} \;\;&=&\; a_1 + a_2, \hspace*{7em} & \stgen{\alpha}_{1,+} \;\;&=&\; b, \\[3pt]
    \stgen{\alpha}_{0,-} \;\;&=&\; x, & \stgen{\alpha}_{0,+} \;\;&=&\; z.
  \end{carray}%
\end{example}

\paragraph{Unital loop-free basis} 

Let~$K$ be a pre-adc equipped with a basis~$\genbasis$. Given~$i \in \N$, we
define a \glossary(.p){$<_i$}{a pre-order on the basis of a
  pre-adc}relation~$<_i$ on~$\genbasis$ as the smallest transitive relation such
that, for~$k,l \in \N$ with~$i < \min(k,l)$, and~$u \in \genbasis_k,v \in
\genbasis_l$ with~${\stgen{u}_{i,+} \land \stgen{v}_{i,-} \neq 0}$, we have~$u
<_i v$. The basis~$\genbasis$ is then said
\begin{itemize}
\item \index{unital basis}\emph{unital} when for all~$u \in \genbasis$,~$\augop(\stgen{u}_{0,-}) = \augop(\stgen{u}_{0,+}) = 1$,

\item \index{loop-free basis}\emph{loop-free} when, for all~$i \in \N$,~$<_i$ is irreflexive.
\end{itemize}
\begin{example}
  Consider the pre-adc~$K$ with basis~$B$ derived from the
  hypergraph~\eqref{eq:not-acyclic}. The basis~$B$ is then unital but not
  loop-free since~$a <_0 b <_0 a$. Now, consider the pre-adc with basis~$B$
  derived from the hypergraph~\eqref{eq:ex-non-wfs}. The basis~$B$ is then not
  unital since~$\augop(\stgen{a}_{0,-}) = \augop(x+y) = 2$, but it is loop-free.
  Now consider the pre-adc~$K$ with basis~$B$ derived from the
  hypergraph~\eqref{eq:ohg-cells}. We have, among others, the relations
  \begin{gather*}
    a <_0 b <_0 c <_0 d <_0 e <_0 f\zbox,
    \qquad
    a <_0 \alpha <_0 \delta <_0 f\zbox,
    \qquad
    \beta <_1 \alpha <_1 \gamma \qtand \beta <_1 \delta <_1 \gamma\zbox.
  \end{gather*}
  It can be verified that~$B$ is unital and loop-free.
\end{example}

\paragraph{The freeness property} In~\cite{steiner2004omega}, the author shows
that, given an adc~$K$ with a loop-free unital basis~$B$, the globular
set~$\adcCell(K)$, together with identity and composition operations, has a
structure of an $\omega$\category which is freely generated by the
atoms~$\stgen{u}$ for~$u \in B$. Using the terminology
of~\Cref{ssec:freeness-def}, this amounts to say that there exist an \opol~$\Q$
and an $\omega$\functor~$F \co \freecat\Q \to \adcCell(K)$ such that
\begin{itemize}
\item $\Q_k = B_k$ for~$k \in \N$,
\item $F(u) = \stgen u$ for~$u \in \Q$,
\item $F$ is an isomorphism.
\end{itemize}
Contrary to parity complexes and pasting schemes, the pre-adc with basis
associated to the \ohg~\eqref{eqn:ce-tf} is not a loop-free adc. Indeed, it is
an adc with unital basis, but the basis is not loop-free since~${A <_1 B <_1
  A}$. Thus, augmented directed complexes are, to the best of our knowledge, the
only formalism of pasting diagrams among the three that we have already
introduced for which the freeness property holds.



%% file: forest.tex
In this section, we introduce \emph{torsion-free complexes}. They are a new
formalism for pasting diagrams based on parity complexes. More precisely,
torsion-free complexes rely on the same notion of cell than parity complexes,
but satisfy different axioms, namely the axioms~\forestax{0} to~\forestax{4}.

\paragraph{Definitions}
\parlabel{par:gpc-ax}

Let~$P$ be an \ohg. Given~$k \in \N$ and~$u \in P_k$, we say that \emph{$u$
  satisfies the segment condition} when, for all~$n \in \N_{k-1}$ and every
$n$\cell~$X$ such that~$\gen{u}_{n,-} \subseteq X_n$, it holds that
both~$\gen{u}_{n,-}$ and~$\gen{u}_{n,+}$ are segments for~$\tl_{X_n}$.
Given~$n,k,l \in \N$ with~$0 < n < \min(k,l)$,~$u \in P_k$, $v \in P_l$ and an
$n$\cell~$X$,~$u$ and~$v$ are said to be \index{torsion}\emph{in torsion with respect to~$X$}
when
\[
  \gen{u}_{n,+} \subseteq X_n
  , \quad
  \gen{v}_{n,-} \subseteq X_n
  ,\quad
  \gen{u}_{n,+} \cap \gen{v}_{n,-} = \emptyset
  \qtand
  \gen{u}_{n,+} \tl_{X_n}
  \gen{v}_{n,-} \tl_{X_n} \gen{u}_{n,+}.
\]
The \ohg~$P$ is then a \index{torsion-free complex}\emph{torsion-free complex}
when it satisfies the following axioms:
\begin{enumerateaxioms}
\item[{\forestaxiom[d]{0}}] (non-emptiness) for all~$u \in P$,~$u^- \neq \emptyset$ and~$u^+ \neq
  \emptyset$;
\item[{\forestaxiom[d]{1}}] (acyclicity)~$P$ is acyclic;
\item[{\forestaxiom[d]{2}}] (relevance) for all~$u \in P$,~$u$ is relevant;
\item[{\forestaxiom[d]{3}}] (segment condition) for~$u \in P$,~$u$ satisfies the segment condition;
\item[{\forestaxiom[d]{4}}] (torsion-freeness) for all~$n,k,l \in \N^*$ with~$n
  < \min(k,l)$,~$u \in P_k$,~$v \in P_l$ and every $n$\cell~$X$,~$u$ and~$v$ are
  not in torsion with respect to~$X$.
\end{enumerateaxioms}
\forestaxiom{1} enforces the same notion of acyclicity than for parity
complexes, forbidding loops like~\eqref{eq:not-acyclic}. \forestaxiom{2}
requires that the generators of the \ohg induce cells, forbidding \ohg{}s
like~\eqref{eq:johnson-non-glob} and~\eqref{eq:ex-non-wfs}. It can be shown that
\forestaxiom{2} entails Axioms~\streetax{1} and~\streetax{2} of parity
complexes. The last axioms deserve their own paragraphs.

\paragraph[The segment Axiom \forestaxlabel{3}]{The segment \forestaxiom{3}}
\parlabel{par:segment}

Our goal is to find conditions on \ohg{}s~$P$ so that the freeness property
holds, \ie the \ocat of cells~$\Cell(P)$ is freely generated by the atoms. In
particular, a technical result states that every cell should be decomposable as
a sequence of \eq{whiskered atoms} (\cf
\Cref{prop:decomp-cells-polextp-strcats}). But there are cells of \ohg{}s
satisfying Axioms~\forestax{0} to~\forestax{2} that cannot be decomposed this
way, because of the constraints that~$\tl$ requires on the composition order. We
illustrate this with an example. Consider the \ohg~$P$ represented
on \Cref{fig:ex-non-segment}
\begin{figure}
  \centering
  \begin{equation}
    \label{eq:ex-non-segment}
    \begin{tikzpicture}[yscale=1.8,xscale=1.2,baseline=(b)]
      \coordinate (a) at (0,1);
      \coordinate (b) at (0,0);
      \coordinate (c) at (0,-1);
      \coordinate (d) at (1.5,0.3);
      \node (A) at (a) {$z$};
      \node (B) at (b) {$y$};
      \node (C) at (c) {$x$};
      \node (D) at (d) {$w$};
      \draw[->] (B) to[bend right=60,"$d'$",swap,pos=0.6,inner sep=0.3ex] (A); 
      \draw[->] (B) to[bend left=60,"$d$"] (A); 
      \draw[->] (C) to[bend right=60,"$a'$"',inner sep=0.3ex] (B); 
      \draw[->] (C) to[bend left=60,"$a$",inner sep=0.3ex] (B); 
      \draw[->] (C) .. controls +(-1,0) and +(-1,0) .. (B) node[midway,left] {$c$}; 
      \draw[->] (B) to[bend left=10,"$b$",pos=0.6,inner sep=0.3ex] (D); 
      \draw[->] (C) to[out=20,in=-100,"$b'$"'] (D); 
      
      \draw[->] (A) .. controls +(-3,0) and +(-3,0) .. (0,-2) node[pos=0.7,left] {$e$}
      .. controls +(2,0) and +(1,-1) .. (D) ;
      \node (X1) at (0,-0.5) {
        $
        \begin{array}{c}
          \myoverset{$\Rightarrow$}{$\alpha_1$} \\[-5pt]
          \myoverset{$\alpha_1'$}{$\Rightarrow$}
        \end{array}
        $
      };
      \node (X2) at (0,0.5) {$\begin{array}{c} \myoverset{$\Rightarrow$}{$\alpha_4$} \\[-5pt]
                                \myoverset{$\alpha_4'$}{$\Rightarrow$} \end{array}$};
      \node at (0.75,-0.2) {$\myoverset{$\Rightarrow$}{$\alpha_2$}$};
      \node at (0,-1.5) {$\myoverset{$\Leftarrow$}{$\alpha_3$}$};
    \end{tikzpicture}
    \quad
    \begin{tikzpicture}[yscale=1.8,xscale=1.2,baseline=(b)]
      \coordinate (a) at (0,1);
      \coordinate (b) at (0,0);
      \coordinate (c) at (0,-1);
      \coordinate (d) at (1.5,0.3);
      \node (A) at (a) {$z$};
      \node (B) at (b) {$y$};
      \node (C) at (c) {$x$};
      \draw[->] (B) to[bend right=60,"$d'$"'] (A); 
      \draw[->] (B) to[bend left=60,"$d$"] (A); 
      \draw[->] (C) to[bend right=60,"$\mathstrut a'$"'] (B); 
      \draw[->] (C) to[bend left=60,"$\mathstrut a$"] (B); 
      
      \node (X1) at (0,-0.5) {\myoverset{$\Rightarrow$}{$\alpha_1$}};
      \node (X2) at (0,0.5) {\myoverset{$\Rightarrow$}{$\alpha_4$}};
    \end{tikzpicture}
    \myoverset{$\Rrightarrow$}{$A$}
    \begin{tikzpicture}[yscale=1.8,xscale=1.2,baseline=(b)]
      \coordinate (a) at (0,1);
      \coordinate (b) at (0,0);
      \coordinate (c) at (0,-1);
      \coordinate (d) at (1.5,0.3);
      \node (A) at (a) {$z$};
      \node (B) at (b) {$y$};
      \node (C) at (c) {$x$};
      \draw[->] (B) to[bend right=60,"$d'$"'] (A); 
      \draw[->] (B) to[bend left=60,"$d$"] (A); 
      \draw[->] (C) to[bend right=60,"$\mathstrut a'$"'] (B); 
      \draw[->] (C) to[bend left=60,"$\mathstrut a$"] (B); 
      \node (X1) at (0,-0.5) {\myoverset{$\Rightarrow$}{$\alpha_1'$}};
      \node (X2) at (0,0.5) {\myoverset{$\Rightarrow$}{$\alpha_4'$}};
    \end{tikzpicture}
  \end{equation}
  \caption{A problematic \ohgwe}
  \label{fig:ex-non-segment}
\end{figure}
where, more precisely,
\begin{align*}
  A^- &= \set{\alpha_1,\alpha_4}, & A^+ &= \set{\alpha'_1,\alpha'_4}, \\
  \alpha_1^- &= {\alpha'_1}^{-} = \set{a}, & \alpha_1^+ &= {\alpha'_1}^+ = \set{a'}, \\
  \alpha_4^- &= {\alpha'_4}^- = \set{d}, & \alpha_4^+ &= {\alpha'_4}^+ = \set{d'},
                                                        \zbox{\quad \etcend}
\end{align*}
One can verify that~$P$ satisfies Axioms~\forestax{0},~\forestax{1}
and~\forestax{2}. In this \ohg, there is a $2$\cell~$X$ given by
\begin{carray}{X_2 = \set{\alpha_1,\alpha_2,\alpha_3,\alpha_4},}
  X_{1,-} \;\;&=&\; \set{a,b}, \hspace*{5em} & X_{1,+} \;\;&=&\; \set{c,d',e}, \\[3pt]
  X_{0,-} \;\;&=&\; \set{x}, & X_{0,+} \;\;&=&\; \set{z}\zbox,
\end{carray}%
and a $3$\cell~$Y$ uniquely defined by $\csrc_2(Y) = X$ and~$Y_3 = \set A$.
Suppose by contradiction that~$\Cell(P)$ is an \ocat which is freely generated
by the atoms. Then,~$Y$ can be written
\begin{equation*}
  Y = \unit\lambda \comp_1 (\unitp 2 l \comp_0 \gen A \comp_0 \unitp 2 r) \comp_1 \unit\rho
\end{equation*}
for some $1$\cells $l,r$ and $2$\cells~$\lambda,\rho$. Thus, $X = \unit \lambda
\comp_1 X' \comp_1 \unit\rho$ where~$X'$ is a $2$\cell such that~$X'_2 = A^-$. Since~${\restrictcat{\Cell(P)}{2} \simeq
  \restrictcat{\Cell(P \setminus \set{A})}{2}}$ and~${P \setminus \set{A}}$ is a
torsion-free complex, using \Lemr{cell-comp-n-1} introduced later, the existence
of the decomposition of~$X$ implies that
\begin{equation*}
  \label{ax3:partition}
  \text{the sets~$\lambda_2$,~$X'_2$ and~$\rho_2$ form a
    partition of~$X_2 = \set{\alpha_1,\alpha_2,\alpha_3,\alpha_4}$,}
\end{equation*}
and, using \Lemr{dep-comp} introduced later, we have that
 \begin{equation*}
   \label{ax3:tlone-restr}
   \text{for~${(\beta,\gamma) \in (\lambda_2 \times (X'_2 \cup \rho_2)) \cup ((\rho_2 \cup X'_2) \times \rho_2)}$,
   we have~$\textstyle\lnot (\gamma \tl_{X_2} \beta)$,}
 \end{equation*}
 or, more simply put, the partition~$\lambda_2,X'_2,\rho_2$ respects the
 relation~$\tl_{X_2}$. But this cannot be
 possible since~$X'_2 = \set{\alpha_1,\alpha_4}$ and~$\alpha_1 \tl \alpha_2 \tl
 \alpha_3 \tl \alpha_4$. Hence,~$\Cell(P)$ is not an \ocat freely generated by
 the atoms. Since $\gen{A}_{2,-}$ is not a segment for~$\tl_{X_2}$,
 \forestaxiom{3} prevents this kind of problem.

\paragraph[The torsion-freeness Axiom \forestaxlabel{4}]{The torsion-freeness \forestaxiom{4}}
\parlabel{par:torsionless}

The notion of torsion captures the essence of the counter-example to the
freeness property of parity complexes and pasting schemes presented in
\Cref{ssec:street}. Indeed,
considering the \ohg~$P$ represented by~\eqref{eqn:ce-tf}, there is a
$2$\cell~$X$ defined by
\begin{carray}{X_2 = \set{\alpha',\beta,\gamma,\delta'},}
  X_{1,-} \;\;&=&\; \set{a,d},\hspace*{5em} & X_{1,+} \;\;&=&\; \set{c,f}, \\[3pt]
  X_{0,-} \;\;&=&\; \set{x}, & X_{0,+} \;\;&=&\; \set{z},
\end{carray}%
which is induced by the pasting diagram
\[
  \begin{tikzcd}[ampersand replacement=\&]
    x \arrow[rr,"b"{name=b,description}]
    \arrow[rr,out=70,in=110,"a"{name=a}] 
    \arrow[rr,out=-70,in=-110,"c"'{name=p}] 
    \arrow[phantom,"\Downarrow\!\alpha'",pos=0.6,from=a,to=b]
    \arrow[phantom,"\Downarrow\!\beta",pos=0.4,from=b,to=p] 
    \& \& y \arrow[rr,"e"{name=e,description}]
    \arrow[rr,out=70,in=110,"d"{name=p}]
    \arrow[rr,out=-70,in=-110,"f"'{name=f}] \& \& z
    \arrow[phantom,"\Downarrow\!\gamma",pos=0.6,from=p,to=e]
    \arrow[phantom,"\Downarrow\!\delta'",pos=0.4,from=e,to=f]
  \end{tikzcd}
  \pbox.
\]
Then, one can verify that~$A$ and~$B$ are in torsion with respect to~$X$, so
that~$P$ does not satisfy \forestaxiom{4} (on the other hand, it satisfies
Axioms~\forestax{0} to~\forestax{3}).

Intuitively, the situations with torsion are the minimal cases where the
freeness property fails for a parity complex~$P$ (and similarly for a pasting
scheme~$P$). When~${u,v \in P}$ are in torsion with respect to a cell~$X$
of~$P$, there are two possible order to compose~$u$ and~$v$: first~$u$ then~$v$,
or first~$v$ then~$u$. And both composites produce equal cells in~$\Cell(P)$.
However, this equality can not be deduced from an exchange law, since the
torsion says basically that~$u$ and~$v$ cross each other, preventing to use the
exchange law to swap them.

\paragraph{More computable axioms}
\parlabel{par:forest-comp}
Axioms~\forestax{3} and~\forestax{4} happen to be hard to check in practice,
since they both involve a quantification on all the cells of an \ohg. Here, we give stronger axioms that are
simpler to verify.

Given an \ohg~$P$, for~$n \in \N$,~$u,v \in P_n$, we \glossary(.q){$\jtlelem$}{a
  pre-order on the generators of an \ohg}write~$u \jtlelem v$ when there
exists~$w \in P_{n+1}$ such that~$u \in w^-$ and~$v \in w^+$ and we write~$\jtl$
for the reflexive transitive closure of~$\jtlelem$. For~$U,V \subseteq P_n$, we
write~$U \jtl V$ when there exist~$u \in U$ and~$v \in V$ such that~$u \jtl v$.
Consider the following axiom on an \ohg~$P$:
\begin{enumerateaxioms}
\item[{\forestaxiomcomp[d]{3}}] for~$k,n \in \N^*$ with~$k < n$ and~$u \in P_n$,
  we do not have~$ \gen{u}_{k,+} \jtl \gen{u}_{k,-}$.
\end{enumerateaxioms}%
Then, \forestaxiom{3} can be replaced by \forestaxiomcomp{3} in the axioms of
torsion-free complexes:
\begin{prop}
  \label{prop:comp-impl-norm-three}
  Let~$P$ be an \ohg satisfying Axioms~\forestax{0}, \forestax{1}
  and~\forestax{2}. If~$P$ satisfies \forestaxiomcomp{3}, then it satisfies
  \forestaxiom{3}.
\end{prop}
\begin{proof}
  Suppose that~$P$ satisfies \forestaxiomcomp{3}. Let~$n,k \in \N$ with~$n <
  k$,~$X$ be an $n$\cell and~$u \in P_k$ such that~${\gen{u}_{n,-} \subseteq
    X_n}$. If~$n = 0$, there is nothing to prove, so we can assume~$n > 0$. By
  contradiction, suppose that~$\gen{u}_{n,-}$ is not a segment for~$\tl_{X_n}$. So
  there are~$r \in \N$ with~$r > 2$ and~$u_1,\ldots,u_r \in X_n$ such that
    $u_1,u_r \in \gen{u}_{n,-}$,
    $u_2,\ldots,u_{r-1} \not \in \gen{u}_{n,-}$
    and
    $u_i \tlone_{X_n} u_{i+1}$
  for~$i \in \N^*_{r-1}$. In particular, there are~$v_1,\ldots,v_{r-1} \in
  P_{n-1}$ such that~${v_i \in u_i^+ \cap u_{i+1}^-}$ for~$i \in \N^*_{r-1}$.
  Given~${w \in X_n}$ such that~${v_1 \in w^-}$, since~$X_n$ is fork-free, we
  have~${w = u_2 \not \in \gen{u}_{n,-}}$. Thus, since~$u$ is relevant by
  \forestaxiom{2},~$v_1 \in \gen{u}_{n,-}^\pm = \gen{u}_{n-1,+}$.
  Similarly, we have that~$v_{r-1} \in \gen{u}_{n-1,-}$. So, $\gen{u}_{n-1,+} \jtl
  \gen{u}_{n-1,-}$, contradicting \forestaxiomcomp{3}. Hence,~$P$ satisfies
  \forestaxiom{3}.
\end{proof}
\noindent Now, consider the following axiom on an \ohg~$P$:
\begin{enumerateaxioms}
\item[{\forestaxiomcomp[d]{4}}] for~$n,k,l \in \N^*$ with~$n < \min(k,l)$,~$u\in
  P_k$ and~$v \in P_l$, if~${\gen{u}_{n,+} \cap \gen{v}_{n,-} = \emptyset}$,
  then at most one of the following holds:
\begin{itemize}
\item $\gen{u}_{n-1,+} \jtl \gen{v}_{n-1,-}$,
  
\item $\gen{v}_{n-1,+} \jtl \gen{u}_{n-1,-}$.
\end{itemize}
\end{enumerateaxioms}
  
\noindent Then, \forestaxiom{4} can be replaced by \forestaxiomcomp{4} in
the axioms of torsion-free complexes:
\begin{prop}
  \label{prop:comp-impl-norm-four}
  Let~$P$ be an \ohg satisfying Axioms~\forestax{0},~\forestax{1}
  and~\forestax{2}. If~$P$ satisfies \forestaxiomcomp{4}, then it satisfies
  \forestaxiom{4}.
\end{prop}
\begin{proof}
  \newcommand\barw{\bar w}%
  Suppose that~$P$ satisfies \forestaxiomcomp{4}. By contradiction, assume
  that~$P$ does not satisfy \forestaxiom{4}. So there are~$n,k,l \in \N^*$
  with~$n < \min(k,l)$,~$u \in P_k$,~$v \in P_l$ and an $n$\cell~$X$ such
  that~$u$ and~$v$ are in torsion with respect to~$X$. That is, $\gen{u}_{n,+}
  \subseteq X_n$, $\gen{v}_{n,-} \subseteq X_n$ $\gen{u}_{n,+} \cap
  \gen{v}_{n,-} = \emptyset$ and $\gen{u}_{n,+} \tl_{X_n} \gen{v}_{n,-}
  \tl_{X_n} \gen{u}_{n,+}$. By the last condition, there are~$r \in \N$ with~$r
  > 1$, and~$w_1,\ldots,w_r \in X_n$ such that
  \[
    w_1 \in \gen{u}_{n,+}
    ,\quad
    w_r \in \gen{v}_{n,-}
    ,\quad
    w_2,\ldots,w_{r-1} \not \in \gen{u}_{n,+} \cup \gen{v}_{n,-}
    ,\qtand
    w_i \tlone_{X_n} w_{i+1}
  \]
  for~$i \in \N^*_{r-1}$. Thus, there are~$\barw_1,\ldots,\barw_{r-1} \in
  P_{n-1}$ such that~$\barw_i \in w_i^+ \cap w_{i+1}^-$ for~$i \in \N^*_{r-1}$.
  Given~$w \in X_n$ with~$\barw_1 \in w^-$, we have~$w = w_2 \not \in
  \gen{u}_{n,+}$ since~$X_n$ is fork-free. Thus, $\barw_1 \in
  \gen{u}_{n,+}^{\pm} = \gen{u}_{n-1,+}$. Similarly,~$\barw_{r-1} \in
  \gen{v}_{n-1,-}$, so~$\gen{u}_{n-1,+} \jtl \gen{v}_{n-1,-}$. Likewise, we
  have~$\gen{v}_{n-1,+} \jtl \gen{u}_{n-1,-}$, which contradicts
  \forestaxiomcomp{4}. Hence,~$P$ satisfies \forestaxiom{4}.
\end{proof}


%% file: cell.tex
In this section, we show that the cells on a torsion-free complex have a
structure of a free \ocat. For the \ocategorical structure, we essentially have
to prove that the composition of two cells is a cell. In order to show this, we
adapt a result of~\cite{street1991parity} and prove a ``gluing theorem'', which
enables building an $(n{+}1)$\cell from an $n$\cell by gluing a set
of~$(n{+}1)$\generators. As a by-product, it also gives some properties
satisfied by composable cells, from which one can derive that the composition of
cells is well-defined. Then, for the freeness, we must first define the meaning
of \eq{freeness} that we use. The most natural notion for our context is the one of
polygraph~\cite{street1976limits,burroni1993higher}, which is a set of
generators for strict categories whose sources and targets are composites made
of other generators, and from which a free strict category can be generated.
This structure adequately encodes the fact that the source and target of each
generator of a torsion-free complexes are themselves pasting diagrams of
generators that must be composed.

In \Cref{ssec:gluing-thm}, we first state the gluing theorem
(\Cref{thm:cell-gluing}), after introducing the adequate terminology for the
\eq{gluing} of a set of generators on a cell. Then, in
\Cref{ssec:cells-are-cat}, we use this property to prove that the cells of a
torsion-free complex have a structure of an \ocat (\Cref{thm:cell-is-ocat}). In
\Cref{ssec:freeness-def}, we introduce the definition of \eq{freeness} for
strict categories that we are going to use for the \ocat of cells by recalling
the definition of polygraphs and the associated free strict category construction. In
\Cref{ssec:freeness}, we prove that the \ocat of cells of a torsion-free complex
is free in this sense, \ie is the free \ocat on a canonical polygraph (\Cref{coro:cells-are-freely-gen}).

\subsection{Gluing sets on cells}
\label{ssec:gluing-thm}

\paragraph{Gluings and activations}
Let~$P$ be an \ohg. Given~$n \in \N$, an $n$\precell~$X$ of~$P$ and a finite
set~$\glueset \subseteq P_{n+1}$, we say that~$\glueset$ is
\index{glueable}\emph{glueable on~$X$} if~$\glueset^{\mp} \subseteq X_n$. If so,
we call \index{gluing}\emph{gluing of~$\glueset$ on~$X$} the $(n{+}1)$\precell~$Y$
of~$P$ defined by
\[
  Y_{n+1} = \glueset
  ,\quad
  Y_{n,-} = X_{n}
  ,\quad
  Y_{n,+} = (X_{n} \cup \glueset^+) \setminus \glueset^-
  \qtand
  Y_{i,\epsilon} = X_{i,\epsilon}
\]
for~$i \in \N_n$ and~$\eps \in \set{-,+}$.
We denote~$Y$ \glossary(Glue){$\cellgluing(X,G)$}{the gluing of~$G$
  on~$X$}by~$\cellgluing(X,\glueset)$. Moreover, we call
\glossary(Act){$\cellactivation(X,\glueset)$}{the activation of~$\glueset$
  on~$X$}\index{activation}\emph{activation of~$\glueset$ on~$X$} the
$n$\precell~$\cellactivation(X,\glueset)$ defined by
\[
  \cellactivation(X,\glueset) = \ctgt_n(\cellgluing(X,\glueset))
\]
We say that~$\glueset$ is \emph{dually glueable on~$X$} when~$\glueset^\pm
\subseteq X_n$, and we define the dual gluing~$\dualcellgluing(X,\glueset)$ and
the dual activation~$\dualcellactivation(X,\glueset)$ similarly. For example,
consider the~\ohg~\eqref{eq:ex-non-segment} %
from \Cref{ssec:gen-pc} and recall there the definitions of~$X$ and~$Y$.
Then~$\set{A}$ is glueable on~$X$ and~$\cellgluing(X,\set{A}) = Y$,
and~$\cellactivation(X,\set{A})$ is the $2$\precell~$\bar X$ with
\begin{carray}{\bar X_2 = \set{\alpha_1,\alpha'_2,\alpha'_3,\alpha_4},}
  \bar X_{1,-} \;\;&=&\; \set{a,b}, \hspace*{5em}& \bar X_{1,+} \;\;&=&\; \set{c,d',e}, \\[3pt]
  \bar X_{0,-} \;\;&=&\; \set{x}, & \bar X_{0,+} \;\;&=&\; \set{z}.
\end{carray}%
Conversely,~$\set{A}$ is dually glueable on~$\bar X$, and we
have~$\dualcellgluing(\bar X,\set{A}) = Y$, and~$\dualcellactivation(\bar X,\set
A) = X$.

\paragraph{The gluing theorem}
We now state the ``gluing theorem''. It is an adapted version
of~\cite[Lemma~3.2]{street1991parity} which enables building new cells using the
gluing and activation operations. The theorem moreover gives additional results
concerning intersections with the source and the target sets of gluing sets.

\begin{restatable}{theo}{theoremcellgluing}%
  \label{thm:cell-gluing}
  Let~$P$ be an \ohg which satisfies Axioms~\forestax{0},~\forestax{1},
  \forestax{2} and~\forestax{3}. Given~${n \in \N}$, an $n$\cell~$X$ of~$P$
  and a finite fork-free set~$\glueset \subseteq P_{n+1}$ such that~$\glueset$
  is glueable on~$X$, we have that
  \begin{enumerate}[label=(\alph*),ref=(\alph*)]
  \item \label{thm:cell-gluing:act} $\cellactivation(X,\glueset)$
    is a cell and~$\glueset^+ \cap X_n = \emptyset$,
  \item \label{thm:cell-gluing:glue} $\cellgluing(X,\glueset)$
    is a cell,
  \item \label{thm:cell-gluing:inter} given a finite fork-free subset~$\dualglueset \subseteq
    P_{n+1}$ which is dually glueable on~$X$,
    $\dualglueset^- \cap \glueset^+ = \emptyset$.
  \end{enumerate}
  and dual properties hold when~$\glueset$ is dually glueable on~$X$.
\end{restatable}
\noindent A representation of the cells of the statement is shown in \Cref{fig:21}.
\input{fig/gluing-cell-rep}%
\begin{proof}
  See \Cref{sec:details-gluing-thm}.
\end{proof}

\subsection[Structure of \texorpdfstring{$\omega$}{omega}-category]{Structure of $\omega$-category}
\label{ssec:cells-are-cat}
Here, we prove that the cells on a torsion-free complex have a structure of an
\ocat. For this purpose, we first show that the composite of two cells is a
cell using \Cref{thm:cell-gluing} shown above. Then, we quickly verify that the
axioms of \ocats are satisfied, which is almost immediate by the definitions of
the operations on cells.

\smallpar We first handle the case of compositions of cells in codimension~$1$ with the
following result:
\begin{lem}
  \label{lem:cell-comp-n-1}
  Let~$P$ be an \ohg satisfying Axioms~\forestax{0}, \forestax{1}, \forestax{2}
  and~\forestax{3}. Given $n \in \N^*$ and two $n$\cells~$X,Y$ of~$P$ that are
  $(n{-}1)$\composable, the following hold:
  \begin{enumerate}[label=(\alph*),ref=(\alph*)]
  \item \label{lemcomposable.a} $X_{n}^- \cap Y_n^+ = \emptyset$,
  \item \label{lemcomposable.abis} $X_n \cap Y_n = \emptyset$,
  \item \label{lemcomposable.b} $X \comp_{n-1} Y$ is an $n$\cell of~$P$. 
  \end{enumerate}
\end{lem}

\begin{proof}
  Using \Thmr{cell-gluing}\ref{thm:cell-gluing:inter} with~$\ctgt_{n-1}
  (X)$,~$X_n$ and~$Y_n$, we get $X_{n}^- \cap Y_n^+ = \emptyset$. Moreover,
  \begin{align*}
    X_n^+ \cap Y_n^+ & = X_n^{\pm} \cap Y_n^+ & \text{(since~$X_n^- \cap Y_n^+ = \emptyset$)} \\
                     & \subseteq X_{n-1,+} \cap Y_n^+ \\
                     & = Y_{n-1,-} \cap Y_n^+ \\
                     & = \emptyset & \text{(by \Thmr{cell-gluing}\ref{thm:cell-gluing:act})\zbox.}
  \end{align*}
  By \forestaxiom{0}, it implies that~$X_n \cap Y_n = \emptyset$.
  Similarly,
    $X_n^- \cap Y_n^- = \emptyset$,
  so~$X_n \cup Y_n$ is fork-free. For~$X \comp_{n-1} Y$ to be a cell,~$X_n \cup
  Y_n$ must move~$X_{n-1,-}$ to~$Y_{n-1,+}$. But, since~$X$ and~$Y$ are cells
  and are $(n{-}1)$\composable, we know that~$X_n$ moves~$X_{n-1,-}$
  to~$X_{n-1,+}$,~$Y_n$ moves~$Y_{n-1,-}$ to~$Y_{n-1,+}$ and~$X_{n-1,+} =
  Y_{n-1,-}$. Since~$X_n^- \cap Y_n^+$, using \Lemr{street-2.3}, we get
  that~$X_n \cup Y_n$ moves~$X_{n-1,-}$ to~$Y_{n-1,+}$. Hence,~$X \comp_{n-1} Y$
  is a cell.
\end{proof}
\noindent We now handle the general case of compositions of cells:
\begin{lem}
  \label{lem:gencomposable}
  Let~$P$ be an \ohg satisfying Axioms~\forestax{0},~\forestax{1},~\forestax{2}
  and~\forestax{3}. Let $i,n \in \N$ with~$i < n$ and~$X,Y$ be two $n$\cells
  of~$P$ that are $i$\composable. Then,
  \begin{enumerate}[label=(\roman*),ref=(\roman*)]
  \item \label{lemgencomposable.a} for~$j \in \N$ with~$i < j \le
    n$,~$(X_{j,-}^- \cup X_{j,+}^-) \cap (Y_{j,-}^+ \cup Y_{j,+}^+) =
    \emptyset$,
  \item \label{lemgencomposable.b} $X \comp_i Y$ is a cell.
  \end{enumerate}
\end{lem}

\begin{proof}
  By induction on~$n - i$. If~$n - i= 1$, the properties follow from
  \Lemr{cell-comp-n-1}. So suppose that~$n - i > 1$. For~$\eps,\eta \in
  \set{-,+}$, by induction hypothesis with~$\csrctgt\epsilon_{n-1} (X)$
  and~$\csrctgt\eta_{n-1} (Y)$, we get $X_{n-1,\epsilon}^- \cap Y_{n-1,\eta}^+ =
  \emptyset$. Therefore, $(X_{n-1,-}^- \cup X_{n-1,+}^-) \cap (Y_{n-1,-}^+ \cup
  Y_{n-1,+}^+) = \emptyset$. We moreover obtain that $(X_{j,-}^- \cup X_{j,+}^-)
  \cap (Y_{j,-}^+ \cup Y_{j,+}^+) = \emptyset$ for~$j \in \N$ with~$i < j < n -
  1$. Let~$Z = \ctgt_{n-1} (X) \comp_{i} \csrc_{n-1} (Y)$. By induction,~$Z$ is
  a $(n{-}1)$\cell and $Z_{n-1} = X_{n-1,+} \cup Y_{n-1,-}$. Using
  \Thmr{cell-gluing}\ref{thm:cell-gluing:inter}, we get
  that  $X_n^- \cap Y_n^+ = \emptyset$
  which concludes the proof of \ref{lemgencomposable.a}.

  For \ref{lemgencomposable.b}, we already know that~$\csrc_{n-1} (X) \comp_i \csrc_{n-1} (Y)$
  and~$\ctgt_{n-1} (X) \comp_i \ctgt_{n-1} (Y)$ are cells by induction. So, in order to prove
  that~$X \comp_i Y$ is a cell, we just need to show that~$X_n \cup Y_n$ is
  fork-free and moves~${X_{n-1,-} \cup Y_{n-1,-}}$ to~${X_{n-1,+} \cup Y_{n-1,+}}$.
  But
  \begin{align*}
    X_n^+ \cap Y_n^+ & = X_n^{\pm} \cap Y_n^+ & \text{(by~\ref{lemgencomposable.a})}\\
                     & \subseteq Z_{n-1} \cap Y_n^+ \\
                     & = \emptyset & \text{(by \Thmr{cell-gluing}\ref{thm:cell-gluing:act})}
  \end{align*}
  and similarly,
    $X_n^- \cap Y_n^- = \emptyset$,
  so~$X_n \cup Y_n$ is fork-free. Using the dual of
  \Thmr{cell-gluing}\ref{thm:cell-gluing:act} with~$Z$ and~$X_n$, we get
  \[
    X_n^- \cap (X_{n-1,+} \cup Y_{n-1,-}) =X_n^- \cap Y_{n-1,-}= \emptyset.
  \]
  Similarly, if~$Z' = \csrc_{n-1} (X) \comp_i \csrc_{n-1} (Y)$ then~$Z'_{n-1} =
  X_{n-1,-} \cup Y_{n-1,-}$. Using \Thmr{cell-gluing}\ref{thm:cell-gluing:act}
  with~$Z'$ and~$X_n$, we have
  \[
    X_n^+ \cap (X_{n-1,-} \cup Y_{n-1,-}) = X_n^+ \cap Y_{n-1,-} =  \emptyset.
  \]
  Since~$X_n$ moves~$X_{n-1,-}$ to~$X_{n-1,+}$, using \Lemr{street-2.2}, we
  deduce that $X_n$ moves~$X_{n-1,-} \cup Y_{n-1,-}$ to~$X_{n-1,+} \cup
  Y_{n-1,-}$. Similarly, $Y_n$ moves~$X_{n-1,+} \cup Y_{n-1,-}$ to~$X_{n-1,+}
  \cup Y_{n-1,+}$. Since~$X_n^- \cap Y_n^+ = \emptyset$, by \Lemr{street-2.3},
  we have that $X_n \cup Y_n$ moves~$X_{n-1,-} \cup Y_{n-1,-}$ to~$X_{n-1,+}
  \cup Y_{n-1,+}$. Hence,~$X \comp_i Y$ is a cell.
\end{proof}
\noindent We can finally conclude that the \ocat of cells has a structure of \ocat given
by the identity and composition operations on cells:
\begin{theo}
  \label{thm:cell-is-ocat}
  Let~$P$ be a torsion-free complex. $(\Cell(P),\csrc,\ctgt,\unit{},\comp)$ is an \ocat.
\end{theo}
\begin{proof}
  We already know that~$\Cell(P)$ is a $\omega$\globular set. By
  \Lemr{gencomposable}, the composition operation~$\comp$ is well-defined on
  composable cells. Moreover, all the axioms of \ocats (given in
  \Cref{text:strcat-firstdef}), follow readily from the definitions
  of~$\csrc,\ctgt,\unit{}$ and~$\comp$. For example, consider the exchange law
  \Axr{cat:id-xch}. Given~$i,j,n \in \N$ with~$i < j \le n$ and~$X,X',Y,Y'
  \in \Cell(P)_n$ such that~${X,Y}$ are $i$\composable,~$X,X'$ are
  $j$\composable and~$Y,Y'$ are $j$\composable, let
  \[
    Z = (X \comp_j Y) \comp_i (X' \comp_j Y') \qtand Z' = (X \comp_i X') \comp_j
    (Y \comp_i Y').
  \]
  One then easily verifies that~$Z_{k,\eps} = Z'_{k,\eps}$ for~$k \le n$
  and~$\eps \in \set{-,+}$, so~$Z = Z'$. Thus,~$\Cell(P)$ satisfies
  \Axr{cat:id-xch}, and the other axioms are shown as easily.
\end{proof}
\begin{remark}
  \label{rem:cell-is-ocat-stronger}
  For the proof of \cref{thm:cell-is-ocat}, we did not use \forestaxiom{4}, so
  that the same property holds for an \ohg which only satisfies Axioms~\forestax{0},
  \forestax{1}, \forestax{2}, \forestax{3}.
\end{remark}

\subsection{The notion of freeness}
\label{ssec:freeness-def}

Our aim is to show that the \ocat of cells on a torsion-free complex is
\eq{free} on the generators of the \ohg. We now give a precise sense to the
notion of freeness that we want to use. It is based on the structure of
polygraph~\cite{street1976limits,burroni1993higher}: the latter describes
generators of strict categories of multiple dimensions, whose sources and
targets are composite of generators of lower dimensions. We recall its
definition following~\cite{burroni1993higher}, using the intermediate notion of
cellular extension. The latter describes sets of $(n{+}1)$\generators specified
on strict $n$\categories. A polygraph is then simply a tower of cellular extensions.

\paragraph{Cellular extensions}
\parlabel{text:cell-extensions}

Given $n \in \N$, an \emph{$n$\cellular extension} is a pair~$(C,S)$ where $C$ is an
$n$\category and $S$ is a set, together with two functions
\[
  \gsrc_n,\gtgt_n\co S \to C_n
\]
such
that, when~$n > 0$, $\csrctgt\eps_{n-1} \circ \gsrc_n = \csrctgt\eps_{n-1} \circ \gtgt_n$
for~$\eps \in \set{-,+}$. The set~$S$ is to be considered as a set of
$(n{+}1)$\generators. Given two $n$\cellular extensions~$(C,S)$ and~$(C',S')$, a
morphism between~$(C,S)$ and~$(C',S')$ is a pair~$(F,f)$ where
\[
  F \co C \to C' \in \nCat n
  \qtand
  f \co S \to S' \in \Set
\]
and such that~$\gsrctgt\eps_n \circ f = F_n \circ \gsrctgt\eps_n$ for~$\eps \in
\set{-,+}$. We write~$\nCatp n$ for the category of $n$\cellular extensions.
There is a canonical functor
\[
  \algtoce_n \co \nCat{n+1} \to \nCatp n
\]
which forgets the operations on the $(n{+}1)$\cells except the globular ones, and we have that
\begin{prop}
  \label{prop:algtoce-ra}
  The functor~$\algtoce_n$ admits a left adjoint.
\end{prop}
\begin{proof}
  By the equational definitions of strict categories, the functor~$\algtoce_n$
  is a functor induced by a morphism of sketches and thus has a left
  adjoint (see for example~\cite[Theorem~3.5]{barr2000toposes}).
\end{proof}
\noindent We write
\[
  \polextp[n] - - \co \nCatp n \to \nCat {n+1}
\]
for such a left adjoint, or even~$\polextp--$ when there is no ambiguity on~$n$.
This functor maps an $n$\cellular extension~$(C,S)$ to its \emph{free
  extension}~$\polextp C S$.
In fact, the functor~$\polextp --$ can be chosen so that the canonical morphism
\[
  C \to \restrict n {\polextp C S}
\]
is the identity for every $(C,S) \in \nCatp n$. The reason is that the theory of
strict categories is \emph{truncable} in the sense
of~\cite{batanin1998computads} (see for
example~\cite[Proposition~1.3.2.10]{forest:tel-03155192}). Given~$(C,S)\in
\nCatp n$ and~$g \in S$, we often abuse notation and write~$g$ for the embedding
of~$g$ in~$\polextp C S$.

\paragraph{Polygraphs}
\parlabel{text:polygraphs}

For $n \in \N$, we define inductively on~$n$ the notion of an
\emph{$n$\polygraph~$\P$} together with a \emph{free $n$\category~$\freecat{\P}$
  on~$\P$}:
\begin{itemize}
\item a $0$-polygraph~$\P$ is a set~$\P_0$ and the free $0$-category
  on~$\freecat{\P}$ is $\P_0$ (seen as a $0$-category),
  
\item an $(n{+}1)$-polygraph~$\P$ is given by an
  $n$-polygraph~$\restrictcat{\P}{n}$ together with an $n$-cellular
  extension~$(\restrictcat{\P}{n},\P_{n+1})$ and the
  free $(n{+}1)$-category~$\freecat{\P}$ on~$\P$ is the free extension
    $\freecat{(\restrictcat{\P}{n})}[\P_{n+1}]$.
\end{itemize}
By induction on~$n$, we naturally define a notion of morphism between $n$\polygraphs:
a morphism between $0$\polygraphs is simply a function between sets, and a morphism
of $(n{+}1)$\polygraphs is the data of a morphism between the underlying $n$\polygraphs together
with a morphism between the underlying $n$\cellular extensions. Thus, for every
$n\in \N$ we obtain
a category~$\nPol n$ of $n$\polygraphs, and a functor
\[
  \freecat-\co \nPol n \to \nCat n\pbox.
\]
\vskip-\belowdisplayskip
\begin{remark}
  An $n$\polygraph~$\P$ can alternatively be described as a diagram in~$\Set$ of
  the form
  \[
    \begin{tikzcd}[column sep=10ex,labels={inner sep=0.5pt}]
      \P_0\ar[d,"\polinj0"{inner sep=2pt}]
      &\P_1
      \ar[dl,shift right,"\gsrc_0"',pos=0.3]
      \ar[dl,shift left,"\gsrc_0",pos=0.3]\ar[d,"\polinj1"{inner sep=2pt}]
      &\P_2\ar[dl,shift right,"\gsrc_1"',pos=0.3]\ar[dl,shift left,"\gsrc_1",pos=0.3]\ar[d,"\polinj2"{inner sep=2pt}]
      &\ldots
      &\P_{n-1}
      \ar[dl,shift right,"\gsrc_{n-2}"',pos=0.3]
      \ar[dl,shift
      left,"\gsrc_{n-2}",pos=0.3]\ar[d,"\polinj{n}"{inner sep=2pt}]
      &\P_{n}\ar[dl,shift right,"\gsrc_{n-1}"',pos=0.3]\ar[dl,shift left,"\gsrc_{n-1}",pos=0.3]\\
      \freecat{\P_0}
      &
      \freecat{\P_1}
      \ar[l,shift right,"{\csrc_0}"']
      \ar[l,shift
      left,"{\ctgt_0}"]
      &\ldots
      \ar[l,shift
      right,"{\csrc_1}"']\ar[l,shift
      left,"{\ctgt_1}"]
      &\freecat{\P_{n-2}}
      &\ar[l,shift
      right,"{\csrc_{n-2}}"']\ar[l,shift
      left,"{\ctgt_{n-2}}"]\freecat{\P_{n-1}}
    \end{tikzcd}
  \]
  where, for~$i \in \N_{n-1}$, $\freecat\P_i$ is the set of cells freely
  generated on the generators of dimensions~$\le i$ with associated
  embedding~$\polinj i \co \P_i \to \freecat\P_i$, and such that
  \[
    {\csrc_i}\circ\gsrc_{i+1}={\csrc_i}\circ\gtgt_{i+1}
    \qqtand
    \ctgt_i\circ\gsrc_{i+1}=\ctgt_i\circ\gtgt_{i+1}
  \]
  for~$i \in \N_{n-1}$. This description of polygraphs can already be found in
  the original paper of Burroni~\cite{burroni1993higher}.
\end{remark}
These constructions naturally extend to $\omega$: an
\emph{$\omega$-polygraph~$\P$} is a sequence~$(\P^k)_{k \ge 0}$ where $\P^k$ is
a $k$-polygraph such that $\restrict k{(\P^{k+1})} = \P^k$ and the
\emph{free $\omega$-category on~$\P$} is defined by
\[
  \freecat\P = \colim_{k \to \omega} \inc[k]\omega {(\freecat{(\P^k)})}\pbox.
\]
where, for every~$k \in \N$,
\[
  \inc[k]\omega{(-)} \co \nCat k \to \nCat\omega
\]
is the left adjoint to the truncation functor $\restrict k - \co \nCat\omega \to
\nCat k$. The notion of morphism of \opol is defined as expected and we obtain a
category~$\oPol$ of \opols together with a functor~$\freecat-\co \oPol \to
\oCat$.

\subsection[Freeness of the \texorpdfstring{$\omega$}{omega}-category of
cells]{Freeness of the $\omega$-category of cells}
\label{ssec:freeness}

Here, we prove that the \ocat of cells on a torsion-free complex is free in the
sense introduced previously, \ie that it is isomorphic to the free \ocat on a
certain polygraph. For this purpose, we introduce the canonical cellular
extensions from which this polygraph is built from, and show inductively that the
adequate restrictions of the \ocat of cells are isomorphic to the free
extensions on these cellular extensions.

\paragraph{The canonical cellular extension}

Let~$P$ be a torsion-free complex~$P$. Given~$n \in \N$, there is an
$n$\cellular extension
\[
  \begin{tikzcd}[cramped,column sep=3.5em]
    \restrictcat{\Cell(P)}{n}
    &
    P_{n+1}
    \ar[l,shift right,"\csrc_n \circ \gen{-}"']
    \ar[l,shift left,"\ctgt_n \circ \gen{-}"]
  \end{tikzcd}
\]
where, for~$x \in P_{n+1}$ and~$\eps \in \set{-,+}$,~$\csrctgt\eps_n \circ
\gen{-} (x) = \csrctgt\eps_n (\gen{x})$, which is an $n$\cell
by~\forestaxiom{2}. We \glossary(CellPn+){$\defcellext$}{the strict
  $(n{+}1)$\category~$\polextp{\restrict
    n{\Cell(P)}}{P_{n+1}}$}write~$\defcellext$ for the $(n{+}1)$\category
\[
  \defcellext = \polextp{\restrict n{\Cell(P)}}{P_{n+1}}
\]
\ie the image of~$(\restrict n {\Cell(P)},P_{n+1}) \in \nCatp n$ by the
functor~$\polextp[n] -- \co \nCatp n \to \nCat{n+1}$. There is a morphism of
$n$\cellular extension
\[ 
  (\restrict n{\Cell(P)},P_{n+1}) \xto{(\unit {\restrict n {\Cell(P)}},\gen -)} (\restrict n {\Cell(P)},\Cell(P)_{n+1})
\]
which maps~$x \in P_{n+1}$ to~$\gen{-}(x) = \gen{x}$. 
\noindent By the universal property of~$\defcellext$ as a free extension,
it induces a unique $(n{+}1)$\functor
\[
  \ev^{n} \colon \defcellext \to \restrictcat{\Cell(P)}{n+1}
\]
often \glossary(eval){$\ev^{n}$, $\ev$}{the evaluation $(n{+}1)$\functor from
  $\defcellext$ to~$\restrictcat{\Cell(P)}{n+1}$}written~$\ev$ for conciseness,
such that
  $\restrict n {\ev^n} = \unit{\restrict n {\Cell(P)}}$
  and
  $\ev(g) = g$
for all~$g \in P_{n+1}$.

\paragraph[Freeness of \texorpdfstring{$\Cell(P)$}{Cell(P)}]{Freeness of $\bm{{\Cell(P)}}$}
We can now assert the freeness of the \ocat~$\Cell(P)$. First, we show that it
is inductively built from the canonical free extensions:
\begin{restatable}{theo}{theoremextfreeness}
  \label{thm:ext-freeness}
  Given a torsion-free complex~$P$, for~$n \in \N$, the
  $(n{+}1)$\functor~$\ev^n$ is an isomorphism between $\defcellext$ and
  $\restrict {n+1}{\Cell(P)}$.
\end{restatable}
\begin{proof}
  See the proof in \Cref{ssec:freeness-proof}.
\end{proof}
\noindent By an inductive argument, we conclude that the \ocat of cells is
freely generated on the \opol made from the atoms:
\begin{coro}
  \label{coro:cells-are-freely-gen}
  Given a torsion-free complex~$P$, there are unique polygraph~$\Q \in \nPol\omega$ and $\omega$\functor
  \[
    F \co
    \freecat\Q \to \Cell(P) \in \nCat\omega
  \]
  such that~$\Q_n = P_n$ for~$n \in \N$
  and~$F(g) = \gen g$ for~$g \in P$. Moreover,~$F$ is an isomorphism.
\end{coro}
\begin{proof}
  We show by induction on~$n \in \N$ that there are unique $n$\polygraph~$\Q^n$
  and morphism
  \[
    F^n \co \freecat{(\Q^n)} \to \restrict
    n {\Cell(P)}
  \]
  such that~$\Q^n_k = P_k$ for~$k \in \N$ and~$F^n(g) = \gen g$ for~$g \in
  \Q^n$, and that~$F^n$ is moreover an isomorphism. This is clear for~$n = 0$.
  So suppose that~$n > 0$. If~$\Q^n$ and~$F^n$ as above exist, then, by the
  unicity property of the induction hypothesis, we have~$\restrict {n-1} {\Q^n}
  = \Q^{n-1}$ and~$\restrict {n-1} {F^n} = F^{n-1}$. The $n$\functor~$F^n$ is
  then uniquely defined by the universal property of~$\freecat{(\Q^n)} =
  \polextp{\freecat{(\Q^{n-1})}}{\Q^n}$ given by \Cref{prop:algtoce-ra} knowing
  that~$F^n(g) = \gen g$ for~$g \in \Q^n_n$. Moreover, the $n$\polygraph
  structure on~$\Q^n$ is unique since
  \begin{equation}
    \label{coro:cells-are-freely-gen:gsrctgt}
    \gsrctgt\eps_{n-1}(g) = \finv {(F^{n-1})} \circ \csrctgt\eps_{n-1}(\gen g)
  \end{equation}
  for~$g \in \Q^n_n$ and~$\eps \in \set{-,+}$. Finally,~$F^n$ is an isomorphism
  since, by \Cref{thm:ext-freeness}, the functor~$\finv {(\ev^{n-1})} \circ F^n$ is the image
  by~$\polextp[n-1]--$ of the isomorphism
  \[
    (F^{n-1},\id {P_n}) \co (\freecat{(\Q^{n-1})},\Q^n_n) \to (\restrict {n-1}
    {\Cell(P)},P_{n}) \in \nCatp {n-1}
  \]
  so that the unicity of~$\Q^n$ and~$F^n$, and the fact that~$F^n$ is an
  isomorphism are proved. For existence, one defines the $n$\polygraph structure
  on~$\Q^n$ from the one on~$\Q^{n-1}$ and
  with~\eqref{coro:cells-are-freely-gen:gsrctgt}, and the $n$\functor~$F^n$ is
  then defined by extending~$F^{n-1}$, using the universal property
  of~$\freecat{(\Q^n)}$.
  By the definition of~$\nPol\omega$, we obtain unique
  $\omega$\polygraph~$\Q$ together with a unique $\omega$\functor $F\co
  \freecat\Q \to \Cell(P)$ as wanted.
\end{proof}


%% file: fig/gluing-cell-rep.tex
\begin{figure}
  \centering
\begin{tikzpicture}[scale=1.5,ampersand replacement=\&]
  \node (Q) at (0,2) {$\glueset$};
  \node (nw) at (-1,2) {};
  \node (ne) at (1,2) {};
  \node (Xn) at (-1,1) {$X_n$};
  \node (R) at (1,1) {$(X_n \cup \glueset^+) \setminus \glueset^-$};
  \node (Xn-1s) at (-1,0) {$X_{n-1,-}$};
  \node (Xn-1t) at (1,0) {$X_{n-1,+}$};
  \node (dots) at (0,-0.5) {\vdots};
  \node (X1s) at (-1,-1) {$X_{1,-}$};
  \node (X1t) at (1,-1) {$X_{1,+}$};
  \node (X0s) at (-1,-2) {$X_{0,-}$};
  \node (X0t) at (1,-2) {$X_{0,+}$};
  \node[green,right] at (dots -| Xn-1s.west) {$X$};
  \node[red,left] at (dots -| R.east) {$\cellactivation(X,\glueset)$};
  \node[blue,left] at (Q -| R.east) {$\cellgluing(X,\glueset)$};
  \draw[->] (X0s) .. controls (X1s) .. (X0t);
  \draw[->] (X0s) .. controls (X1t) .. (X0t);
  \draw[->] (Xn-1s) .. controls (Xn) .. (Xn-1t);
  \draw[->] (Xn-1s) .. controls (R) .. (Xn-1t);
  \draw[->] (Xn) .. controls (Q) .. (R);
  \draw[green] (X0s.south west) -| (Xn-1s.north west) |- (Xn.north west) --
  (Xn.north east -| Xn-1s.north east) |-
   (Xn-1t.north east -| R.east) |- (X0t.south east) -- cycle;
  \draw[red] ([xshift=-0.2em,yshift=-0.2em]X0s.south west) -| ([xshift=-0.2em,yshift=+0.2em]Xn-1s.north west) -|
  ([xshift=-0.2em,yshift=0.2em]R.north west) -- ([xshift=0.2em,yshift=0.2em]R.north east) |- cycle;
  \draw[blue] ([xshift=-0.4em,yshift=-0.4em]X0s.south west) -|
  ([xshift=-0.4em]Xn-1s.west) |-
  ([yshift=0.4em]Q.north) -| ([xshift=0.4em]R.east) |- cycle;
\end{tikzpicture}
\caption{Cells involved and their movements in \protect\Thmr{cell-gluing}}
  \label{fig:21}
\end{figure}

%% file: relating.tex
In this section, we relate all the introduced formalisms together. In
particular, we show that the formalism of torsion-free complexes is a Rosetta
stone which can express the other ones (after correcting the defect of parity
complexes and pasting schemes). Embedding parity complexes into torsion-free
complexes is almost direct, since they share the same definition of cells and
several axioms. However, additional developments are needed for translating
pasting schemes and augmented directed complexes into torsion-free complexes.
Indeed, in the first case, one needs to show that a definition of cells
analogous to the ones of pasting schemes can be used for torsion-free complexes
before being able to relate the axioms of the two formalisms. In the second
case, one needs to link the abelian group setting of augmented directed
complexes to the set setting of torsion-free complexes.

We first introduce two other set-based definitions of cells for torsion-free
complexes: \emph{\clwf fgs's} and \emph{\mwf fgs's} (\Cref{text:cl-max-cells}).
The former is similar to the well-formed fgs of pasting schemes, while the
latter is a convenient intermediate between the cells of torsion-free complexes
and \clwf fgs's. The \ocats of cells induced by these two other definitions is
then isomorphic to the one obtained with the initial definition
(\Cref{thm:ctoprinc-m-iso,thm:ctocl-m-iso}). Using the more natural definition
of cells as \emph{\clwf fgs's}, we give a characterization of polygraphs that
can be represented by torsion-free complex
(\Cref{prop:criterion-for-tfc-representability}). Next, we show the embeddings
of parity complexes (\Cref{ssec:enc-street}) and pasting schemes
(\Cref{ssec:enc-johnson}) into torsion-free complexes. Then, we develop the
relation between the set-based and group-based definitions of cells before
showing the embedding augmented directed complexes into torsion-free complexes
(\Cref{ssec:enc-steiner}). Finally, we illustrate that those are the only
embeddings between the formalisms by providing counter-examples to the other
ones (\Cref{ssec:incl-ce}).

\subsection{Closed and maximal cells}
\label{text:cl-max-cells}

\input{alt-rep}

\subsection{Embedding parity complexes}
\label{ssec:enc-street}
\input{enc-street}

\subsection{Embedding pasting schemes}
\label{ssec:enc-johnson}
\input{enc-johnson}

\subsection{Embedding augmented directed complexes}
\label{ssec:enc-steiner}
\input{enc-steiner}

\subsection{Absence of other embeddings}
\label{ssec:incl-ce}
\input{str-inc}


%% file: alt-rep.tex
In this section, we introduce two other set-based definitions of cells for
torsion-free complexes, namely \clwf fgs's and \mwf fgs's, together with identity
and compositions operations for them. We moreover provide translation functions
between the different definitions of cells, and show that the \ocats of cells
with the new definitions are isomorphic to the one with the original definition
of cells (\Cref{thm:ctoprinc-m-iso,thm:ctocl-m-iso}). Finally, using this different
representation, we characterize the polygraphs that can be represented by
torsion-free complexes (\Cref{prop:criterion-for-tfc-representability}).

\paragraph{Definitions}
\parlabel{text:cl-cells-max-cells}

Let~$P$ be an \ohg. Recall the definitions of fgs and closed fgs from
\Cref{ssec:johnson}. We \glossary(ClosedP){$\Closed(P)$}{the
  graded set of closed fgs's of~$P$}write
\[
  \Closed(P)
\]
for the graded set of closed fgs's of~$P$. Given an $n$-fgs $X$ of~$P$, $x \in
X$ is said to be \index{maximal!generator}\emph{maximal in $X$} when for all $y
\in P$ such that $x \clrel y$ and $x \neq y$, it holds that $y \not \in X$. We
write $\max(X)$ for the $n$-fgs of~$P$ made of the maximal elements of~$X$. The
$n$-fgs $X$ is then said to be \index{maximal!fgs}\emph{maximal} when $\max(X) =
X$. We \glossary(MaxP){$\Princ(P)$}{the graded set of maximal fgs's of the
  \ohg~$P$}write
\[
  \Princ(P)
\]
for the graded set of
maximal fgs. Given $n \in \N$ and $X$ an $n$-pre-cell of~$P$, we write $\cup X$
for the $n$-fgs of~$P$ given by
\[
  \cup X = \bigcup_{i \in \N_n} (X_{i,-} \cup X_{i,+})\zbox.
\]

\paragraph{Maximality lemma}

Let~$P$ be an \ohg. In order to relate the cells of~$\Cell(P)$ with the fgs's
of~$\Princ(P)$, we give here a simple criterion to characterize the maximal
elements in a cell of~$\Cell(P)$:
\begin{lem}[Maximality lemma]
  \label{lem:maximality}
  Suppose that~$P$ satisfies Axioms~\forestax{0}, \forestax{1},
  \forestax{2} and \forestax{3}. Let~${k,n \in \N}$ with~$k<n$ and $X \in
  \Cell(P)_n$. For~$x \in X_{k,-}$ (\resp~${x \in X_{k,+}}$) with~$x$ not
  maximal in~$\cup X$, we have~$x \in X_{k+1,-}^{\mp}$ (\resp~$x \in
  X_{k+1,+}^{\pm}$).
\end{lem}
\begin{proof}
  We prove this property by induction on~$l = n - k$. By symmetry, we only
  prove the case where $x \in X_{k,-}$. Since $x$ is not maximal, by definition
  of~$\clset$, there exist
  \[
    p \in \N^*,
    \quad
    \eta \in \set{-,+},
    \quad
    x_0,x_1,\ldots,x_p \in P
    \qtand
    \eps_1,\ldots,\eps_p \in \set{-,+}
  \]
  such that
  \hrule height0pt depth0pt\relax
  \[
    x_0 = x,
    \quad
    x_p \in X_{k+p,\eta}
    \qtand
    x_i \in x_{i+1}^{\eps_{i+1}}
    \quad
    \text{for $i \in \N_{p-1}$.}
  \]
  Suppose that $p = 1$. By \Lemr{street-2.1}, we have $X_{k,-} \cap
  X_{k+1,\eta}^+ = \emptyset$. Since $x \in x_1^{\eps_1}$ and $x_1 \in
  X_{k+1,\eta}$, we have $\eps_1 = -$ and $x \in X_{k+1,\eta}^{\mp}$. Hence,
  by \Lemr{move-eq}, $x \in X_{k+1,-}^{\mp}$.

  \smallpar Otherwise, suppose that $p > 1$. Let $y \in X_{k+p,\eta}$ be the
  smallest of~$X_{k+p,\eta}$ for~$\tl_{X_{k+p,\eta}}$ such that~${y \clrel
    x_{p-1}}$. If $x_{p-1} \in y^-$, then, by minimality of~$y$, there is no
  $\bar y \in X_{k+p,\eta}$ such that~${x_{p-1} \in \bar y^+}$. Therefore,
    $x_{p-1} \in
    X_{k+p,\eta}^\mp \subseteq X_{k+p-1,-}$.
  Hence, $x$ is not minimal in $\csrc_{k+p-1} (X)$ and we conclude by induction.
  We now consider the case~${x_{p-1} \in y^+}$. Let
  \[
    G = \set{z \in X_{k+p,\eta} \mid z \tl_{X_{k+p,\eta}} y} \cup
    \set{y} \qtand Y = \cellactivation(\csrc_{k+p-1} (X), G).
  \]
  We have $x \in Y_{k,-}$ and $x_{p-1} \in Y_{k+p-1}$. Moreover, by
  \Thmr{cell-gluing}, $Y$ is a cell. By induction hypothesis, we have~${x \in
    Y_{k+1,-}^{\mp}}$. Since $X_{k+1,-}$ and $Y_{k+1,-}$ both move $X_{k,-}$ to
  $X_{k,+}$, by \Lemr{move-eq}, we have~${x \in X_{k+1,-}^\mp}$ which concludes
  the proof.
\end{proof}
\noindent We then have a simple description of the set of maximal
elements of a cell of~$\Cell(P)$:
\begin{lem}
  \label{lem:cell-max-elts}
  Suppose that~$P$ satisfies Axioms~\forestax{0}, \forestax{1},
  \forestax{2} and \forestax{3}. Let $k,n \in \N$ with $k < n$, an
  $n$\cell~$X \in \Cell(P)_n$ and~$\eps \in
  \set{-,+}$. Then,
    $\max(\cup X) \cap P_k = X_{k,-} \cap X_{k,+}$.
\end{lem}
\begin{proof}
  By \Cref{lem:maximality,lem:unmoved-eq}, 
  \[
    \max(\cup X) \cap P_k = (X_{k,-} \setminus X_{k+1,-}^{\mp}) \cup (X_{k,+}
    \setminus X_{k+1,+}^{\pm}) = X_{k,-} \cap X_{k,+}. \qedhere.
  \]
\end{proof}

\paragraph{The translation functions}
\parlabel{par:rel-funs} We now provide \index{translation
  function}\emph{translation functions} between the three graded
sets~$\Cell(P)$,~$\Princ(P)$ and~$\Closed(P)$ and introduce several properties
on them. The functions we introduce are the ones represented on the diagram
\[
  \begin{tikzcd}[row sep=huge]
    & \Princ(P) \ar[dl,shift left,"\princtoc"] \ar[dr,shift left,"\princtocl"]& \\
    \PCell(P) \ar[rr,shift left,"\ctocl"] \ar[ur,shift left,"\ctoprinc"] & &
    \Closed(P) \ar[ul,shift left,"\cltoprinc"] \ar[ll,shift left,"\cltoc"]
  \end{tikzcd}
\]
and are defined as
\glossary(T){$\ctoprinc,\princtoc,\princtocl,\cltoprinc,\ctocl,\cltoc$}{translation
  functions between~$\PCell(P)$,~$\Princ(P)$ and~$\Princ(P)$}follows:
\begin{itemize}
\item $\ctoprinc \colon \PCell(P) \to \Princ(P)$ is defined by 
    $\ctoprinc(X) = \max(\cup X)$ for $X\in \PCell(P)$,
\item $\princtoc \colon \Princ(P) \to \PCell(P)$ is such that, for $n \in \N$
  and~$X \in \Princ(P)$, $\princtoc(X)$ is the $n$\nbd-pre-cell $Y$ of~$P$
  defined by $Y_n = X_n$, and, for~$i \in \N_{n-1}$,
  \[
    Y_{i,-} = X_i \cup Y_{i+1,-}^{\mp} \qquad\qquad Y_{i,+} = X_i \cup Y_{i+1,+}^{\pm}\zbox,
  \]
  
\item $\princtocl \colon \Princ(P) \to \Closed(P)$ is defined by
    $\princtocl(X) = \clset(X)$  for $X\in \Princ(P)$,
 
\item $\cltoprinc \colon \Closed(P) \to \Princ(P)$ is defined by
  $
    \cltoprinc(X) = \max(X)
    $
  for $X \in \Closed(P)$,
\item $\ctocl \colon \PCell(P) \to \Closed(P)$ is defined by
    $\ctocl(X) = \clset(\cup X)$  for $X \in \PCell(P)$,
  
\item $\cltoc \colon \Closed(P) \to \PCell(P)$ is defined by
    $\cltoc = \princtoc \circ \cltoprinc$.
\end{itemize}

\noindent These operations can be related to each other, as state the following lemmas.

\begin{prop}
  \label{prop:princ-cl-bij}
  We have $\princtocl \circ \cltoprinc = \id{\Closed(P)}$ and $\cltoprinc \circ
  \princtocl = \id{\Princ(P)}$.
\end{prop}
\begin{proof}
  Let $X \in \Closed(P)$. By the definitions, we have $\princtocl \circ
  \cltoprinc(X) \subseteq X$. Moreover, given $x \in X$, since $X$ is finite,
  there exists $y \in \max(X)$ with $y \clrel x$. It implies that $y \in
  \cltoprinc(X)$ and $x \in \princtocl \circ \cltoprinc(X)$. Therefore,
  $X \subseteq \princtocl \circ \cltoprinc (X)$.
  
  For the other equality, note that, for all $n$-fgs $X$ of~$P$, $\clset(X)$ has
  the same maximal elements as $X$. Thus, $\cltoprinc \circ \princtocl = \id{\Princ(P)}$.
\end{proof}
\begin{lem}
  \label{lem:ctoprinc-describe}
  Suppose that $P$ satisfies Axioms~\forestax{0}, \forestax{1},
  \forestax{2} and~\forestax{3}.
  Let $n \in \N$, $X \in \Cell(P)_n$ and $Y = \ctoprinc(X)$. Then,
    $Y_n = X_n$ and $Y_i = X_{i,-} \cap X_{i,+}$  for $i \in \N_{n-1}$.
\end{lem}
\begin{proof}
  This is a direct consequence of \Lemr{cell-max-elts}.
\end{proof}
\begin{prop}
  \label{prop:ctoprinc-retract}
  Suppose that $P$ satisfies Axioms~\forestax{0}, \forestax{1}, \forestax{2} and
  \forestax{3}. Then, given a cell~$X \in \Cell(P)$, we have ${\princtoc \circ
    \ctoprinc(X) = X}$.
\end{prop}

\begin{proof}
  Let $n \in \N$, $X \in \Cell(P)_n$, $Y = \ctoprinc(X)$ and $Z = \princtoc(Y)$.
  For $i \in \N_n$ and $\eps \in \set{-,+}$, we show that $X_{i,\eps} =
  Z_{i,\eps}$ by a decreasing induction on $i$. By \Lemr{ctoprinc-describe}, we
  have
    $Z_n = Y_n = X_n$
  and, for $i \in \N_{n-1}$, by \Lemr{move-canform}, we have
  \[
    Z_{i,-}= Y_i \cup Z_{i+1,-}^{\mp}
            = (X_{i,-} \cap X_{i,+}) \cup X_{i+1,-}^{\mp}
            = X_{i,-} \zbox.
          \]
  Similarly,
  $ Z_{i,+} = X_{i,+} $,
  so $X = Z$. Hence, $\princtoc \circ \ctoprinc(X) = X$.
\end{proof}

\begin{prop}
  \label{prop:ctocl-decomp}
  We have $\princtocl \circ \ctoprinc = \ctocl$.
\end{prop}
\begin{proof}
  It readily follows from the definitions.
\end{proof}

\paragraph{Sources and targets}
\parlabel{text:cl-pr-src-tgt}

Let $P$ be an \ohg. We now define source and target operations
for~$\Closed(P)$ and~$\Princ(P)$. Given $n \in \N^*$ and~$X \in \Closed(P)$, we define the
\glossary(droundb){$\clsrc_i,\cltgt_i$}{the source and target operations on
  closed fgs's}\index{source!of a
  closed fgs}\emph{source $\clsrc_{n-1}
  (X)$} (\resp \index{target!of a closed fgs}\emph{target $\cltgt_{n-1}(X)$}) of~$X$ as the closed $(n{-}1)$-fgs~$Y$
defined by
\[
  Y = \clset(X \setminus (X_n \cup \clset(X_n^+))) \quad \text{(\resp $\clset(X \setminus (X_n
    \cup \clset(X_n^-)))$).}
\]
Respectively, given $n \in \N^*$ and a maximal $n$-fgs~$X$, we define the
\glossary(droundm){$\prsrc_i,\prtgt_i$}{the source and target operations on
  maximal fgs's}\index{source!of a maximal fgs}\emph{source $\prsrc_{n-1} (X)$} (\resp
\index{target!of a maximal fgs}\emph{target~${\prtgt_{n-1} (X)}$})
of $X$ as the maximal $(n{-}1)$-fgs $Y$ such that
\[
  Y_{n-1} = X_{n-1} \cup X_{n}^{\mp}
  \text{ (\resp $\displaystyle Y_{n-1} = X_{n-1} \cup X_n^{\pm}$)}
  \qtand
  Y_{i} = X_i
  \quad
  \text{for~$i \in \N_{n-2}$.}
\]
\noindent We have the following compatibility results between the source and target operations
and the translation functions:
\begin{prop}
  \label{prop:ctoprinc-compat-srctgt}
  If $P$ satisfies Axioms~\forestax{0}, \forestax{1}, \forestax{2} and
  \forestax{3}, then, for~$n \in \N^*$, $\eps \in \set{-,+}$ and~${X \in
    \Cell(P)_n}$, we have
  $
    \ctoprinc(\csrctgt\eps_{n-1} (X)) = \prsrctgt\eps_{n-1} (\ctoprinc(X))
    $.
\end{prop}
\begin{proof}
  Let $Y = \ctoprinc ( \csrctgt\eps_{n-1}(X))$, $X' = \ctoprinc(X)$ and $Z =
  \prsrctgt\eps_{n-1}(X')$. By \Lemr{ctoprinc-describe}, we have
    $Y_{n-1} = X_{n-1, \eps}$
    and
    $Y_i = X_{i,-} \cap X_{i,+}$ for $i \in \N_{n-1}$.
  Moreover,
    $X'_{n} = X_n$
    and
    $X'_{i} = X_{i,-} \cap X_{i,+}$
    for $i \in \N_{n-1}$.
  If $\eps = -$, then, by \Lemr{move-canform},
  \begin{align*}
    Z_{n-1} = (X_{n-1,-} \cap X_{n-1,+}) \cup X_{n}^{\mp}
            = X_{n-1,-} 
  \end{align*}
  and $Z_i = X'_i = X_{i,-} \cap X_{i,+}$ for $i \in \N_{n-1}$, so $Y = Z$.
  Similarly, if $\eps = +$, we have $Y = Z$.
\end{proof}

\begin{prop}
  \label{prop:princtocl-compat-srctgt}
  For $n \in \N^*$, $\eps \in \set{-,+}$ and $X \in \Princ(P)_n$, we have
  \[
    \princtocl(\prsrctgt\eps_{n-1} (X)) = \clsrctgt\eps_{n-1} (\princtocl(X)).
  \]
\end{prop}
\begin{proof}
  By symmetry, it is sufficient to handle the case~$\eps = -$. Let
    $Y = \princtocl(\prsrc_{n-1} (X))$
    and
    $Z = \clsrc_{n-1} (\princtocl(X))$.
  By unfolding the definitions, we have
  \[
    \makebox[0cm][r]{$Y = \clset( (X \setminus X_n) \cup X_n^{\mp})$}
    \qqtand
    \makebox[0cm][l]{$Z = \clset( \clset(X) \setminus (X_n \cup
      \clset(X_n^+)))$.}
  \]
  In order to show that $Y \subseteq Z$, we only need to prove that $Y' \subseteq Z$
  where 
    $Y' = (X \setminus X_n) \cup X_n^{\mp}$. 
  First, we have that $Y' \subseteq \clset(X)$. Moreover,
  \begin{align*}
     Y' \cap (X_n \cup \clset(X_n^+)) &= ((X \setminus X_n) \cup X_n^{\mp}) \cap (X_n
       \cup \clset(X_n^+)) \\
    &= ((X \setminus X_n) \cup X_n^{\mp}) \cap
       \clset(X_n^+) 
    \\
    &=  (X \setminus X_n) \cap \clset(X_n^+) 
\\
    &= X \cap \clset(X_n^+) = \emptyset && \text{(since $X$ is maximal).}
  \end{align*}
  So $Y' \subseteq Z$, which implies that $Y \subseteq Z$. Similarly, in order
  to show that $Z \subseteq Y$, we only need to prove that $Z' \subseteq Y$
  where
  $Z' = \clset(X) \setminus (X_n \cup \clset(X_n^+))$.
But
\begin{gather*}
  Z' \subseteq Y \Leftrightarrow \clset(X) \subseteq Y \cup X_n \cup \clset(X_n^+)
  \shortintertext{and}
  \begin{aligned}
    Y \cup X_n \cup \clset(X_n^+) &= \clset( (X \setminus X_n) \cup X_n^{\mp}) \cup X_n \cup
    \clset(X_n^+) \\
    &= \clset( (X \setminus X_n) \cup X_n^{\mp} \cup X_n^+)
    \cup X_n \\
    &= \clset( (X \setminus X_n) \cup X_n^- \cup X_n^+) \cup
    X_n \\
    &= \clset((X \setminus X_n) \cup X_n^- \cup X_n^+ \cup X_n) = \clset(X).
  \end{aligned}
\end{gather*}
So $Z' \subseteq Y$, which implies that $Z \subseteq Y$. Hence, $Y = Z$, which
concludes the proof.
\end{proof}
\begin{prop}
  \label{prop:ctocl-compat-srctgt}
  \ifzerotothreehyp then, for~$n \in \N^*$, $\eps \in \set{-,+}$ and~${X \in
    \Cell(P)_n}$,
    $\ctocl(\csrctgt\eps_{n-1}(X)) = \clsrctgt\eps_{n-1} (\ctocl(X))$.
\end{prop}
\begin{proof}
  We compute that
  \begin{align*}
    \ctocl(\csrctgt\eps_{n-1}(X))
    &= \princtocl \circ \ctoprinc (\csrctgt\eps_{n-1}(X)) &&
                                                               \text{(by
                                                               \Cref{prop:ctocl-decomp})}
    \\
    &= \princtocl( \prsrctgt\eps_{n-1}(\ctoprinc (X))) && \text{(by
                                                     \Cref{prop:ctoprinc-compat-srctgt})}
    \\
    &= \clsrctgt\eps_{n-1}( \princtocl \circ \ctoprinc (X)) && \text{(by
                                                          \Cref{prop:princtocl-compat-srctgt})}
    \\
    &= \clsrctgt\eps_{n-1} (\ctocl (X))\zbox.&&\qedhere
  \end{align*}
\end{proof}

\paragraph{Identities and compositions}

Let $P$ be an \ohg. Here, we define identity and composition operations for the graded
sets~$\Princ(P)$ and~$\Closed(P)$, and prove some compatibility results with the
translations functions.

\medskip\noindent Given $n \in \N$ and a closed (\resp maximal) $n$-fgs~$X$, we
define the \index{identity operation!for closed or maximal fgs's}\emph{identity
  of~$X$} as the closed (\resp maximal) $(n{+}1)$-fgs~$\unitp{n+1}{}(X)$ defined
by
\[
  \unitp{n+1}{}(X)
  =
  (X_0,\ldots,X_n,\emptyset).
\]
Given $i,n \in \N$ with $i < n$ and two maximal $n$-fgs~$X,Y$, we define the
\glossary(.racomppr){$\comppr_i$}{the composition operation on maximal
  fgs's}\index{composition operation!for maximal fgs's}\emph{maximal
  $i$-composition of $X$ and $Y$} as the maximal $n$-fgs~$X \comppr_i Y$ defined
by
\[
  X \comppr_i Y = \max( \clset(X) \cup \clset(Y) ).
\]
Respectively, given $i,n \in \N$ with $i < n$ and two closed $n$-fgs~$X,Y$, we
define the \glossary(.rbcompcl){$\compcl_i$}{the composition operation on closed
  fgs's}\index{composition operation!for closed fgs's}\emph{closed
  $i$-composition of $X$ and $Y$} as the closed $n$-fgs~$X \compcl_i Y$ defined
by
\[
  X \compcl_i Y = X \cup Y.
\]
For simplicity, we sometimes write $\compcl$ (\resp $\comppr$) for $\compcl_i$
(\resp $\comppr_i$).
\noindent We now prove several compatibility results of the identity and
composition operations with the translation functions.
\begin{prop}
  \label{prop:ctocl-compat-id}
  For $n \in \N$ and an $n$-cell $X \in \Cell(P)$,
    $\ctocl(\unitp{n+1}{}(X)) = \unitp{n+1}{}(\ctocl(X))$.
\end{prop}
\begin{proof}
  It readily follows from the definitions.
\end{proof}
\begin{prop}
  \label{prop:ctoprinc-compat-id}
  For $n \in \N$ and an $n$-cell $X \in \Cell(P)$,
    $\ctoprinc(\unitp{n+1}{}(X)) = \unitp{n+1}{}(\ctoprinc(X))$.
\end{prop}
\begin{proof}
  It readily follows from the definitions.
\end{proof}
\begin{prop}
  \label{prop:ctocl-compat-comp}
  For $i,n \in \N$ with $i < n$, and $i$-composable $n$-cells $X$ and $Y$ in
  $\Cell(P)$,
  \[
    \ctocl(X \comp_i Y) = \ctocl(X) \compcl_i \ctocl(Y).
  \]
\end{prop}
\begin{proof}
  Let $Z = X \comp_i Y$. We have
    $\ctocl (X \comp_i Y) = \clset( \cup Z)$
  and
  \[
    \ctocl(X) \compcl_i \ctocl(Y) = \clset(\cup X) \cup \clset(\cup Y) = \clset(
    (\cup X) \cup (\cup Y)).
  \]
  By definition of composition, $\cup Z \subseteq (\cup X) \cup (\cup
  Y)$, so
    $\ctocl (X \comp_i Y) \subseteq \ctocl(X) \compcl_i \ctocl(Y)$.
  For the other inclusion, note that $X_{j,\eps} \subseteq Z_{j,\eps}$ for $j
  \in \N_n$ and $\eps\in\set{-,+}$ with $(j,\eps) \neq (i,+)$, and
  \[
    X_{i,+} = (X_{i,-} \cup X_{i+1,-}^+) \setminus
              X_{i+1,-}^- \subseteq Z_{i,-} \cup Z_{i+1,-}^+ \subseteq \clset( \cup Z)
    \]
  so $\cup X \subseteq \clset(\cup Z)$. Similarly, $\cup Y \subseteq \clset(\cup Z)$, thus
  $
    (\cup X) \cup (\cup Y) \subseteq \clset(\cup Z),
    $
  which implies that
  \[
    \ctocl(X) \compcl_i \ctocl(Y) \subseteq \ctocl (X \comp_i Y).
    \qedhere
  \]
\end{proof}

\begin{prop}
  \label{prop:cltoprinc-compat-comp}
  For $i,n \in \N$ with $i < n$, and $X,Y \in \Closed(P)_n$,
  \[
    \cltoprinc(X \compcl_i Y) = \cltoprinc(X) \comppr_i \cltoprinc(Y).
  \]
\end{prop}
\begin{proof}
  We compute that
  \begin{align*}
    \cltoprinc(X) \comppr_i \cltoprinc(Y)
    &= \max( \clset( \cltoprinc(X) ) \cup \clset(\cltoprinc(Y))) \\
    &= \max( X \cup Y ) && \text{(by \Cref{prop:princ-cl-bij})} \\
    &= \cltoprinc ( X \compcl_i Y)\zbox. &&\qedhere
  \end{align*}
\end{proof}

\begin{prop}
  \label{prop:ctoprinc-compat-comp}
  For $i,n \in \N$ with $i < n$, and $i$-composable $n$-cells $X$ and $Y$
  of~$P$,
  \[
    \ctoprinc(X \comp_i Y) = \ctoprinc(X) \comppr_i \ctoprinc(Y).
  \]
\end{prop}
\begin{proof}
  We compute that
  \begin{align*}
    \ctoprinc(X \comp_i Y)
    &= \cltoprinc \circ \ctocl(X \comp_i Y) && \text{(by
                                               \Cref{prop:princ-cl-bij,prop:ctocl-decomp})}\\
    &= \cltoprinc ( \ctocl(X) \compcl_i \ctocl(Y)) && \text{(by
                                                    \Cref{prop:ctocl-compat-srctgt})} \\
    &=  \cltoprinc \circ \ctocl(X) \comppr_i \cltoprinc \circ \ctocl(Y) &&
                                                                         \text{(by
                                                                         \cref{prop:cltoprinc-compat-comp})}
    \\
    &= \ctoprinc(X) \comppr_i \ctoprinc(Y) && \text{(by
                                            \Cref{prop:princ-cl-bij,prop:ctocl-decomp})\zbox.}
  \end{align*}
\end{proof}

\paragraph{Well-formed cells}

We defined above source, target, identity and composition operations for both
$\Closed(P)$ and $\Princ(P)$. However, these operations are not expected to
equip the graded sets~$\Closed(P)$ and~$\Princ(P)$ with a structure of \ocat (in
fact, not even a structure of $\omega$\globular set). In order to obtain an
\ocat, we need to restrict to subsets of ``well-formed'' elements
of~$\Closed(P)$ and~$\Princ(P)$. Then, we can show that the two induced \ocat of
cells are isomorphic to~$\Cell(P)$.

\medskip\noindent Let~$P$ be an \ohg. Given $n \in \N$ and $X \in \Closed(P)_n$,
we say that $X$ is \index{closed-well-formed}\emph{\clwf} when
\begin{itemize}
\item $X_n$ is fork-free,
\item $\clsrc_{n-1}(X)$ and $\cltgt_{n-1}(X)$ are \clwf,
\item if $n \ge 2$, $\clsrc_{n-2} \circ \clsrc_{n-1} (X) = \clsrc_{n-2} \circ \cltgt_{n-1}(X)$ and 
$\cltgt_{n-2} \circ \clsrc_{n-1} (X) = \cltgt_{n-2} \circ \cltgt_{n-1}(X)$.
\end{itemize}
We \glossary(ClosedWFP){$\MClosed(P)$}{the graded set of closed-well-formed
  fgs's of the \ohg~$P$}write~$\MClosed(P)$ for the graded set of \clwf fgs of~$P$. Respectively,
given $n \in \N$ and $X \in \Princ(P)_n$, we say that $X$ is \index{maximal-well-formed}\emph{\mwf} when
\begin{itemize}
\item $X_n$ is fork-free,
\item $\prsrc_{n-1}(X)$ and $\prtgt_{n-1}(X)$ are \mwf,
\item if $n \ge 2$, $\prsrc_{n-2} \circ \prsrc_{n-1} (X) = \prsrc_{n-2} \circ \prtgt_{n-1}(X)$ and 
$\prtgt_{n-2} \circ \prsrc_{n-1} (X) = \prtgt_{n-2} \circ \prtgt_{n-1}(X)$.
\end{itemize}
We \glossary(MaxWFP){$\MPrinc(P)$}{the graded set of maximal-well-formed fgs's
  of the \ohg~$P$}write~$\MPrinc(P)$ for the graded set of \mwf fgs of~$P$. We now aim at
proving that both~$\MClosed(P)$ and~$\MPrinc(P)$ are \ocats isomorphic
to~$\Cell(P)$ when~$P$ satisfies enough axioms of torsion-free complexes. We
first show this property for~$\MPrinc(P)$ after introducing several technical
results.

\begin{lem}
  \label{lem:ctoprinc-mprinc}
  \ifzerotothreehyp then, for $n \in \N$ and ~${X \in \Cell(P)_n}$, we
  have~${\ctoprinc (X) \in \MPrinc(P)_n}$.
\end{lem}
\begin{proof}
  We proceed by induction on~$n$. If $n = 0$, the result is trivial. So suppose
  that~${n > 0}$ and let~${Y = \ctoprinc (X)}$. Since $Y_n = X_n$, $Y_n$ is
  fork-free. Moreover, by \Cref{prop:ctoprinc-compat-srctgt},
  we have
  \[
    \prsrctgt\eps_{n-1} (Y) = \ctoprinc (\csrctgt\eps_{n-1} (X))
    \quad
    \text{for $\eps \in \set{-,+}$.}
  \]
  By the induction hypothesis, $\prsrctgt\eps_{n-1} (Y)$ is
  \mwf. And, when $n \ge 2$, for~$\eta \in \set{-,+}$, we have
  \begin{align*}
    \prsrctgt\eta_{n-2} \circ \prsrc_{n-1} (Y)
    &= \ctoprinc (\csrctgt\eta_{n-2} \circ \csrc_{n-1} (X)) && \text{(by \Cref{prop:ctoprinc-compat-srctgt})}\\
    &= \ctoprinc (\csrctgt\eta_{n-2} \circ \ctgt_{n-1} (X)) \\
    &=\prsrctgt\eta_{n-2} \circ \prtgt_{n-1}(Y).
  \end{align*}
  Hence, $Y$ is \mwf.
\end{proof}
\begin{lem}
  \label{lem:ctoprinc-onto}
  \ifzerotothreehyp then, given~$n\in\N$ and an fgs~$X \in \MPrinc(P)_n$, there exists an
  $n$\nbd-cell~$Y\in \Cell(P)_n$ such that $\ctoprinc (Y) = X$.
\end{lem}
\begin{proof}
  We proceed by induction on~$n$. If~${n = 0}$, the result is trivial. So
  suppose that~${n > 0}$. By induction, let $S,T \in \Cell(P)_{n-1}$ be such
  that $\ctoprinc (S) = \prsrc_{n-1} (X)$ and $\ctoprinc (T) = \prtgt_{n-1} (X)$. When $n \ge
  2$, for~${\eps \in \set{-,+}}$, we have
  \begin{align*}
    \csrctgt\eps_{n-2} (S)
    &= \princtoc \circ \ctoprinc ( \csrctgt\eps_{n-2} (S) ) && \text{(by
                                                           \Cref{prop:ctoprinc-retract})}
    \\
    &= \princtoc ( \prsrctgt\eps_{n-2} (\ctoprinc (S))) && \text{(by
                                                   \Cref{prop:ctoprinc-compat-srctgt})}
    \\
    &= \princtoc ( \prsrctgt\eps_{n-2} \circ \prsrc_{n-1}(X)) \\
    &= \princtoc ( \prsrctgt\eps_{n-2} \circ \prtgt_{n-1}(X))&& \text{(because $X$ is
      \mwf)} \\
                    &= \princtoc ( \prsrctgt\eps_{n-2} (\ctoprinc (T))) \\
                    &= \princtoc \circ \ctoprinc (\csrctgt\eps_{n-2} (T)) = \csrctgt\eps_{n-2} (T).
  \end{align*}
  Moreover,
  \vskip-\abovedisplayskip\vskip-\baselineskip\vskip\abovedisplayshortskip
  \begin{align*}
    (S_{n-1} \cup X_n^+) \setminus X_n^- &= (X_{n-1} \cup X_n^{\mp} \cup X_n^+)
                                           \setminus X_n^- \\
                                         &= X_{n-1} \cup X_n^{\pm} = T_{n-1}.
  \end{align*}
  Similarly, $(T_{n-1} \cup X_n^-) \setminus X_n^+ = S_{n-1}$ so $X_n$ moves
  $S_{n-1}$ to $T_{n-1}$. Thus, the $n$-pre-cell~$Y$ defined by
    $Y_n = X_n$,
    $Y_{n-1,-} = S_{n-1}$,
    $Y_{n-1,+} = T_{n-1}$
    and
    $Y_{i,\delta} = S_{i,\delta}$
    for $i \in \N_{n-2}$ and $\delta \in \set{-,+}$,
  is an $n$\cell. Let $Z = \ctoprinc(Y)$. We have $Z_n = X_n$ and
  \begin{align*}
    \prsrc_{n-1} (Z) &= \prsrc_{n-1} (\ctoprinc(Y)) \\
    &= \ctoprinc (\csrc_{n-1} (Y)) && \text{(by \Cref{prop:ctoprinc-compat-srctgt})}
    \\
             &= \ctoprinc(S) = \prsrc_{n-1} (X)\zbox.
  \end{align*}
  So, by definition of $\prsrc$, we have
    $Z_{n-1} \cup X_{n}^{\mp} = X_{n-1} \cup X_{n}^{\mp}$
    and
    $Z_i = X_i$
    for $i \in \N_{n-2}$.
  Since $X$ and $Z$ are maximal, we have
    $X_{n-1} \cap X_{n}^{\mp} = Z_{n-1} \cap X_{n}^{\mp} = \emptyset$.
  Hence, $X_{n-1} = Z_{n-1}$ and $X = Z = \ctoprinc(Y)$ which concludes the
  proof.
\end{proof}
\begin{lem}
  \label{prop:ctoprinc-m-bij}
  \ifzerotothreehyp then, $\ctoprinc$ induces a bijection between~$\Cell(P)$ and~$\MPrinc(P)$.
\end{lem}
\begin{proof}
  By \Lemr{ctoprinc-onto}, $\ctoprinc \colon \Cell(P) \to \MPrinc(P)$ is
  surjective and, by \Cref{prop:ctoprinc-retract}, it is injective, so it is
  bijective.
\end{proof}
\noindent We can now deduce that \mwf fgs's are an adequate alternative
definition of cells for torsion-free complexes:

\begin{theo}
  \label{thm:ctoprinc-m-iso}
  \ifzerotothreehyp then, $\MPrinc(P)$ is an \ocat and $\ctoprinc$ induces an
  isomorphism between $\Cell(P)$ and~$\MPrinc(P)$.
\end{theo}
\begin{proof}
  By definition of $\MPrinc(P)$, the functions $\prsrc_k,\prtgt_k$ for $k \in \N$
  equip $\MPrinc(P)$ with a structure of $\omega$\globular set. We first prove
  that the composition operation~$\comppr$ restricts to $\MPrinc(P)$. Let~${i, n
    \in \N}$~with $i < n$, and $X,Y \in \MPrinc(P)_n$ be such that $\prtgt_i (X) =
  \prsrc_i (Y)$. By \cref{prop:ctoprinc-m-bij}, there exist $X',Y' \in \Cell(P)_n$ such
  that $\ctoprinc(X') = X$ and $\ctoprinc(Y') = Y$. By
  \Cref{prop:ctoprinc-compat-srctgt}, we have
  \[
    \ctoprinc(\ctgt_i (X')) = \prtgt_i (X) = \prsrc_i (Y) = \ctoprinc(\csrc_i (Y')),
  \]
  and, by \cref{prop:ctoprinc-m-bij}, $\ctgt_i (X') = \csrc_i (Y')$ so~$X'$ and~$Y'$
  are $i$-composable. By \cref{prop:ctoprinc-m-bij}, we have~${\ctoprinc (X' \comp_i Y') \in
  \MPrinc(P)}$ and, by \cref{prop:ctoprinc-compat-comp}, ${X \comppr_i Y \in
    \MPrinc(P)}$.

  By
  \cref{prop:ctoprinc-compat-srctgt,prop:ctoprinc-compat-id,prop:ctoprinc-compat-comp},
  $\ctoprinc$ commutes with the source, target, identity and composition
  operations and is a bijection when restricted to $\MPrinc(P)$, so that
  $\MPrinc(P)$ is an \ocat since $\Cell(P)$ is (by
  \cref{thm:cell-is-ocat,rem:cell-is-ocat-stronger}), and $\ctoprinc$ induces an
  isomorphism of \ocats.
\end{proof}
\noindent We prove a similar property for \clwf fgs's after showing some
technical results.
\begin{lem}
  \label{prop:princtocl-m-bij}
  $\princtocl$ induces a bijection between $\MPrinc(P)$ and $\MClosed(P)$.
\end{lem}
\begin{proof}
  We already know that $\princtocl$ is a bijection by \Cref{prop:princ-cl-bij}.
  For~${n \in \N}$, we show that $\princtocl$ sends a \mwf $n$-fgs $X$ to a
  \clwf $n$-fgs by induction on $n$. If~${n = 0}$, the result is trivial. So
  suppose that~${n > 0}$. Let~${Y = \princtocl (X)}$. Then, $Y_n = X_n$ is
  fork-free and, for~${\eps \in \set{-,+}}$, we have~${\clsrctgt\eps_{n-1} (Y) =
    \princtocl (\prsrctgt\eps_{n-1} (X))}$ by
  \Cref{prop:princtocl-compat-srctgt}, and it is \clwf by induction. Moreover,
  when $n \ge 2$,
  \begin{align*}
    \clsrctgt\eps_{n-2} \circ \clsrc_{n-1} (Y)
    &= \princtocl (\prsrctgt\eps_{n-2} \circ \prsrc_{n-1}(X))
    && \text{(by \Cref{prop:princtocl-compat-srctgt})}
    \\
    &= \princtocl (\prsrctgt\eps_{n-2} \circ \prtgt_{n-1} (X)) \\
    &= \clsrctgt\eps_{n-2} \circ \cltgt_{n-1}(Y)
  \end{align*}
  so $Y$ is \clwf. Similarly, $\cltoprinc$ sends \clwf fgs to \mwf fgs, which
  concludes the proof.
\end{proof}
\begin{lem}
  \label{prop:ctocl-m-bij}
  \ifzerotothreehyp then, $\ctocl$ induces a bijection between~$\Cell(P)$
  and~$\MClosed(P)$.
\end{lem}
\begin{proof}
  The result is a consequence of \cref{prop:ctocl-decomp,prop:ctoprinc-m-bij,prop:princtocl-m-bij}.
\end{proof}
\noindent We can now conclude that \clwf fgs's are an adequate alternative
definition of cells for torsion-free complexes:
\begin{theo}
  \label{thm:ctocl-m-iso}
  \ifzerotothreehyp then, $\MClosed(P)$ is an \ocat and $\ctocl$ induces an
  isomorphism between~$\Cell(P)$ and~$\MClosed(P)$.
\end{theo}
\begin{proof}
  By a proof similar to the one of \Thmr{ctoprinc-m-iso}, using
  \Cref{prop:ctocl-compat-srctgt,prop:ctocl-m-bij,prop:ctocl-compat-id,prop:ctocl-compat-comp}.
\end{proof}

\paragraph{From polygraphs to torsion-free complexes}

We saw earlier (\Cref{coro:cells-are-freely-gen}) that torsion-free complexes
induce free \ocats on a canonical \opol. However, in practice, we are often
interested in the inverse operation, \ie representing the cells of an \ocat
freely generated on an \opol by the cells of a torsion-free complex. Here, we
define the \ohg~$\poltohyp\P$ associated to an \opol~$\P$ and, in the case
where~$\poltohyp\P$ is a torsion-free complex, give conditions under which the
\ocat~$\MClosed(\poltohyp\P)$ is isomorphic to the free \ocat~$\freecat\P$. In
order to define~$\poltohyp\P$, we will use the support function introduced by
Makkai for free strict categories~\cite{makkai2005word}. Given a polygraph~$\P$,
this function maps a cell of~$\freecat\P$ to the set of generators \eq{it
  involves}. We first recall its definition before dealing with the other
matters.

Given a set~$S$, we write~$\finsets S$ for the set of finite subsets of~$S$.
Given $n \in \Ninf$ and an $n$\polygraph~$\P$, we define the
\glossary(supp){$\supp[\P]$}{the support function for the
  polygraph~$\P$}\index{support function}\emph{support function}
\[
  \supp[\P]\co \freecat\P
  \to \finsets {\sqcup_{i \in \N_n} \P_i}
\]
or simply,~$\supp$, as the unique function such that, given $u \in \freecat\P$,
\begin{itemize}
\item $\supp(u) = \set{g}$ when $u = g$ for some~$g \in \P_0$,
\item $\supp(u) = \set{g}\cup\supp(\csrc_{k-1}(g))\cup\supp(\ctgt_{k-1}(u))$
  when~$u = g$ for some~$k \in \N^*_n$ and~$g \in \P_k$,
\item $\supp(u) = \supp(u')$
  when~$u = \unit {u'}$ for some~$u' \in \freecat\P$,
\item $\supp(u) = \supp(u_1) \cup \supp(u_2)$ when $u = u_1 \comp_i u_2$ for
  some $i,k \in \N^*_n$ with~$i < k$ and $i$\composable~$u_1,u_2\in \freecat\P_k$
\end{itemize}
The above definition completely defines~$\supp$, since the cells of~$\freecat\P$
are precisely the classes of formal composites of generators of~$\P$. It can
moreover be shown well-defined (see the original proof of
Makkai~\cite[Lemma~5]{makkai2005word} or
\cite[Proposition~2.4.3.2]{forest:tel-03155192}).

 Given $\P \in \nPol\omega$,
we define an \glossary(.rcH){$\poltohyp\P$}{the \ohg associated to the
  polygraph~$\P$}\ohg~$\poltohyp \P$ by putting~$\poltohyp \P_n = \P_n$ for $n
\in \N$ and, when~$n > 0$,
\[
  g^- = \supp(\gsrc_{n-1}(g)) \cap \P_{n-1}
  \qquad
  g^+ = \supp(\gtgt_{n-1}(g)) \cap \P_{n-1}
\]
for~$g \in \poltohyp\P_n$. Under this definition, $\supp[\P]$ can be seen as a
function~$\freecat\P \to \Closed(\poltohyp \P)$. We then have the following
criterion to know whether $\freecat\P$ can be faithfully represented by the
\clwf fgs's of~$\poltohyp\P$:

\begin{theo}
  \label{prop:criterion-for-tfc-representability}
  Let $\P \in \nPol\omega$ such that $\poltohyp\P$ is a torsion-free complex.
  Then,~$\supp[\P]$ is the underlying function of an $\omega$\functor~$F\co
  \freecat\P \to \MClosed(\poltohyp \P)$ if and only if, for $n \in \N^*$, $g
  \in \P_n$ and~$\eps \in \set{-,+}$, we have
    $\supp(\gsrctgt\eps_{n-1}(g)) = \clset(g^\eps)$.
  In this case, $F$ is moreover an isomorphism.
\end{theo}
\begin{remark}
  If the condition of \Cref{prop:criterion-for-tfc-representability} is
  satisfied, then $\cltoc \circ F \co \freecat\P \to \Cell(\poltohyp\P)$ is the
  unique isomorphism given by \Cref{coro:cells-are-freely-gen} which maps~$g \in
  \P$ to~$\gen g \in \Cell(\poltohyp \P)$.
\end{remark}
\begin{proof}
  If $\supp[\P]$ induces an $\omega$\functor $F\co \freecat\P \to \MClosed(\poltohyp
  \P)$, then we have
  \begin{align*}
    \supp(\gsrctgt\eps_{n-1}(g)) 
    &= F(\gsrctgt\eps_{n-1}(g)) \\
    &= \clsrctgt\eps_{n-1}(F(g)) \\
    &= \clsrctgt\eps_{n-1}(\clset(g)) \\
    &= \clset(g^\eps) & \text{(by definition of $\clsrctgt\eps_{n-1}$)}
  \end{align*}
  which proves the necessity. For sufficiency, we prove by induction on~$n \in
  \N$ that $\supp[\P]$ is the underlying function of an $n$\functor~$F^n \co
  \restrict n {(\freecat \P)} \to \restrict n {\MClosed(\poltohyp \P)}$. This is
  clear for $n = 0$, and, when $n > 0$, we define~$F^n$ by extending~$F^{n-1}$
  and so that~$F^n(g) = \clset(g)$ using the universal property of $\restrict n
  {(\freecat \P)} = \polextp{\restrict {n-1}{(\freecat \P)}}{\P_n}$. This is
  possible since the condition of the statement implies that
  \[
    F^{n-1}(\gsrctgt\eps_{n-1}(g)) = \clsrctgt\eps_{n-1}(\clset(g))
  \]
  for $g \in \P_n$ and $\eps \in \set{-,+}$. We then obtain an $\omega$\functor
  $F \co \freecat\P \to \MClosed(\poltohyp \P)$ using
  \Cref{rem:ocat-limit}, which satisfies that $F(g) = \clset(g)$ for~$g
  \in \P$. Then, by \Cref{thm:ctocl-m-iso}, $\cltoc \circ F$ is an
  $\omega$\functor $\freecat\P \to \Cell(P)$ which maps~$g$ to~$\gen g$. It is
  then an isomorphism by \Cref{coro:cells-are-freely-gen}, so that~$F$ is an
  isomorphism too.
\end{proof}

\begin{example}
  Let $\P$ be the $\omega$\polygraph with
  \begin{gather*}
    \P_0 = \set{x,y,z}\quad \P_1 = \set{f \co x \to y\; g,g' \co y \to z}
    \quad\P_2 = \set{\alpha,\alpha' \co g \To g'}
    \\
    \P_3 = \set{ A \co \unitp 2 f \comp_0 \alpha \To \unitp 2 f \comp_0 \alpha'}
  \end{gather*}
  and~$\P_k = \emptyset$ for~$k \in \N$ with $k \ge 4$ as in
  \[
    \begin{tikzcd}[column sep=3em]
      x
      \ar[r,"f"]
      &
      y
      \ar[r,bend left=70,"g",""{auto=false,name=top}]
      \ar[r,bend right=70,"g'"',""{auto=false,name=bot}]
      \ar[from=top,to=bot,phantom,"\alpha\!\Downarrow\,\Downarrow\!\alpha'"]
      &
      z
    \end{tikzcd}
    \qtand
    \begin{tikzcd}[column sep=4em]
      x
      \ar[r,bend left=70,"f \comp_0 g",""{auto=false,name=top}]
      \ar[r,bend right=70,"f \comp_0 g'"',""{auto=false,name=bot}]
      \ar[from=top,to=bot,phantom,"\Downarrow\!\unitp 2 f\!\comp_0\!\alpha"]
      &
      z
    \end{tikzcd}
    \xTO{A}
    \begin{tikzcd}[column sep=4em]
      x
      \ar[r,bend left=70,"f \comp_0 g",""{auto=false,name=top}]
      \ar[r,bend right=70,"f \comp_0 g'"',""{auto=false,name=bot}]
      \ar[from=top,to=bot,phantom,"\Downarrow\!\unitp 2 f\!\comp_0\!\alpha'"]
      &
      z
    \end{tikzcd}
    \pbox.
  \]
  We can verify that~$\poltohyp\P$ is a torsion-free complex. But, by
  \Cref{prop:criterion-for-tfc-representability}, the function~$\supp[\P]$ does
  not induce an $\omega$\functor $\freecat\P \to \MClosed(\poltohyp\P)$ since
  \[
    \supp(\gsrc_2(A)) = \set{x,y,z,f,g,g',\alpha} \neq \set{y,z,g,g',\alpha}
    = \clset(A^-)\zbox.
  \]
  However, by considering a modified version of $\P$ where
    $\P_3 = \set{A \co \alpha \To \alpha'}$
  it can be verified that $\poltohyp\P$ is still a torsion-free complex and
  that, by \Cref{prop:criterion-for-tfc-representability}, the
  function~$\supp[\P]$ induces an $\omega$\functor $\freecat\P \to
  \MClosed(\poltohyp\P)$ which is an isomorphism.
\end{example}


%% file: enc-street.tex
In this section, we show that parity complexes are a particular case of
torsion-free complexes, under two reasonable caveats. Firstly, since parity
complexes do not require all the generators to be relevant, there are parity
complexes that are not torsion-free complexes. But, by~\cite[Theorem
4.2]{street1991parity}, irrelevant generators of a parity complex~$P$ do not
play any role in the generated $\omega$\category~$\Cell(P)$, so that, by
restraining $P$ to the \ohg~$\bar P$ of relevant generators, we have~${\Cell(P)
  = \Cell(\bar P)}$. Thus, it is reasonable to assume that all the parity
complexes we are considering for embedding in torsion-free complexes have
relevant generators, \ie satisfy \forestaxiom{2}. Secondly, as discussed in
\Cref{ssec:street}, general parity complexes are not freely generated by their
atoms and, since the latter property is supposed to be the \emph{raison d'être}
of such structures, it is reasonable to only consider the parity complexes that
satisfy this property. We believe that \forestaxiom{4} is the minimal additional
condition to require for the \ocat of cells of a parity complex to be freely
generated, so we will only consider parity complexes that moreover satisfy
\forestaxiom{4}.

Under the assumptions given above, we are only left to derive \forestaxiom{3}
from the axioms of a parity complex. We show below that it is essentially a
consequence of the tightness requirements stated by \streetaxiom{5}. First, we
recall from~\cite{street1994parity} the link between tightness and the segment
property:
\begin{prop}[{\cite[Proposition~1.4]{street1994parity}}]
  \label{prop:tightness-segment}
  Let $P$ be an \ohg. For $n \in \N^*$, subsets $U,V \subseteq P_n$ with $U$
  tight, $V$ fork-free and $U \subseteq V$, we have that $U$ is a segment for
  $\tl_V$.
\end{prop}
\begin{proof}
  Let $x,y,z \in V$ such that $x,z \in U$ and $x \tlone_V y \tl_V z$. Then,
  there is $w \in x^+ \cap y^-$. By definition of tightness, since $y \tl_V z$,
  we have $y^- \cap U^\pm = \emptyset$. So there is $\bar y \in U$ such that $w \in
  \bar y^-$. Since $V$ is fork-free, $y = \bar y$. Hence, $U$ is a segment for $\tl_V$.
\end{proof}
\noindent Then, we show how to derive the segment property from the axioms of
parity complexes:
\begin{lem}
  \label{lem:pc-segment}
  Let $P$ be a parity complex which satisfies \forestaxiom{2}. Given $n \in \N$
  and $x \in P_n$, $x$ satisfies the segment condition.
\end{lem}
\begin{proof}
  
  Let $k,n \in \N$ with $k < n$, $x \in P_n$ and $X$ be a $k$-cell. Suppose
  first that~${\gen{x}_{k,-} \subseteq X_k}$. By \streetaxiom{5}, the
  set~$\gen{x}_{k,-}$ is tight, so that, by~\cref{prop:tightness-segment},
  $\gen{x}_{k,-}$ is a segment for $\tl_{X_k}$.

  Now suppose that~$\gen{x}_{k,+} \subseteq X_k$. By contradiction, assume
  that~$\gen{x}_{k,+}$ is not a segment for~$\tl_{X_k}$. By definition of
  $\tl_{X_k}$, there exist $p > 1$ and $u_0,\ldots,u_p \in X_k$ such that
  \[
    u_0,u_p \in \gen{x}_{k,+},
    \quad
    u_1,\ldots,u_{p-1} \not \in \gen{x}_{k,+}
    \qtand
    u_i \tlone_{X_k} u_{i+1}.
  \]
  By definition of $\tlone_{X_k}$, there exist $z_0,\ldots,z_{p-1}$ such that
  $z_i \in u_i^+ \cap u_{i+1}^-$. Note that ${z_0 \in \gen{x}_{k,+}^{\pm}}$.
  Indeed, if $z_0 \in v^-$ for some $v \in X_k$, then, since $X_k$ is fork-free,
  ${v = u_1}$, so $v \not \in \gen{x}_{k,+}$. Similarly, we have~${z_{p-1} \in
    \gen{x}_{k,+}^{\mp}}$. Since $x$ is relevant by~\forestaxiom{2}, we have
    $\gen{x}_{k+1,+}^{\pm} = \gen{x}_{k,+} \subseteq X_k$.
  By~\cite[Lemma 3.2]{street1991parity} (which is the analogous for parity
  complexes of \Thmr{cell-gluing}) and \forestaxiom{2}, we have that
    $\gen x_{k,-} \cap X_n \subseteq \gen x_{k+1,+}^- \cap X_n =
    \emptyset$
  and the $k$-pre-cell ${Y = \dualcellactivation(X,\gen{x}_{k+1,+})}$ is a
  $k$\cell. Moreover, by \cref{lem:unmoved-eq},
  \[
    Y_k
    = (X_k \cup \gen{x}_{k+1,+}^-) \setminus \gen{x}_{k+1,+}^+
    = (X_k \setminus \gen{x}_{k,+}) \cup \gen{x}_{k,-}\zbox.
  \]
  Thus, $\gen{x}_{k,-} \subseteq Y_k$ and, similarly as above,
  $\gen{x}_{k,-}$ is a segment for~$\tl_{Y_k}$. Since
    $\gen{x}_{k,-}^{\mp} = \gen{x}_{k,+}^{\mp}$
    and
    $\gen{x}_{k,-}^{\pm} = \gen{x}_{k,+}^{\pm}$, 
  there
  exist $\tilde u_0,\tilde u_p \in \gen{x}_{k,-}$ such that $z_0 \in \tilde u_0^+$ and $z_{p-1}
  \in \tilde u_p^-$. So
  \[
    \tilde u_0 \tlone_{X_k} u_1 \tlone_{X_k} \!\cdots\, \tlone_{X_k}
    u_{p-1} \tlone_{X_k} \tilde u_p
  \]
  with $u_1,\ldots,u_{p-1} \not \in \gen{x}_{k,-}$ (since $\gen x_{k+1,+}^- \cap
  X_n = \emptyset$), contradicting the fact that $\gen{x}_{k,-}$ is a segment
  for $\tl_{Y_k}$. Thus, $\gen{x}_{k,+}$ is a segment for~$\tl_{X_k}$.
  Hence,~$x$ satisfies the segment condition.
\end{proof}
\noindent We conclude that parity complexes are embedded into torsion-free
complexes:
\begin{theo}
  \label{thm:pc-is-gpc}
  Given a parity complex~$P$ which satisfies \forestaxiom{2} and
  \forestaxiom{4}, $P$ is a torsion-free complex.
\end{theo}
\begin{proof}
  \forestaxiom{0} is a consequence of \streetaxiom{0}. \forestaxiom{1} is a
  consequence of \streetaxiom{3}.
  And \forestaxiom{3} is a consequence of \Lemr{pc-segment}.
\end{proof}
\begin{remark}
  Given $P$ as in \cref{thm:pc-is-gpc}, the \ocat~$\Cell(P)$ of cells of the
  parity complex~$P$ is, of course, exactly the \ocat~$\Cell(P)$ of cells of the
  torsion-free complex~$P$.
\end{remark}


%% file: enc-johnson.tex
In this section, we show that loop-free pasting schemes are a particular case of
torsion-free complexes, under the caveat that we only consider loop-free pasting
schemes that satisfy \forestaxiom{4} since, like for parity complexes, loop-free
pasting schemes do not induce free $\omega$\categories in general. We think that
it is a reasonable requirement since we also believe that \forestaxiom{4} is the
minimal additional condition to add to the axioms of loop-free pasting schemes
for this property to hold.

In order to embed pasting schemes into torsion-free complexes, our main concerns
will be to derive Axioms~\forestax{2} and~\forestax{3} from Axioms~\johnsonax{3}
and~\johnsonax{4}. For this purpose, we will need to relate the cells of
torsion-free complexes with the wfs's (defined in \Cref{ssec:johnson}), using
\clwf fgs's (defined in \Cref{text:cl-max-cells}) as an intermediate. In fact,
we will prove that the latter are exactly the wfs's. First, we prove a technical
result about the relations~$\jB$ and~$\jE$:

\begin{lem}
  \label{lem:rel-switch}
  Let $P$ be a pasting scheme, $k,n \in \N$ with $k < n$, $x \in P_n$ and $y \in P_k$.
  If $x {\jB^n_{n-1}} {\clrel^{n-1}_k} y$ then
    $y \in \jB^{n}_k(x)$
    or
    $x {\jE^{n}_{n-1}} {\clrel^{n-1}_k} y$.
  Dually, if $x \comprel{\jE^{n}_{n-1}}{\clrel^{n-1}_k} y$ then
    $y \in \jE^{n}_k(x)$
    or
    ${x \comprel{\jB^{n}_{n-1}}{\clrel^{n-1}_k} y}$.
\end{lem}

\begin{proof}
  We do an induction on $n - k$. If $k = n-1$, the result is trivial. If $k= n -
  2$, the result is a consequence of \johnsonaxiom{1}. So suppose that $k < n -
  2$. We will only prove the first part, since the second is dual. So assume
  that $y \not \in \jB^n_k(x)$. By the definition of $\jB$, we have
  \[
    \neg (x \comprel{\jB^n_{n-1}}{\jB^{n-1}_k} y) \quad \text{or} \quad \neg (x
    \comprel{\jB^n_{n-1}}{\jE^{n-1}_k} y).
  \]
  By symmetry, we can suppose that $\neg (x \comprel{\jB^n_{n-1}}{\jE^{n-1}_k}
  y)$. Let $u\in P_{n-1}$ be minimal for~$\tl$ such that
    ${x \jB^{n}_{n-1} u
      \clrel^{n-1}_k y}$.
  Then, there are two possible cases: either $u
  \comprel{\jB^{n-1}_{n-2}}{\clrel^{n-2}_k} y$ or $u
  \comprel{\jE^{n-1}_{n-2}}{\clrel^{n-2}_k} y$.

  \medskip\noindent In the first case, let $v \in P_{n-2}$ be such that $u
  \jB^{n-1}_{n-2} v \clrel^{n-2}_k y$. By the minimality of $u$, we have
    $\neg (x
    \comprel{\jB^{n}_{n-1}}{\jE^{n-1}_{n-2}} v)$,
  so $\neg (x \jB^n_{n-2} v)$ by definition of $\jB$. By \johnsonaxiom{1},
  we have $x \comprel{\jE^{n}_{n-1}}{\jE^{n-1}_{n-2}} v$. So $x
  \comprel{\jE^{n}_{n-1}}{\clrel^{n-1}_k} y$.

  \medskip\noindent In the second case, since we supposed $\neg (x
  \comprel{\jB^n_{n-1}}{\jE^{n-1}_k} y)$, we have $\neg (u \jE^{n-1}_k y)$. By
  induction hypothesis, we deduce~${u \comprel{\jB^{n-1}_{n-2}}{\clrel^{n-2}_k}
    y}$ and we can conclude using the first case.
\end{proof}
\noindent Then, we prove that the source and target of wfs's computed by the
operations defined for pasting schemes in \Cref{ssec:johnson} are the same as the
ones computed with the operations defined for closed fgs's in
\Cref{text:cl-max-cells}:
\begin{lem}
  \label{lem:eq-csrctgt-clsrctgt}
  Let $P$ be a loop-free pasting scheme. Given $n \in \N^*$, $\eps \in
  \set{-,+}$ and an $n$-wfs~$X$ of~$P$, we have~${\csrctgt\eps_{n-1} (X) =
    \clsrctgt\eps_{n-1} (X)}$.
\end{lem}
\begin{proof}
  We only prove the case $\eps = -$. Recall that 
  \[
    \csrc_{n-1} (X) = X \setminus \jE(X)
    \qtand
    \clsrc_{n-1} (X) = \clset(X \setminus (X_n \cup \clset(X_n^+))).
  \]
  We first prove $\clsrc_{n-1} (X) \subseteq \csrc_{n-1} (X)$, that is,
  \[
    \clset(X \setminus (X_n \cup \clset(X_n^+))) \subseteq  X \setminus \jE(X).
  \]
  Since $X \setminus \jE(X)$ is closed (by~\cite[Theorem
  12]{johnson1989combinatorics}), it is equivalent to
  $
    X \setminus (X_n \cup \clset(X_n^+)) \subseteq  X \setminus \jE(X)
    $
  which is itself equivalent to
  $
    \jE(X) \subseteq (X_n \cup \clset(X_n^+))
    $
  which holds. We now prove that we have~$\csrc_{n-1} (X) \subseteq \clsrc_{n-1} (X)$, that is,
  \[
    X \setminus \jE(X) \subseteq  \clset(X \setminus (X_n \cup \clset(X_n^+))) = \clsrc_{n-1}(X).
  \]
  Let $k \in \N_{n-1}$ and $x \in (X \setminus \jE(X))_k$. If $x \not \in
  \clset(X_n^+)$ then $x \in \clsrc_{n-1}(X)$. So suppose that $x \in \clset(X_n^+)$.
  Since $\jE(X)_{n-1} = X_n^+$, it implies that ${k < n-1}$. By definition of
  $\clset(X_n^+)$, there exists $y \in X_n$ such that $y {\jE^n_{n-1}}
  {\clrel^{n-1}_k} x$ and, by \johnsonaxiom{2}, we can take $y$ minimal
  for $\tl$ satisfying this property. By \Lemr{rel-switch}, it holds that $y
  \comprel{\jB^n_{n-1}}{\clrel^{n-1}_k} x$. Let $z \in P_{n-1}$ be such that $y
  \jB^n_{n-1} z \clrel^{n-1}_k x$. Then, there is no ${\bar y \in X_n}$ such that
  $\bar y \jE^n_{n-1} z$: otherwise, $\bar y \comprel{\jE^n_{n-1}}{\clrel^{n-1}_k} x$
  and $\bar y \tl y$, contradicting the minimality of $y$. So $z \not \in
  \clset(X_n^+)$ and $z \clrel x$. It implies that $z \in X \setminus (X_n \cup
  \clset(X_n^+))$ and $x \in \clsrc_{n-1} (X)$.
\end{proof}
\noindent We can then prove the inclusion of wfs's into \clwf fgs's:
\begin{prop}
  \label{prop:wfs-are-clwf}
  Let $P$ be a loop-free pasting scheme. Given $n\in\N$ and~$X \in
  \wfset(P)_n$, we have~${X \in \MClosed(P)_n}$.
\end{prop}
\begin{proof}
  We prove this lemma by induction on~$n$. If $n = 0$, the result is trivial. So
  suppose~${n > 0}$. Since~$X$ is well-formed, $X_n$ is fork-free. Moreover, by
  \Lemr{eq-csrctgt-clsrctgt}, for $\eps \in \set{-,+}$, we have
  that~${\clsrctgt\eps_{n-1}(X) = \csrctgt\eps_{n-1}(X)}$ is a well-formed
  $(n{-}1)$-fgs. By induction, $\clsrctgt\eps_{n-1}(X) \in \MClosed(P)_{n-1}$.
  Moreover, when $n \ge 2$, since $\csrctgt\eps_{n-2} \circ \csrc_{n-1} (X) =
  \csrctgt\eps_{n-2} \circ \ctgt_{n-1}(X)$, by \Lemr{eq-csrctgt-clsrctgt},
  \[
    \clsrctgt\eps_{n-2} \circ \clsrc_{n-1} (X) = \clsrctgt\eps_{n-2} \circ
    \cltgt_{n-1}(X)\zbox.
  \]
  Hence, $X\in \MClosed(P)_n$.
\end{proof}
\noindent Next, we prove an analogue of the gluing \Cref{thm:cell-gluing} for wfs's:
\begin{lem}
  \label{lem:pasting-on-gen}
  Let $P$ be a loop-free pasting scheme,~$n \in \N$,~$X$ be an $n$-wfs, $S
  \subseteq P_{n+1}$ be a finite subset with $S$ fork-free and $S^{\mp}
  \subseteq X$, and $Y = X \cup \clset(S)$. Then, $Y$ is an $(n{+}1)$-wfs of~$P$
  and $\csrc_{n} (Y) = X$.
\end{lem}

\begin{proof}
  We show this lemma by induction on $k = \setsize S$. If $k = 0$, the result is
  trivial. If $k = 1$, the result is a consequence of~\cite[Proposition
  8]{johnson1989combinatorics}. So suppose that~${k > 1}$. By
  \johnsonaxiom{2}, take $x \in S$ minimal for $\tl$. By minimality, we
  have
  $
    x^- \subseteq S^{\mp} \subseteq X
    $.
  Using~\cite[Proposition 8]{johnson1989combinatorics}, $X \cup \clset(x)$ is
  well-formed. By \johnsonaxiom{5}, $X \cap \jE(x) = \emptyset$, so we have
  that~$\csrc_n(X \cup \clset(x)) = X$. Let
  \[
    \bar X = \ctgt_n(X \cup \clset(x))
    \qtand
    \bar S = S \setminus \set{x}. 
  \]
   We have
   \abovelongtoshortskip
  \begin{align*}
    \bar S^{\mp} \subseteq \bar X_n &\Leftrightarrow \bar S^- \subseteq \bar X_n \cup \bar S^+  \Leftrightarrow S^- \subseteq \bar X_n \cup \bar S^+ \cup x^- \\
                         &\Leftrightarrow S^- \subseteq (X_n \setminus x^-) \cup x^+  \cup \bar S^+ \cup
      x^-  \Leftrightarrow S^- \subseteq X_n \cup S^+ \Leftrightarrow S^{\mp} \subseteq X_n
  \end{align*}
  so $\bar S^{\mp} \subseteq \bar X$. By induction, $\bar X \cup \clset(\bar S)$ is well-formed
  and $\csrc_n(\bar X \cup \clset(\bar S)) = \bar X$. Since $\wfset(P)$ has the structure of
  an $\omega$-category by~\cite[Theorem 12]{johnson1989combinatorics}, we can
  compose $X \cup \clset(x)$ and $\bar X \cup \clset(\bar S)$. So
  \[
    X \cup \clset(S) = X \cup \clset(x) \cup \bar X \cup \clset(\bar S)
  \]
  is well-formed and $\csrc_n(X \cup \clset(S)) = X$.
\end{proof}
\noindent We can now prove the converse inclusion of \clwf fgs's into wfs's:
\begin{prop}
  \label{prop:mclosed-are-wf}
  Let $P$ be a loop-free pasting scheme. Given $n \in \N$ and~$X \in
  \MClosed(P)_n$, we have~$X \in \wfset(P)_n$.
\end{prop}

\begin{proof}
  We prove this lemma by induction on~$n$. If~${n = 0}$, the result is trivial.
  So suppose $n > 0$. Let~${Y = \clsrc_{n-1} (X)}$. By definition
  of~$\MClosed(P)$, $Y \in \MClosed(P)$ and, by induction, $Y \in \wfset(P)$. By
  definition of~$\clsrc$, we have~${X_n^{\mp} \subseteq Y}$. Moreover, by
  \Lemr{pasting-on-gen}, $Y \cup \clset(X_n)$ is well-formed. But~${Y = \clset(X
  \setminus (X_n \cup \clset(X_n^+)))}$, so that $X = Y \cup \clset(X_n)$ is
  well-formed.
\end{proof}
\noindent We now give a simple form for the sources and targets of atomic wfs's:
\begin{lem}
  \label{lem:srctgt-gen}
  Let $P$ be a loop-free pasting scheme. Given $i,n \in \N$ such that $i < n$, $\eps
  \in \set{-,+}$ and~${x \in P_n}$, we have
  $
    \csrctgt\eps_i (\clset(x)) = \clset( \gen{x}_{i,\eps})
    $.
\end{lem}
\begin{proof}
  By symmetry, we can suppose that $\eps = -$. We compute that
  \begin{align*}
    \csrc_i (\clset(x)) &= \csrc_i (\princtocl(\set{x})) \\
    &= \princtocl(\prsrc_i (\set{x})) && \text{(by \Cref{prop:princtocl-compat-srctgt,lem:eq-csrctgt-clsrctgt})} \\
    &= \princtocl(\gen{x}_{i,-}) = \clset (\gen{x}_{i,-})\zbox.
  \end{align*}
  \vskip-\baselineskip\vskip-\belowdisplayskip
\end{proof}
\noindent Using the above lemma, we deduce the relevance of the generators:
\begin{lem}
  \label{lem:ps-relevant}
  Let $P$ be a loop-free pasting scheme. Given $n \in \N$ and $x \in P_n$, $x$
  is relevant.
\end{lem}
\begin{proof}
  By~\johnsonaxiom{3}, $\clset(x)$ is well-formed. So, for $i \in
  \N_{n-1}$ and~$\eps \in \set{-,+}$, $\csrctgt\eps_i (\clset(x))$ is
  well-formed. Then, by \Lemr{srctgt-gen},~$\gen{x}_{i,-}$ and~$\gen{x}_{i,+}$
  are fork-free. We show that~${\gen{x}_{i+1,-}^{\pm} = \gen{x}_{i,+}}$
  and~${\gen{x}_{i+1,+}^{\mp} = \gen{x}_{i,-}}$. We have
  \[
    \gen{x}_{n,-}^\pm =
    \gen{x}^\pm = x^+ = \gen{x}_{n-1,+}
  \]
  and, similarly, $\gen{x}_{n,+}^\mp = \gen{x}_{n-1,-}$. For $i \in \N_{n-1}$,
  we have
  \begin{align*}
    \gen{x}_{i+1,-}^\pm
    &= (\ctgt_i(\clset(\gen{x}_{i+1,-})))_i & \text{(by definition of~$\ctgt_i$)} \\
    &= (\ctgt_i \circ \csrc_{i+1}(\clset(x)))_i & \text{(by \cref{lem:srctgt-gen})} \\
    &= (\ctgt_i (\clset(x)))_i & \text{(by globularity)} \\
    &= (\clset(\gen x_{i,+}))_i &\text{(by \cref{lem:srctgt-gen})} \\
    &= \gen{x}_{i,+}
  \end{align*}
  and similarly, $\gen{x}_{i+1,+}^\mp = \gen{x}_{i,-}$. Moreover, we have
  \begin{align*}
    (\gen x_{i,-} \cup \gen x_{i+1,-}^+) \setminus \gen x_{i+1,-}^-
    &= ((\gen x_{i+1,-}^- \setminus \gen x_{i+1,-}^+) \cup \gen x_{i+1,-}^+) \setminus \gen x_{i+1,-}^- \\
    & = \gen x_{i+1,-}^+ \setminus \gen x_{i+1,-}^- = \gen x_{i,+}
  \end{align*}
  and similarly
    $(\gen x_{i,+} \cup \gen x_{i+1,-}^-) \setminus \gen x_{i+1,-}^+
    = \gen x_{i,-}$.
  Thus, $\gen x_{i+1,-}$ moves~$\gen x_{i,-}$ to~$\gen x_{i,+}$ and so
  does $\gen{x}_{i+1,+}$. Hence, $\gen{x}$ is a cell.
\end{proof}

\noindent We now prove that the cells (in the sense of \Cref{ssec:street}) of
pasting schemes are sent to wfs's by~$\ctocl$, and that all the generators
satisfy the segment condition:
\begin{lem}
  \label{lem:ps-segment}
  Let $P$ be a loop-free pasting scheme and $n \in \N$. The following hold:
  \begin{enumerate}[label=(\roman*),ref=(\roman*)]
  \item \label{lem:ps-segment-seg} for $x \in P_n$, $x$ satisfies the segment condition,
  \item \label{lem:ps-segment-wf} for $X \in \Cell(P)_n$, $\ctocl(X) \in \wfset(P)_n$.
  \end{enumerate}
\end{lem}
\begin{proof}
  We prove this lemma by an induction on $n$. If $n = 0$, the result is trivial.
  So suppose that~${n > 0}$.

  We start with the proof of~\ref{lem:ps-segment-seg}. Let $k \in \N_{n-1}$, $x \in P_n$,
  $X$ be a $k$\cell such that $\gen{x}_{k,-} \subseteq X_k$, and~${Y =
    \ctocl(X)}$. By induction, $Y \in \wfset(P)$. Moreover, by
  \Lemr{srctgt-gen},
  \[
    \csrc_k (\clset(x)) = \clset(\gen{x}_{k,-}) \subseteq Y\zbox.
  \]
  So,
  by~\johnsonaxiom{4}, $\gen{x}_{k,-}$ is a segment for $\tl_{Y_k} = \tl_{X_k}$.
  Hence, $x$ satisfies the segment condition.

  We now prove~\ref{lem:ps-segment-wf}. Let $X \in \Cell(P)_n$. By
  \Cref{prop:mclosed-are-wf}, it is enough to show that $\ctocl(X)$ is \clwf. This
  latter property can be obtained from \Thmr{ctocl-m-iso} which requires the
  full segment axiom. But we can consider the restriction of~$P$ to an
  \ohg~$\bar P$ where
  $ \bar P_i = P_i$ for $i \le n$
  and
    $\bar P_i = \emptyset$  for $i > n$.
  By \ref{lem:ps-segment-seg}, $\bar P$ satisfies
  \forestaxiom{3}. Then, using \Thmr{ctocl-m-iso}, $\ctocl(X)$ is \clwf and is
  still \clwf in $P$. Hence, by \Cref{prop:mclosed-are-wf}, $\ctocl(X) \in
  \wfset(P)$.
\end{proof}
\noindent We can conclude the embedding of pasting schemes into torsion-free
complexes:
\begin{theo}
  \label{thm:ps-is-gpc}
  Let $P$ be a loop-free pasting scheme. Then,~$P$ satisfies
  Axioms~\forestax{0}, \forestax{1}, \forestax{2} and~\forestax{3}. In
  particular, if~$P$ satisfies~\forestaxiom{4}, then~$P$ is a torsion-free
  complex.
\end{theo}
\begin{proof}
  The different axioms of torsion-free complexes can be deduced as follows:
  \forestaxiom{0} is a consequence of \johnsonaxiom{0}, \forestaxiom{1} is a
  consequence of \johnsonaxiom{2}, \forestaxiom{2} is a consequence of
  \Lemr{ps-relevant} and \forestaxiom{3} is a consequence of \Lemr{ps-segment}.
\end{proof}
\noindent Moreover, one translates the cells of the pasting scheme to the wfs's
using the operation~$\ctocl$:
\begin{theo}
  \label{thm:cells-wfset-isom}
  Let $P$ be a loop-free pasting scheme. $\ctocl$ is an isomorphism between the
  \ocats~$\Cell(P)$ and~$\wfset(P)$. Moreover, for all $x \in P$,
  $\ctocl(\gen{x}) = \clset(x)$.
\end{theo}
\begin{proof}
  By \Cref{prop:wfs-are-clwf,prop:mclosed-are-wf}, we have
    $\MClosed(P) = \wfset(P)$
  as graded sets and, by \cref{lem:eq-csrctgt-clsrctgt} and the definition
  of~$\unit{}$, $\compcl$ and~$\comp$, the two have the same structure of \ocat.
  Thus, by \Cref{thm:ps-is-gpc,thm:ctocl-m-iso}, $\ctocl \co \Cell(P) \to
  \wfset(P)$ is an isomorphism. Moreover, by \Cref{prop:ctocl-decomp}, for $x
  \in P$, we have
  \[
    \ctocl(\gen{x}) = \princtocl \circ \ctoprinc(\gen{x}) = \princtocl(\set x) =
    \clset(x).\qedhere
  \]
\end{proof}


%% file: enc-steiner.tex
\newcommand{\Padc}{P}

In this section, we embed augmented directed complexes with loop-free unital
basis into torsion-free complexes. More precisely, given an adc with a loop-free
unital basis, we prove that the basis induces an \ohg which is a torsion-free
complex such that the \ocat of cells of the adc is isomorphic to the \ocat of
cells of this torsion-free complex. For this purpose, we relate properties
defined for \ohg{}s, like fork-freeness
(\Cref{ssec:hypergraph}) and movement (\Cref{ssec:street}), to analogous
properties in augmented directed complexes, and define translation functions
between the cells of augmented directed complexes and the ones of the associated
\ohg{}s.

\paragraph[Adc's as \texorpdfstring{\protect\ohg{}s}{omega-hypergraphs}]{Adc's as \bmohg{}s}
\parlabel{par:adc-to-ohg}

Here, dually to the translation given in \Cref{ssec:steiner}, we associate a
canonical \ohg to an adc with basis.
Let~$\genadc$ be an adc with a basis~$P$. Note that~$P$ is canonically a graded
set and, in the following, given~$n \in \N$ and~$x \in \Padc_n$, we
write~$\setel{x}$ to refer to~$x$ \emph{as an element of the graded set~$P$}
whereas~$x$ alone refers to~$x$ \emph{as an element of the
  monoid~$\deffreemon_n$}. Given~${n \in \N}$,
\begin{itemize}
\item for~$s \in \deffreemon_n$, we \glossary(Sn){$\montoset{n}(s)$}{the set
    associated to an element of an adc with basis}write~$\montoset{n}(s)$
  for~$\set{\setel{x} \in P_n \mid x \le s}$,
\item for a finite subset~$S \subseteq P_n$, we
  \glossary(Sigman){$\settomon{n}(S)$}{the adc element associated to a
    subset~$S$ of the basis of the adc}write~$\settomon{n}(S)$ for~$\sum_{x \in
    S} x$.

\end{itemize}
From these definitions, we readily have:
\begin{lem}
 \label{lem:settomon-mono}
 For all~$n \in \N$,~$\montoset{n} \circ \settomon{n} = \id{\finsets {P_n}}$.
\end{lem}
\noindent For~${n \in \N^*}$ and~${\setel{x} \in P_{n+1}}$, we define subsets~$\setel{x}^-,\setel{x}^+ \subseteq P_n$ such that
\[
  \setel{x}^- = \montoset{n}(x^-) \qtand \setel{x}^+ = \montoset{n}(x^+)
\]
where~$x^-,x^+$ are the elements of~$K_{n-1}$ defined in \Cref{ssec:steiner}. We
thus obtain an \ohg~$(P,(-)^-,(-)^+)$ that we call the \emph{\ohg associated
  to~$K$}. In the following, we prove that, when~$P$ is a unital loop-free basis
of~$K$,~$P$ is a torsion-free complex. We already have:
\begin{lem}
  \label{lem:stsrctgt-nonempty}
  If~$P$ is a unital basis of~$K$, given~$n \in \N^*$ and~$\setel{x} \in P_n$,
  we have~$\setel{x}^- \neq \emptyset$ and~$\setel{x}^+ \neq \emptyset$. That
  is,~$P$ satisfies~\forestaxiom{0}.
\end{lem}
\begin{proof}
  By contradiction, if~$\setel{x}^- = \emptyset$, it implies that~$\stgen{x}_{n-1,-} = 0$. Hence,~$\stgen{x}_{i,-} = 0$ for~$i \in \N_{n-1}$. In
  particular,~$\augop(\stgen{x}_{0,-}) = 0$, contradicting the fact that the
  basis is unital. Hence,~$\setel{x}^- \neq \emptyset$ and,
  similarly,~${\setel{x}^+ \neq \emptyset}$.
\end{proof}

\paragraph{Fork-freeness and radicality}

We now define an analogue for adc's of the notion of fork-freeness defined for
\ohg{}s, and relate the notions between the two settings.

Let~$\genadc$ be an adc with a loop-free unital basis~$P$. Given~${n \in \N^*}$,
an element~$s \in \deffreemon_n$ is said \index{fork-free!for an adc}\emph{fork-free} when for all~${{x},{y}
  \in P_n}$ such that~$x + y \le s$, it holds that
  $\setel{x}^\eps \cap \setel{y}^\eps = \emptyset$ for~$\eps \in \set{-,+}$.
Moreover, in dimension~$0$,~$s \in \deffreemon_0$ is said to be fork-free when~$\augop(s) = 1$. We extend the notion of fork-freeness to cells: given~$n \in
\N$ and~${X \in \adcCell(K)}$,~$X$ is said \index{fork-free!for an adc cell}\emph{fork-free} when, for~$i \in
\N_n$ and~${\eps \in \set{-,+}}$,~$X_{i,\eps}$ is fork-free.

Contrary to subsets of the \ohg~$P$, an element of~$P$ can appear in an element
of~$\deffreemon_n$ with a multiplicity greater than one (since~$\deffreemon_n$
is the free monoid on~$P_n$). It is then useful to distinguish the elements
of~$\deffreemon_n$ where generators appear with multiplicity at most one: given~$n \in \N$ and~$s \in \deffreemon_n$,~$s$ is said \index{radical}\emph{radical} when for
all~${z \in \deffreemon_n}$ such that~$2 z \le s$, we have~$z = 0$. We then
readily have:
\begin{lem}
  \label{lem:montoset-mono}
  For all~$n \in \N$ and~$s \in \deffreemon_n$ radical,~$\settomon{n} \circ \montoset{n}(s) = s$
\end{lem}
\noindent Moreover, fork-freeness implies radicality:
\begin{lem}
  \label{lem:ff-impl-rad}
  Given~$n \in \N$ and~$s \in \deffreemon_n$, if~$s$ is fork-free, then~$s$ is
  radical.
\end{lem}
\begin{proof}
  If~$n = 0$,~$s \in \deffreemon_n$ can be written~$s = \sum_{1 \le i \le k}
  x_i$ for some~$k \in \N$ and~$x_i \in \Padc_0$ for~$i \in \N^*_k$.
  So~${\augop(s) = k}$, and, by fork-freeness,~$k=1$. Hence,~$s$ is radical.
  
  Otherwise, assume that~$n > 0$. By contradiction, suppose that there is~$\setel{x} \in P_n$ such that~$2 x \le s$. By
  \Lemr{stsrctgt-nonempty}, it means that~$\setel{x}^- \cap \setel{x}^- \neq \emptyset$, contradicting the fact that~$s$
  is fork-free. Hence,~$s$ is radical.
\end{proof}
\noindent Like for cells of torsion-free complexes, cells of adc's with
loop-free basis are fork-free:
\begin{lem}
  {\label{lem:stcell-ff}} Given~$n \in \N$ and~$X \in \adcCell(K)_n$,~$X$ is
  fork-free.
\end{lem}
\begin{proof}
  We prove this lemma using an induction on~$n$. If~$n = 0$, since~$\augop(X_0)
  = 1$,~$X$ is fork-free by definition.

  Otherwise, suppose that~$n > 0$. By induction,~$\csrc_{n-1} (X)$ and~$\ctgt_{n-1} (X)$ are fork-free, so~$X_{i,\eps}$ is fork-free for~$i \in
  \N_{n-1}$ and~$\eps \in \set{-,+}$. Let~$\setel{x},\setel{y} \in P_n$ be such
  that~$x + y \le X_n$. By contradiction, suppose that there is~$\setel{z} \in
  P_{n-1}$ such that~$\setel{z} \in \setel{x}^- \cap \setel{y}^-$.
  By~\cite[Proposition 5.4]{steiner2004omega}, there are
  \[
    k \ge 1,
    \quad
    \setel{x}_1,\ldots,\setel{x}_k \in P_n
    \qtand
    X^1,\ldots,X^k \in \adcCell(K)
  \]
  with~$X^i_n = \setel{x}_i$ for~$i \in \N^*_k$ and such that
    $X = X^1 \comp_{n-1} \cdots \comp_{n-1} X^k$,
  so~$X_n = x_1 + \cdots + x_k$. Hence, there are~$1 \le i_1,i_2 \le k$ with~$i_1 \neq i_2$ such that~$x_{i_1} = x$ and~$x_{i_2} = y$. By symmetry, we can
  suppose that~$i_1 < i_2$. If there is some~$i$ such that~$\setel{z} \in
  \setel{x}_i^+$, by~\cite[Proposition~5.4]{steiner2004omega}, we have~$i <
  i_1$. So, for~$i_1 \le i \le i_2$, it holds that~$\setel{z} \not \in
  \setel{x}_i^+$. Let
  $
    Y = X^{i_1} \comp_{n-1} X^{i_1+1} \comp_{n-1} \cdots
    \comp_{n-1} X^{i_2}
    $.
  We have that~$Y \in \adcCell(K)$ and
  \begin{gather*}
    Y_{n-1,-} = \sum_{i_1 \le i \le i_2}
    \stgen{x_i}_{n-1,-} - \sum_{i_1 \le i \le i_2} \stgen{x_i}_{n-1,+} +
    Y_{n-1,+}
    \shortintertext{with}
    2z \le \sum_{i_1 \le i \le i_2} \stgen{x_i}_{n-1,-} \qtand \lnot (z \le
    \sum_{i_1 \le i \le i_2} \stgen{x_i}_{n-1,+}) \qtand Y_{n-1,+} \ge 0
  \end{gather*}
  so~$2z \le Y_{n-1,-}$, contradicting the fact that~$\csrc_{n-1} (Y)$ is
  radical by \cref{lem:ff-impl-rad}. Thus~$\setel{x}^- \cap \setel{y}^- =
  \emptyset$ and, similarly,~$\setel{x}^+ \cap \setel{y}^+ = \emptyset$. Hence,~$X$ is fork-free.
\end{proof}
\noindent We now give several compatibility results for the
operations~$\settomon n$ with sets and the structure of \ohg on~$P$:
\begin{lem}
  \label{lem:settomon-props}
  Let~$n \in \N$,~$U,V \subseteq P_n$ be finite subsets and~$x \in P_n$. The
  following hold:
  \begin{enumerate}[label=(\roman*),ref=(\roman*)]
  \item \label{lem:settomon-props-cup} if~$U \cap V = \emptyset$, then~$\settomon{n}(U) \land \settomon{n}(V) = 0$
    and~$\settomon{n}(U \cup V) = \settomon{n}(U)
    + \settomon{n}(V)$,
    
  \item \label{lem:settomon-props-setminus} if~$U \subseteq V$, then~$\settomon{n}(U) \le \settomon{n}(V)$ and~$\settomon{n}(V \setminus U) = \settomon{n}(V) - \settomon{n}(U)$,

  \item \label{lem:settomon-props-eps} if~$n > 0$, then~$\settomon{n-1}(\setel{x}^\eps) = x^\eps$,
  \item \label{lem:settomon-props-ff} Suppose that~$U$ is fork-free. Then~$\settomon{n}(U)$
    is fork-free. Moreover, in the case where~${n >0}$,
    we have~${\settomon{n-1}(U^\eps) = (\settomon{n}(U))^\eps}$.
  \end{enumerate}
\end{lem}
\begin{proof}
  \ref{lem:settomon-props-cup} and \ref{lem:settomon-props-setminus} are direct
  consequences of the definitions.
  For~\ref{lem:settomon-props-eps}, note that~$\setel{x}^\eps = \montoset{n-1} (x^\eps)$. By \Lemr{stcell-ff},~$\stgen{x}_{n-1,\eps}$ is fork-free and, by \Lemr{ff-impl-rad}, it is
  radical. So, by \Lemr{montoset-mono},
  we have~$\settomon{n-1} (\setel{x}^\eps) = x^\eps$.

  For~\ref{lem:settomon-props-ff}, suppose that~$U \subseteq P_n$ is fork-free.
  If~$n = 0$, the result is trivial. So suppose that~${n > 0}$. Given~${x},{y}
  \in P_n$ with~$x \le \settomon{n}(U)$ and~$y \le \settomon{n}(U)$,~$\setel{z}
  \in P_{n-1}$ and~$\eps \in \set{-,+}$ such that~${z \le x^\eps}$ and~${z \le
    y^\eps}$, we have~$\setel{z} \in \setel{x}^\eps$ and~$\setel{z} \in
  \setel{y}^\eps$. Since~$U$ is fork-free,~$x = y$. Also,~$\settomon{n}(U)$ is
  radical by definition of~$\settomon n$, so that~${\neg (x + y \le
    \settomon{n}(U))}$. Hence,~$\settomon{n}(U)$ is fork-free. For the second
  part, note that, for~${x},{y} \in U$ with~$x \neq y$, we have~$\setel{x}^\eps
  \cap \setel{y}^\eps = \emptyset$. Hence, by~\ref{lem:settomon-props-cup} and~\ref{lem:settomon-props-eps},
  \[
    \settomon{n-1}(U^\eps) = \settomon{n-1}(\cup_{\setel{x} \in U} \setel{x}^\eps)  
                            = \sum_{\setel{x}
                             \in U} \settomon{n-1}(\setel{x}^\eps) 
                           = \sum_{\setel{x} \in U} x^\eps 
                           = (\settomon{n}(U))^\eps.  \qedhere
                         \]
\end{proof}
\noindent We give analogous compatibility results for the
operations~$\montoset n$ with the group structure of~$K_n$ and the
operations~$(-)^-$ and~$(-)^+$ defined on~$K_n$:
\begin{lem}
  \label{lem:montoset-props}
  Let~$n \in \N$,~$u,v \in \deffreemon_n$ be such that~$u,v$ are radical and~$z
  \in \Padc_n$. The following hold:
  \begin{enumerate}[label=(\roman*),ref=(\roman*)]
  \item \label{lem:montoset-props-sum} if~$u \land v = 0$, then~${\montoset{n} (u)} \cap {\montoset{n} (v)} = \emptyset$ and~$\montoset{n}(u + v) = {\montoset{n} (u)} \cup {\montoset{n} (v)}$,
  \item \label{lem:montoset-props-sub} if~$u \le v$, then~${\montoset{n} (u)} \subseteq {\montoset{n} (v)}$ and~$\montoset{n}(v - u) = (\montoset{n}(v)) \setminus (\montoset{n}(u))$,
  \item \label{lem:montoset-props-eps} if~$n > 0$, then~$\montoset{n-1}(z^\eps) = \setel{z}^\eps$,
  \item \label{lem:montoset-props-ff} Suppose that~$u$ is fork-free. Then,~${\montoset{n} (u)}$ is fork-free. Moreover, in the case where~${n > 0}$, we
    have~${\montoset{n-1}(u^\eps) = (\montoset{n}(u))^\eps}$.
  \end{enumerate}
\end{lem}
\begin{proof}
  \ref{lem:montoset-props-sum}, \ref{lem:montoset-props-sub} and
  \ref{lem:montoset-props-eps} are direct consequences of the definitions. For
  \ref{lem:montoset-props-ff}, suppose that~$u$ is fork-free. If~$n = 0$, the
  result is trivial, so suppose that~$n > 0$. Given~$\setel{x},\setel{y} \in
  {\montoset{n} (u)}$,~$\setel{z} \in \P_{n-1}$ and~$\eps \in \set{-,+}$ such
  that~$\setel{z} \in \setel{x}^\eps \cap \setel{y}^\eps$, we have~$z \le
  x^\eps$ and~$z \le y^\eps$. By fork-freeness,~$\neg (x+y \le u)$. But~${x \le
    u}$ and~${y \le u}$, so that~$x = y$. Thus,~$\montoset{n}(u)$ is fork-free.
  For the second part, note that, for~$x,y \in \Padc_n$ with~$x \neq y$,~$x \le
  u$ and~$y \le u$, we have~$x^\eps \land y^\eps = 0$. Hence,
  by~\ref{lem:montoset-props-sum} and~\ref{lem:montoset-props-eps},
  \[
    \montoset{n-1}(u^\eps) = \montoset{n-1}(\sum_{x \in \Padc_n, x \le u} x^\eps) 
    =
    \bigcup_{x \in \Padc_n, x \le u} \montoset{n-1}(x^\eps)  
    =  \bigcup_{x \in \Padc_n, x \le u}
    \setel{x}^\eps   
    = (\montoset{n}(u))^\eps.  \qedhere
  \]
\end{proof}

\paragraph{Movement properties}

We now relate the movement properties of \ohg{}s (as defined in
\Cref{ssec:street}) to properties of augmented directed complexes. Let~$\genadc$
be an adc with a loop-free unital basis~$P$. We first prove a compatibility
result of the functions~$\settomon n$ with the operations~$(-)^\mp$
and~$(-)^\pm$ on \ohg{}s and adc's:
\begin{lem}
  \label{lem:montoset-mp-pm}
  Let~$n \in \N^*$,~$u \in \deffreemon_n$ fork-free and~$U = \montoset{n}(u)$.
  We have
  \[
    u^\mp = \settomon{n-1}(U^\mp) \qtand u^\pm =\settomon{n-1} (U^\pm).
  \]
\end{lem}
\begin{proof}
  We compute that
  \begin{align*}
    \diffop (u) &= u^\pm - u^\mp = u^+ - u^- \\
    &= \settomon{n-1} (U^+) - \settomon{n-1} (U^-) && \text{(by \Lemr{settomon-props})} \\
                &= (\settomon{n-1} (U^\pm) {+} \settomon{n-1} (U^+ \cap U^-))
         -  (\settomon{n-1} (U^\mp) + \settomon{n-1} (U^+ \cap U^-))&& \text{(by \Lemr{settomon-props})}\\
    &= \settomon{n-1} (U^\pm) - \settomon{n-1} (U^\mp).
  \end{align*}
  Since~$U^\pm \cap U^\mp = \emptyset$, we have~$\settomon{n-1} (U^\pm) \land
  \settomon{n-1} (U^\mp) = \emptyset$. By uniqueness of the decomposition, we
  have $ u^\mp = \settomon{n-1} (U^\mp)$ and $u^\pm =\settomon{n-1} (U^\pm) $.
\end{proof}
\noindent Now, we show a compatibility of the operations~$\settomon n$ with movement:
\begin{lem}
  \label{lem:settomon-move-diff}
  Let~$n \in \N$,~$S \subseteq P_{n+1}$ be a finite and fork-free set and~$U,V
  \subseteq P_n$ be finite sets such that~$S$ moves~$U$ to~$V$. Then,~$\diffop(
  \settomon{n+1}(S)) = \settomon{n}(V) - \settomon{n}(U)$.
\end{lem}

\begin{proof}
  By definition of movement,~$V = (U \cup S^+) \setminus S^-$.
  Hence,
  \begin{align*}
    \settomon{n}(V)&= \settomon{n}((U \cup S^+) \setminus S^-) \\
                   &= \settomon{n}(U \cup S^+) - \settomon{n}( S^-) && \text{(by
                                                                       \Lemr{settomon-props},
                                                                       since~$S^-
                                                                       \subseteq U
                                                                       \cup S^+$)} \\
    &= \settomon{n}(U) + \settomon{n}(S^+) - \settomon{n}(S^-) && \text{(since~$U
                                                                \cap S^+ =
                                                                \emptyset$ by
                                                                \Lemr{street-2.1})}\\
    &= \settomon{n}(U) + (\settomon{n+1}(S))^+ - (\settomon{n+1}(S))^- && \text{(by \Lemr{settomon-props})}\\
    &= \settomon{n}(U) + \diffop (\settomon{n+1}(S)). && \qedhere
  \end{align*}
\end{proof}
\noindent Conversely, we prove sufficient conditions for the operations~$\montoset n$ to induce movement:
\begin{lem}
  \label{lem:montoset-diff-move}
  Let~$n \in \N$,~$s \in \deffreemon_{n+1}$ fork-free,~$u,v \in \deffreemon_n$
  with $u,v$ radical, such that
  \[
    \diffop (s) = v - u,
    \quad
    u \land s^+ = 0
    \qtand
    s^- \land v = 0.
  \]
  Then,~${\montoset{n+1} (s)}$ moves~${\montoset{n} (u)}$ to~${\montoset{n}
    (v)}$.
\end{lem}
\begin{proof}
  Let~$S = \montoset{n+1}(s)$,~$U = \montoset{n}(u)$ and~$V = \montoset{n}(v)$.
  Since~$\diffop (s) = v - u$, we have
  \[
    s^- \le s^- + v = u + s^+
  \]
  so
  $S^- = \montoset{n}(s^-) \subseteq \montoset{n}(u + s^+) = U \cup S^+$.
  We compute that
  \begin{align*}
    \settomon{n}((U \cup S^+) \setminus S^-) &= \settomon{n}(U \cup S^+) -
                                               \settomon{n}(S^-) \\
                                             &= \settomon{n} \circ \montoset{n} (u + s^+) - s^- && \text{(by \Lemr{settomon-props})}\\
                                             &= u + s^+ - s^- \\
                                             &= u + \diffop (s) = v = \settomon{n}(V)\zbox.
  \end{align*}
  By \Lemr{settomon-mono},~$V = (U \cup S^+) \setminus S^-$. Similarly,~$U
  = (V \cup S^-) \setminus S^+$. So,~$S$ moves~$U$ to~$V$.
\end{proof}
\noindent Finally, we show empty intersection results for cells
of~$\adcCell(K)$, analogous to the ones for~$\Cell(P)$:%
\begin{lem}
  \label{lem:stcell-move-prop}
  Let~$n \in \N^*$ and~$X \in \adcCell(K)_n$. Then, for~$i \in \N_{n-1}$
  and~${\eps \in \set{-,+}}$, we have
  \[
     X_{i,-} \land  X_{i+1,\eps}^+ = 0
    \qtand
    X_{i+1,\eps}^- \land X_{i,+} = 0
    .
  \]
\end{lem}

\begin{proof}
  By contradiction, suppose given~$n \in \N^*$,~$X \in \adcCell(K)_n$,~$i \in
  \N_{n-1}$ and~$\eps \in \set{-,+}$ that give a counter-example for this
  property. By applying~$\csrc,\ctgt$ sufficiently, we can suppose that~$i =
  n-1$. Also, by symmetry, we only need to handle the first case, that is, when
  there is~$z \in \Padc_{n-1}$ such that~${z \le X_{n-1,-} \land X_{n}^+}$. So
  there is~$x \in \Padc_n$ such that~$x \le X_{n}$ and~$z \le x^+$. By the
  definition of a cell, we have~$\diffop (X_n) = X_{n-1,+} - X_{n-1,-}$, thus
  \begin{align*}
    X_{n-1,+} + \sum_{u \in P_n,u \le X_n} u^- &= X_{n-1,-} + \sum_{u \in P_n,u \le X_n} u^+ \ge 2z
  \end{align*}
  and, since~$X_{n-1,+}$ is radical, there is~$y \in \Padc_n$ with~$y \le X_n$
  such that~$z \le y^-$. By~\cite[Proposition~5.1]{steiner2004omega}, there
  are~$k \in \N^*$,~$x_1,\ldots,x_k \in \Padc_n$ with~$x_1 + \cdots +
  x_k=X_n$, $i_1,i_2 \in \N^*_k$ with~$i_1 < i_2$, $x_{i_1} = x$ and~$x_{i_2} =
  y$, and~$X^1,\ldots,X^k \in \adcCell(K)$ with~$X^i_n = x_i$ for~$i \in \N^*_k$
  such that we have the decomposition~$X = X^1 \comp_{n-1} \cdots \comp_{n-1} X^k$. Let~${Y = X^1
    \comp_{n-1} \cdots \comp_{n-1} X^{i_1}}$. Since~$Y$ is a cell, we have
  \[
    Y_{n-1,+} + \sum_{1 \le i \le k} x_i^- = Y_{n-1,-} +  \sum_{1 \le i \le k}x_i^+ = X_{n-1,-} + \sum_{1 \le i \le k} x_i^+  \ge 2z.
  \]
  Moreover, since~$X$ is fork-free and~$z \le x_{i_2}^-$, we have~$\lnot (z \le
  x_i^-)$ for~$i \in \N_{i_1}$. So~$2z \le Y_{n-1,+}$, contradicting the fact
  that~$Y_{n-1,+}$ is radical derived from Lemmas~\ref{lem:stcell-ff}
  and~\ref{lem:ff-impl-rad}. Hence, we have~${X_{i,-} \land X_{n}^+ = 0}$.
\end{proof}
\paragraph{The translation operations}

We now introduce translation functions between the cells of augmented directed
complexes and the cells of their associated \ohg{}s, and show that these
translations are bijective.

Let~$\genadc$ be an adc with a loop-free unital basis~$P$. We extend the
operations~$\settomon n$ and~$\montoset n$ to translation functions between the
pre-cells of~$P$ and the pre-cells of~$K$. Given~$n \in \N$ and
an~$n$\precell~${X \in \PCell(P)_n}$, we \glossary(Sigma){$\ctost(X)$}{the adc
  pre-cell associated to the \ohg pre-cell~$X$}define~$\ctost(X) \in
\adcCell(K)$ as the~$n$\precell~$Y$ such that
  $Y_{i,\eps} = \settomon{i}(X_{i,\eps})$ for~$i \in \N_n$ and~$\eps \in
  \set{-,+}$.
Similarly, given an $n$\precell~$X \in \adcPCell(K)$, we
\glossary(S){$\sttoc(X)$}{the \ohg pre-cell associated to the adc
  pre-cell~$X$}define~$\sttoc(X) \in \PCell(P)$ as the $n$\precell~$Y$ such that
  $Y_{i,\eps} = \montoset{i}(X_{i,\eps})$ for~$i \in \N_n$ and~$\eps \in
  \set{-,+}$.
 We then have:
\begin{prop}
  \label{prop:ctost-bij}
  $\ctost$ induces a bijection with inverse~$\sttoc$ from~$\Cell(P)$
  to~$\adcCell(K)$. Moreover, given~$x \in P$, we have~$\sttoc (\stgen{x}) =
  \gen{\setel{x}}$.
\end{prop}
\begin{proof}
  Let~$n \in \N$ and~$X \in \Cell(P)_n$. Then, by \Lemr{settomon-props}, for~${i
    \in N_n}$ and~${\eps \in \set{-,+}}$,~$\settomon{i}(X_{i,\eps})$ is
  fork-free. Plus, when~$i < n$, by \Lemr{settomon-move-diff}, we have
    $\diffop (\settomon{i+1}(X_{i+1,\eps})) = \settomon{i}(X_{i,+}) - \settomon{i}
    (X_{i,-})$
  so~$\ctost(X) \in \adcCell(K)$. Conversely, let~$n \in \N$ and~$X \in
  \adcCell(K)_n$. By \Lemr{montoset-props}, given~$i \in \N_n$ and~$\eps \in
  \set{-,+}$,~$\montoset{i}(X_{i,\eps})$ is fork-free. Moreover, when~$i <n$, by
  Lemmas~\ref{lem:montoset-diff-move} and~\ref{lem:stcell-move-prop}, we have that
    $\montoset{i+1}(X_{i+1,\eps})$ moves~$\montoset{i}(X_{i,-})$ to~$\montoset{i}
    (X_{i,+})$
  so~$\sttoc(X) \in \Cell(P)$. By \Lemr{settomon-mono}, for~$X \in \Cell(P)$,
    ${\sttoc \circ \ctost (X) = X}$,
  and, by Lemmas~\ref{lem:stcell-ff}, \ref{lem:ff-impl-rad}
  and~\ref{lem:montoset-mono}, for~$X \in \adcCell(K)$,
    ${\ctost \circ \sttoc (X) = X}$.
  Hence,~$\ctost$ and~$\sttoc$ induce bijections between~$\Cell(P)$ and~$\adcCell(K)$ and are inverse of each other.
  
  Now let~$n \in \N$,~$x \in P_n$ and~$X = \sttoc(\stgen{x})$. We have~$X_n =
  \montoset{n}(\stgen{x}_n) = \set{x}$. We show by a decreasing induction on~$i$
  that~${X_{i,\eps} = \gen{x}_{i,\eps}}$ for~${i \in \N_{n-1}}$ and~$\eps \in
  \set{-,+}$. We have~$\stgen{x}_{i,-} = \stgen{x}_{i+1,-}^\mp$ so, by
  \cref{lem:stcell-ff,lem:montoset-mp-pm},
    $X_{i,-} =
    \montoset{i} (\stgen{x}_{i+1,-}^\mp) = X_{i+1,-}^\mp$.
  Thus,~$X_{i,-} = \gen{x}_{i,-}$. Similarly,~$X_{i,+} = \gen{x}_{i,+}$. Hence,~$\sttoc (\stgen{x}) = \gen{\setel{x}}$.
\end{proof}

\paragraph{Adc's are torsion-free complexes} We now prove that the \ohg{}s
associated to adc's equipped with loop-free unital bases are torsion-free
complexes. In fact, we will show that they moreover satisfy the stronger
Axioms~\forestaxcomp{3} and~\forestaxcomp{4}.

Let~$\genadc$ be an adc with a loop-free unital basis~$P$. We have already shown
how to derive \forestaxiom{0} for~$P$ in \cref{lem:stsrctgt-nonempty}, and we
now derive the other ones in the following lemmas.
\begin{lem}
  \label{lem:adc-axiom1}
  $P$ satisfies \forestaxiom{1}.
\end{lem}
\begin{proof}
  Note that, for~$n \in \N^*$ and~$\setel{x},\setel{y} \in P_n$,~$\setel{x}
  \tlone_{P_n} \setel{y}$ implies~$\setel{x} <_{n-1} \setel{y}$. So, by
  transitivity, we have~$\tl_{P_n} \subseteq {<_{n-1}}$. Since the basis~$P$ is
  loop-free,~$<_{n-1}$ is irreflexive and so is~$\tl_{P_n}$. Hence,~$\tl$ is
  irreflexive.
\end{proof}

\begin{lem}
  \label{lem:adc-axiom2}
  $P$ satisfies \forestaxiom{2}.
\end{lem}
\begin{proof}
  Given~$\setel{x} \in P$, we have~$\sttoc (\stgen{x}) = \gen{\setel{x}}$
  By~\cref{prop:ctost-bij}. Moreover, by \cref{prop:ctost-bij}, we have~${\sttoc
    (\stgen{x}) \in \Cell(P)}$. Hence,~$\setel{x}$ is relevant.
\end{proof}

\begin{lem}
  \label{lem:adc-axiomcomp3}
  $P$ satisfies \forestaxiomcomp{3}.
\end{lem}
\begin{proof}
  By contradiction, suppose that there are~$i,n \in \N$ with~$i < n$
  and an element~$\setel{x} \in P_n$ such that~$\gen{\setel{x}}_{i,+} \jtl
  \gen{\setel{x}}_{i,-}$. So there are~$k \ge 1$,~$\setel{y}_1,\ldots,\setel{y}_k \in P_i$ such that
    $\setel{y}_1 \in \gen{\setel{x}}_{i,+}$,
    $\setel{y}_k \in \gen{\setel{x}}_{i,-}$
    and
    $\setel{y}_j \jtlelem \setel{y}_{j+1}$  for~$1 \le j < k$.
  By definition of~$\jtlelem$, it gives~$\setel{z}_1,\ldots,\setel{z}_{k-1} \in
  P_{i+1}$ with~$\setel{y}_j \in \setel{z}_j^-$ and~$\setel{y}_{j+1} \in
  \setel{z}_j^+$ for~$1 \le j < k$. So we have
  $
    {x} <_i {z}_1 <_i \cdots <_i {z}_{k-1}
    <_i {x}$,
contradicting the loop-freeness of the basis~$P$. Hence,~$P$
  satisfies \forestaxiomcomp{3}.
\end{proof}

\begin{lem}
  \label{lem:adc-axiomcomp4}
  $P$ satisfies \forestaxiomcomp{4}.
\end{lem}
\begin{proof}
  By contradiction, suppose that there are~$i \in \N^*$,~$m,n \in \N$ with~$m >
  i$ and~$n > i$,~$\setel{x} \in P_m$ and~$\setel{y} \in P_n$ such that $
  \gen{\setel{x}}_{i,+} \cap \gen{\setel{y}}_{i,-} = \emptyset$,
  $\gen{\setel{x}}_{i-1,+} \jtl \gen{\setel{y}}_{i-1,-}$ and
  $\gen{\setel{y}}_{i-1,+} \jtl \gen{\setel{x}}_{i-1,-}$. By the same method as
  for \Lemr{adc-axiomcomp3}, we get~$r,s \in \N$,~${u}_1,\ldots,{u}_r \in
  P_{i}$,~${v}_1,\ldots,{v}_s \in P_i$ such that
  \[
    {x} <_i {u}_1 <_i \cdots <_i {u}_r <_i {y} <_i
    {v}_1 <_i \cdots <_i {v}_s <_i {x},
  \]
  contradicting the loop-freeness of the basis~$P$. Hence,~$P$ satisfies
  \forestaxiomcomp{4}.
\end{proof}
\noindent We can conclude that:
\begin{theo}
  \label{thm:adc-is-gpc}
  The \ohg~$P$ associated to~$K$ is a torsion-free complex.
\end{theo}
\begin{proof}
  This follows from \Cref{lem:stsrctgt-nonempty,lem:adc-axiom1,lem:adc-axiom2,lem:adc-axiomcomp3,prop:comp-impl-norm-three,lem:adc-axiomcomp4,prop:comp-impl-norm-four}.
\end{proof}
\noindent Finally, we show that~$\ctost$ exhibits an isomorphism between the two
\ocats of cells:
\begin{theo}
  \label{thm:ctost-functor}
  $\ctost$ induces an isomorphism of~\ocats between~$\Cell(P)$ and~$\adcCell(K)$.
  Moreover, for~${\setel{x} \in P}$, we have~$\ctost(\gen{\setel{x}}) =
  \stgen{x}$.
\end{theo}
\begin{proof}
  By definition,~$\ctost$ commutes with the source, target and identity
  operations defined on the \ocats~$\Cell(P)$ and~$\adcCell(K)$. We show that it
  commutes with the composition operations. Given~${i,n \in \N}$ with~${i <
    n}$,~$i$\composable cells~$X,Y \in \Cell(P)_n$, by~\cref{lem:gencomposable},
  we have $ {X_{j,\eps}\cap Y_{j,\eps} = \emptyset} $ for~$j \in \N$ with~$i < j
  \le n$ and~$\eps \in \set{-,+}$. Thus, by \Lemr{settomon-props}, it follows
  readily that~$\settomon{n}(X \comp_i Y) = \settomon{n}(X) \comp_i
  \settomon{n}(Y)$. Thus,~$\ctost$ is a morphism of~\ocats. We conclude with
  \cref{prop:ctost-bij}.
\end{proof}




%% file: str-inc.tex
We conclude our comparison of the pasting diagram formalisms by showing that
there are no embeddings between the four formalisms except the ones already
proved, that is, that parity complexes, pasting scheme and augmented directed
complexes are particular cases of torsion-free complexes (under the caveats
stated for parity complexes and pasting schemes). We show these nonexistence
results by simply exhibiting counter-examples to the other embeddings.

Since adc's are not exactly \ohg{}s, we should make the following precision.
When we say that ``there is no embedding of adc's with loop-free unital bases
into the formalism X'', we mean that, in general, the \ohg obtained from an adc
with loop-free unital basis
(as described in \Cref{ssec:enc-steiner}) is not an instance of X. Conversely,
when we say that ``there is no embedding of the formalism X into adc's with
loop-free unital bases'', we mean that, in general, the pre-adc with basis
obtained from an \ohg which is an instance of X %
(as described in \Cref{ssec:steiner}) is not an adc with loop-free unital basis.

\paragraph{No embedding in parity complexes}

We show that there are no embeddings into parity complexes of the other
formalisms. Considering the axioms of parity complexes, \streetaxiom{4} is
relatively strong, and it has no real equivalent in the other formalisms, so it
can be used to build a counter-example to embeddings. The \ohg~\eqref{eq:not-3b}
is a pasting scheme satisfying \forestaxiom{4} (and thus is a torsion-free
complex) and is an adc with loop-free unital basis. But it is not a parity
complex %
as we have seen in \Cref{ssec:street}, because it does not satisfy
\streetaxiom{4}. So pasting schemes, augmented directed complexes and
torsion-free complexes are not parity complexes in general.

\paragraph{No embedding in pasting schemes}
We now show that there are no embeddings into pasting schemes of the other
formalisms. We use the relatively strong \johnsonaxiom{2} to build a
counter-example to the embeddings. The following $\omega$-hypergraph is a parity
complex satisfying \forestaxiom{4} (and thus it is a torsion-free complex) and
is an adc with loop-free unital basis but it is not a pasting scheme:
\begin{equation}
  \label{eq:ce-inc-johnson}
\begin{tikzpicture}[yscale=1.5,xscale=1,baseline=(b)]
 \coordinate (a) at (0,1);
 \coordinate (b) at (0,0);
 \coordinate (c) at (0,-1);
 \coordinate (d) at (1.5,0.3);
 \node (A) at (a) {$z$};
 \node (B) at (b) {$y$};
 \node (C) at (c) {$x$};
 \node (D) at (d) {$w$};
 \draw[->] (B) to[bend right=60,"$d'$"{xshift=1.2em,yshift=1em}] (A); 
 \draw[->] (B) to[bend left=60,"$d$"] (A); 
 \draw[->] (C) to[bend right=60,"$a'$"',inner sep=0.3ex] (B); 
 \draw[->] (C) to[bend left=60,"$a$",inner sep=0.3ex] (B); 
 \draw[->] (C) .. controls +(-1,0) and +(-1,0) .. (B) node[midway,left] {$c$}; 
 \draw[->] (B) to[bend left=10,"$b$",pos=0.6,inner sep=0.3ex] (D); 
 \draw[->] (C) to[out=20,in=-100,"$b'$"'] (D); 
 
 \draw[->] (A) .. controls +(-3,0) and +(-3,0) .. (0,-2) node[pos=0.7,left] {$e$}
               .. controls +(2,0) and +(1,-1) .. (D) ;
 \node (X1) at (0,-0.5) {\myoverset{$\Rightarrow$}{$\alpha_1$}};
 \node (X2) at (0,0.5) {\myoverset{$\Rightarrow$}{$\alpha_4$}};

 \node at (0.75,-0.2) {$\myoverset{$\Rightarrow$}{$\alpha_2$}$};
 \node at (0,-1.5) {$\myoverset{$\Leftarrow$}{$\alpha_3$}$};
\end{tikzpicture}
.
\end{equation}
Indeed, \johnsonaxiom{2} is not satisfied because $\alpha_2 \tl \alpha_3$ and $y
\in B(\alpha_2) \cap E(\alpha_3) \neq \emptyset$. Note
that~\eqref{eq:ce-inc-johnson} is essentially the \ohg \eqref{eq:ex-non-segment}
without the $3$-generator~$A$ and the $2$-generators~$\alpha'_1$
and~$\alpha'_4$.

\paragraph{No embedding in augmented directed complexes}

Finally, we prove that there are no embeddings into augmented directed complexes
with loop-free unital basis of the other formalisms. As shown in
\Cref{ssec:enc-steiner}, such adc's satisfy \forestaxiomcomp{4}, which is a
stronger version of \forestaxiom{4}. Thus, we can find a counter-example to
embedding into adc's with loop-free unital basis by considering an adequate \ohg
which satisfies \forestaxiom{4} but not \forestaxiomcomp{4}. Consider the
\ohg~$P$ from \Cref{fig:ce-emb-into-adc}
\begin{figure}
  \centering
  \[
    \begin{array}{c}
      \begin{tikzcd}[ampersand replacement=\&,cramped]
        w \arrow[rr,"b"{description},myname=b]
        \arrow[rr,out=70,in=110,"a",myname=a] 
        \arrow[rr,out=-70,in=-110,"c"',myname=p] 
        \arrow[phantom,"\alpha \Downarrow",from=a,to=b]
        \arrow[phantom,"\beta \Downarrow",from=b,to=p] 
        \& \& x \arrow[rr,"e"{description},myname=e] \arrow[rr,out=70,in=110,"d",myname=p]
        \arrow[rr,out=-70,in=-110,"f"',myname=f] 
        \arrow[phantom,"\gamma \Downarrow",from=p,to=e]
        \arrow[phantom,"\delta \Downarrow",from=e,to=f]
        \& \& y \arrow[rr,out=70,in=110,"g",myname=g] \arrow[rr,"h"{description},myname=h] 
        \arrow[rr,out=-70,in=-110,"i"',myname=i] 
        \arrow[phantom,"\eps \Downarrow",from=g,to=h]
        \arrow[phantom,"\zeta \Downarrow",from=h,to=i]
        \& \& z
      \end{tikzcd}
      \\
      A~\Ddownarrow
      \\
      \begin{tikzcd}[ampersand replacement=\&,cramped]
        w \arrow[rr,"b"{description},myname=b]
        \arrow[rr,out=70,in=110,"a",myname=a] 
        \arrow[rr,out=-70,in=-110,"c"',myname=p] 
        \arrow[phantom,"\alpha \Downarrow",from=a,to=b]
        \arrow[phantom,"\beta' \Downarrow",from=b,to=p] 
        \& \& x \arrow[rr,"e"{description},myname=e] \arrow[rr,out=70,in=110,"d",myname=p]
        \arrow[rr,out=-70,in=-110,"f"',myname=f] 
        \arrow[phantom,"\gamma' \Downarrow",from=p,to=e]
        \arrow[phantom,"\delta \Downarrow",from=e,to=f]
        \& \& y \arrow[rr,out=70,in=110,"g",myname=g] \arrow[rr,"h"{description},myname=h] 
        \arrow[rr,out=-70,in=-110,"i"',myname=i] 
        \arrow[phantom,"\eps \Downarrow",from=g,to=h]
        \arrow[phantom,"\zeta \Downarrow",from=h,to=i]
        \& \& z
      \end{tikzcd}
      \\
      B~\Ddownarrow
      \\
      \begin{tikzcd}[ampersand replacement=\&,cramped]
        w \arrow[rr,"b"{description},myname=b]
        \arrow[rr,out=70,in=110,"a",myname=a] 
        \arrow[rr,out=-70,in=-110,"c"',myname=p] 
        \arrow[phantom,"\alpha \Downarrow",from=a,to=b]
        \arrow[phantom,"\beta' \Downarrow",from=b,to=p] 
        \& \& x \arrow[rr,"e"{description},myname=e] \arrow[rr,out=70,in=110,"d",myname=p]
        \arrow[rr,out=-70,in=-110,"f"',myname=f] 
        \arrow[phantom,"\gamma' \Downarrow",from=p,to=e]
        \arrow[phantom,"\delta' \Downarrow",from=e,to=f]
        \& \& y \arrow[rr,out=70,in=110,"g",myname=g] \arrow[rr,"h"{description},myname=h] 
        \arrow[rr,out=-70,in=-110,"i"',myname=i] 
        \arrow[phantom,"\eps' \Downarrow",from=g,to=h]
        \arrow[phantom,"\zeta \Downarrow",from=h,to=i]
        \& \& z
      \end{tikzcd}
      \\
      C~\Ddownarrow
      \\
      \begin{tikzcd}[ampersand replacement=\&,cramped]
        w \arrow[rr,"b"{description},myname=b]
        \arrow[rr,out=70,in=110,"a",myname=a] 
        \arrow[rr,out=-70,in=-110,"c"',myname=p] 
        \arrow[phantom,"\alpha' \Downarrow",from=a,to=b]
        \arrow[phantom,"\beta' \Downarrow",from=b,to=p] 
        \& \& x \arrow[rr,out=70,in=110,"d",myname=p]
        \arrow[rr,out=-70,in=-110,"f"',myname=f] 
        \arrow[phantom,"\gamma'' \Downarrow",from=p,to=f]
        \& \& y \arrow[rr,out=70,in=110,"g",myname=g] \arrow[rr,"h"{description},myname=h] 
        \arrow[rr,out=-70,in=-110,"i"',myname=i] 
        \arrow[phantom,"\eps' \Downarrow",from=g,to=h]
        \arrow[phantom,"\zeta' \Downarrow",from=h,to=i]
        \& \& z
      \end{tikzcd}
    \end{array}
  \]
  \caption{The \protect\ohg $P$}
  \label{fig:ce-emb-into-adc}
\end{figure}
where the $3$\generators~${A,B,C}$ are such that
\begin{align*}
  A^- &= \set{\beta,\gamma}, & B^- &= \set{\delta,\epsilon},&C^- &= \set{\alpha,\gamma',\delta',\zeta}, \\
  A^+ &= \set{\beta',\gamma'}, & B^+ &= \set{\delta',\epsilon'},& C^+ &= \set{\alpha',\gamma'',\zeta'}.
\end{align*}
It can be shown that it is a parity complex and a pasting scheme. It moreover
satisfies \forestaxiom{4} so that it is a torsion-free complex by
\cref{thm:ps-is-gpc}. But its associated pre-adc is an adc with a basis which is
not loop-free unital. Indeed, we have $e \le \stgen{A}_{1,+} \land
\stgen{B}_{1,-}$, $h \le \stgen{B}_{1,+} \land \stgen{C}_{1,-}$ and $b \le
\stgen{C}_{1,-} \land \stgen{A}_{1,+}$, so that $ A <_1 B <_1 C <_1 A $. Hence,
the basis of the associated augmented directed complex is not loop-free.


%% file: details-gluing-thm.tex
This section is devoted to the proof of \Cref{thm:cell-gluing}, which allows
gluing sets of generators to existing cells in order to get new cells. This
result requires some technical results about movement, which appears in the
definition of cells. We first introduce these results and then discuss the proof
of the gluing theorem.

\subsection{Movement properties}

Here, we state and prove here several such properties, some of which coming
already present in~\cite{street1991parity}.

\medskip\noindent In the following, we suppose given an \ohg~$P$. We first state
a criterion for movement:

\begin{lem}[{\cite[Proposition 2.1]{street1991parity}}]
  \label{lem:street-2.1}
  For~$n \in \N$, finite subsets~$U \subseteq P_n$ and~$S \subseteq P_{n+1}$, there
  exists~$V \subseteq P_{n}$ such that~$S$ moves~$U$ to~$V$ if and only if~$S^\mp
  \subseteq U$ and~$U \cap S^+ = \emptyset$.
\end{lem}

\begin{proof}
  If~$S$ moves~$U$ to~$V$, then, by definition,
  $S^\mp \subseteq (V \cup S^-) \setminus S^+ = U$ and \[U \cap S^+ = ((V \cup
    S^-) \setminus S^+ ) \cap S^+ = \emptyset.\] Conversely, if~$S^\mp \subseteq U$
  and~$U \cap S^+ = \emptyset$, let~$V = (U \cup S^+) \setminus S^-$.
  Then
  \begin{align*}
    (V \cup S^-) \setminus S^+ &= (U \cup S^+ \cup S^-) \setminus S^+ \\
    &= (U \setminus S^+) \cup (S^- \setminus S^+) \\
    &= U \cup S^\mp & \text{(since~$U \cap S^+ = \emptyset$)} \\
    &= U & \text{(since~$S^\mp \subseteq U$)}
  \end{align*}
  and~$S$ moves~$U$ to~$V$.
\end{proof}
\noindent The next property states that it is possible to modify a movement by adding or
removing ``independent'' elements.
\begin{lem}[{\cite[Proposition 2.2]{street1991parity}}]
  \label{lem:street-2.2}
  Let~$n \in \N$,~$U,V \subseteq P_n$ and~$S \subseteq P_{n+1}$ be finite subsets
  such that~$S$ moves~$U$ to~$V$. Then, given~$X,Y \subseteq P_n$ such that
  $
    X
    \subseteq U$,
    $X \cap S^\mp = \emptyset$
    and
    $Y \cap (S^- \cup S^+) =
    \emptyset$,
  we have that~$S$ moves~$(U \cup Y) \setminus X$ to~$(V \cup Y) \setminus X$.
\end{lem}
\begin{proof}
  By \Lemr{street-2.1}, we have~$S^\mp \subseteq U$ and~$U \cap S^+ =
  \emptyset$. Using the hypothesis, we can refine both equalities to $S^\mp
  \subseteq (U \cup Y) \setminus X$ and $((U \cup Y) \setminus X) \cap S^+ =
  \emptyset$. Using \Lemr{street-2.1} again,~$S$ moves~$(U \cup Y) \setminus X$
  to~$W$ where
  \begin{align*}
  W &= (((U \cup Y) \setminus X) \cup S^+) \setminus S^- \\
    &= ((U \cup S^+ \cup Y) \setminus X) \setminus S^- & \text{(since~$X \cap
                                                          S^+ \subseteq U \cap S^+
                                                          = \emptyset$)}\\
    &= (((U \cup S^+) \setminus S^-) \cup Y) \setminus X & \text{(since~$Y
                                                            \cap S^- = \emptyset$)}\\
    &= (V \cup Y) \setminus X . & 
  \end{align*}
  \vskip-\baselineskip\vskip-\belowdisplayskip
\end{proof}
\noindent The following property gives sufficient conditions for composing movements.
\begin{lem}[{\cite[Proposition 2.3]{street1991parity}}]
  \label{lem:street-2.3}
  Let~$n \in \N$, and~$U,V,W \subseteq P_n$,~$S,T \subseteq P_{n+1}$ be finite
  subsets such that~$S$ moves~$U$ to~$V$ and~$T$ moves~$V$ to~$W$, if~$S^- \cap
  T^+ = \emptyset$ then~$S \cup T$ moves~$U$ to~$W$.
\end{lem}
\begin{proof}
  We compute~$(U \cup (S \cup T)^+) \setminus (S \cup T)^-$:
  \begin{align*}
    (U \cup S^+ \cup T^+) \setminus (S^- \cup T^-) &= (((U \cup S^+) \setminus
                                                     S^-) \cup T^+) \setminus
                                                     T^- = (V \cup T^+) \setminus T^- = W .
  \end{align*}
  Similarly,~$(W \cup (S \cup T)^-) \setminus (S \cup T)^+ = U$ and~$S \cup T$
  moves~$U$ to~$W$.
\end{proof}
\noindent Conversely, the next property enables decomposing movements, under a
condition of orthogonality: given~$n \in \N$ and finite sets~$S,T \subseteq
P_n$, we say that~$S$ and~$T$ are \index{orthogonal}\emph{orthogonal},
\glossary(.r){$S \perp T$}{the property that~$S$ and~$T$ are
  orthogonal}written~$S \perp T$, when~$(S^- \cap T^-) \cup (S^+ \cap T^+) =
\emptyset$. We then have:
\begin{lem}[{\cite[Proposition 2.4]{street1991parity}}]
  \label{lem:street-2.4}
  Given~$n \in \N$, finite subsets~$U,W \subseteq P_n$ and~$S,T \subseteq P_{n+1}$ such
  that~${S \cup T}$ moves~$U$ to~$W$ and~$S^\mp \subseteq U$, if~$S \perp T$ then
  there exists~$V$ such that~$S$ moves~$U$ to~$V$ and~$T$ moves~$V$ to~$W$.
\end{lem}
\begin{proof}
  Let~$R = S \cup T$. By \Lemr{street-2.1},~$R^\mp \subseteq U$ and~$U \cap S^+ \subseteq U \cap R^+ = \emptyset$. By \Lemr{street-2.1}
  again,~$S$ moves~$U$ to~$V = (U \cup S^+) \setminus S^-$.
  Moreover,
  \begin{align*}
    S^- \cap T^+ & = S^\mp \cap T^+ & \text{(since~$S^+ \cap T^+ = \emptyset$, by~$S \perp T$)} \\
                 & \subseteq U \cap T^+ & \text{(since~$S^\mp \subseteq U$, by hypothesis)}\\
                 & \subseteq U \cap (S \cup T)^+ \\
                 & = \emptyset & \text{(by \Lemr{street-2.1}).}
  \end{align*}
  \belowlongtoshortskip
  so that
  \abovelongtoshortskip
  \begin{align*}
    && R^\mp &\subseteq U \\
    \lequiv && ((S^- \cup T^-) \setminus T^+) \setminus S^+ &\subseteq U  \\
    \lequiv && ((T^- \setminus T^+) \cup S^- ) \setminus S^+ &\subseteq U
                                                              & \text{(since~$S^- \cap T^+ = \emptyset$)}\\
    \lequiv && T^\mp \cup S^-   &\subseteq U \cup S^+\\
    \lequiv && T^\mp &\subseteq (U \cup S^+) \setminus S^- & \text{(since~$T^\mp \cap S^- = \emptyset$, by~$S \perp T$)}.
  \end{align*}
  Hence,~$T^\mp \subseteq (U \cup S^+) \setminus S^- = V$ and
  \begin{align*}
    V \cap T^+ \subseteq (U \cup S^+) \cap T^+ \subseteq (U \cap R^+) \cup (S^+ \cap T^+) = \emptyset.
  \end{align*}
  By \Lemr{street-2.1},~$T$ moves~$V$ to~$(V \cup T^+) \setminus T^-$. Moreover,
  \begin{align*}
    S^- \cap T^+ &= S^\mp \cap T^+ & \text{(since~$S \perp T$)} \\
                 &\subseteq U \cap R^+ & \text{(since~$S^\mp \subseteq U$ by hypothesis)}\\
                 &= \emptyset.
  \end{align*}
  \belowlongtoshortskip
  Therefore,
  \abovelongtoshortskip
  \begin{align*}
    (V \cup T^+) \setminus T^- &= (((U \cup S^+) \setminus S^-) \cup T^+)
    \setminus T^- \\
    &= (U \cup S^+ \cup T^+) \setminus (S^- \cup T^-) & \text{(since~$S^-
                                                         \cap T^+ = \emptyset$)}
    \\
    &= W .
  \end{align*}
  \belowlongtoshortskip
  Hence,~$T$ moves~$V$ to~$W$.
\end{proof}
\noindent The next three properties (not in~\cite{street1991parity}) describe
which elements are touched or left untouched by movement.
\begin{lem}
  \label{lem:move-eq}
  Given~$n \in \N$, finite subsets~$U,V \subseteq P_n$ and~$S \subseteq P_{n+1}$, if~$S$ moves~$U$ to~$V$, then
  \[
    {S^\mp = U \setminus V} \qtand S^\pm = V \setminus U\zbox.
  \]
  In particular, if~$T$ moves~$U$ to~$V$, then~$S^{\mp} = T^{\mp}$ and~$S^{\pm} = T^{\pm}$.
\end{lem}
\begin{proof}
  By the definition of movement, we have $V = (U \cup S^+) \setminus S^-$ and $U
  = (V \cup S^-) \setminus S^+$. Therefore, since~$U \cap S^+ = \emptyset$, $ U
  \cap V = U \cap ( (U \setminus S^-) \cup S^\pm) = U \setminus S^\mp$.
  Similarly,~$U \cap V = V \setminus S^\pm$. Hence,~$S^\mp = U \setminus V$
  and~$S^\pm = V \setminus U$.
\end{proof}
\begin{lem}
  \label{lem:unmoved-eq}
  Given~$n \in \N$, finite subsets~$U,V \subseteq P_n$ and~$S \subseteq P_{n+1}$, if~$S$ moves~$U$ to~$V$, then
  \[
    U \setminus S^- = U \setminus S^{\mp} = U \cap V = V \setminus S^{\pm} = V \setminus S^+\zbox.
  \]
\end{lem}
\begin{proof}
  We compute that
  \begin{align*}
  U \setminus S^- &= U \setminus S^{\mp} & \text{(since~$U \cap S^+ =
                                            \emptyset$, by definition of
                                            movement)} \\
                      &= U \cap V & \text{(by \Lemr{move-eq})} \\
                      &= V \setminus S^\pm \\
    &= V \setminus S^{+} & \text{(since~$V \cap S^- = \emptyset$, by
                              definition of movement)} 
  \end{align*}
\qedunskipendalign
\end{proof}

\begin{lem}
  \label{lem:move-canform}
  For~$n \in \N$, finite subsets~$U,V \subseteq P_n$ and~$S \subseteq P_{n+1}$, if~$S$ moves~$U$ to~$V$, then
  \[
    U = (U \cap V) \sqcup S^\mp \qtand V = (U \cap V) \sqcup S^\pm.
  \]
\end{lem}
\begin{proof}
  By \Lemr{unmoved-eq}, we have
  \begin{gather*}
    U = (V \cup S^-) \setminus S^+  = (V \setminus S^+) \cup (S^- \setminus S^+)  = (U \cap V) \cup S^\mp\zbox.
    \shortintertext{Moreover,}
    (U \cap V) \cap S^\mp  \subseteq V \cap S^-  = ((U \cup S^-) \setminus S^-) \cap S^-  = \emptyset\zbox.
  \end{gather*}
  Hence,~$U = (U \cap V) \sqcup S^\mp$. Similarly~$V = (U \cap V) \sqcup S^\pm$.
\end{proof}
\noindent
For~$n \in \N$, ${V \subseteq U \subseteq P_n}$, we say that~$V$ is
\index{initial set}\index{terminal set}\emph{initial (\resp terminal) in~$U$}
for~$\tl_U$ when, for all~$u \in U$, whenever there exists~$v \in V$ such
that~$u \tl_U v$ (\resp~$v \tl_U u$), we have~$u \in V$. The last lemma enables
decomposing a moving set starting from a subset which is a segment:
\begin{lem}
  \label{lem:segment-decomp-in-3}
  For~$n \in \N$, finite subsets~$U,V \subseteq P_n$,~$S \subseteq P_{n+1}$ and~$T \subseteq S$ such that~$S$ is fork-free and moves~$U$ to~$V$, and~$T$ is a
  segment in~$S$ for~$\tl_S$, there exist~$L,R \subseteq S$ and~$A,B \subseteq
  P_{n}$ such that
  \begin{itemize}
  \item $L,T,R$ is a partition of~$S$,
  \item $L$ is initial in~$S$ for~$\tl_S$ and~$R$ is final in~$S$ for~$\tl_S$,
  \item $L$ moves~$U$ to~$A$,~$T$ moves~$A$ to~$B$ and~$R$ moves~$B$ to~$V$.
  \end{itemize}
\end{lem}
\begin{proof}
  Let $ L = \set{ x \in S \mid x \tl_{S} T}$ and $R = S \setminus (L \cup T)$.
  $L,T,R$ is a partition of~$S$, and since~$S$ is fork-free, we have $ L \perp
  T$, $L \perp R$ and $T \perp R$. Since~$T$ is a segment for~$\tl_S$, we have
  that~$L^- \cap T^+ = \emptyset$, and, by definition of~$L$ and~$R$,~$L^- \cap
  R^+ = \emptyset$ so that~$L$ is initial in~$S$. In particular,~$L^\mp
  \subseteq U$. Thus, by \Lemr{street-2.3}, writing~$A$ for the set~$(U \cup
  L^+) \setminus L^-$, we have that~$L$ moves~$U$ to~$A$. Furthermore, since~$L
  \cap R = \emptyset$, we have~$T^- \cap R^+ = \emptyset$ so that~$R$ is
  terminal in~$S$. In particular,~$R^\pm \subseteq V$. Thus, by the dual of
  \Lemr{street-2.3}, writing~$B$ for~${(V \cup R^-) \setminus R^+}$, we have
  that~$R$ moves~$B$ to~$V$.
\end{proof}

\subsection{Proof of the gluing theorem}

We have now enough material to prove the gluing theorem:%
\ifx\theoremcellgluing\undefined
\begin{theo}
  Let~$P$ be an \ohg which satisfies Axioms~\forestax{0},~\forestax{1},
  \forestax{2} and~\forestax{3}. Given~${n \in \N}$, an $n$\cell~$X$ of~$P$
  and a finite fork-free set~$\glueset \subseteq P_{n+1}$ such that~$\glueset$
  is glueable on~$X$, we have that
  \begin{enumerate}[label=(\alph*),ref=(\alph*)]
  \item $\cellactivation(X,\glueset)$
    is a cell and~$\glueset^+ \cap X_n = \emptyset$,
  \item $\cellgluing(X,\glueset)$
    is a cell,
  \item given a finite fork-free subset~$\dualglueset \subseteq
    P_{n+1}$ which is dually glueable on~$X$,
    $\dualglueset^- \cap \glueset^+ = \emptyset$.
  \end{enumerate}
  and dual properties hold when~$\glueset$ is dually glueable on~$X$.
\end{theo}
\else
  \theoremcellgluing*
\fi
\begin{proof}
  See \Figr{21} for a representation of the cells in the statement of the
  theorem. In the following, write~$\movedset$ for
  \[
    S = \cellactivation(X,\glueset)_n = (X_n \cup \glueset^+) \setminus
    \glueset^-.
  \]
  We prove the different sub-properties (and their duals) of the theorem by
  induction on~$n$.


  \proofof{thm:cell-gluing:act} We prove \ref{thm:cell-gluing:act} in two steps:
  first, in the case where~$\setsize G = 1$, then, in the general case.

  \medskip \emph{Step~1: \ref{thm:cell-gluing:act} holds when~$\setsize{\glueset} = 1$.} Let~$x \in P_{n+1}$ be such that~$\set{x} =
  \glueset$. If~$n = 0$, then there exists~$y \in P_0$ such that~$X_0 =
  \set{y}$. By Axioms~\forestax{1} and~\forestax{2}, there exists~$z \in
  P_0$ with~$y \neq z$ such that~$x^- = \set{y}$ and~$x^+ = \set{z}$. So~$\cellactivation(X,\glueset) = \set{z}$ is a cell. So suppose that~$n > 0$.
  Then, we have~${\movedset=(X_n \cup x^+) \setminus x^-}$ and, in order to prove that~$\cellactivation(X,\glueset)$ is a cell, we are required to show that
\begin{itemize}
\item $\movedset$ moves~$X_{n-1,-}$ to~$X_{n-1,+}$;
\item $\movedset$ is fork-free.
\end{itemize}
Using \forestaxiom{3}, we get that~$x^-$ is a segment in~$X_n$ for~$\tl_{X_n}$. By \Lemr{segment-decomp-in-3}, we can decompose the set~$X_n$ as a
partition
\[
  X_n = \afstdecset \cup x^- \cup \asnddecset
\]
with~$\afstdecset$ initial and~$\asnddecset$ final in~$X_n$ and, writing~$A,B
\subseteq P_{n-1}$ for
\begin{gather*}
  A = (X_{n-1,-} \cup \afstdecset^+) \setminus \afstdecset^-
  \qtand
  B = (X_{n-1,+} \cup \asnddecset^-) \setminus \asnddecset^+
  \shortintertext{we have that}
  \text{$\afstdecset$ moves~$X_{n-1,-}$ to~$\afstborder$,}
  \qquad
  \text{$x^-$ moves~$\afstborder$ to~$\asndborder$,}
  \qquad
  \text{$\asnddecset$ moves~$\asndborder$ to~$X_{n-1,+}$}
\end{gather*}
\input{fig/gluing-cell-dec}%
as pictured on \Figr{gluing-a-dec}. In the following, for~$Z \subseteq P_{n-1}$,
we write~$D(Z)$ for the $(n{-}1)$\precell of~$P$ defined by
\[
  D(Z)_{n-1} = Z \quad\qtand\quad D(Z)_{i,\epsilon} = X_{i,\epsilon} \quad \text{for~$i \in \N_{n-2}$ and~$\epsilon \in \set{-,+}$}.
\]
  Since
    $D(\afstborder) = \cellactivation(D(X_{n-1,-}),\afstdecset)$,
    $D(\asndborder) = \cellactivation(D(\afstborder),x^-)$,
  and~${D(X_{n-1,-}) = \csrc_{n-1} (X)}$ is an $(n{-}1)$\cell and both~$\afstdecset$
  and~$x^-$ are fork-free, by using two times the induction hypothesis of
  \Thmr{cell-gluing}, first on~$D(X_{n-1,-})$, then on~$D(A)$, we get that
  \begin{equation}
    \label{eq:gluing-border-cells}
    \text{$D(\afstborder)$ and~$D(\asndborder)$ are cells.}
  \end{equation}
  By \forestaxiom{2}, we have that
  \begin{equation}
    \label{eq:gluing-xplus-ff}
    \text{$x^+$ is fork-free.}
  \end{equation}
  Since~$x^-$ moves~$\afstborder$ to~$\asndborder$, by \Lemr{street-2.1}, we get
  \begin{equation}
    \label{eq:gluing-afstborder-cap-xmp}
    \text{$\afstborder \cap x^{-+} = \emptyset$.}
  \end{equation}
  By \forestaxiom{2}, it holds that~$x^{+\mp} = x^{-\mp} \subseteq \afstborder$.
  By~\eqref{eq:gluing-border-cells} and~\eqref{eq:gluing-xplus-ff}, using the
  induction hypothesis of \Thmr{cell-gluing} on~$D(\afstborder)$, we
  get
\begin{equation}
  \label{eq:gluing-afstborder-cap-xpp}
  \text{$\afstborder \cap x^{++} = \emptyset$.}
\end{equation}
By \Lemr{street-2.1}, there exists~$\asndborder'$ such that~$x^+$ moves~$\afstborder$ to~$\asndborder'$, and
\begin{align*}
  \asndborder' &= (\afstborder \cup x^{++}) \setminus x^{+-} \\
     &= (\afstborder \setminus x^{+-}) \cup (x^{++} \setminus x^{+-}) \\
     &= (\afstborder \setminus x^{+\mp}) \cup x^{+\pm} & \text{(by~\eqref{eq:gluing-afstborder-cap-xpp})} \\
     &= (\afstborder \setminus x^{-\mp}) \cup x^{-\pm} & \text{(since~$x^{+\mp} = x^{-\mp}$, by \forestaxiom{2})} \\
     &= (\afstborder \setminus x^{--}) \cup (x^{-+} \setminus x^{--}) & \text{(by~\eqref{eq:gluing-afstborder-cap-xmp})} \\
     &= (\afstborder \cup x^{-+}) \setminus x^{--} \\
     &= \asndborder & \text{(since~$x^-$ moves~$\afstborder$ to~$\asndborder$).}
\end{align*}
Hence,
\begin{equation}
  \label{eq:gluing-xplus-moves}
  \text{$x^+$ moves~$\afstborder$ to~$\asndborder$.}
\end{equation}


\noindent Since~$x^{+\mp} \subseteq D(\afstborder)_{n-1}$ and~$\afstdecset^\pm
\subseteq D(\afstborder)_{n-1}$, using the induction hypothesis
of \Thmr{cell-gluing}, by~\ref{thm:cell-gluing:inter} we get
\begin{equation}
  \label{eq:gluing-xpp-cap}
  \afstdecset^- \cap x^{++} = \emptyset.
\end{equation}
Similarly, with~$D(\asndborder)$, we get
\begin{equation}
  \label{eq:gluing-xpm-cap}
  x^{+-} \cap \asnddecset^+ = \emptyset.
\end{equation}
By definition,~$\afstdecset$ moves~$X_{n-1,-}$ to~$\afstborder$, and~$x^+$
moves~$\afstborder$ to~$\asndborder$ by~\eqref{eq:gluing-xplus-moves}. Moreover,
by~\eqref{eq:gluing-xpp-cap}, we have that~$\afstdecset^- \cap x^{++} = \emptyset$. Using
\Lemr{street-2.3}, we deduce that
\begin{equation}
  \label{eq:gluing-afstdecset-moves}
  \text{$\afstdecset \cup x^+$ moves~$X_{n-1,-}$ to~$\asndborder$.}
\end{equation}
Since~$\afstdecset$ and~$\asnddecset$ are disjoint and respectively initial and
terminal in~$X_n$,~$\afstdecset^- \cap \asnddecset^+ =
\emptyset$. Also, by~\eqref{eq:gluing-xpm-cap}, we have~$(x^{+-} \cap
\asnddecset^+) = \emptyset$, therefore
\[
  (\afstdecset \cup x^{+})^- \cap \asnddecset^+  \subseteq (\afstdecset^- \cap \asnddecset^+) \cup (x^{+-} \cap \asnddecset^+) 
                                                  = \emptyset.
\]
Using~\eqref{eq:gluing-afstdecset-moves} and \Lemr{street-2.3}, knowing that~$\movedset = \afstdecset \cup x^+ \cup
\asnddecset$, we deduce that
\begin{equation}
  \label{eq:gluing-ms-moves}
  \text{$\movedset$ moves~$X_{n-1,-}$ to~$X_{n-1,+}$.}
\end{equation}
The set~$\afstdecset \cup \asnddecset$ is fork-free as a subset of the fork-free~$X_n$, and~$x^+$ is fork-free since~$x$ is relevant by~\forestaxiom{2}.
Moreover,
  \begin{align*}
    \afstdecset^- \cap x^{+-} &= \afstdecset^- \cap x^{+\mp}  & \text{(by~\eqref{eq:gluing-xpp-cap})}\\
                              &\subseteq \afstdecset^- \cap \afstborder  & \text{(by~\eqref{eq:gluing-xplus-moves} and \Lemr{street-2.1})}\\
                              &= \emptyset & \text{(since~$\afstdecset$ moves~$X_{n-1,-}$ to~$\afstborder$),}\\
    \afstdecset^+ \cap x^{++} &= \afstdecset^{\pm} \cap x^{++}  & \text{(by~\eqref{eq:gluing-xpp-cap})}\\
                              & \subseteq \afstborder \cap x^{++} & \text{(by \Lemr{street-2.1} since~$\afstdecset$ moves~$X_{n-1,-}$ to~$\afstborder$)}\\
                              &= \emptyset &\text{(by~\eqref{eq:gluing-xplus-moves} and \Lemr{street-2.1})}.
  \end{align*}
  So~$\afstdecset \perp x^+$. Similarly,~$x^+ \perp \asnddecset$. Hence, since~$\movedset = \afstdecset \cup x^+ \cup \asnddecset$,
\begin{equation}
  \label{eq:gluing-ms-ff}
  \text{$\movedset$ is fork-free.}
\end{equation}
Then, by~\eqref{eq:gluing-ms-moves} and~\eqref{eq:gluing-ms-ff},
$\cellactivation(X,\glueset)$ is a cell.

We now prove the second part of~\ref{thm:cell-gluing:act}. By \forestaxiom{1},~$x^- \cap x^+ = \emptyset$ .
Since~$\afstdecset \perp x^+$ and~$x^+ \perp \asnddecset$
(by~\eqref{eq:gluing-ms-ff}), using \forestaxiom{0}, we deduce that
$ \afstdecset \cap x^+ = x^+ \cap \asnddecset = \emptyset $
so that
\[
  X_n \cap x^+ = (\afstdecset \cup x^- \cup \asnddecset) \cap x^+ = \emptyset
\]
which concludes the proof of the Step~1.

\medskip
  \emph{Step~2: \ref{thm:cell-gluing:act} holds.}
  We prove this by induction on~$\setsize{\glueset}$. If~$\setsize{\glueset} =
  0$, then the result is trivial. Moreover, the case~$\setsize{\glueset} = 1$
  was proved in Step~1. So suppose that~$\setsize{\glueset} \ge 2$. Since the
  relation~$\tl$ is acyclic by~\forestaxiom{1}, we can consider a
  minimal~$x\in\glueset$ for~$\tl_\glueset$. Let
  \[
    \bsmallerglueset = \glueset
    \setminus \set{x}
    ,\quad
    \bfstdecset=(X_n \cup x^+) \setminus x^-
    ,\quad
    \bsnddecset=(\bfstdecset \cup \bsmallerglueset^+) \setminus \bsmallerglueset^-
  \]
  and recall that we defined~$S$ as~$(X_n \cup \glueset^+) \setminus
  \glueset^-$. In order to show that~$\cellactivation(X,\glueset)$ is a cell, we
  are required to prove that $\movedset$ moves~$X_{n-1,-}$ to~$X_{n-1,+}$, and
  that $\movedset$ is fork-free.
  For this purpose, we will first move~$X_{n}$ with~$\set{x}$ to~$\bfstdecset$
  and use Step~1, then move~$\bfstdecset$ by~$\bsmallerglueset$ to~$\bsnddecset$
  and use the induction of Step~2. Finally, we will prove that~$\bsnddecset =
  \movedset$. So, using Step~1 with~$X$ and~$\set{x}$, we get that
  \begin{itemize}
  \item $\cellactivation(X,\set{x})$ is a cell;
  \item in particular,~$\bfstdecset$ is fork-free and, when~$n > 0$,~$\bfstdecset$ moves~$X_{n-1,-}$
    to~$X_{n-1,+}$;
  \item $X_{n} \cap x^+ = \emptyset$.
  \end{itemize}
  By \Lemr{street-2.1}, we deduce that \set{$x$} moves~$X_n$ to~$U$.
  Moreover,
  \begin{align*}
    \bsmallerglueset^\mp & = \bsmallerglueset^- \setminus \bsmallerglueset^+ \\
    & = (\glueset^- \setminus x^-) \setminus (\glueset^+\setminus x^+) & \text{(since fork-freeness implies
      that~$\glueset^\eps = \sqcup_{u \in \glueset} u^\eps$)} \\
    & \subseteq ((\glueset^- \setminus x^-) \setminus \glueset^+) \cup x^+ \\
    & = ((\glueset^- \setminus \glueset^+) \setminus x^- )\cup x^+ \\
    & \subseteq (X_n \setminus x^- )\cup x^+ & \text{(since~$\glueset^\mp \subseteq
      X_n$ by \Lemr{street-2.1})} \\
    & \subseteq (X_n \cup x^+) \setminus x^- & \text{(since~$x^- \cap x^+ = \emptyset$ by \forestaxiom{1})}\\
    & = \bfstdecset.
  \end{align*}
  Also,~$\bsmallerglueset$ is fork-free as a subset of the fork-free set~$\glueset$. Using the
  induction hypothesis of Step~2 for~$\bsmallerglueset$, we get that
  \begin{itemize}
  \item $\cellactivation(\cellactivation(X,\set{x}),\bsmallerglueset)$ is a cell;
  \item In particular,~$\bsnddecset = (\bfstdecset \cup \bsmallerglueset^+)
    \setminus \bsmallerglueset^-$ is fork-free, and, when~$n > 0$,~$\bsnddecset$
    moves~$X_{n-1,-}$ to~$X_{n-1,+}$;
  \item $\bfstdecset \cap \bsmallerglueset^+ = \emptyset$.
  \end{itemize}
  By \Lemr{street-2.1}, we deduce that~$\bsmallerglueset$ moves~$U$ to~$V$.
  Also, note that~$x^- \cap \bsmallerglueset^+ = \emptyset$ since~$x$ was taken minimal in~$\glueset$.
  Using \Lemr{street-2.3}, we deduce that~$\glueset = \{x\} \cup \bsmallerglueset$ moves~$X_n$
  to~$\bsnddecset$. But~${\movedset = (X_n \cup \glueset^+) \setminus
    \glueset^-}$ so that~$\movedset = \bsnddecset$.
  
  \medskip \noindent The second part of \ref{thm:cell-gluing:act} is left to show, that is,~$X_n \cap \glueset^+ = \emptyset$. We compute that
  \begin{align*}
  X_n \cap \glueset^+ & =  (\bfstdecset \cup x^- \setminus x^+) \cap \glueset^+ & \text{(by \Lemr{street-2.1}, since \set{$x$} moves~$X_n$ to~$U$)}\\
  & =   ((\bfstdecset \cup x^-) \cap \glueset^+) \setminus x^+ &  \\
  & =  (\bfstdecset \cap \glueset^+) \setminus x^+ & \text{(since~$x^- \cap \glueset = x^- \cap (x^+ \cup \bsmallerglueset)  = \emptyset$)}\\
  & =  (\bfstdecset \cap \bsmallerglueset^+) =  \emptyset
  \end{align*}
  which concludes the proofs of Step~2 and \ref{thm:cell-gluing:act}.

  \proofof{thm:cell-gluing:glue} By \ref{thm:cell-gluing:act},~$\cellactivation(X,\glueset)$ is a cell. To conclude, we need to show that~$\glueset$ moves~$X_n$ to~$\movedset$. By definition of~$\movedset$, we have
  that~$\movedset = (X_n \cup \glueset^+) \setminus \glueset^-$. Moreover,
  \begin{align*}
    (\movedset \cup \glueset^-) \setminus \glueset^+ & = (((X_n \cup \glueset^+) \setminus \glueset^-)\cup \glueset^-)\setminus \glueset^+ \\
    & =  (X_n \cup \glueset^+ \cup \glueset^-) \setminus \glueset^+ \\
    & =  (X_n\setminus \glueset^+) \cup \glueset^\mp \\
    & =  X_n \cup \glueset^\mp & \text{(since~$X_n \cap \glueset^+ = \emptyset$ by~\ref{thm:cell-gluing:act})} \\
    & =  X_n & \text{(since~$\glueset$ is glueable on~$X$).}
  \end{align*}
  Hence,~$\cellgluing(X,\glueset)$ is a cell.

  \proofof{thm:cell-gluing:inter} By contradiction, suppose that~$\dualglueset^-
  \cap \glueset^+ \neq \emptyset$. Then, there are~${\dualel \in
    \dualglueset}$,~${\normalel \in \glueset}$ and~$\commonel \in \dualel^- \cap
  \normalel^+$. Consider $ \dfstdecset = \{ \dualelvar \in \dualglueset \mid
  \dualel \tl_{\dualglueset} \dualelvar \} \cup \set{\dualel}, $ and $
  \dsnddecset = \{ \normalelvar \in \glueset \mid \normalelvar \tl_\glueset
  \normalel \}\cup \set{\normalel} $. By the acyclicity~\forestaxiom{1}, we have
  $ \dfstdecset^+ \cap \dsnddecset^- = \emptyset $. Since~$\dfstdecset$ is a
  terminal set for~$\tl_{\dualglueset}$, we have in particular~$\dfstdecset^+
  \cap \dualglueset^- \subseteq \dfstdecset^-$. So,
\begin{align*}
  \dfstdecset^+ = (\dfstdecset^+ \setminus \dualglueset^-) \cup (\dfstdecset^+ \cap \dualglueset^-) \subseteq \dualglueset^\pm \cup \dfstdecset^-. 
\end{align*}
Hence,~$\dfstdecset^\pm \subseteq \dualglueset^\pm \subseteq X_n$ (since~$\dualglueset$ is dually glueable on~$X$). Similarly,~$\dsnddecset^\mp \subseteq
X_n$. Using the dual version of \ref{thm:cell-gluing:act}, the $n$\precell~${\dactcell =
  \dualcellactivation(X,\dfstdecset)}$ is an $n$\cell where~${\dactcell_{n}
  = (X_n \cup \dfstdecset^-) \setminus \dfstdecset^+}$ (see \Figr{23}) and we
have
\begin{align*}
  \dsnddecset^\mp &= \dsnddecset^\mp \setminus \dfstdecset^+ & \text{(since~$\dsnddecset^- \cap \dfstdecset^+ = \emptyset$)} \\
                  &\subseteq X_n \setminus \dfstdecset^+ & \text{(since~$\dsnddecset^\mp \subseteq X_n$)}\\
                  &\subseteq (X_n \cup \dfstdecset^-) \setminus \dfstdecset^+= \dactcell_{n}.
\end{align*}%

\input{fig/gluing-contra}
\noindent Using \Thmr{cell-gluing}\ref{thm:cell-gluing:act} with~$\dactcell$
and~$\dsnddecset$, we get $\dactcell_{n} \cap \dsnddecset^+ = \emptyset$. But,
since~$\commonel \in \dfstdecset^\mp \subseteq \dactcell_{n}$ (by
\forestaxiom{1}) and~$\dfstdecset^\mp \subseteq \dactcell_{n}$,~$\commonel \in
\dactcell_{n} \cap \dsnddecset^+$, which is a contradiction. Hence, $
\dualglueset^- \cap \glueset^+ = \emptyset $ which ends the proof of
\ref{thm:cell-gluing:inter}.
\end{proof}


%% file: fig/gluing-cell-dec.tex
\begin{figure}
  \centering
\begin{tikzpicture}[scale=0.7]
 \coordinate (dn-2s) at (0,3);
 \coordinate (xn-2s) at (0,1);
 \coordinate (dn-2t) at (0,-3);
 \coordinate (xn-2t) at (0,-1);
 \draw[red,line width=2pt] ([xshift=-1pt]dn-2s) |- ([xshift=-1pt]xn-2s);
 \draw[red,line width=2pt] ([shift=(90:1cm+1pt)]0,0) arc (90:270:1cm+1pt);
 \draw[red,line width=2pt] ([xshift=-1pt]dn-2t) -- ( [xshift=-1pt]dn-2t |- xn-2t);
 \node[red] at (-1.3,0) {$\afstborder$};
 \draw[blue,line width=2pt] ([xshift=1pt]dn-2s) |- ([xshift=1pt]xn-2s);
 \draw[blue,line width=2pt] ([shift=(90:1cm+1pt)]0,0) arc (90:-90:1cm+1pt);
 \draw[blue,line width=2pt] ([xshift=1pt]dn-2t) -- ( [xshift=1pt]dn-2t |- xn-2t);
 \node[blue] at (1.3,0) {$\asndborder$};
 \draw (0,0) circle (1); 
 \draw (0,0) circle (3); 
 \node at (dn-2s) {$\bullet$};
 \node at (dn-2t) {$\bullet$};
 \node at (xn-2s) {$\bullet$};
 \node at (xn-2t) {$\bullet$};
 \node at (-3.8,0)  {$X_{n-1,-}$};
 \node at (3.8,0)  {$X_{n-1,+}$};
 \node at (0,3.5)  {$X_{n-2,-}$};
 \node at (0,-3.5)  {$X_{n-2,+}$};
 \draw (dn-2s) -- (xn-2s);
 \draw (dn-2t) -- (xn-2t);
 \node (flechex) at (0,0) {\myoverset{$\Rightarrow$}{$x^-$}};
 \node (flecheS) at (-2,0) {\myoverset{$\Rightarrow$}{$\afstdecset$}};
 \node (flecheT) at (2,0) {\myoverset{$\Rightarrow$}{$\asnddecset$}};
\end{tikzpicture}
  \caption{The decomposition of $X_n$}
  \label{fig:gluing-a-dec}
\end{figure}%

%% file: fig/gluing-contra.tex
\begin{figure}
  \centering
\begin{tabular}{ccc}
\begin{tikzpicture}[baseline=(current bounding box.center),scale=0.6]
 \coordinate (source) at (0,4);
 \coordinate (target) at (0,-4);
 \coordinate (centrew) at ($ (0,0) + (-2,-2) $);
 \coordinate (debutW) at ($(0,-1)$);
 \coordinate (debutW') at ($ (centrew) + (90:0.5) $);
 \coordinate (finW) at (0,-3);
 \coordinate (finW') at ($ (centrew) + (-90:0.5) $);
 \coordinate (centrev) at ($(2,2)$);
 \coordinate (debutV) at ($(0,3)$);
 \coordinate (debutV') at ($ (centrev) + (90:0.5) $);
 \coordinate (finV) at ($ (centrev) + (-90:0.5) $);
 \coordinate (finV') at (0,1);

 \draw (0,0) circle (4);
 \draw[red] (source) -- (target);
 \node[red,right] at (0,0) {$X_n$};
 \draw (debutV) -- (debutV');
 \draw (finV) -- (finV');
 \draw (debutW) -- (debutW');
 \draw (finW) -- (finW');
 \draw (centrev) circle (0.5);
 \draw[very thick,->] ([shift=(30:0.5)]centrev) arc (30:-30:0.5);
 \node[right] at ([shift=(0:0.5)]centrev) {$\commonel$};
 \node at (centrev) {\begin{tabular}{c}$\normalel$\\[-5pt]$\Rightarrow$\end{tabular}};
 \draw (centrew) circle (0.5);
 \draw[very thick,<-] ([shift=(-150:0.5)]centrew) arc (-150:-210:0.5);
 \node[left] at ($ (centrew) + (-0.5,0) $) {$\commonel$};
 \node at (centrew) {\begin{tabular}{c}$\dualel$\\[-5pt]$\Rightarrow$\end{tabular}};
 \node at (-2,0) {\begin{tabular}{c}$\dualglueset$\\[-5pt]$\Rightarrow$\end{tabular}};
 \node at (2,0) {\begin{tabular}{c}$\glueset$\\[-5pt]$\Rightarrow$\end{tabular}};
 \node at ($ (centrew) + (1.2,0) $) {\begin{tabular}{c}$\dfstdecset$\\[-5pt]$\Rightarrow$\end{tabular}};
 \node at ($ (centrev) + (-1.2,0) $) {\begin{tabular}{c}$\dsnddecset$\\[-5pt]$\Rightarrow$\end{tabular}};

\end{tikzpicture}
&
$\xrightarrow{\makebox[3cm]{\text{\Thmr{cell-gluing}}}}$
&
\begin{tikzpicture}[baseline=(current bounding box.center),scale=0.6]
 \coordinate (source) at (0,4);
 \coordinate (target) at (0,-4);
 \coordinate (centrew) at ($ (0,0) + (-2,-2) $);
 \coordinate (debutW) at ($(0,-1)$);
 \coordinate (debutW') at ($ (centrew) + (90:0.5) $);
 \coordinate (finW) at (0,-3);
 \coordinate (finW') at ($ (centrew) + (-90:0.5) $);
 \coordinate (centrev) at ($(2,2)$);
 \coordinate (debutV) at ($(0,3)$);
 \coordinate (debutV') at ($ (centrev) + (90:0.5) $);
 \coordinate (finV) at ($ (centrev) + (-90:0.5) $);
 \coordinate (finV') at (0,1);
 \draw[blue] (source) -- (debutW) -- (debutW');
 \draw[blue] (finW') -- (finW) -- (target);
 \draw (debutV) -- (debutV');
 \node[blue,right] at (0,0) {$\dactcell_n$};
 \draw (finV) -- (finV');
 \draw (centrev) circle (0.5);
 \draw[very thick,->] ([shift=(30:0.5)]centrev) arc (30:-30:0.5);
 \node[right] at ([shift=(0:0.5)]centrev) {$\commonel$};
 \node at (centrev) {\begin{tabular}{c}$\normalel$\\[-5pt]$\Rightarrow$\end{tabular}};
 \draw[blue] ([shift=(90:0.5)]centrew) arc (90:270:0.5);
 \draw[very thick,<-] ([shift=(-150:0.5)]centrew) arc (-150:-210:0.5);
 \node[left] at ($ (centrew) + (-0.5,0) $) {$\commonel$};
 \node at ($ (centrev) + (-1.2,0) $) {\begin{tabular}{c}$\dsnddecset$\\[-5pt]$\Rightarrow$\end{tabular}};
\end{tikzpicture}
\end{tabular}
  \caption{$\dsnddecset$, $\dfstdecset$ and $\dactcell_n$}
  \label{fig:23}
\end{figure}

%% file: details-freeness.tex
This section is devoted to prove \Cref{thm:ext-freeness}, stating that the \ocat
of cells of a torsion-free complex is a tower of cellular extensions. For this
purpose, we first need a concrete description of cells of free extensions. Such
a description was already introduced by Makkai when he gave a solution to the
word problem on free strict categories~\cite{makkai2005word}. It is based on an
alternate definition of strict categories using another set of operations.
Indeed, the standard set of operations, \ie identity
operations~$\unit{}$ and operations~$\comp$ which allow composing pairs of cells
of homogeneous dimensions, is inconvenient for finding such a description, since
this homogeneous constraint requires an intensive use of identity operations to
lift dimensions of cells, in order to compose them with other cells of higher
dimensions. Instead, strict categories can be seen as instances of another kind
of higher categories called \emph{precategories} which satisfy an additional
condition. In precategories, the composition operation~$\pcomp$ of precategories
allows composing cells of heterogeneous dimensions, which is more convenient for
formal handling. Based on this new definition, syntactical devices called
\emph{contexts} and \emph{context classes} can be developed. These formally
represent whiskering operations on generators as composites with one hole. Using
them, cells of free extensions can be described as adequately quotiented
sequences of context classes formally applied to generators. From this
description, the freeness of the \ocat of cells of torsion-free complexes can be
proved.

In \Cref{ssec:scat-another-def}, we first introduce the definition of
precategories and show that strict categories can be interpreted as
precategories satisfying an additional exchange condition. Then, in
\Cref{text:ctxts-ctxtcls}, we introduce contexts and context classes for strict
categories, together with several natural operations on them. In
\Cref{ssec:polextp-description}, we give a description, in the form of
\Cref{thm:descr-polextp}, of the functor $\polextp--$ from cellular extensions
to strict categories using the intermediate notion of categorical actions, the
latter describing the structure that context classes have with regard to the underlying
strict category. Concretely, the cells of free extensions will be classes of
sequences of formally applied context classes. In \Cref{ssec:cell-dec}, we prove
that the cells of torsion-free complexes
admit a decomposition of this form (\Cref{thm:cell-dec}). In
\Cref{ssec:freeness-proof}, we finally prove \Cref{thm:ext-freeness} by showing
that this decomposition is essentially unique, which precisely characterizes
that a strict category is a free extension. If required, more detailed proofs
can be found in~\cite{forest:tel-03155192}.

\subsection{Another definition of strict categories}
\label{ssec:scat-another-def}

In this section, we introduce an alternate definition of strict categories as
particular instances of precategories.

\paragraph{Precategories}
\parlabel{text:precategories-def}

Precategories can be described, in a sense that will be made precise
in the next paragraph, as ``strict categories without exchange law'' and
generalize in higher dimensions the $2$\dimensional theory of sesquicategories
defined by Street in~\cite{street1996categorical}. They were already introduced
by \citeauthor{makkai2005word} in order to describe the cells of free
strict categories~\cite{makkai2005word}.

Given~$n\in\N \cup \set\omega$, an \index{precategory}\emph{$n$\precategory}~$C$ is an $n$\globular
set together with, for~$k \in \N_{n-1}$, identity \glossary(idabpcat){$\unitp k
  {}$}{the identity operation for precategories}operations
\[
  \unitp{k+1} {}\colon C_{k}\to C_{k+1}
\]
for which we use the same notation conventions than the identity operations on
strict categories,
and, for~$k,l \in \N^*_n$, composition \glossary(.cab){$\pcomp_i$}{the
  composition operation for precategories}operations
\[
  \pcomp_{k,l}\co C_k\times_{\min(k,l)-1}C_l\to C_{\max(k,l)}
\]
which satisfy the axioms below.
Given~$i,k,l \in \N_n$ with~${i = \min(k,l)}$, since the dimensions of the cells
determine the functions to be used, we often write~$\pcomp_i$
for~$\pcomp_{k,l}$. This way, we still display the most important information
which is the dimension~$i$ of composition.
The axioms of $n$\precategories are the
following:
\begin{enumerate}[label=(P-\roman*),ref=(P-\roman*)]
\item \label{precat:first}\label{precat:src-tgt-unit}for~$k \in \N_{n-1}$ and~$u\in C_k$,
  \[
    \csrc_k(\unitp {k+1}{u})=u=\ctgt_k(\unitp {k+1}{u}),
  \]
\item \label{precat:csrc-tgt}for~$i,k,l \in \N_n$ such that~$i = \min(k,l) -1$,~$(u,v)\in C_k\times_iC_l$, and~${\eps \in \set{-,+}}$,
  \begin{align*}
    \csrctgt\eps(u \pcomp_i v)
    &=
    \begin{cases}
      u\pcomp_i \csrctgt\eps(v)& \text{if~$k < l$,}\\
      \csrc(u)&\text{if~$k=l$ and~$\eps = -$,}\\
      \ctgt(v)&\text{if~$k=l$ and~$\eps = +$,}\\
      \csrctgt\eps(u)\pcomp_i v&\text{if~$k>l$,}
    \end{cases}
  \end{align*}
\item \label{precat:compat-id-comp}for~$i,k,l \in \N_n$ with~$i = \min(k,l) -
  1$, given~$(u,v)\in C_{k-1}\times_{i}C_l$,
  \begin{align*}
    \unit u\pcomp_i v
    &=
    \begin{cases}
      v&\text{if~$k \le l$,}\\
      \unit{u\pcomp_i v}&\text{if~$k > l$,}
    \end{cases}
  \shortintertext{and, given~$(u,v) \in C_k \times_i C_{l-1}$,}
    u\pcomp_i\unit{v}
    &=
    \begin{cases}
      u&\text{if~$l\le k$,}\\
      \unit{u\pcomp_i v}&\text{if~$l > k$,}
    \end{cases}
  \end{align*}

\item \label{precat:before-last} \label{precat:assoc}for~$i,k,l,m \in \N_n$ with~$i = \min(k,l) - 1 =
  \min(l,m) - 1$, and~$u \in C_k$,~$v \in C_l$ and~$w \in C_w$ such that~$u,v,w$
  are $i$\composable,
  \[
    (u\pcomp_iv)\pcomp_iw
    =
    u\pcomp_i(v\pcomp_iw),
  \]
\item \label{precat:distrib} for every~$i,j,k \in \N_{n}$ satisfying~$i < j <
  k$, and cells~$u_1,u_2 \in C_{i+1}$,~$v_1,v_2 \in C_{j+1}$ and~$w \in C_k$
  such that~$u_1,w,u_2$ are $i$\composable and~$v_1,w,v_2$ are $j$\composable,
  we have
    \[
      u_1 \pcomp_i (v_1 \pcomp_j w \pcomp_j v_2) \pcomp_i  u_2 = (u_1 \pcomp_i
      v_1 \pcomp_i u_2) \pcomp_j (u_1
      \pcomp_i w \pcomp_i u_2) \pcomp_j (u_1 \pcomp_i v_2 \pcomp_i u_2)\zbox.
    \]
\end{enumerate}
\noindent Given two $n$\precategories~$C$ and~$D$, a \index{morphism!of
  precategories}\emph{morphism of $n$-precategories between~$C$ and~$D$} (also
called \index{prefunctor@$n$\prefunctor}\emph{$n$\prefunctor}), is a morphism of
$n$\globular sets~$F\co C \to D$ which moreover commutes with the
operations~$\unit{}$ and~$\pcomp$. We \glossary(PCatn){$\nPCat n$}{the category
  of $n$\precategories}write~$\nPCat n$ for the category of $n$\precategories.

\paragraph{Strict categories as precategories}
\parlabel{text:eq-strcat-precat}

In this section, we recall from~\cite{makkai2005word} how
strict categories can be expressed as precategories satisfying a condition
analogous to the exchange law.

For~$n \in \Ninf$ and~$C \in \nPCat n$, we write~\ref{cond:pcat-to-strict} for the
following property on~$C$:
\begin{enumerateaxioms}[label=(E),ref=(E)]
\item\label{cond:pcat-to-strict} for~$i,k,l \in \N_n$ with~$1 \le i =
  \min(k,l)-1$,~$u \in C_k$ and~$v \in C_l$, if~$u,v$ are $(i{-}1)$\composable,
  then
  \[
    (u \pcomp_{i-1} \csrc_{i}(v)) \pcomp_{i} (\ctgt_{i}(u) \pcomp_{i-1} v) =
    (\csrc_{i}(u) \pcomp_{i-1} v) \pcomp_{i} (u \pcomp_{i-1} \ctgt_{i}(v))\zbox.
  \]
\end{enumerateaxioms}
Let~$\smash{\nPCatstr n}$\glossary(PCatE){$\nPCatstr n$}{the category of
  precategories which satisfy the exchange condition~\ref{cond:pcat-to-strict}}
be the full subcategory of~$\nPCat n$ of those $n$\precategories~$C$ that
satisfy~\ref{cond:pcat-to-strict}. The condition~\ref{cond:pcat-to-strict} can
be thought as an equivalent for precategories of the exchange law~\ref{cat:xch}
of strict categories.

Given~$n \in \Ninf$ and~$C \in \nCat n$, we define a structure of
$n$\precategory on the underlying $n$\globular set of~$C$. We keep the
identities given by the strict $n$\category structure and define the composition
operations~$\pcomp_{(-)}$ on~$C$ based on the composition
operations~$\comp_{(-)}$. Given~$i,k,l \in \N_n$ with~$i = \min(k,l) - 1$,~$u
\in C_k$ and~$v\in C_l$ such that~$u,v$ are $i$\composable, we put
\[
  u\pcomp_iv =
  \unitp{m}{u}\comp_i\unitp{m}{v}
\]
where~$m = \max(k,l)$. 
We can then prove that we indeed obtain an $n$\precategory which
satisfies~\ref{cond:pcat-to-strict} and that the construction is functorial.

Conversely, given~$C \in \smash{\nPCatstr n}$, we define a structure of strict
$n$\category on the underlying $n$\globular set of~$C$. We keep the identities
given by the structure of $n$\precategory of~$C$ as before and define the
multiple composition operations~$\comp_{(-)}$ based on the precategorical
composition operations~$\pcomp_{(-)}$. For~$i,k \in \N_{n}$ with~${i < k}$, we
define~$u\comp_iv$ for $i$\composable~$u,v\in C_{k}$ by induction on~$k - i$.
If~$i = k-1$, we put~$u\comp_i v = u\pcomp_i v$. Otherwise, if~$i < k-1$, we
define~$u\comp_i v$ inductively by
\begin{equation*}
  \label{eq:comp-from-pcomp}
  u\comp_iv=(u\pcomp_i\csrc_{i+1}(v))\comp_{i+1}(\ctgt_{i+1}(u)\pcomp_iv)\zbox.
\end{equation*}
We can then prove that we obtain a strict $n$\category and that the construction is functorial.
\begin{theo}
  \label{prop:strcat-pcat-isom}
  The two constructions define an isomorphism between $\nCat n$ and~$\nPCatstr n$.
\end{theo}
\noindent Thus, a strict $n$\category~$C$ is canonically an $n$\precategory
satisfying~\ref{cond:pcat-to-strict} (and \viceversa). For our purposes, we will
often prefer the definition of strict categories and use the precategorical
compositions~$\pcomp_{(-)}$ on a strict category without invoking
\Cref{prop:strcat-pcat-isom}.

\subsection{Contexts and contexts classes}
\label{text:ctxts-ctxtcls}

Here, we introduce \emph{contexts} and \emph{context classes}, that represent
formal cells of strict categories with ``holes'' in them. Our definitions are
similar to the one of context given by \citeauthor{metayer2008cofibrant}
in~\cite{metayer2008cofibrant}, but with a stronger syntactical perspective.

\paragraph{Definition}
\parlabel{text:ctxt-ctxtcl-def}

Let~$n \in \Ninf$ and~$G$ be an $n$\globular set. Given~$k \in \N_n$ and~$u,v
\in X_k$, $u$ and~$v$ are said \emph{parallel} when $k = 0$
or~$\csrctgt\eps_{k-1}(u) = \csrctgt\eps_{k-1}(v)$ for~$\eps \in \set{-,+}$.
Given~$m \in \N_n$, an \index{type}\emph{$m$\type} is a pair~$(u,u')$ of
parallel $(m{-}1)$\globes of~$G$ (we use the convention that there is a unique
$({-}1)$\globe~$\ast$ which is parallel with itself). Given~$k \in \N_n$ with~$k
\ge m$ and~${v \in G_k}$, the \emph{$m$\type of~$v$} is the
$m$\type~$(\csrc_{m-1}(v),\ctgt_{m-1}(v))$ so that every $k$\cell can be
implicitly considered as an $m$\type.

\medskip \noindent Let~$C \in \nCat n$. For every~$m \in \N_n$ and
$m$\type~$(u,u')$, we define, by induction on~$m$,
\begin{itemize}
\item the notion of \index{context (class)}\emph{$m$\context of type~$(u,u')$ of~$C$},
  
\item the notion of \emph{$m$\context class of type~$(u,u')$ of~$C$}, 
  
\item for~$k \in \N_n$ with~$k \ge m$, the \glossary(EwFw){$E[w]$ ($F[w]$)}{the
    evaluation of a context~$E$ (class~$F$) at a cell~$w$}\index{evaluation!of a
    context (class)}\emph{evaluation} of an $m$\context~$E$ (\resp $m$\context
  class~$F$) of type~$(u,u')$ at a cell~$w \in C_k$ of type~$(u,u')$ which is a
  $k$\cell denoted~$E[w]$ (\resp~$F[w]$).
\end{itemize}
For~$m \in \N_n$, an $m$\context class of type~$(u,u')$ of~$C$ will be an
equivalence class of $m$\contexts of type~$(u,u')$ under a relation
\glossary(.cd){$\ctxtrel_m$}{the relation which witnesses that two $m$\contexts
  are equivalent}denoted~$\ctxtrel_m$, so that we write~$\ctxtcl E$ for the
associated $m$\context class of an $m$\context~$E$. This relation witnesses that
two contexts are equivalent up to the equalities~\ref{cond:pcat-to-strict}
considered in dimension~$m$. 

There is a unique $0$\context, denoted~$[-]$, and the relation~$\ctxtrel_0$ is
the identity relation, so that a $0$\context class is exactly a $0$\context.
Given~$k \in \N_n$ and $k$\cell~$v \in C_k$, the evaluation of the unique
$0$\context (class)~$[-]$ at~$v$ is~$v$.


Given~$m \in \N_{n-1}$ and an $(m{+}1)$\type~$(u,u')$, an
$(m{+}1)$\context of type~$(u,u')$ is defined as a triple~$E = (l,F,r)$ where
\begin{itemize}
\item $F$ is an $m$\context class of type~$(\csrc_{m-1}(u),\ctgt_{m-1}(u'))$,
  
\item $l$ and~$r$ are $(m{+}1)$\cells of~$C$ such that~$\ctgt_{m}(l) = F[u]$ and~$\csrc_{m}(r) = F[u']$.
\end{itemize}
Moreover, given~$k \in \N_n$ with $k \ge m+1$ and~$w \in C_k$ of type~$(u,u')$,
the evaluation~$E[w]$ of~$E$ at~$w$ is the $k$\cell
\[
  E[w] = l \pcomp_m F[w] \pcomp_m r.
\]
We define the relation~$\ctxtrel_{m+1}$ on $(m{+}1)$\contexts of type~$(u,u')$.
When~$m = 0$, for all $1$\contexts~$E_1$ and~$E_2$ of type~$(u,u')$, we put~$E_1
\ctxtrel_1 E_2$ if and only if~$E_1 = E_2$. When~$m > 0$, we
define~$\ctxtrel_{m+1}$ to be the reflexive symmetrical transitive closure
of~$\ctxtrel^1_{m+1}$, where~$\ctxtrel^1_{m+1}$ is the relation such that, for all
$(m{+}1)$\contexts
\[
  E_1= (l_1,F_1,r_1) \qtand E_2=(l_2,F_2,r_2)
\]
of type~$(u,u')$, we have~$E_1 \ctxtrel^1_{m+1} E_2$ if there exist $m$\contexts
\[
  E'_1 = (l'_1,F'_1,r'_1)
  \qtand
  E'_2 = (l'_2,F'_2,r'_2)
\]
of type~$(\csrc_{m-1}(u),\ctgt_{m-1}(u'))$ with~$F_i = \ctxtcl{E'_i}$ for~$i \in \set{1,2}$, and~$l,r,w \in C_{m+1}$ such that at least one of the two
sets of conditions \ref{cond:ctxtrel-left} and \ref{cond:ctxtrel-right} is
satisfied, where the set of conditions~\ref{cond:ctxtrel-left} is \begingroup
\leqnomode
\begin{equation}
  \label{cond:ctxtrel-left}
  \tag*{($\ctxtrel$-L)}
  \begin{aligned}
    l_1 &= l \pcomp_m (w \pcomp_{m-1} F'_1[u] \pcomp_{m-1} r'_1)
    &
      r_1
    &= r
            \\
    l_2 &= l
    &
      r_2 &=
            (w \pcomp_{m-1} F'_2[u'] \pcomp_{m-1} r'_2) \pcomp_m r
    \\
    l'_1 &= \ctgt_{m}(w)
           &
    r'_1 &= r'_2
    \\
    l'_2 &= \csrc_{m}(w)
    & F'_1 &= F'_2
  \end{aligned}
\end{equation}
and the set of conditions~\ref{cond:ctxtrel-right} is
\begin{equation}
  \label{cond:ctxtrel-right}
  \tag*{($\ctxtrel$-R)}
  \begin{aligned}
    l_1 &= l \pcomp_m (l'_1 \pcomp_{m-1} F'_1[u] \pcomp_{m-1} w)  
    &
      r_1 &= r
    \\
    l_2 &= l
    &
      r_2 &=
            (l'_2 \pcomp_{m-1} F'_2[u'] \pcomp_{m-1} w) \pcomp_m r
    \\
    l'_1 &= l'_2 & r'_1 &= \ctgt_m(w) \\
    r'_2 &= \csrc_m(w)
           &
             F'_1 &= F'_2\zbox.
  \end{aligned}
\end{equation}
\endgroup
\noindent
We give a graphical representation of \ref{cond:ctxtrel-left} and
\ref{cond:ctxtrel-right} in the case~$m = 1$ in \Cref{fig:ctxtrel-left-right-dim2}.
\begin{figure}
  \centering
  \begin{gather*}
    \makebox[0pt][r]{\ref{cond:ctxtrel-left}\qquad}
    \begin{tikzpicture}[scale=0.4,baseline=(w.base)]
      \draw (0,0) circle (3);
      \draw (0,0) circle (1);
      \draw (180:3) arc (180:360:1);
      \draw[dotted] (180:3) arc (180:0:1);
      \draw (0:1) -- (0:3);
      \node at (90:2) {$l$};
      \node at (-90:2) {$r$};
      \node (w) at (180:2) {$w$};
    \end{tikzpicture}
    \qquad
    \ctxtrel^1_2
    \qquad
    \begin{tikzpicture}[scale=0.4,baseline=(w.base)]
      \draw (0,0) circle (3);
      \draw (0,0) circle (1);
      \draw[dotted] (180:3) arc (180:360:1);
      \draw (180:3) arc (180:0:1);
      \draw (0:1) -- (0:3);
      \node at (90:2) {$l$};
      \node at (-90:2) {$r$};
      \node (w) at (180:2) {$w$};
    \end{tikzpicture}
    \\
    \makebox[0pt][r]{\ref{cond:ctxtrel-right}\qquad}
    \begin{tikzpicture}[scale=0.4,baseline=(w.base)]
      \draw (0,0) circle (3);
      \draw (0,0) circle (1);
      \draw (180:1) -- (180:3);
      \draw (0:1) arc (180:360:1);
      \draw[dotted] (0:1) arc (180:0:1);
      \node at (90:2) {$l$};
      \node at (-90:2) {$r$};
      \node (w) at (0:2) {$w$};
    \end{tikzpicture}
    \qquad
    \ctxtrel^1_2
    \qquad
    \begin{tikzpicture}[scale=0.4,baseline=(w.base)]
      \draw (0,0) circle (3);
      \draw (0,0) circle (1);
      \draw (180:1) -- (180:3);
      \draw[dotted] (0:1) arc (180:360:1);
      \draw (0:1) arc (180:0:1);
      \node at (90:2) {$l$};
      \node at (-90:2) {$r$};
      \node (w) at (0:2) {$w$};
    \end{tikzpicture}
  \end{gather*}
  \caption{The rules \ref{cond:ctxtrel-left} and \ref{cond:ctxtrel-right} for $\ctxtrel^1_2$}
  \label{fig:ctxtrel-left-right-dim2}
\end{figure}
An $(m{+}1)$\context class of type~$(u,u')$ is an equivalence class of
$(m{+}1)$\contexts of type~$(u,u')$ under~$\ctxtrel_{m+1}$. Note that if~$E_1
\ctxtrel_{m+1} E_2$ and~$w$ is a $k$\cell of type~$(u,u')$, then~$E_1[w] =
E_2[w]$, so that we can define the evaluation~$F[w]$ of an $(m{+}1)$\context
class~$F$ by a $k$\cell~$w$, both of type~$(u,u')$, as~$E[w]$, where~$E$ is an
$(m{+}1)$\context of type~$(u,u')$ such that~$F = \ctxtcl{E}$.

\paragraph{Source and target of contexts}
Let~$n \in \Ninf$ and~$C \in \nCat n$. Given~$m \in \N^*_n$, an
$m$\type~$(u,u')$ and an $m$\context~$E = (l,F,r)$ of type~$(u,u')$ of~$C$, the
\index{source!of a context (class)}\emph{source} and the \index{target!of a
  context (class)}\emph{target} of~$E$ are respectively the $(m{-}1)$\cells
\[
  \csrc_{m-1}(E) = \csrc_{m-1}(l) \qtand \ctgt_{m-1}(E) = \ctgt_{m-1}(r).
\]
When~$m > 1$, for~$\eps \in \set{-,+}$, we easily verify that
\[
  \csrctgt\eps_{m-2} \circ \csrc_{m-1}(E) = \csrctgt\eps_{m-2} \circ \ctgt_{m-1}(E)\zbox.
\]
The operations~$\csrc,\ctgt$ on $m$\contexts extend to $m$\context classes since
they are compatible with the~$\approx_m$ relation. Given~$i \in \N_{m-1}$
and~$\eps \in \set{-,+}$ and an $m$\context~$E$ (\resp an $m$\context
class~$F$), we write~$\csrctgt\eps_i(E)$ for~$\csrctgt\eps_i \circ
\csrctgt\eps_{m-1}(E)$ (\resp~$\csrctgt\eps_i \circ \csrctgt\eps_{m-1}(F)$).
Thus, for~$i \in \N_{n-1}$, we can extend the notion of $i$\composable sequences
of globes of globular sets to sequences~$X_1,\ldots,X_l$ for some~$l \in \N^*$
where~$X_s$ is either an $m$\context, an $m$\context class, or a cell of~$C$
for~$s \in \N^*_l$, and say that~$X_1,\ldots,X_l$ is $i$\composable
when~$\ctgt_i(X_s) = \csrc_i(X_{s+1})$ for~$s \in \N^*_{l-1}$.

\paragraph{Composition operations}

Let~$n \in \Ninf$ and~$C \in \nCat n$. Given~$i,m \in \N_n$ with~$i < m$, an
$m$\context~$E = (l,F,r)$ of some $m$\type~$(u,u')$ of~$C$, and~$v \in C_{i+1}$,
if~$(v,E)$ is $i$\composable, we define an $m$\context~$v \pcomp_i E$ by
induction on~$m - i$ with
\[
  v \pcomp_i E =
  \begin{cases}
    (v \pcomp_i l,F,r) & \text{if~$i + 1 = m$,} \\
    (v \pcomp_i l,v\pcomp_i F,v\pcomp_i r) & \text{if~$i + 1 < m$,}
  \end{cases}
\]
and, since it can be verified that the~$\pcomp_i$~operation is compatible
with~$\ctxtrel_m$, we extend the operation on $m$\context classes and put~$v
\pcomp_i \ctxtcl{E} = \ctxtcl{v\pcomp_i E}$. Similarly, if~$(E,v)$ is
$i$\composable, we define an $m$\context~$E \pcomp_i v$ using an induction on~$m
- i$ by
\[
  E \pcomp_i v =
  \begin{cases}
    (l,F,r\pcomp_i v) & \text{if~$i + 1 = m$,} \\
    (l\pcomp_i v, F\pcomp_i v, r\pcomp_i v) & \text{if~$i + 1 < m$,}
  \end{cases}
\]
and we put~$\ctxtcl{E} \pcomp_iv = \ctxtcl{E \pcomp_i v}$. These composition
operations satisfy properties similar to strict $(n{+}1)$\categories (unitality,
associativity, condition~\ref{cond:pcat-to-strict}, \etc).

\subsection{A description of free extensions}
\label{ssec:polextp-description}

We now introduce a concrete description of the free extension functor~$\polextp
--$ of strict categories using context and context classes. It is based on the
intermediate notion of \emph{categorical action}, which encodes the structure of
context and context classes with regard to the underlying $n$\category. On the
one hand, the cells of free categorical action on a cellular extension will then
be characterized as formally applied context classes on the generators. On the
other hand, the cells of the free strict category on a categorical action will
be characterized as adequately quotiented sequences of the cells of the
categorical action. This will result in the characterization of cells of free
extensions as quotiented sequences of formally applied context classes.

\paragraph{Categorical actions}
\parlabel{text:catact-def}

Let~$n \in \N$. An \index{categorical action}\emph{$n$\categorical action} is
the data of an $n$\cellular extension~$(C,C_{n+1})$ together with, for~$k \in
\N^*_{n}$, composition operations
\[
  \pcomp_{k,n+1} \co C_{k} \times_{k-1} C_{n+1} \to C_{n+1}
  \qtand
  \pcomp_{n+1,k}\co C_{n+1}
  \times_{k-1} C_{k} \to C_{n+1}
\]
satisfying the axioms given below. We extend the convention used for
precategories, meaning that, for~${i,k,l \in \N_{n+1}}$ with $ {i = \min(k,l) -
  1} $ and $ {\max(k,l) = n+1} $, given~${(u,v) \in C_k \times_i C_l}$, we
write~$u \pcomp_i v$ for~${u \pcomp_{k,l} v}$. The axioms satisfied by
$n$\categorical actions are then the following:
\begin{enumerateaxioms}[label=(A-\roman*),ref=(A-\roman*)]
\item \label{act:csrc-tgt}\label{act:first-comp} for~$i,k,l \in \N_{n+1}$ satisfying
  \[
    i = \min(k,l) -1 \le n - 1
    \qtand
    \max(k,l) = n+1,
  \]
  and~$(u,v)\in C_k\times_i C_l$ and~${\eps \in \set{-,+}}$,
  \begin{align*}
    \csrctgt\eps_n(u \pcomp_i v)
    &=
      \begin{cases}
        u\pcomp_i \csrctgt\eps_n(v)& \text{if~$k < l$,}\\
        \csrctgt\eps_n(u)\pcomp_i v&\text{if~$k>l$,}
      \end{cases}
  \end{align*}

\item \label{act:assoc} for~$i,k,l,m \in \N_{n+1}$ satisfying
  \[
    i = \min(k,l) - 1 = \min(l,m) - 1 \le n-1
    \qtand
    \max(k,l,m) = n+1,
  \]
  and~$(u,v,w) \in C_k \times_i C_l \times_i C_m$,
  \[
    (u \pcomp_i v) \pcomp_i w = u \pcomp_i (v \pcomp_i w),
  \]
\item \label{act:distrib}\label{act:last-comp} for~$i,j\in \N_{n-1}$ with~$i < j$,~$u_1,u_2
  \in C_{i+1}$,~$v_1,v_2 \in C_{j+1}$ and~$w \in C_{n+1}$ such that~$u_1,w,u_2$
  are $i$\composable and~$v_1,w,v_2$ are $j$\composable,
  \[
    u_1 \pcomp_i (v_1 \pcomp_j w \pcomp_j v_2) \pcomp_i  u_2 = (u_1 \pcomp_i
    v_1 \pcomp_i u_2) \pcomp_j (u_1
    \pcomp_i w \pcomp_i u_2) \pcomp_j (u_1 \pcomp_i v_2 \pcomp_i u_2),
  \]

\item \label{act:exch} for~$i,k,l \in \N^*_{n+1}$ satisfying
  \[
    i = \min(k,l) - 1 \le n-1
    \qtand
    \max(k,l) = n+1,
  \]
  and~$(u,v) \in C_k \times_{i-1} C_l$,
  \[
    (u \pcomp_{i-1} \csrc_{i}(v)) \pcomp_{i} (\ctgt_{i}(u) \pcomp_{i-1} v) =
    (\csrc_{i}(u) \pcomp_{i-1} v) \pcomp_{i} (u \pcomp_{i-1} \ctgt_{i}(v)).
  \]
\end{enumerateaxioms}
Axioms~\ref{act:csrc-tgt},~\ref{act:assoc} and~\ref{act:distrib} above closely
match Axioms~\ref{precat:csrc-tgt},~\ref{precat:assoc} and~\ref{precat:distrib}
of precategories. \Axr{act:exch} is
analogous to the condition~\ref{cond:pcat-to-strict} satisfied by precategories
derived from strict categories %
(\cf~\Cref{ssec:scat-another-def}). %
An \index{morphism!of categorical actions}\emph{$n$\categorical action morphism}
between~$(C,C_{n+1})$ and~$(D,D_{n+1})$ is a morphism of $n$\cellular extension
\[
  (F,f) \co (C,C_{n+1}) \to (D,D_{n+1}) \in \nCatp n
\]
which is moreover commutes with the~$\pcomp_{k,n+1}$ and~$\pcomp_{n+1,k}$
operations for~$k \in \N^*_n$. We \glossary(CatAn){$\nCata n$}{the category of
  $n$\categorical actions}write~$\nCata n$ for the category of $n$\categorical
actions.

\paragraph{Free action on a cellular extension}
\parlabel{text:freeact-on-cell-ext}

There is
a forgetful functor
\[
  \mfunctorb U\co\nCata n \to \nCat n^+
\]
which forgets the data of the~$\pcomp_{k,n+1}$ and~$\pcomp_{n+1,k}$ operations,
for~$k \in\N^*_{n}$. In this section, we use the formalism of contexts and
contexts classes to define a left adjoint $(-,\freeacttop -)$ to the functor~$\mfunctorb U\co \nCata
n \to \nCatp n$: given an $n$\cellular extension~$(C,X)$, the elements of
$\freeacttop X$ will be the pairs $(g,F)$, where $g \in X$ and~$F$ is an
adapted $n$\context class, \ie $\freeacttop X$ is the set of context classes formally applied to
generators of~$X$.


Let~$n \in \N$. Given an $n$\cellular extension~$(C,X)$, an
$n$\categorical action~$\freeact C X = (C,\freeacttop X)$ can be defined as
follows:~$\freeacttop X$ is the set of pairs~$(g,F)$ with~$g \in X$ and~$F$ an
$n$\context class of type~$g$. The $n$\source and $n$\target of such a
pair~$(g,F)$ are defined respectively as the $n$\cells
\[
  \csrc_n((g,F)) = F[\gsrc_n(g)] 
  \qtand
  \ctgt_n((g,F)) = F[\gtgt_n(g)].
\]
and they equip~$(C,\freeacttop X)$ with a structure of an $n$\cellular
extension. We extend the operations~$\pcomp_i$ defined for $n$\context classes
to such pairs by putting
\[
  u \pcomp_i (g,F) = (g,u \pcomp_i F)
  \qtand
  (g,F) \pcomp_i v = (g,F\pcomp_iv)
\]
for~$i \in \N_{n-1}$ and~$u,v \in C_{i+1}$ such that~$u,(g,F)$ and~$(g,F),v$ are
$i$\composable. We then have:
\begin{prop}
  \label{prop:freeact-catact}
  \label{prop:freeact-free}
  The operations~$\pcomp_i$ defined above equip~$\freeact C X$ with the
  structure of an $n$\categorical action. It is the free categorical action
  relatively to the forgetful functor~$\mfunctorb U$.
\end{prop}
\noindent The construction~$(C,X) \mapsto \freeact C X$ of the above proof
uniquely extends to a \glossary(.d){$\freeact - -$, $\freeact C A$}{the functor
  which maps an $n$\cellular extension to the associated free $n$\categorical
  action}functor
\[
  \freeact - -\co \nCext n \to \nCata n
\]
which is left adjoint
to~$\mfunctorb U$. Given~$(H,h)\co (C,X) \to (D,Y)$ in~$\nCext n$, the
$n$\categorical action morphism
\[
  \freeact H h\co \freeact C X \to \freeact D Y \in \nCata n
\]
is defined by $ \freeact H h_i = H_i $ and $ \freeact H h_{n+1}((g,F)) =
(h(g),H(F)) $ for~$i \in \N_n$ and~$(g,F) \in\freeacttop X$.

\paragraph[Free \texorpdfstring{$(n{+}1)$}{(n+1)}-categories on \texorpdfstring{$n$}{n}-categorical actions]{Free $\bm{(n{+}1)}$-categories on $\bm n$-categorical actions}
\parlabel{ssec:freecat-on-act}


There is a forgetful functor
\[
  \mfunctorb{U}'\co \nCat {n+1} \to \nCata n
\]
which maps an $(n{+}1)$\category~$C$ to an $n$\categorical action~$(\restrictcat
C n,C_{n+1})$ by forgetting the~$\pcomp_n$ operation (where we consider the
$(n{+}1)$\precategory structure of~$C$). In this section, we describe explicitly
a left adjoint $\acttocat--$ to this functor: given~$(C,A) \in \nCata n$, we
show that the $(n{+}1)$\cells of~$\acttocat C A$ can be described as sequences
of composable elements of~$A$ that are adequately quotiented.

Let $n \in \N$ and $(C,A) \in \nCata n$. We define the
\glossary(.da){$\acttolist{A}$}{the set of $n$\sequences
  over~$A$}set~$\acttolist{A}$ of \emph{$n$\composable sequ\-ences} (or simply,
\index{sequence@$n$-sequence}\emph{$n$\sequences}) of~$(C,A)$ as the set of
terms of the form
\[
  \cseq{u_1,\ldots,u_k}
\]
for some~$k \in \N$ and~$u_1,\ldots,u_k \in A$ such that~$u_1,\ldots,u_k$ are
$n$\composable. When~$k = 0$, by convention, there is an empty
sequence~$\unitseq u$ for each~$u \in C_n$. Given~$v = \cseq{v_1,\ldots,v_k} \in
\acttolist A$, we say that~$k$ is the \index{length!of an
  $n$-sequence}\emph{length} of~$v$ and we \glossary(.e){$\len u$}{the length of
  the cell~$u$}write~$\len v$ for~$k$. Moreover, we define a source~$\csrc_n(v)$
and a target~$\ctgt_n(v)$ for~$v$ by putting
\[
  \csrc_n(v) = \csrc_n(v_1)
  \qtand
  \ctgt_n(v) = \ctgt_n(v_k)
\]
where, by convention, if $v = \unitseq u$ for some $u \in C_n$, then~$\csrc_n(v)
= \ctgt_n(v) = u$. Thus, we obtain an $n$\cellular extension whose set of
$(n{+}1)$\globes is~$\acttolist A$ and whose underlying $n$\category is~$C$. We
now define composition operations for the $n$\sequences. Given~$i \in \N_{n-1}$,
a~cell~$u \in C_{i+1}$ and an $n$\sequence~${v = \cseq{v_1,\ldots,v_l} \in
  \acttolist{A}}$ such that~$u,v$ are $i$\composable, we put
\[
  u \pcomp_i v = \cseq{u \pcomp_i v_1,\ldots,u\pcomp_i v_l}
\]
where, by convention, if~$v = \unitseq {\tilde v}$ for some~$\tilde v \in C_n$,
then~$u \pcomp_i v = \unitseq {u \pcomp_i \tilde v}$. Given $n$\composable
$n$\sequences~$u = \cseq{u_1,\cdots,u_k}$ and~${v = \cseq{v_1,\cdots,v_l}}$
in~$\acttolist{A}$, we put
\[
  u \pcomp_n v = \cseq{u_1,\ldots,u_k,v_1,\ldots,v_l}\zbox.
\]

\noindent In order to obtain a strict $(n{+}1)$\category from~$C$
and~$\acttolist A$, we need to quotient~$\acttolist A$ so that the exchange
condition~\ref{cond:pcat-to-strict} on precategories holds (\cf
\Cref{prop:strcat-pcat-isom}). For this purpose, we define a
\glossary(X){$\exchwit$}{the exchange relation for the top cells of categorical
  actions}relation $ \exchwit \subseteq A^6 $ such that, given~$l,l',r,r',u,v
\in A$,~$\exchwit(l,l',r,r',u,v)$ holds when~$u,v$ are $(n{-}1)$\composable and
the following equalities hold in~$A$
\begin{align*}
  l & = u \pcomp_{n-1} \csrc_n(v) & r &= \csrc_n(u) \pcomp_{n-1} v \\
  l' &= \ctgt_n(u) \pcomp_{n-1} v & r' &= u \pcomp_{n-1} \ctgt_{n}(v)\zbox.
\end{align*}
In \Cref{fig:2-dimen-exchwit-example}, we illustrate this condition in the case
of a $1$\categorical action.
\begin{figure}
  \centering
  \begin{align*}
    l
    &=
      \begin{tikzcd}[ampersand replacement=\&]
        x
        \ar[r,bend left=70,"f",""{auto=false,name=f}]
        \ar[r,bend right=70,"{f'}"'{pos=0.53},""{auto=false,name=fp}]
        \&
        y
        \ar[r,"g"]
        \&
        z
        \ar[from=f,to=fp,"\Downarrow u",phantom]
      \end{tikzcd}
    &
      r
    &=
      \begin{tikzcd}[ampersand replacement=\&]
        x
        \ar[r,"f"]
        \&
        y
        \ar[r,bend left=70,"g",""{auto=false,name=f}]
        \ar[r,bend right=70,"{g'}"'{pos=0.48},""{auto=false,name=fp}]
        \&
        z
        \ar[from=f,to=fp,"\Downarrow v",phantom]
      \end{tikzcd}
    \\
    l'
    &=
      \begin{tikzcd}[ampersand replacement=\&]
        x
        \ar[r,"f'"]
        \&
        y
        \ar[r,bend left=70,"g",""{auto=false,name=f}]
        \ar[r,bend right=70,"{g'}"'{pos=0.48},""{auto=false,name=fp}]
        \&
        z
        \ar[from=f,to=fp,"\Downarrow v",phantom]
      \end{tikzcd}
    &
      r'
    &=
      \begin{tikzcd}[ampersand replacement=\&]
        x
        \ar[r,bend left=70,"f",""{auto=false,name=f}]
        \ar[r,bend right=70,"{f'}"'{pos=0.53},""{auto=false,name=fp}]
        \&
        y
        \ar[r,"g'"]
        \&
        z
        \ar[from=f,to=fp,"\Downarrow u",phantom]
      \end{tikzcd}
  \end{align*}
  \caption{A configuration of $2$-cells $l,l',r,r',u,v$ such that
    $\exchwit(l,l',r,r',u,v)$}
  \label{fig:2-dimen-exchwit-example}
\end{figure}
Given top-level elements~$l,l',r,r' \in A$, we write~$\exchwit(l,l',r,r')$ when
there exist~$u,v \in A$ such that we have~$\exchwit(l,l',r,r',u,v)$. We define
an equivalence \glossary(.ea){$\actrel$}{the relation which witnesses that two
  $n$\sequences are equivalent}relation~$\actrel$ on~$\acttolist{A}$ as the
reflexive sym\-me\-trical transitive closure of~$\actrel^1$, where, for~$l =
\cseq{l_1,\ldots,l_k}$ and~${r = \cseq{r_1,\ldots,r_k}}$ in~$\acttolist{A}$,~$l
\actrel^1 r$ when there is~$i \in \N^*_{k-1}$ such that $
\exchwit(l_i,l_{i+1},r_i,r_{i+1}) $ and $ l_j = r_j $ for~$j \in \N^*_k
\setminus \set{i,i+1}$. We \glossary(.eb){$\acttocattop A$}{the set of
  $n$\sequence classes over~$A$}write~$\acttocattop A$ for the quotient
set~$\acttolist{A}/\actrel$ of \index{sequence class@$n$-sequence
  class}\emph{$n$\sequence classes} and 
\[
  \actcl -\co \acttolist A \to \acttocattop A
\]
for the associated projection. We remark that, if~$u,v \in \acttolist A$ are
such that~$u \actrel v$, then~$\len u = \len v$. Thus, the length given for the
members of~$\acttolist A$ induces a \index{length!for a free strict
  category}\emph{length} for the members of~$\acttocattop A$. The
operations~$\csrctgt\eps_n$ for~$\eps \in \set{-,+}$ on~$\acttolist{A}$ can be
shown compatible with the relation~$\actrel$, so that they are well-defined
on~$\acttocattop A$ as well. Thus, we obtain an $n$\cellular
extension~$\acttocat C A$ by extending the strict $n$\category~$C$
with~$\acttocattop A$. Similarly, the operations~$\pcomp_i$ for~$i \in \N_{n-1}$
and~$\pcomp_n$ defined for~$\acttolist{A}$ are compatible with the
relation~$\actrel$, so that they are well-defined on~$\acttocat C A_{n+1} =
\acttocattop A$ as well. We add an identity operation by putting~${\unitp {n+1}
  u}={\actcl{\unitseq u}}$ for~${u \in C_n}$. We then have:
\begin{prop}
  \label{prop:acttocat-free}
  $\acttocat C A$ has a structure of a strict $(n{+}1)$\category. It is the free
  $(n{+}1)$\category on the action~$(C,A)$ relatively to the forgetful
  functor~$\mfunctorb {U}'$.
\end{prop}
\noindent In the following, for all $n$\categorical action~$(C,A)$, we
write~$\acttocat C A$ for~$\acttocat C A$ as above. The construction~$(C,A)
\mapsto \acttocat C A$ uniquely extends to a \glossary(.f){$\acttocat--$,
  $\acttocat C A$}{the functor which maps an $n$\categorical action to the
  associated free strict $(n{+}1)$\category}functor
\[
  \acttocat - -\co \nCata n \to \nCat {n+1}
\]
which is left adjoint to~$\mfunctorb U'$. Given~$(H,h)\co (C,A) \to (D,B)$ in~$\nCext n$, the $(n{+}1)$\functor 
\[
  \acttocat H h\co \acttocat C A \to \acttocat D B \in \nCat {n+1}
\]
is defined by
  $\acttocat H h_i = H_i$
  and
  $
  \acttocat H h_{n+1}(\cseqcl{u_1,\ldots,u_k}) = \cseqcl{h(u_1),\ldots,h(u_k)}
  $
for~$i \in \N_n$ and~$\cseq{u_1,\ldots,u_k} \in\acttolist A$.

\paragraph{Free categories on cellular extensions}
\parlabel{text:polextp-another-description}

Let $n \in \N$. We can sum up the content of the previous sections to give a
concrete description of the functor
\[
  \polextp - -\co \nCext n \to \nCat {n+1}\zbox.
\]
Indeed, since its right adjoint $\algtoce_{n} \co \nCat {n+1} \to \nCext n$ is
the composite of the right adjoints
  $\mfunctorb {U}'$
  and
  $\mfunctorb U$,
we have, as a consequence of \Cref{prop:freeact-free,prop:acttocat-free}:
\begin{theo}
  \label{thm:descr-polextp}
  The composite
  \begin{equation*}
    \label{eq:isom-freecat}
    (\acttocat - -) \circ (\freeact - -)\co \nCatp n \to \nCat {n+1}
  \end{equation*}
  is a left adjoint for~$\algtoce_n$. In particular, it is isomorphic
  to~$\polextp--$.
\end{theo}
\noindent Our description of~$\polextp - -$ also induces a decomposition
property for the $(n{+}1)$\cells of free extensions:
\begin{prop}
  \label{prop:decomp-cells-polextp-strcats}
  Given an $n$\cellular extension~$(C,X)$ and~$u \in \polextp C X_{n+1}$,~$u$
  can be written
  \[
    F_1[g_1] \pcomp_n \cdots \pcomp_n F_k[g_k]
  \]
  where~$k\in \N$,~$g_i \in X$ and~$F_i$ is an $n$\context class of type~$g_i$
  for~$i \in \N^*_k$. Moreover,~$k$ is unique for~$u$.
\end{prop}

\subsection{Cell decompositions}
\label{ssec:cell-dec}

Here, we use the machinery of context classes in order to prove a decomposition
property for the cells of a torsion-free complex. More precisely, given a
torsion-free complex~$P$, we prove that the $n$\cells of~$\Cell(P)$ can be
written as sequences of applied $(n{-}1)$\context classes. Actually, we prove
the stronger statement that such a composite exists for any total ordering,
called \emph{linear extensions}, of the top-level $n$\generators that respects
the relation~$\tl$. This result will be a first step towards the proof that
$\Cell(P)$ is freely generated on the atoms. The next one, tackled in the following
section, will be to show that the above decomposition is unique up to the
relation~$\actrel$ defined in the previous section.

\paragraph{Linear extensions} 

Given a finite poset~$(S,<)$, a \index{linear extension}\emph{linear extension
  of~$(S,<)$} is the data of a bijection~$\sigma \co \N^*_{\setsize S} \to S$
such that, for~$i,j \in \N^*_{\setsize S}$, if~$\sigma(i) < \sigma(j)$, then~$i
< j$. Given two linear extensions~$\sigma,\sigma' \colon \N_{|S|} \to S$, a
\index{morphism!of linear extensions}\emph{morphism of linear extensions
  of~$(S,<)$} between~$\sigma$ and~$\sigma'$ is a function~${\rho \colon
  \N^*_{\setsize S} \to \N^*_{\setsize S}}$ such that the triangle
\[
  \begin{tikzcd}
    \N^*_{\setsize S} \ar[rr,"\rho"] \ar[rd,"\sigma"'] & & \N^*_{\setsize S} \ar[ld,"\sigma'"] \\
    & S
  \end{tikzcd}
\]
is commutative (in particular,~$\rho$ is a bijection). We
\glossary(LinExtS){$\linext(S)$}{the category of linear extensions
  of~$(S,<)$}write~$\linext(S)$ for the category of linear extensions of~$S$.
Given~$n \in \N$ and a bijection~$\rho \co \N^*_n \to \N^*_n$, we write
$\Inv(\rho) \in \N$ for the number of \emph{inversions} of~$\rho$, \ie
\[
  \Inv(\rho) = \setsize{ \set{ (i,j) \in \N^*_n \times \N^*_n \mid i < j \qtand \rho(i) > \rho(j)}}\zbox.
\]
Moreover, given~$i,j \in \N^*_{n}$ such that~$i \neq j$, we write~$\tau_{i,j}$
for the bijection~$\N^*_n \to \N^*_n$ which is the transposition of~$i$ and~$j$.
We show that the morphisms of linear extensions are generated by the
transpositions:
\begin{lem}
  \label{lem:linext-dec}
  Given a poset~$(S,<)$ and~$\sigma,\sigma' \in \linext(S)$ and~$\rho \co \sigma
  \to \sigma' \in \linext(S)_1$, there exist~$p \in \N$
  and~${\sigma_0,\ldots,\sigma_p \in \linext(S)}$ with~$\sigma = \sigma_0$
  and~$\sigma_p = \sigma'$, and~$\rho_i \co \sigma_{i-1} \to \sigma_i \in
  \linext(S)$ for~$i \in \N^*_p$ such that $ \rho = \rho_1 \comp_0 \cdots
  \comp_0 \rho_p $ and $\rho_i$ is a transposition for~$i \in \N^*_p$.
\end{lem}

\begin{proof}
  We prove the result by induction on the number~$\Inv(\rho)$ of inversions of
  the bijection~$\rho$. If~$\Inv(\rho) = 0$, then $ \rho = \id{\N^*_{\setsize
      S}} = \unitp{1}\sigma $. So suppose that~$\Inv(\rho) > 0$. Thus, there
  exists~$k \in \N^*_{\setsize S-1}$ such that~$\rho(k) > \rho(k+1)$. The
  bijection~$\bar\sigma = \sigma \circ \tau_{k,k+1}$ is then a linear extension
  of~$(S,<)$ as in
  \[
    \begin{tikzcd}[column sep=huge]
      \N^*_{\setsize S} \ar[r,"\tau_{k,k+1}"] \ar[rd,"\sigma"'] & \N^*_{\setsize S} \ar[d,"\bar\sigma"]
      \ar[r,"\rho \circ \tau_{k,k+1}"] & \N^*_{\setsize S} \ar[ld,"\sigma'"] \\
      & S 
    \end{tikzcd}
    .
  \]
  Indeed, for~$i, j \in \N^*_{\setsize S}$ such that~$i \neq j$
  and~$\bar\sigma(i) < \bar\sigma(j)$,
  \begin{itemize}
  \item if~$i = k+1$ and~$j = k$, then~$\sigma(k) < \sigma(k+1)$,
    so~$\sigma'(\rho(k)) < \sigma'(\rho(k+1))$ and~$\rho(k) < \rho(k+1)$,
    contradicting the hypothesis;
  
  \item if~$i = k+1$ and~$j \neq k$, then~$\sigma(i-1) < \sigma(j)$, so~$i-1 <
    j$, and, since~$j \neq i$,~$i < j$;
  \item otherwise, we are able to prove that~$i < j$ easily.  
  \end{itemize}
  Moreover, the number of inversions of~$\rho \circ \tau_{k,k+1}$ is~$\Inv(\rho) -
  1$. By induction hypothesis,~$\rho \circ \tau_{k,k+1}$ can be written as
  \[
    \rho \circ \tau_{k,k+1} = \rho_2 \comp_0 \cdots \comp_0 \rho_p
  \]
  for some~$p \in \N$ and transpositions~$\rho_i \co \sigma_{i-1} \to \sigma_i
  \in \linext(S)_1$ for~$i \in \N^*_{p-1}$, so that
  \[
    \rho = \tau_{k,k+1} \comp_0\rho_2 \comp_0 \cdots \comp_0 \rho_p
  \]
  is of the wanted form.
\end{proof}

\paragraph{Decomposition theorem}
Here, we write~$P$ for a torsion-free complex. We show that cells
of~$\Cell(P)$ can be decomposed as composites of applied context classes that
respect the relation~$\tl$. First, we state a simple criterion for equality
in~$\Cell(P)$, which readily follows from the definitions of cells and
source/target operations:
\begin{lem}
  \label{lem:cell-eq}
  Given~$k,n \in \N$ with~$k < n$,~$\eps \in \set{-,+}$ and~$X,Y \in \Cell(P)_n$
  such that
  $
    \csrctgt\eps_{k} (X) = \csrctgt\eps_{k} (Y)
  $
  and
  $
    X_{i,\eps} = Y_{i,\eps}  
  $
  for~$i
  \in \set{k+1,\ldots,n}$, we have~$X = Y$.
\end{lem}
\noindent Next, we show that we can write a cell as a composition by extracting a minimal
element for~$\tl$:

\begin{lem}
  \label{lem:toplevel-decomp-min-gen}
  Let~$n \in \N^*$ and~$X$ be an $n$\cell and~$g$ be a minimal element of~$X_n$
  for~$\tl_{X_n}$. Then, there exist $n$\cells~$Y$ and~$Z$ that are
  $(n{-}1)$\composable such that $ Y_n = \set{g} $, $ Z_n = X_n \setminus
  \set{g} $ and $ X = Y \comp_{n-1} Z $.
\end{lem}
\begin{proof}
  Since~$g$ is minimal for~$\tl_{X_n}$, we have~$\set{g}^\mp \subseteq
  X_{n-1,-}$. Moreover, since~$X$ is an $n$\cell,~$X_n$ is fork-free so
  that~$\set{g} \perp (X_n \setminus \set{g})$. Thus, by \Lemr{street-2.4},
  writing~$V$ for~$(X_{n-1,-} \cup g^+) \setminus g^-$, we have that
    $\set{g}$ moves~$X_{n-1,-}$ to~$V$ and~$X_n \setminus \set g$ moves~$V$ to~$X_{n-1,+}$.
  By \Thmr{cell-gluing}, the cell~$Y = \cellgluing(\csrc_{n-1}(X),\set{g})$ is
  an $n$\cell which satisfies that
  \[
    Y_n = \set g,
    \quad
    {\csrc_{n-1}(Y) = \csrc_{n-1}(X)}
    \qtand
    {Y_{n-1,+} = V}.
  \]
  By \Thmr{cell-gluing} again,~$Z = \cellgluing(\ctgt_{n-1}(Y),X_n \setminus
  \set g)$ is an $n$\cell such that
  \[
    Z_n = X_n \setminus \set g,
    \quad
    \csrc_{n-1}(Z) = \ctgt_{n-1}(Y)
    \qtand
    Z_{n-1,+} = X_{n-1,+},
  \]
  so that~$\ctgt_{n-1}(Z)=\ctgt_{n-1}(X)$. Then, by the definition
  of~$\comp_{n-1}$, we have~$X = Y \comp_{n-1} Z$.
\end{proof}
\noindent The previous lemma implies that we can write a cell as a composite of
cells with a single top-level generator, that are moreover ordered by a given
linear extension:
\begin{lem}
  \label{lem:toplevel-decomp-gen}
  Let~$n \in \N^*$ and~$X$ be an $n$\cell of~$P$,~$p =
  \setsize{X_n}$ and~$\sigma \co \N^*_{p} \to (X_n,\tl_{X_n})$ be a linear
  extension. There exist $n$\cells~$X^1,\ldots,X^p$ that are
  $(n{-}1)$\composable and such that
  \[
    X^i_n = \set{\sigma(i)}
    \quad
    \text{for~$i \in \N^*_p$} 
    \qtand
    X = X^1 \comp_{n-1} \cdots \comp_{n-1} X^p.
  \]
\end{lem}
\begin{proof}
  We prove this property by induction on~$p$. When~$p = 0$ or~$p = 1$, then the
  property is trivial. So suppose that~$p > 1$. Note that~$\sigma(1)$ is minimal
  in~${X_n}$ for~$\tl_{X_n}$. By \Lemr{toplevel-decomp-min-gen}, we can write~$X
  = X^1 \comp_{n-1} X'$ where~$X^1$ and~$X'$ are $(n{-}1)$\composable $n$\cells
  such that~$X^1_n = \set{\sigma(1)}$ and~$X'_n = X_n \setminus \set
  {\sigma(1)}$. By induction hypothesis, we have that~$X' = X^2 \comp_{n-1}
  \cdots \comp_{n-1} X^p$ for some $(n{-}1)$\composable $n$\cells~$X^2, \ldots,
  X^p$ such that~$X^i_n = \set{\sigma(i)}$ for~$i \in \set{2,\ldots,p}$, which
  concludes the proof.
\end{proof}
\noindent Next, we give a sufficient criterion for a cell to be written as an
applied context class:
\begin{lem}
  \label{lem:cell-to-applied-ctxt}
  Let~$k,n \in \N$ with~$k < n$,~$g \in P_n$ and~$X$ be an
  $n$\cell such that
  $
    X_{i,\eps} = \gen{g}_{i,\eps}
  $
    for~${i \in \set{k+1,\ldots,n}}$ and~${\eps \in \set{-,+}}$.
  There exists a $k$\context class~$F$ of type~$\gen{g}$ such that we have~$X =
  F[\gen{g}]$.
\end{lem}
\begin{proof}
  We show this property by induction on~$k$. When~$k = 0$, we have that~$X_{i,\eps} = \gen{g}_{i,\eps}$ for~${i \in \N^*_n}$ and~${\eps \in \set{-,+}}$.
  Moreover, since~$X$ is an $n$\cell, we have that~$\gen{g}_{1,-}$ moves~$X_{0,-}$ to~$X_{0,+}$, so that
  $
    \gen{g}_{1,-}^\mp = \gen{g}_{0,-} \subseteq
    X_{0,-}
  $.
  Since~$X_{0,-}$ is fork-free, $\setsize{X_{0,-}} = 1$. Thus,~$X_{0,-}
  = \gen{g}_{0,-}$. Similarly,~$X_{0,+} = \gen{g}_{0,+}$. Hence, we have~$X
  = \gen{g}$ and the property of the statement is verified with the unique
  $0$\context class.

  So suppose that~$k > 0$. We have that~$X_{k+1,\eps} = \gen{g}_{k+1,\epsilon}$
  moves~$X_{k,-}$ to~$X_{k,+}$, so~$\gen{g}_{k,-} \subseteq X_{k,-}$.
  By~\forestaxiom{3},~$\gen{g}_{k,-}$ is a segment for~$\tl_{X_{k,-}}$ and, by
  \Lemr{segment-decomp-in-3}, there exist~$U,V \subseteq X_{k,-}$ and~$A,B
  \subseteq P_{k-1}$ such that
  \begin{itemize}
  \item $U,\gen{g}_{k,-},V$ is a partition of~$X_{k,-}$,
    
  \item $U$ moves~$X_{k-1,-}$ to~$A$,~$\gen{g}_{k,-}$ moves~$A$ to~$B$ and~$V$
    moves~$B$ to~$X_{k-1,+}$.
  \end{itemize}
  Writing
  \[
    L = \cellgluing(\csrc_{k-1}(X),U) \quad X^k =
    \cellgluing(\ctgt_{k-1}(L),\gen{g}_{k,-})
    \quad
    R =
    \cellgluing(\ctgt_{k-1}(X^k),V)\zbox,
  \]
  by \Thmr{cell-gluing}, we have that~$L,X^k,R$ are $k$\cells that are
  $(k{-}1)$\composable and such that
  \[
    \csrc_k(X) = L \pcomp_{k-1} X^k \pcomp_{k-1} R
  \]
  By induction on~$i \in \set{k+1,n}$, we define $i$\cells~$X^i$ such
  that~$\csrc_{i-1}(X^i) = X^{i-1}$ and~$X^i_i = \gen{g}_{i,-}$ by putting $ X^i
  = \cellgluing(X^{i-1},\gen{g}_{i,-}) $, which is indeed a cell by
  \Thmr{cell-gluing}. Then,~$X^n$ is an $n$\cell such that $ \csrc_k(X^n) = X^k
  $ and $ X^n_{i,\eps} = \gen{g}_{i,\eps} $ for~$i \in \set{k,\ldots,n}$.
  Moreover, since~$\csrc_k(X^n) = X^k$,
  \[
    \text{$L,X^n,R$ are $(k{-}1)$\composable}
    \qtand
    \csrc_k(L \pcomp_{k-1} X^n \pcomp_{k-1} R) = \csrc_k(X).
  \]
  Furthermore, we have that
  \[
    X_{i,-} = \gen{g}_{i,-} = X^n_{i,-} = (L \pcomp_{k-1} X^n \pcomp_{k-1}
    R)_{i,-}
  \]
  for~$i \in \set{k+1,n}$ so that, by \Lemr{cell-eq}, we have
  $
    X = L \pcomp_{k-1} X^n \pcomp_{k-1} R
  $.
  By induction hypothesis, there exists a $(k{-}1)$\context class~$F'$ such
  that~$X^n = F'[\gen{g}]$. Writing~$F$ for the $k$\context class~$\ctxtcl{(L,F',R)}$, we have that~$X = F[\gen{g}]$ as wanted.
\end{proof}
\noindent We can now prove the following decomposition theorem:

\begin{theo}
  \label{thm:cell-dec}
  Given~$n \in \N^*$, an $n$\cell~$X \in \Cell(P)$ and a linear
  extension
  \[
    \sigma \co \N^*_p \to (X_n,\tl_{X_n})
  \]
  with~$p = \setsize{X_n}$,
  there exist $(n{-}1)$\context classes~$F_1,\ldots,F_p$ of~$\Cell(P)$ of
  types~$\gen{\sigma(1)},\ldots,\gen{\sigma(p)}$ respectively such that
  \[
    X = F_1[\gen{\sigma(1)}] \pcomp_{n-1} \cdots \pcomp_{n-1} F_p[\gen{\sigma(p)}]\zbox.
  \]
\end{theo}
\begin{remark}
  By \forestaxiom{1}, given~$n \in \N^*$ and a finite subset~$S \subseteq
  P_n$, there always exists a linear extension~$\sigma \co \N^*_{\setsize P} \to
  (S,\tl_S)$, so that an $n$\cell~$X$ of~$P$ has at least one decomposition of
  the form given by \Thmr{cell-dec}.
\end{remark}
\begin{proof}
  By \Lemr{toplevel-decomp-gen},~$X$ can be written
  $
    X = X^1 \pcomp_{n-1} \cdots \pcomp_{n-1} X^p
  $
  for some $n$\cells~$X^1,\ldots,X^p$ such that~$X^i_n = \set{\sigma(i)}$ for~$i
  \in \N^*_p$. We conclude with \Lemr{cell-to-applied-ctxt}.
\end{proof}
\noindent We verify with the following property that
\Thmr{cell-dec} does not miss other possible decompositions:
\begin{prop}
  \label{prop:lemdec.gen}
  Given~$n \in \N^*$ and~$X \in \Cell(P)_{n}$ such that
  \[
    X = F_1[\gen{x_1}] \pcomp_{n-1} \cdots \pcomp_{n-1} F_k[\gen{x_k}]
  \]
  for some~$k \in \N$,~$x_1,\ldots,x_k \in P_{n}$ and $(n{-}1)$\context
  classes~$F_1,\ldots,F_k$ of~$\Cell(P)$, we have
  \begin{enumerate}[label=(\roman*),ref=(\roman*)]
  \item \label{lemdec.gen.1} $X_{n}= \set{x_1,\ldots,x_k}$,
    \ndr{on a perdu une prop dans le transfert. pê la remettre plus tard}
  \item \label{lemdec.gen.2} for~$i,j \in \N^*_k$ with $i \neq j$, we have~$x_i
    \neq x_j$,
    
  \item \label{lemdec.gen.3} the function~$p \mapsto x_p$ of type~$\N^*_k \to
    X_{n}$ is a linear extension of~$(X_{n},\tl_{X_{n}})$.
  \end{enumerate}
  In particular, if
  $
    X = F'_1[\gen{y_1}] \pcomp_{n-1} \cdots \pcomp_{n-1} F'_{l}[\gen{y_l}]
  $
  for some~$l \in \N$,~$y_1,\ldots,y_l \in P_{n}$ and $(n{-}1)$\context classes~$F'_1,\ldots,F'_l$, then~$k = l$ and
  $
    \set{x_1,\ldots,x_k} = \set{y_1,\ldots,y_{l}}
$.
\end{prop}
\begin{proof}
  \ndr{preuve effacée ici à redéplacer plus tard si nécessaire}

  Given~$m < n$,~$x \in P_n$ and an $m$\context class~$F$ of type~$\gen{x}$, by
  a simple induction on~$m$, one can prove that~$(F[\gen{x}])_{n} = \set{x}$.
  Thus, by definition of~$\comp_{n-1}$, we have~$X_n = \set{x_1,\ldots,x_k}$, so
  \ref{lemdec.gen.1} holds.
  Let~$i,j \in \N^*_k$ with~$i < j$, and~$Y,Z$ be the $n$\cells defined by
  \[
    Y =
    F_1[\gen{x_1}] \pcomp_{{n-1}} \cdots \pcomp_{{n-1}} F_i[\gen{x_i}] \qtand Z =
    F_{i+1}[\gen{x_{i+1}}] \pcomp_{{n-1}} \cdots \pcomp_{{n-1}} F_k[\gen{x_k}].
  \]
  Then~$x_i \in Y_{n}$,~$x_j \in Z_{n}$ and~$Y$,$Z$ are $(n{-}1)$\composable.
  By \Lemr{cell-comp-n-1},~${Y_{n} \cap Z_{n} = \emptyset}$.
  Hence,~$x_i \neq x_j$, thus \ref{lemdec.gen.2} holds. Moreover, by
  \Lemr{cell-comp-n-1} again,~$(Y_{n})^- \cap (Z_{n})^+ =
  \emptyset$, so that~${\neg (x_j \tl^1_{X_{n}} x_i)}$. Thus, by contrapositive,
  given~$i,j \in \N^*_k$ such that~$x_{i} \tl^1_{X_{n}} x_{j}$, we have~$i \le
  j$, and in fact~$i < j$ by \forestaxiom{1}. Since~$\tl_{X_{n}}$ is the
  transitive closure of~$\tl^1_{X_{n}}$, given~$i,j \in \N^*_k$, we have
  that~${x_i \tl_{X_{n}} x_j}$ implies~$i < j$, so the function~$p \mapsto
  x_p$ is a linear extension of~$(X_{n},\tl_{X_{n}})$, which concludes the
  proof of~\ref{lemdec.gen.3}.
\end{proof}

\subsection{Freeness for torsion-free complexes}
\label{ssec:freeness-proof}

In this section, we give a proof to \Cref{thm:ext-freeness}, which states that
the \ocat of cells of a torsion-free complex is, in each dimension, a free
extension over itself. By the characterization of the functor~$\polextp - - \co
\nCatp n \to \nCat {n+1}$ %
given in \Cref{ssec:polextp-description}, we are only left to prove that the
canonical forms
$
  F_1[x_1] \pcomp_n \cdots \pcomp_n F_p[x_p]
$
from the previous section are unique, up to the relation~$\actrel$ defined
in~\Cref{text:ctxts-ctxtcls}. We first prove the unicity of the decomposition
in the case~$p = 1$, and then handle the general case afterwards.

\noindent In this section, we write~$P$ for a torsion-free complex.

\paragraph{Freeness of decompositions of length one}

We first prove two technical lemmas on the manipulation of contexts by mutual
induction. The first states that, as long as we respect the relation~$\tl$, we
can modify the whiskers of the contexts:
\begin{lem}
  \label{lem:dep-xch}
  Let~$k,n \in \N^*$ with~$k < n$,~$\eps \in \set{-,+}$,~$g \in P_n$ and~${E =
    (L,F,R)}$ be a $k$\context of type~$\gen g$ of~$\Cell(P)$. Consider the following
  subsets of~$P_k$:
  \begin{align*}
    S &= L_k \cup R_k,
    &
    S' &= S \cup \gen{g}_{k,\eps},
    \\
    U &= \setof{ y \in S}{ y \tl_{S'} \gen{g}_{k,\eps}},
    &
    V &= \setof{y \in S}{ \gen{g}_{k,\eps}\tl_{S'} y }.
  \end{align*}
  Then, for every partition~$U'\cup V'$ of~$S$ such that~$U \subseteq U'$,~$V
  \subseteq V'$,~$U'$ is initial in~$S$ and~$V'$ is final in~$S$, there exists
  a $k$\context~$E' = (L',F',R')$ of type~$X$ such that
  \begin{align*}
    L'_k &= U',&
    R'_k &= V',&
    E &\ctxtrel_k E'.
  \end{align*}
\end{lem}

\noindent For~$k = 2$, \Cref{lem:dep-xch} states that, given~$g \in P_n$ for
some~$n > 2$ and a $2$\context~$E = (L,F,R)$ of type~$\gen g$
\Cref{fig:illustration:lem:dep-xch},~$E$ is equivalent through~$\ctxtrel_2$ to a
$2$\context~$E' = (L',F',R')$ as on the right of \Cref{fig:illustration:lem:dep-xch}.
\begin{figure}
  \centering
  \input{fig/dep-xch.tex}
  \caption{Illustration of \Cref{lem:dep-xch}}
  \label{fig:illustration:lem:dep-xch}
\end{figure}
The
second lemma gives sufficient conditions under which two composable context
classes can be decomposed in a way that allows them to be exchanged by the
relations~$\ctxtrel_k$ or~$\actrel$ defined in~\Cref{text:ctxts-ctxtcls}:
\begin{lem}
  \label{lem:linext-comp}
  Let~$k,n_1,n_2 \in \N^*$ with~$k < \min(n_1,n_2)$,~$g_1 \in P_{n_1}$,~$g_2 \in
  P_{n_2}$, and~$F_1,F_2$ be $k$\context classes of~$\Cell(P)$, of
  type~$\gen{g_1}$ and~$\gen{g_2}$ respectively, such that
  \[
    F_1[\ctgt_k(\gen{g_1})] = F_2[\csrc_k(\gen{g_2})]
    \qtand
    \gen{g_1}_{k,+} \cap \gen{g_2}_{k,-} = \emptyset.
  \]
   There exist
  $k$\context classes~$\bar F_1,\bar F_2$ of type~$\gen{g_1}$ and~$\gen{g_2}$
  respectively, such that
  \begin{itemize}
  \item either~$\bar F_1,\bar F_2$ are $(k{-}1)$\composable and
    \begin{align*}
      F_1 &= \bar F_1 \pcomp_{k-1} \bar F_2[\csrc_k(\gen{g_2})] 
      &
        F_2 &= \bar F_1[\ctgt_k(\gen{g_1})] \pcomp_{k-1} \bar F_2\zbox,
    \end{align*}
    
  \item or~$\bar F_2,\bar F_1$ are $(k{-}1)$\composable and
    \begin{align*}
      F_1 &= \bar F_2[\csrc_k(\gen{g_2})] \pcomp_{k-1} \bar F_1
      &
        F_2 &= \bar F_2 \pcomp_{k-1} \bar F_1[\ctgt_k(\gen{g_1})]\zbox.
    \end{align*}
  \end{itemize}
\end{lem}
\begin{proof}
  We prove the two lemmas by induction on~$k$. { \newcommand{\midcset}{S'}
    \newcommand{\ctxtset}{S} \newcommand{\gset}{Y} \newcommand{\gsetbis}{Z}
    \newcommand{\gvar}{x} \newcommand{\tempc}{T} \medbreak\noindent\textit{Proof
      of \Lemr{dep-xch}.}
    Let~$p = \setsize{L_k}$. Since~$U'$ is initial in~$\ctxtset$,~$U' \cap L_k$
    is initial for~$\tl_{L_k}$, so there exists a
    linear extension
    \[
      \sigma \colon \N^*_p \to (L_k,\tl_{L_k})
    \]
    such that~$\set{ i \in \N^*_p \mid \sigma(i) \in U'} = \set{1,\ldots,i_0}$
    for some~$i_0 \in \N_p$. Writing~$x_i$ for~$\sigma(i)$ for~$i \in \N^*_p$,
    by \Thmr{cell-dec},~$L$ can be decomposed as
    \[
      L = F_1[\gen{x_1}] \pcomp_{k-1} \cdots \pcomp_{k-1} F_p[\gen{x_p}]
    \]
    for some $(k{-}1)$\context classes~$F_1,\ldots,F_p$.
    For~${i \in \set{i_0 + 1,\ldots,p}}$, we aim at transferring~$F_i[x_{i}]$
    from~$L$ to~$R$ using the relation~$\ctxtrel_k$ on~$k$\contexts. If~$k = 1$,
    then~$\gen{x_1},\ldots,\gen{x_p},\csrctgt\eps_1(\gen{g})$ are
    $0$\composable, so that
    \[
      x_1 \tl_{\midcset} \cdots \tl_{\midcset} x_p \tl_{\midcset}
      \gen{g}_{1,\eps}
    \]
    which implies that~$x_1,\ldots,x_p \in U'$ and~$i_0 = p$. Thus, we can
    suppose that~$k > 1$. Assume moreover that~${i_0 < p}$. To transfer the~$F_i[x_i]$'s, our plan is to use \Lemr{linext-comp}. We only need to show
    how to do this for~${i = p}$, and then iterate this procedure for~$i \in
    \set{i_0+1,\ldots,p-1}$.

    Note that~$F_p[\ctgt_{k-1}(\gen {x_p})] = F[\csrc_{k-1}(\gen{g})]$. Moreover, since~$x_p
    \notin U'$, we have~$x_p \notin U$, so that
    \[
      \gen{x_p}_{k-1,+} \cap \gen{g}_{k-1,-} = \emptyset\zbox.
    \]
    Thus, using \Lemr{linext-comp} inductively, we get $(k{-}1)$\context
    classes~$\bar F_p$ and~$\bar F$ of type~$\gen{x_p}$ and~$\gen{g}$ such
    that
    \begin{itemize}
    \item either~$\bar F_p,\bar F$ are $(k{-}2)$\composable and
      \begin{align*}
        F_p &= \bar F_p \pcomp_{k-2} \bar F[\csrc_{k-1}(\gen{g})] 
        &
          F &= \bar F_p[\ctgt_{k-1}(\gen{x_p})] \pcomp_{k-2} \bar F
      \end{align*}
      
    \item or~$\bar F,\bar F_p$ are $(k{-}2)$\composable and
      \begin{align*}
        F_p &= \bar F[\csrc_{k-1}(\gen{g})] \pcomp_{k-2} \bar F_p
        &
          F &= \bar F \pcomp_{k-2} \bar F_p[\ctgt_{k-1}(\gen{x_p})]
      \end{align*}
    \end{itemize}
    By symmetry, we can suppose that we are in the first situation. Then, by
    axiom~\ref{cond:ctxtrel-left} of~$\ctxtrel_k$, we get that~${E \ctxtrel_k \tilde E}$
    where~$\tilde E = (\tilde L,\tilde F,\tilde R)$ is such that
    \begin{align*}
      \tilde L &= F_1[\gen{x_1}] \pcomp_{k-1} \cdots \pcomp_{k-1} F_{p-1}[\gen{x_{p-1}}]
               \\
      \tilde F &= \bar F_p[\csrc_{k-1}(\gen{x_p})] \pcomp_{k-2} \bar F
               \\
      \tilde R &= (\bar F_p[\gen{x_p}] \pcomp_{k-2} \bar F[\ctgt_{k-1}(\gen{g})]) \pcomp_{k-1} R
    \end{align*}
    By iterating the above procedure for~$i \in \set{i_0+1,\ldots,p-1}$, we obtain
    a $k$\context~$E' = (L',F',R')$ of type~$\gen{g}$ such that
    \[
      E \ctxtrel_k E'
      \qquad
      L'_k = L_k \cap U'
      \qquad
      R'_k = R_k \cup (L_k \setminus U')\zbox.
    \]
    Using a similar method to transfer elements from~$R'$ to~$L'$, we get a
    $k$\context~$E'' = (L'',F'',R'')$ of type~$\gen{g}$ such that
    \[
      E' \ctxtrel_k E'' \qquad L''_k = L'_k \cup (R'_k \setminus V') \qquad
      R''_k = R'_k \cap V'\zbox.
    \]
    Then, we have~$E \ctxtrel_k E''$ and we compute that
    \begin{align*}
      L''_k &=  L'_k \cup (R'_k \setminus V') \\
            &= (L_k \cap U') \cup (R_k \setminus V') \cup (L_k \setminus (U' \cup V')) \\
            &= (L_k \cap U') \cup (R_k \cap U') & \text{(since~$L_k \cup R_k = U' \cup V'$)} \\
            &= U'
    \end{align*}
    and, similarly,~$R''_k = V'$. Thus,~$E''$ satisfies the wanted properties.
  }%

  {%
    \newcommand{\midcell}{M}%
    \newcommand{\midcset}{S'}%
    \newcommand{\ctxtset}{S}%
    \newcommand{\beforeset}{U}%
    \newcommand{\afterset}{V}%
    \newcommand{\partset}{Q}%
    \newcommand{\cumulset}{R}%
\medbreak\noindent\textit{Proof of \Lemr{linext-comp}.} Let~$E_k =
(L^k,F'_k,R^k)$ be such that~$\ctxtcl {E_k} = F_k$ for~$k \in \set{1,2}$.
Consider
\begin{align*}
  \midcell & =F_1[\ctgt_k(\gen{g_1})] && \text{(or, equivalently,~$F_2[\csrc_k(\gen{g_2})]$)}, \\
  \ctxtset_i & = L^i_k \cup R^i_k && \text{for~$i \in \set{1,2}$},\\
  \midcset & = \midcell_k,\\
  \beforeset_1 & = \set{ x \in \ctxtset_1 \mid x \tl_{\midcset}
                 \gen{g_1}_{k,+}}&
  \afterset_1 & = \{x \in \ctxtset_1 \mid \gen{g_1}_{k,+} \tl_{\midcset}
                x\}\\ 
  \beforeset_2 & = \set{ x \in \ctxtset_2 \mid x \tl_{\midcset}
                 \gen{g_2}_{k,-}} &
  \afterset_2 & = \{x \in \ctxtset_2 \mid \gen{g_2}_{k,-} \tl_{\midcset}
                x\}
\end{align*}
Since, by \forestaxiom{4},~${g_1}$ and~${g_2}$ are not in torsion with respect
to~$F_1[\ctgt_k(\gen{g_1})]$, we have
\[
  \text{either\quad$\neg (\gen{g_1}_{k,+} \tl_{\midcset} \gen{g_2}_{k,-})$\quad or\quad$\neg(\gen{g_2}_{k,-} \tl_{\midcset} \gen{g_1}_{k,+})$.}
\]
By symmetry, we can suppose that~$\neg (\gen{g_2}_{k,-} \tl_{\midcset} \gen{g_1}_{k,+})$. Since we
can use \Lemr{dep-xch} (which is proved for the current value of~$k$) to change~$E_1$ and~$E_2$, we can suppose that
\begin{align*}
  L^1_k & = \beforeset_1, 
  &
    R^1_k & = \ctxtset_1 \setminus \beforeset_1, \\
  L^2_k & = \ctxtset_2 \setminus \afterset_2,
  & R^2_k & = \afterset_2.
\end{align*}
Then,
\[
  (\beforeset_1 \cup \gen{g_1}_{k,+}) \cap (\gen{g_2}_{k,-} \cup \afterset_2) = \emptyset
\]
since, otherwise, it would contradict the condition~$\neg (\gen{g_2}_{k,-}
\tl_{\midcset} \gen{g_1}_{k,+})$. Consider the following sets:
\abovelongtoshortskip
\begin{align*}
  \partset_1 &= \beforeset_1, & \partset_2 &= \gen{g_1}_{k,+}, \\
  \partset_3 &= \zbox{$\midcset \setminus (\beforeset_1 \cup \gen{g_1}_{k,+} \cup 
               \gen{g_2}_{k,-} \cup \afterset_2)$,} \\
  \partset_4 &= \gen{g_2}_{k,-}, &
                                   \partset_5 &= \afterset_2.
\end{align*}
Then~$\partset_1,\partset_2,\partset_3,\partset_4,\partset_5$ form a partition of~$\midcset$. Moreover, this partition is compatible with~$\tl_{\midcset}$. Indeed,
given~$x,y \in \midcset$ such that~$x \tl_{\midcset} y$,
\begin{itemize}
\item if~$x \in \partset_2$, then we can not have~$y \in \partset_1$ since, by \forestaxiom{3},~$\gen{g_1}_{k,+}$ is a segment for~$\tl_{\midcset}$,
  
\item if~$x \in \partset_3$, then we can not have~$y \in \partset_1 \cup
  \partset_2$ (otherwise, we would have~$x \in \beforeset_1 \cup
  \gen{g_1}_{k,+}$),
  
\item if~$x \in \partset_4$, then either~$y \in \partset_4$ or~$y \in
  \partset_5$ by definition of~$\partset_5$,
  
\item if~$x \in \partset_5$, then~$y \in \partset_5$ since, by \forestaxiom{3},~$\gen{g_2}_{k,-}$ is a segment for~$\tl_{\midcset}$.
\end{itemize}
Thus, there exists a linear extension for~$(\midcset,\tl_{\midcset})$
\[
  \sigma \colon \N_{\midcset} \to \midcset
\]
such that, for~$i,j \in \N_{\setsize\midcset}$ and~$r,s \in \N^*_5$,
if~$\sigma(i) \in \partset_r$ and~$\sigma(j) \in \partset_s$ with~$r < s$,
then~$i < j$. Since~${\midcset = \midcell_k}$, using \Thmr{cell-dec},~$\midcell$
can be written
\[
  \midcell = \prod_{i = 1}^{\setsize{\midcset}} F_i [ \gen{\sigma(i)} ]
\]
for some $(k{-}1)$\context classes~$F_1,\ldots,F_{\setsize\midcset}$. By
gathering the terms corresponding to~$\partset_1,\ldots,\partset_5$
respectively, we obtain five
$k$\cells~$\midcell^1,\midcell^2,\midcell^3,\midcell^4$,~$\midcell^5 \in
\Cell(P)_k$ where
\[
  \midcell^j = \prod_{i \in \sigma^{-1}(\partset_j)} F_i[\gen{\sigma(i)}]
\]
as in \Cref{fig:midcell-dec} and such that
\begin{figure}
  \centering
  \[
    \begin{tikzpicture}[scale=0.4,rotate=-90]
      \def\orad{5}
      \def\irad{1}
      \def\ampl{40}
      \path (0,0) coordinate (center) -- +(\orad/3 + \irad/3,0) coordinate (centerR) -- +(-\orad/3 - \irad/3,0) coordinate (centerL);
      \draw (center) circle (\orad);
      \draw (centerL) circle (\irad);
      \draw (centerR) circle (\irad);
      \draw ($(centerL) + (90:\irad)$) -- ($(center) + (180 - \ampl:\orad)$);
      \draw ($(centerL) + (-90:\irad)$) -- ($(center) + (-180 + \ampl:\orad)$);
      \draw ($(centerR) + (90:\irad)$) -- ($(center) + (\ampl:\orad)$);
      \draw ($(centerR) + (-90:\irad)$) -- ($(center) + (-\ampl:\orad)$);
      \node at ($(center) + (-2/3*\orad -2/3*\irad,0)$) {$M^1$};
      \node at ($(centerL)$) {$M^2$};
      \node at ($(center)$) {$M^3$};
      \node at ($(centerR)$) {$M^4$};
      \node at ($(center) + (2/3*\orad +2/3*\irad,0)$) {$M^5$};
    \end{tikzpicture}
  \]
  \caption{The decomposition of~$\midcell$}
  \label{fig:midcell-dec}
\end{figure}
\[
  \midcell = \midcell^1 \pcomp_{k-1} \midcell^2 \pcomp_{k-1} \midcell^3\pcomp_{k-1} \midcell^4\pcomp_{k-1} \midcell^5.
\]
\goodbreak
\noindent Since
\[
  \csrc_{k-1} (L^1) = \csrc_{k-1} (\midcell) = \csrc_{k-1} (\midcell^1)
  \qtand
  L^1_k = \beforeset_1 = \midcell^1_k,
\]
by \Lemr{cell-eq}, we have~$L^1=\midcell^1$. Moreover, since
\[
  \csrc_{k-1} (F'_1[\gen{g_1}]) =
  \ctgt_{k-1}(L^1) = \ctgt_{k-1} (\midcell^1) = \csrc_{k-1} (\midcell^2)
\]
and
\[
  (F'_1[\ctgt_k(\gen{g_1})])_k = \gen{g_1}_{k,+} = \midcell^2_k,
\]
by \Lemr{cell-eq}, it follows that
\[
  F'_1[\ctgt_k(\gen{g_1})] = \midcell^2\zbox.
\]
Similarly, we can show that
\[
  F'_2[\csrc_k(\gen{g_2})] = \midcell^4 \qtand R^2 = \midcell^5.
\]
Moreover, since
\begin{gather*}
  \csrc_{k-1}(L^2) = \csrc_{k-1} (M) = \csrc_{k-1}(\midcell^1 \pcomp_{k-1}
  \midcell^2 \pcomp_{k-1} \midcell^3)
  \shortintertext{and}
  L^2_k   = \ctxtset_2 \setminus \afterset_2  
                     = \midcset \setminus (\gen{g_2}_{k,-} \cup \afterset_2) 
                    = \partset_1 \cup \partset_2 \cup \partset_3\zbox,
\end{gather*}
by \Lemr{cell-eq}, we have
\[
  L^2 = \midcell^1 \pcomp_{k-1} \midcell^2 \pcomp_{k-1} \midcell^3.
\]
Similarly, we have
\[
  R^1 = \midcell^3 \pcomp_{k-1} \midcell^4 \pcomp_{k-1} \midcell^5.
\]
Hence, writing
\[
  \bar F_1 = \ctxtcl {(L^1,F'_1,\unitp k {F'_1[\ctgt_{k-1}(\gen{g_1})]})}
  \qtand
  \bar F_2 = \ctxtcl{(M^3,F'_2,R^2)}
\]
we have
$F_1 = \bar F_1 \pcomp_{k-1} \bar F_2[\csrc_k(g_2)]$  and
$F_2 = \bar F_1[\ctgt_k(g_1)]
  \pcomp_{k-1} \bar F_2$
as wanted.}%
\end{proof}
\noindent We deduce that applied context classes are
completely determined by their sources (or targets):
\begin{theo}
  \label{thm:ext-env}
  Given~$k,n \in \N$ with~$k < n$,~$g \in P_n$ and $k$\context classes~$F_1,F_2$
  of type~$\gen{g}$ such that
  \[
    \csrc_k (F_1[\gen{g}]) = \csrc_k (F_2[\gen{g}])
    \qtext{or}
    \ctgt_k (F_1[\gen{g}]) = \ctgt_k (F_2[\gen{g}]),
  \]
  we have~$F_1 = F_2$.
\end{theo}
\begin{proof}
  \newcommand{\midcell}{M}
  \newcommand{\midcset}{{S'}}
  \newcommand{\ctxtset}{S}
  \newcommand{\beforeset}{U}
  \newcommand{\afterset}{V}
  By symmetry, it is enough to prove the case where~$\csrc_k (F_1[\gen{g}]) =
  \csrc_k (F_2[\gen{g}])$. We prove this property by an induction on~$k$. If~$k =
  0$, the result is trivial. So suppose that~$k > 0$. Let
  \[
    E_1 = (L^1,F'_1,R^1)
    \qtand
    E_2 = (L^2,F'_2,R^2)
  \]
  be $k$\contexts such that~$F_i = \ctxtcl{E_i}$ for~$i \in \set{1,2}$. Thus,
  \[
    L^1 \pcomp_{k-1} F'[\csrc_k(\gen g)] \pcomp_{k-1}R^1 =L^2 \pcomp_{k-1} F'[\csrc_k(\gen
    g)] \pcomp_{k-1}R^2
  \]
  In particular,~$L^1_k \cup \gen g_{k,-} \cup R^1_k = L^2_k \cup \gen g_{k,-}
  \cup R^2$ and, by \Lemr{cell-comp-n-1}, both sides are partitions, so that
  we have~${L^1_k \cup R^1_k = L^2_k \cup R^2_k}$.
  Consider the following subsets of~$P_k$:
  \begin{align*}
    \ctxtset &= L^1_k \cup R^1_k, &
    \midcset &= \ctxtset \cup \gen{g}_{k,-}, \\
    \beforeset &= \set{x \in \ctxtset \mid x \tl_{\midcset} \gen{g}_{k,-}}, &
    \afterset &= \ctxtset \setminus \beforeset.
  \end{align*}
  By \Lemr{dep-xch}, we can suppose that
  \[
    L^1_k = L^2_k = U
    \qtand
    R^1_k = R^2_k = V.
  \]
  For~$i \in \set{1,2}$, we have
  \[
    \csrc_{k-1}(L^i) = \csrc_{k-1} (F_i[\gen{g}])
    = \csrc_{k-1} \circ \csrc_k (F_i[\gen{g}])
  \]
  so that~$\csrc_{k-1}(L^1) = \csrc_{k-1}(L^2)$. Thus, by \Lemr{cell-eq}, we
  have
  $
    L^1 = L^2
  $
  and, by a similar argument,~$R^1 = R^2$.
  Moreover, for~$i \in \set{1,2}$,~$\ctgt_{k-1} (L^i) = \csrc_{k-1}
  (F'_i[\gen{g}])$,
  so
  \[
    \csrc_{k-1} (F'_1[\gen{g}]) = \csrc_{k-1} (F'_2[\gen{g}]).
  \]
  By induction hypothesis, we have~$F'_1 = F'_2$. Hence,~$F_1 = F_2$.
\end{proof}


\paragraph{Freeness of general decompositions}
\parlabel{text:freeness-proof}

We now handle the case of general decompositions of arbitrary lengths. First, we
show an analogous of~\Cref{thm:cell-dec}, \ie that the decompositions
in~$\defcellext$ can also be reordered by linear extensions:
\begin{lem}
  \label{lem:dep-comp}
  Let~$n \in \N$ and~$X$ be an $(n{+}1)$\cell of~$\defcellext$ such that
  \[
    X = F_1[\gene{x_1}] \pcomp_{n} \cdots \pcomp_{n} F_p[\gene{x_p}]
  \]
  for some~$p \in \N$,~$x_1,\ldots,x_p \in P_{n+1}$ and $n$\context
  classes~$F_1,\ldots,F_p$ of~$\Cell(P)$. Then, we have that the function~$q
  \mapsto x_q$ of type~$\N^*_q \to X_{n+1}$ is a linear extension
  of~$(X_{n+1},\tl_{X_{n+1}})$. Moreover, if~$\sigma$ is a linear extension of~$(X_{n+1},\tl_{X_{n+1}})$, then there exist $n$\context classes~$\bar
  F_1,\ldots,\bar F_p$ of respective types~${\gen{\sigma(1)},\ldots,\gen{\sigma(p)}}$ such that
  \[
    X = \bar F_1[\gene{{\sigma(1)}}] \pcomp_n \cdots \pcomp_n \bar F_p[\gene{{\sigma(p)}}]
    .
  \]
\end{lem}
\begin{proof}
  Write~$\rho \co \N_p \to X_{n+1}$ for the function such that
  $
    \rho(i) = x_i
  $
  for~$i \in \N^*_p$. By the functoriality of~$\ev$, we have
  \[
    \ev(X) = F_1[\gen{x_1}] \pcomp_n \cdots \pcomp_n F_p[\gen{x_p}]
  \]
  so that~$\rho$ is a linear extension by \Cref{prop:lemdec.gen}. We are left to
  prove the second part of the statement. We have a morphism of linear extensions
  \[
    f = \finv \sigma \circ \rho
  \]
  between~$\sigma$ and~$\rho$. By \Lemr{linext-dec}, we can suppose that~$f =
  \tau_{i,i+1}$ for some~${i \in \N^*_{p-1}}$. To conclude, we only need to show
  that~$\gene{x_i}$ and~$\gene{x_{i+1}}$ can be swapped in the decomposition
  of~$X$ as~${F_1[\gene{x_1}] \pcomp_n \cdots \pcomp_n F_p[\gene{x_p}]}$. By
  contradiction, suppose that~$\gen{x_i}_{n,+} \cap \gen{x_{i+1}}_{n,-} \neq
  \emptyset$. In particular, we have~${\rho(i) \tl_{X_{n+1}} \rho(i+1)}$.
  Since~$\rho = \sigma \circ \tau_{i,i+1}$, it implies~${\sigma(i+1)
    \tl_{X_{n+1}} \sigma(i)}$. Thus, since~$\sigma$ is a linear extension, we
  deduce that~$i+1 < i$, which is a contradiction. So~${\gen{x_i}_{n,+} \cap
    \gen{x_{i+1}}_{n,-} = \emptyset}$. By~\Lemr{linext-comp}, there exist
  $n$\context classes~$\bar F_i$ and~$\bar F_{i+1}$ such that, in~$\acttocat
  {\restrict n {\Cell(P)}}{P_{n+1}}$,
  \begin{gather*}
    \cseq{(x_i,F_i),(x_{i+1},F_{i+1})} \actrel \cseq{(x_{i+1},\bar F_{i}),(x_i,\bar F_{i+1})}
    \shortintertext{so that}
    \cseq{(x_1,F_1), \ldots, (x_p,F_p)}\\
    \llap{$\actrel\;$}
    \cseq{(x_1,F_1),\ldots,(x_{i-1},F_{i-1}),(x_{i+1},\bar F_i),(x_i,\bar F_{i+1}),(x_{i+2},F_{i+2}),\ldots,(x_p,F_p)}
  \end{gather*}
  \ie in~$\defcellext$,
  \[
    X = F_1[\gene{x_1}] \pcomp_n \cdots \pcomp_n F_{i-1}[\gene{x_{i-1}}] \pcomp_n
    \bar F_{i}[\gene{x_{i+1}}] \pcomp_n \bar F_{i+1}[\gene{x_{i}}] \pcomp_n
    F_{i+2}[\gene{x_{i+2}}]
    \pcomp_n \cdots \pcomp_n
    F_p[\gene{x_p}]
  \]
  which concludes the proof.
\end{proof}
\noindent We can now deduce that~$\restrict {n+1} {\Cell(P)}$ is canonically a
free extension on~$\restrict n {\Cell(P)}$:
\ifx\theoremextfreeness\undefined
\begin{theo}
  Given a torsion-free complex~$P$, for~$n \in \N$, the
  $(n{+}1)$\functor~$\ev^n$ is an isomorphism between $\defcellext$ and
  $\restrict {n+1}{\Cell(P)}$.
\end{theo}
\else
  \theoremextfreeness*
\fi
\begin{proof}
  Since~$\restrict n \ev = \unit {\restrict n {\Cell(P)}}$, it is enough to
  prove that~$\ev$ induces a bijection on the $(n{+}1)$\cells. By
  \Thmr{cell-dec}, it is surjective, so we are left to prove injectivity.
  Let~$X^1$ and~$X^2$ be $(n{+}1)$\cells of~$\defcellext$, such that~${\ev(X^1)
    = \ev(X^2)}$ and
  \[
    X^i = F^i_1[\gene {x^i_1}] \pcomp_{n} \cdots \pcomp_{n} F^i_{p_i}[\gene{x^i_{p_i}}]
  \]
  for some~$p_i \in \N$,~$x^i_1,\ldots,x^i_{p_i} \in P_{n+1}$ and $n$\context
  classes~$F^i_1,\ldots,F^i_{p_i}$ for~$i \in \set{1,2}$. By functoriality
  of~$\ev$, we have
  \[
    \ev(X^i) = F^i_1[\gen {x^i_1}] \pcomp_{n} \cdots \pcomp_{n} F^i_{p_i}[\gen{x^i_{p_i}}]
  \]
  for~$i \in \set{1,2}$, so that, by \Cref{prop:lemdec.gen}, we have~$p_1 = p_2$, and
  we write~$p$ for the common value. Moreover,
  $
    \set{x^1_1,\ldots,x^1_p} = \set{x^2_1,\ldots,x^2_p}
  $.
  By \Lemr{dep-comp}, we can suppose that~$x^1_j = x^2_j$ for~$j \in \N^*_p$,
  and we write~$x_j$ for the common value. Since~${\csrc_n(X^i) =
  \csrc_n(F^i_1[x_1])}$ for~$i \in \set{1,2}$, we have
  \[
    \csrc_n(F^1_1[x_1]) = \csrc_n(F^2_1[x_1])
  \]
  so that, by \Thmr{ext-env},~$F^1_1 = F^2_1$. In particular,~$\ctgt_n(F^1_1[x_1]) =
  \ctgt_n(F^2_1[x_1])$, so that
  \[
    \csrc_n(F^1_2[\gene {x_2}] \pcomp_{n} \cdots \pcomp_{n} F^1_{p}[\gene{x_{p}}])
    =
    \csrc_n(F^2_2[\gene {x_2}] \pcomp_{n} \cdots \pcomp_{n} F^2_{p}[\gene{x_{p}}])
    \zbox.
  \]
  Thus, we can iterate the above procedure to show that~$F^1_j = F^2_j$ for~$j
  \in \set{1,\ldots,p}$, so that we get~$X^1 = X^2$. Hence, the $(n{+}1)$\functor~$\ev$
  is an isomorphism.
\end{proof}


%% file: fig/dep-xch.tex
\[
  \begin{tikzpicture}[scale=0.6,baseline=(current bounding box.center),rotate=90,xscale=.9,yscale=.9]
    \draw[decorate,decoration={brace,amplitude=4,mirror}] (3,3.3) -- node[left]{$L_2\ $} (0.2,3.3);
    \draw[decorate,decoration={brace,amplitude=4}] (-3,3.3) -- node[left]{$R_2\ $} (-.2,3.3);
    \clip (0,0) circle (3.2);
    \coordinate (centrez) at (0,0);
    \coordinate (dn-2s) at (0,3);
    \coordinate (dn-2t) at (0,-3);
    \coordinate (zn-2s) at ($ (centrez) + (0,0.75) $);
    \coordinate (zn-2t) at ($ (centrez) + (0,-0.75) $);
    \node[circle,inner sep=0pt] (dn-2sn) at (dn-2s) {$.$};
    \node[circle,inner sep=0pt] (dn-2tn) at (dn-2t) {$.$};
    \node[circle,inner sep=0pt] (zn-2sn) at (zn-2s) {$.$};
    \node[circle,inner sep=0pt] (zn-2tn) at (zn-2t) {$.$};
    \node at (centrez) {$\gene{g}$};
    \draw[->] (95:3) arc (95:265:3);
    \draw[->] (85:3) arc (85:-85:3);
    \draw[->] (centrez) + (100:0.75) arc (100:260:0.75);
    \draw[->] (centrez) + (80:0.75) arc (80:-80:0.75);
    \draw[bend right,->] (dn-2sn) to[bend left=60] (zn-2sn.north west);
    \draw[bend right,->] (zn-2tn.south east) to[bend right=60] (dn-2tn);
    \draw (zn-2sn) -- ($ (0,0) + (150:3) $);
    \draw (zn-2sn) -- ($ (0,0) + (30:3) $);
    \draw (zn-2tn) -- ($ (0,0) + (-150:3) $);
    \draw (zn-2tn) -- ($ (0,0) + (-30:3) $);
    \node (flecheS) at (-1.875,0) {$\Downarrow V$};
    \node (flecheT) at (1.875,0) {$\Downarrow U$};
  \end{tikzpicture}
  \quad\ctxtrel_2\quad
  \begin{tikzpicture}[scale=0.6,baseline=(current bounding box.center),rotate=90,xscale=.9,yscale=.9]
    \draw[decorate,decoration={brace,amplitude=4}] (3,-3.3) -- node[right]{\ $L'_2$} (.2,-3.3);
    \draw[decorate,decoration={brace,amplitude=4,mirror}] (-3,-3.3) -- node[right]{\ $R'_2$} (-.2,-3.3);
    \clip (0,0) circle (3.2);
    \coordinate (centrez) at (0,0);
    \coordinate (dn-2s) at (0,3);
    \coordinate (dn-2t) at (0,-3);
    \coordinate (zn-2s) at ($ (centrez) + (0,0.75) $);
    \coordinate (zn-2t) at ($ (centrez) + (0,-0.75) $);
    \node[circle,inner sep=0pt] (dn-2sn) at (dn-2s) {$.$};
    \node[circle,inner sep=0pt] (dn-2tn) at (dn-2t) {$.$};
    \node[circle,inner sep=0pt] (zn-2sn) at (zn-2s) {$.$};
    \node[circle,inner sep=0pt] (zn-2tn) at (zn-2t) {$.$};
    \node at (centrez) {$\gene{g}$};
    \draw[->] (95:3) arc (95:265:3);
    \draw[->] (85:3) arc (85:-85:3);
    \draw[->] (centrez) + (100:0.75) arc (100:260:0.75);
    \draw[->] (centrez) + (80:0.75) arc (80:-80:0.75);
    \draw[bend right,->] (dn-2sn) to[bend right=60] (zn-2sn.south west);
    \draw[bend right,->] (zn-2tn.north east) to[bend left=60] (dn-2tn);
    \draw (zn-2sn) -- ($ (0,0) + (150:3) $);
    \draw (zn-2sn) -- ($ (0,0) + (30:3) $);
    \draw (zn-2tn) -- ($ (0,0) + (-150:3) $);
    \draw (zn-2tn) -- ($ (0,0) + (-30:3) $);
    \node (flecheS) at (-1.875,0) {$\Downarrow V$};
    \node (flecheT) at (1.875,0) {$\Downarrow U$};
  \end{tikzpicture}
\]